\documentclass[a4paper, 12pt, titlepage]{report}
\usepackage[utf8]{inputenc}
\usepackage{amsmath}
\usepackage{amssymb}
\usepackage{amsthm}
\usepackage[small]{titlesec}
\usepackage{graphicx}
\usepackage{standalone}

\theoremstyle{plain}
\newtheorem{thm}{Theorem}[section]

\newtheorem{cor}[thm]{Corollary}
\newtheorem{lem}[thm]{Lemma}
\newtheorem{prop}[thm]{Proposition}
\newtheorem{res}[thm]{Result}
\newtheorem{prob}[thm]{Problem}

\theoremstyle{definition}
\newtheorem{defn}{Definition}[chapter]
\newtheorem{exmp}[defn]{Example}
\newtheorem{cons}[defn]{Construction}
\newtheorem{remark}[defn]{Remark}

\begin{document}
\begin{titlepage}

 \begin{figure}[h]
 \centering
  \includegraphics[scale=0.5]{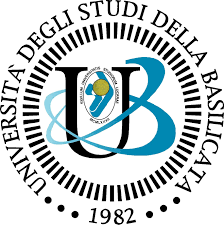}
  \includegraphics[scale=0.3]{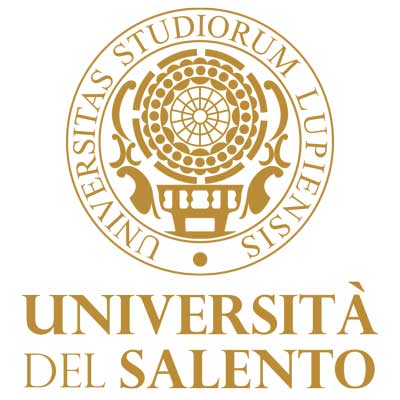}
 \end{figure}

 \begin{center}
  {{\textsc{UNIVERSIT\`{A} DEGLI STUDI DELLA BASILICATA UNIVERSIT\`{A} DEL SALENTO}}} \rule[0.1cm]{11.1cm}{0.1mm}
  \rule[0.5cm]{11.1cm}{0.6mm}
  {\small{\bf DOTTORATO DI RICERCA IN MATEMATICA ED INFORMATICA\\
  XXXV CICLO\\
  Coordinatore Prof. Angela Anna Albanese}}
 \end{center}

 \vspace{10mm}
 \begin{center}
{\LARGE{\bf On Geometry and Combinatorics}}\\
\vspace{3mm}
{\LARGE{\bf of Finite Classical Polar Spaces }}\\

\vspace{10mm} {\large{\bf Settore Scientifico Disciplinare MAT/03}}
\end{center}
\vspace{25mm}
\par
\noindent
\begin{minipage}[t]{0.47\textwidth}
{\large{\bf Tutor:\\
Prof. G\'{a}bor Korchm\'{a}ros\\
Prof. Antonio Cossidente\\}}\\
\\
\\

\end{minipage}
\hfill
\begin{minipage}[t]{0.47\textwidth}\raggedleft
{\large{\bf Dottorando:\\
Valentino Smaldore\\}}
\end{minipage}

\end{titlepage}

\newpage

\cleardoublepage
\newpage
 \textit{Il modello non \`{e} la sfera, che non \`{e} superiore alle parti, dove ogni punto \`{e} equidistante dal centro e non vi sono differenze tra un punto e l’altro. Il modello \`{e} il poliedro, che riflette la confluenza di tutte le parzialit\`{a} che in esso mantengono la loro originalit\`{a}. Sia l’azione pastorale sia l’azione politica cercano di raccogliere in tale poliedro il meglio di ciascuno. L\`{i} sono inseriti i poveri, con la loro cultura, i loro progetti e le loro proprie potenzialit\`{a}. Persino le persone che possono essere criticate per i loro errori, hanno qualcosa da apportare che non deve andare perduto. \`{E} l’unione dei popoli, che, nell’ordine universale, conservano la loro peculiarit\`{a}; \`{e} la totalit\`{a} delle persone in una societ\`{a} che cerca un bene comune che veramente incorpora tutti.}
 \vspace{0.3 cm}\\
 \small{Papa Francesco, Evangelii Gaudium, 236.}
 \vspace{1.5 cm}
 \begin{figure}[h]
  \centering
  \includegraphics[scale=2]{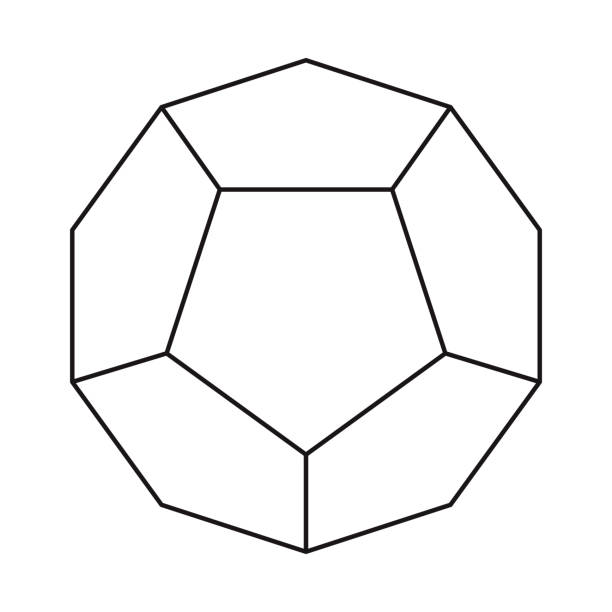}
 \end{figure}

 \newpage

\cleardoublepage
\newpage
\chapter*{Acknowledgements}

\tableofcontents

\chapter*{Preface}
\normalsize
Polar spaces over finite fields are fundamental in combinatorial geometry.  The concept of polar space was firstly introduced by F. Veldkamp who gave a system of 10 axioms in the spirit of Universal Algebra, \cite{Vel1,Vel2}. Later the axioms were simplified by J. Tits, who introduced the concept of \textit{subspaces}, \cite[Chapters 7-9]{buildings}. Later on, from the point of view of incidence geometry, axioms of polar spaces were also given by F. Buekenhout and E. Shult in 1974, \cite{BSh}. The reader can find the three systems of axioms of polar spaces in Appendix \ref{ap0}.
Examples of polar spaces are the so called \textit{finite classical polar spaces}, i.e. incidence structures arising from quadrics, symplectic forms and Hermitian forms, and in \cite{Birkhoff} it is shown that finite classical polar spaces are in correspondance with reflexive sesquilinear forms.
It is still an open problem to show whether or not classical polar spaces are the only example of finite polar spaces.

Nowadays, some research problems related to finite classical polar space are:
\begin{itemize}
 \item existence of spreads and ovoids;
 \item existence of regular systems and $m$-ovoids;
 \item upper or lower bounds on partial spreads and partial ovoids.
\end{itemize}
Moreover, polar spaces are in relation with combinatorial objects as regular graphs, block designs and association schemes. In this Ph.D. Thesis we investigate the geometry of finite classical polar spaces, giving contributions to the above problems.

When I was writing my Master Degree thesis, my supervisors G\'{a}bor Korchm\'{a}ros and Angelo Sonnino, proposed me to work on Hermitian polar spaces, starting from the seminal paper \cite{Segre} \textit{Forme e geometrie hermitiane, con paricolare riguardo al caso finito}. The thesis, which had title \textit{Variet\`{a} hermitiane sopra campi finiti}, ended with some research open problems, related to the search of hemisystems on the Hermitian surface $H(3,q^2)$. Starting my Ph.D. Professor Korchm\'{a}ros was one of my supervisors, and naturally the first research problem given to me was the hemisystem problem. At the same time, I started working with the co-supervisor Antonio Cossidente, and with the other members of the research group Giuseppe Marino and Francesco Pavese. With Francesco Pavese I had also the great chance to study spectral graph problems related to polar spaces, and I went on this directions during my abroad period in Brussel, with the research group of Jan De Beule.

The thesis is organized as follows.
Part \ref{pI} is more focused on the geometric aspects of polar spaces, while in Part \ref{pII} some combinatorial objects are introduced such as regular graphs, association schemes and combinatorial designs.
\begin{itemize}
 \item In Chapter \ref{ch1} it is given a series of preliminaries and classical results about finite fields (Sections \ref{sec11}, \ref{sec12}), finite vector and projective spaces (Section \ref{sec13}) and finite classical polar spaces (Sections \ref{sec14}, \ref{sec15}). Moreover, in Section \ref{sec16} they are given examples of polar spaces of small rank.
\item In Chapter \ref{ch2} we investigate regular systems in other polar spaces, as in the papers \cite{VS2, VS4}, joint work with A. Cossidente, G. Marino and F. Pavese and joint work with V. Pallozzi Lavorante, respectively. Sections \ref{sec21}-\ref{sec22} introduce regular systems starting from the concepts of spread of projective spaces and polar spaces, while in Section \ref{sec26} it is given the first construction of hemisystem on the Hermitian surface due to B. Segre, in the case $H(3,9)$. New results on regular systems are listed in Sections \ref{sec23}-\ref{sec25} ad \ref{sec27}-\ref{sec28}: in Section \ref{sec23} we find hemisystems of elliptic quadrics by partitioning generators of an elliptic quadric into generators of hyperbolic quadrics embedded in it; in Section \ref{sec24} we give a construction of regular systems by means of a $k$-system, in hyperbolic quadrics and in the parabolic quadric $Q(6,3)$; in Section \ref{sec25} we find regular systems of hyperbolic and elliptic quadrics arising from the field reduction map; sections \ref{sec27}-\ref{sec28} present the result contained in \cite{VS4} in which we extend a well-known construction of hemisystem on the Hermitian surface made by G. Korchm\'{a}ros, G. P. Nagy and P. Speziali, finding a new family of hemisystems.
\item In Chapter \ref{ch3} we give the definition of the dual of a spread: the ovoid. All results presented on new family of ovoids, belong to the paper \cite{VS5}, joint work with M. Ceria, J. De Beule and F. Pavese. In Section \ref{sec31} we describe a new families of partial ovoids of the symplectic spaces $W(3,q)$ (or dually, a new family of spreads of the parabolic quadric $Q(4,q)$) and $W(5,q)$. Introducing tangent sets, new and old results on the symplectic spaces are used in Section \ref{sec32} to find examples of partial ovoids on the Hermitian spaces $H(2n,q^{2})$, $n\in\{2,3,4\}$.
\item In Chapter \ref{ch4} we introduce distance regular graphs and association schemes. In Section \ref{sec41} we give some basic tools of spectral graph theory, while in Sections \ref{sec42}-\ref{sec45} we analyze association schemes and their matrix algebra. Finally, in Section \ref{sec46} we list some results of \cite{VS2} on regular systems, arising from the distance regular graph $\mathcal{D}^{i}_{\mathcal{P}_{d,e}}$ and its related association scheme.
\item In Chapter \ref{ch5} we introduce strongly regular graphs and we give examples arising from polar spaces: from the hemisystem in \cite{VS4} it arise the linear representation graph of an $m$-ovoid and the related two-weight linear codes (Section \ref{sec51}), and the graph on the hemisystem's lines (Section \ref{sec52}); from the ovoids found in \cite{VS5} we get results on the collinearity graph of $W(3,q)$ and $H(4,q^2)$ (Section \ref{sec53}).
\item In Chapter \ref{ch6} we define the tangent graph $NU(n+1,q^2)$, and we show that if $n\neq3$ it is not determined by its spectrum, following the paper \cite{VS1}, joint work with F. Ihringer and F. Pavese. In Section \ref{sec62} we analyze the case $n=2$, in which a cospectral graph is given by considering an abstract unital instead of the Hermitian curve $H(2,q^2)$, while in Section \ref{sec63} we give the result for $n>3$ by applying a switching technique on vertices of $NU(n+1,q^2)$.
\item In Chapter \ref{ch7}, starting from the results on $NU(3,q^2)$, we show that the automorphism group of the graph is $P\Gamma U(3,q)$, for $q>2$. The result is contained in the paper \cite{VS4}, joint work with F. Romaniello. Moreover, in Sections \ref{sec73}-\ref{sec74}, the result is extended considering the automorphism group of all non-classical unitals, in a work in progress with S. Adriaensen, J. De Beule and F. Romaniello.
\end{itemize}

Finally Appendix \ref{apA}, \ref{apB} and \ref{apC} are dedicated to give more details on, respectively, maximal curves, linear codes and combinatorial designs, giving useful results and definitions.

\part{Spreads and $m$-regular systems of finite classical polar spaces}\label{pI}
\chapter{Finite classical polar spaces}\label{ch1}
In this chapter we introduce some notations on finite fields and related geometries, giving some classical definitions and folklore results. We follow the approach of the notes \cite{pavesefinite}, see also \cite{Grove, Hirschfeld1, Hirschfeld2, Hirschfeld3, KSz, Mazzocca, taylor} for more details on finite classical polar spaces.
\section{Generalities on finite fields}\label{sec11}
 \begin{defn}
  A \textit{field} $(F, +, \cdot)$ is an algebraic structure with operations $+$ and $\cdot$, usually called addition and multiplication, such that:
   \begin{enumerate}
     \item $\forall a, b, c\in F$, $(a+b)+c=a+(b+c)$
     \item $\forall a, b, c\in F$, $(a\cdot b)\cdot c=a\cdot(b\cdot c)$
     \item $\forall a, b\in F$, $a+b=b+a$
     \item $\forall a, b\in F$, $a\cdot b=b\cdot a$
     \item $\exists 0\in F$, $\forall a\in F$, $a+0=0+a=a$
     \item $\exists 1\in F$, $\forall a\in F$, $a\cdot1=1\cdot a=a$
     \item $\forall a\in F$, $\exists(-a)\in F$, $a+(-a)=(-a)+a=0$
     \item $\forall a\in F\setminus \{0\}$, $\exists a^{-1}\in F$, $a\cdot a^{-1}=a^{-1}\cdot a=1$
     \item $\forall a, b, c\in F$, $a\cdot(b+c)=(a\cdot b)+(a\cdot c)$
   \end{enumerate}
  \end{defn}
  A \textit{finite field} has a finite number $q$ of elements. A first example is the field $\mathbb{Z}_{p}$, $p$ prime. We know that $\mathbb{Z}_{n}$ is a field if and only if $n=p$ prime.

  \begin{prop}
   Let $F_{q}$ be a finite field of $q$ elements, then $q=p^{n}$, $p$ prime.
  \end{prop}

  \begin{proof}
   Since $F_{q}$ is finite, there exist integers $r$ such that $r1=1+1+\ldots+1=0$. We call $r$ the minimum such integer, then $r$ is prime since if $r=ab$, $0=r1=(ab)1=a(b1)=(a1)(b1)=ab.$  But $r$ is minimal, and at least one among $a1$ e $b1$ is 0, that is a contradiction since a field does not contain non-zero zero divisors. Let now $F_{r}$ be the additive subgroup generated by $r$. Therefore, $F_{r}$ is a subfield of $F_{q}$ isomorphic to $\mathbb{Z}_{p}$, $r=p$. Hence $F_{q}$ can be seen as a vector space over $F_{p}$ with operations of vector sum and scalar-vector product. Since $F_{q}$ is finite, it has finite vector dimension over $F_{p}$. Fixing a basis $\{x_{1}, \ldots, x_{n}\}$, each $x\in F_{q}$ can be written as $x=a_{1}x_{1}+\ldots+a_{n}x_{n},$
   $a_{i}\in F_{p}$, $\forall i$. Since $|F_{p}|=p$, $|F_{q}|=p^{n}$.
  \end{proof}

  We call the integer $p$ \textit{characteristic} of the field $F_{q}$, denoted also by $charF_{q}$. If $charF_{q}=p$, $px=0$, $\forall x\in F_{q}$.

  It is possible to characterize fields with a finite number of elements.
  \begin{defn}
   The \textit{splitting field} of a polynomial $g\in F[X]$ is the smallest field containing $F$ in which $g$ factorises into linear factors.
  \end{defn}
 \begin{thm}
  A finite field $F$ with $q=p^{n}$ elements is the splitting field of the  polynomial $X^{q}-X\in\mathbb{Z}_{p}[X]$.
 \end{thm}
 \begin{proof}
 We show that every $a\in F$ is a root of $X^{q}-X$, i.e., $a^{q}=a$. Let $a\in F\setminus\{0\}$, consider $\Lambda=\{xa|x\in F\setminus\{0\}\}$, $\Lambda'=\{x|x\in F\setminus\{0\}\}$, then $\Lambda=\Lambda'$. Thus,
 $$1=\prod_{\lambda'\in\Lambda'}\lambda'=\prod_{\lambda\in\Lambda}\lambda=\prod_{x\in F\setminus\{0\}}xa=a^{q-1}\prod_{x\in F\setminus\{0\}}x=a^{q-1}.$$
\end{proof}
\section{Finite field automorphisms}\label{sec12}
\begin{defn}
 An \textit{automorphsim} of the field $F$ is a map from $F$ to itself which preserves the two operations.
\end{defn}
  \begin{prop}
   \label{frobenius}
   Consider the finite field $F_{q}$, $q=p^{n}$. Then $\forall x,y\in F_{q}$, $(xy)^p=x^{p}y^{p}$ and
   $(x+y)^{p}=x^{p}+y^{p}$.
  \end{prop}
  \begin{proof}
   First equality is obvious from the field axioms. Now from the binomial expansion:
   \begin{equation}
    \label{Newton}
    (x+y)^{n}=\sum_{k=0}^{n}\binom{n}{k}x^{n-k}y^{k},
   \end{equation}
   we see that $\binom{p}{k}=\frac{p!}{k!(p-k)!}$ is such that $p$ divides only the numerator, since $p$ is prime, i.e. $0<k<p$, $\binom{p}{k}=0$, while $\binom{p}{p}=\binom{p}{0}=1$, and Equation \eqref{Newton} gives the thesis.
  \end{proof}

  Proposition \ref{frobenius} shows that the injective morphism
  \begin{equation*}
   \begin{array}{lccc}
    \sigma: & F & \longrightarrow & F \\
    & a & \mapsto & a^{p}\\
   \end{array}
  \end{equation*}
  is a field automorphism said \textit{Frobenius automorphism}. It is possible to show that the automorphism group $Aut(F_{q})$ is a cyclic group of order $n$ generated by $\sigma$.
  \section{Vector spaces and projective spaces over finite fields}\label{sec13}
  We call now $V=V(r,q)$ the $r$-dimensional vector space over $F_{q}$, $|V|=q^{r}.$
   \begin{prop}
    \label{basi}
    $V$ contains, for each $h$ such that
    $1\leq h\leq r$, exactly $(q^{r}-1)(q^{r}-q)\ldots(q^{r}-q^{h-1})$ $h$-tuples of linearly independent vectors.
   \end{prop}
   Hence, $V$ has got exactly $(q^{r}-1)(q^{r}-q)\ldots(q^{r}-q^{r-1})$ distinct bases.
   \begin{defn}
    The \textit{Gaussian binomial coefficient} of $V$ is the number of $h$-dimensional vector subspaces. This quantity is denoted by $\begin{bmatrix}
         r \\
         h \\
        \end{bmatrix}_{q}$.
    \end{defn}
    If $h=0$, we say $\begin{bmatrix}
         r \\
         h \\
        \end{bmatrix}_{q}=1$, otherwise if $1\leq h\leq r$,
   \begin{equation}
    \begin{bmatrix}
         r \\
         h \\
        \end{bmatrix}_{q}=\frac{(q^{r}-1)(q^{r}-q)\ldots(q^{r}-q^{h-1})}{(q^{r}-1)(q^{r}-q)\ldots(q^{r}-q^{r-1})}=\frac{(q^{r}-1)(q^{r-1}-1)\ldots(q^{r-h+1}-1)}{(q^{h}-1)(q^{h-1}-1)\ldots(q-1)}
   \end{equation}
   \begin{prop}
   The following properties hold:
   \begin{enumerate}
     \item $\begin{bmatrix}
         r \\
         h \\
        \end{bmatrix}_{q}=\begin{bmatrix}
         r \\
         r-h \\
        \end{bmatrix}_{q}$
     \item $\begin{bmatrix}
         r \\
         h \\
        \end{bmatrix}_{q}=\begin{bmatrix}
         r-1 \\
         h \\
        \end{bmatrix}_{q}+q^{r-h}\begin{bmatrix}
         r-1 \\
         h-1 \\
        \end{bmatrix}_{q}$
     \item $\begin{bmatrix}
         r+1 \\
         h \\
        \end{bmatrix}_{q}=\begin{bmatrix}
         r \\
         h-1 \\
        \end{bmatrix}_{q}+q^{h}\begin{bmatrix}
         r \\
         h \\
        \end{bmatrix}_{q}$
   \end{enumerate}
 \end{prop}

 \begin{defn}
     An \textit{automorphism} of $V$ is a bijection $L:V\rightarrow V$ such that $\forall x,y\in V$ $\forall \alpha,\beta\in F_{q}$
     $$L(\alpha x+\beta y)=\alpha L(x)+\beta L(y)$$
    \end{defn}
   The automorphisms of $V$ define a permutation group over vectors of $V$, denoted by $GL(V)$. Fixing a basis of $V$ we can describe the action of $L$ on $V$ as the product with a nonsingular matrix $A_{L}$, defining a group isomorphism between $GL(V)$ and the \textit{general linear group} $GL(r,q)$ of invertible $r\times r$ matrices over $F_{q}$.

    \begin{prop}
     $|GL(r,q)|=q^{\frac{r(r-1)}{2}}\prod_{j=1}^{r}(q^{j}-1)$.
    \end{prop}

    \begin{proof}
      We have $q^{r}-1$ possibilities to choose the first row of a matrix in $GL(r,q)$, as the number of linearly independent $r$-tuples. Second row cannot be a scalar multiple of the first one, i.e. we have $q^{r}-q$ possibilities, and so on, getting
     $$|GL(r,q)|=\prod_{j=1}^{r}(q^{r}-q^{j})=q^{1+2+\ldots+(r-1)}\prod_{j=1}^{r}(q^{j}-1)=q^{\frac{r(r-1)}{2}}\prod_{j=1}^{r}(q^{j}-1).$$
    \end{proof}

   Having fixed a basis, from $L\in GL(V)$ we write a matrix $A_{L}\in GL(r,q)$, and the corresponding \textit{linear isomorphism} is described as follows:
   \begin{equation*}
    \begin{array}{lccc}
     \varphi_{A}: & V & \longrightarrow & V \\
     & x & \mapsto & xA^{T}.\\
   \end{array}
  \end{equation*}
  $L$ describes in $GL(r,q)$ an equivalence class of similar matrices. Since the determinant is a similitude invariant, we may define the subgroup $SL(V)\leq GL(V)$, called \textit{special linear group}, of automorphisms with determinant 1. Fixing a bases wa associate a matrix in the group
  $$SL(r,q)=\{M\in GL(r,q)|detM=1\}.$$
  \begin{prop}
  $SL(r,q)\unlhd GL(r,q)$ and $\frac{GL(r,q)}{SL(r,q)}\cong F_{q}\setminus\{0\}$, hence $SL(r,q)=GL(r,q)$ if and only if $q=2$. Moreover
   $|SL(r,q)|=\frac{|GL(r,q)|}{q-1}$.
  \end{prop}

   By the way, let $\sigma$ an automorphism of the field $F_{q}$, it is possible to define in an analogous way the group of \textit{semilinear automorphisms}
   \begin{defn}
     A \textit{semilinear automorphism} of $V$ is a function $T:V\rightarrow V$ such that $\forall x,y\in V$ $\forall \alpha,\beta\in F_{q}$
     $$T(\alpha x+\beta y)=\sigma(\alpha)T(x)+\sigma(\beta)T(y)$$
    \end{defn}
    Semilinear automorphisms define a permutation group $\Gamma L(V)$, overgroup of $GL(V)$. $\Gamma L(V)=GL(V)$ if and only if the only field automorphism is the identity (as $\mathbb{R}$, $\mathbb{Q}$ or in case $p$ prime elements).
   Even in this case we get a description of the \textit{semilinear automorphisms} $\varphi_{A,\sigma}$:
   \begin{equation*}
    \begin{array}{lccc}
     \varphi_{A,\sigma}: & V & \longrightarrow & V \\
     & x & \mapsto & x^{\sigma}A^{T}.\\
   \end{array}
   \end{equation*}
  The following result describe the structure of the semilinear group.
      \begin{prop}
     \label{nucleo}
     $GL(V)\unlhd\Gamma L(V)$ e $\frac{\Gamma L(V)}{GL(V)}\cong AutF_{q}$.\\
     Moreover, if $V$ is a vector space over $F_{q}$, $q=p^{n}$,
     $$|\Gamma L(V)|=n|GL(V)|.$$
    \end{prop}

  Now let $V$ be a vector space of dimension $r+1$ over a field $F$. The \textit{projective space} $PG(r,F)$, or $PG(V)$, is the set of all $1$-dimensional subspaces of $V$, called \textit{projective points}. Each $(h+1)$-dimensional subspace $W$ of $V$ is associated to a projective space $[W]=PG(h,F)$. We have only one $(-1)$-dimensional subspace associated to the zero subspace of $V$, while $0$-dimensional projective subspaces are associated to points in $V$, \textit{projective points} are associated to lines in $V$, \textit{projective lines} are associated to planes in $V$, and $(r-1)$-dimensional \textit{projective hyperplanes} are associated to $r$-dimensional hyperplanes in $V$. The notions of generators and linear independence in a projective space arise from the analogous in the vectorial case, then $\{[v_{1}],\ldots,[v_{k}]\}$ is a basis in $PG(V)$ if and only if $\{v_{1},\ldots,v_{k}\}$ is a basis in $V$, so $k=r+1$.

    \begin{defn}
      A \textit{projective frame} is a $(r+2)$-tuple $\mathfrak{R}=\{A^{0},A^{1},\ldots,A^{r},A\}$ of linearly independent projective points such that no hyperplane contains $r+1$ of them, i.e. each $r+1$ points are linearly independent.
    \end{defn}

    \begin{prop}
    Consider the projective space of dimension $r$ over the finite field $F_{q}$. The number of $h$-dimensional projective subspaces is equal to
     $$\begin{bmatrix}
         r+1 \\
         h+1 \\
        \end{bmatrix}_{q}=\frac{(q^{r+1}-1)(q^{r}-1)\ldots(q^{r-h+1}-1)}{(q^{h+1}-1)(q^{h}-1)\ldots(q-1)}.$$
    For example, the number of projective points and of hyperplanes is
    $$\theta_{r}=1+q+\ldots+q^{r}=\frac{q^{r+1}-1}{q-1}=\begin{bmatrix}
         r+1 \\
         1 \\
        \end{bmatrix}_{q}=\begin{bmatrix}
         r+1 \\
         r \\
        \end{bmatrix}_{q}.$$
    \end{prop}

    \begin{figure}[h]
    \label{Fano}
    \centering
    \includegraphics[scale=0.5]{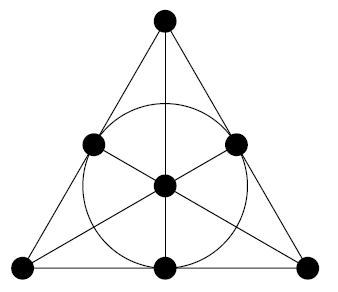}
    \caption{\footnotesize{The smallest example of projective space: the \textit{Fano plane} $PG(2,2)$, with 7 points and 7 lines.}}
   \end{figure}

    We may associate, to linear and semilinear transformation in $V$, projective transformation in $PG(V)$ called respectively \textit{projectivities} and \textit{collineations}.

    \begin{defn}
     A collineation of $PG(V)$ is a permutation of projective points which preserves incidence, i.e., $\theta$ is a collineation if
for any two subspaces $\Pi,\Pi'$ of $PG(V)$, $\Pi\subseteq\Pi'$ implies $\Pi^{\theta}\subseteq\Pi'^{\theta}$.
    \end{defn}
   Since $\varphi_{A,\sigma}(\lambda x)=\lambda^{\sigma}x^{\sigma}A^{T}=\lambda^{\sigma}\varphi_{A,\sigma}(x)$, semilinear automorphisms define permutations over projective points, and it can be easily shown that those are exactly the collineations. If $\sigma=I$, the collineation is called projectivity.
   The group of all projectivities is the \textit{projective linear group} $PGL(V)$, given by:
    \begin{equation}
     PGL(V)\cong\frac{GL(V)}{Z(GL(V))}\cong\frac{GL(r+1,F)}{Z(GL(r+1,F))}.
    \end{equation}
    The \textit{projective special (linear) group} $PSL(V)$ can be defined as follows:
    \begin{equation}
     PSL(V)\cong\frac{SL(V)}{Z(SL(V))}\cong\frac{SL(r+1,F)}{Z(SL(r+1,F))}.
    \end{equation}
    The \textit{projective semilinear group} $P\Gamma L(V)$ of all collineations can be defined as follows:
    \begin{equation}
     P\Gamma L(V)\cong\frac{\Gamma L(V)}{Z(GL(V))}.
    \end{equation}
    Moreover, when $F=F_{q}$ and $q=p^{n}$:
    \begin{eqnarray}
     |PGL(r,q)| &=& q^{\frac{r(r-1)}{2}}\prod_{j=2}^{r}(q^{j}-1) \\
     |PSL(r,q)| &=& \frac{|PGL(r,q)|}{gcd(r,q-1)}\\
     |P\Gamma L(V)| &=& n|PGL(V)|.
   \end{eqnarray}
   More generally the following holds true.
   \begin{thm}[\textbf{Fundamental Theorem of Projective Geometry}]
   \label{fund}
   Let $V$ and $W$ be two vector spaces of the same dimension over a finite field $F_q$. Every collineation between $PG(V)$ and $PG(W)$ is induced by a semilinear map $\varphi_{A,\sigma}:V\rightarrow W$.
   \end{thm}
\section{$\sigma$-sesquilinear forms and polarities}\label{sec14}
We assume now $dimV\geq3$.
 \begin{defn}
  A \textit{correlation} of the projective geometry $PG(V)$ is a bijection from $PG(V)$ to itself, which reverses inclusion.
 \end{defn}
 For example, a correlation sends points to hyperplanes and vice versa. The product of two correlations is a collineation.
 \begin{defn}
  A \textit{polarity} of $PG(V)$ is a correlation $\pi$ of order 2 and the pair $(PG(V),\pi)$ is called \textit{polar geometry}.
 \end{defn}
If $\sigma$ is an automorphism of $F_{q}$, a $\sigma$-sesquilinear form on $V$ is a map
$$\beta:V\times V\rightarrow F_{q}$$
such that, $\forall x,y,z\in V$, $\forall a,b\in F_{q}$:
\begin{enumerate}
 \item $\beta(x+y,z)= \beta(x,z)+\beta(y,z)$;
 \item $\beta(x,y+z)=\beta(x,y)+\beta(x,z)$;
 \item $\beta(ax,by)=a\sigma(b)\beta(x,y)$.
\end{enumerate}
If $\sigma=I$, then $\beta$ is said to be bilinear. A $\sigma$-sesquilinear form is non-degenerate if $\beta(x,y)=0$, $\forall y\in V$ implies
$x=0$ and vice versa. Any non-degenerate $\sigma$-sesquilinear form gives rise to a correlation $\rho$ of $PG(n,q)$ by:
$$\Pi^{\rho}= \{y \in PG(n,q) |\beta(x,y) =0, \forall x\in \Pi\},$$
where $\Pi$ is a subspace of $PG(n,q)$.
A \textit{reflexive} $\sigma$-sesquilinear form is such that, $\forall x,y \in V$:
$$\beta(x,y)=0 \iff \beta(y,x)=0.$$
\begin{prop}
 $\beta$ is a non-degenerate reflexive $\sigma$-sesquilinear form if and only if the corresponding correlation $\pi$ is a polarity.
\end{prop}
\begin{proof}
 If $\beta$ is a non-degenerate reflexive sesquilinear form,
$$y\in x^{\pi}\iff \beta(x,y)=0\iff\beta(y,x)=0\iff x\in y^{\pi}.$$
On the other hand, being $\pi$ a polarity
$$\beta(x,y)=0\iff y\in x^{\pi}\iff x\in y^{\pi}\iff\beta(y,x)=0.$$
\end{proof}
The following result from \cite{Birkhoff}, gives a complete classification of polar geometries. To prove it, we will follow the steps of the proof in \cite{taylor}.
\begin{thm}\cite[Theorem 7.1]{taylor}[\textbf{Birkhoff-von Neumann Theorem}]
\label{BVN}
\\Let $(PG(V),\pi)$ be a polar geometry, with $dimV\geq3$. Then, $\pi$ arises from a non-degenerate reflexive $\sigma$-sesquilinear form $\beta$, which falls in one of the following types:
\begin{enumerate}
 \item Alternating: in this case $\sigma=I$ and $\beta(x,x)=0$, $\forall x\in V$.
 \item Symmetric: in this case $\sigma=I$ and $\beta(x,y)=\beta(y,x)$, $\forall x,y\in V$.
 \item Hermitian: in this case $\sigma^2=I$ and $\beta(x,y)=\sigma(\beta(y,x))$, $\forall x,y\in V$.
\end{enumerate}
\end{thm}
\begin{proof}
 The maps $v\rightarrow\beta(-,v)$ and $u\rightarrow\beta(u,-)$ both induce $\pi$. It can be shown that, by Theorem \ref{fund}, exists $\lambda\in F_{q}$ such that $\beta(u,v)=\sigma^{-1}\beta(v,u)\lambda$ and $\sigma(a)=\lambda^{-1}\sigma^{-1}(a)\lambda$, i.e.
 $$\beta(u,v)=\sigma^{-1}(\sigma^{-1}\beta(u,v)\lambda)\lambda=\sigma^{-2}\beta(u,v)\sigma^{-1}(\lambda)\lambda.$$
 Now take $u,v\in V$ such that $\beta(u,v)=1$, we find $1=\sigma^{-1}(\lambda)\lambda$, i.e. $\sigma(\lambda)=\lambda^{-1}$. Moreover $\sigma(a)=\lambda^{-1}\sigma^{-1}(a)\lambda$, i.e.
 $$\sigma^{2}(a)=\sigma(\sigma(a))=\sigma(\lambda^{-1}\sigma^{-1}(a)\lambda)=\lambda a\lambda^{-1}=\lambda^{-1}a\lambda.$$
 Now if $\forall a\in F_{q}$, $\lambda^{-1}a+\sigma(a)=0$, we fix $a=\lambda$, and consider $\lambda^{-1}\lambda+\sigma(\lambda)=0$, i.e. $\lambda^{-1}=\sigma(\lambda)=-1$, $\lambda=-1$. In this case $-a+\sigma(a)=0$, that means $\sigma=I$,$\beta(u,v)=\beta(v,u)$, and the form $\beta$ is alternating if $q$ is odd, and symmetric if $q$ is even.\\
 On the converse, let suppose $\exists b=\lambda^{-1}a+\sigma(a)\neq0$. Hence,
 $$\sigma(b)=\sigma(\lambda^{-1})\sigma(a)+\sigma^{2}(a)=\lambda\sigma(a)+\lambda^{-1}a\lambda=b\lambda.$$
 Set $\overline{\beta(u,v)}=\beta(v,u)b$ and $\overline{\sigma(x)}=b^{-1}\sigma(x)b$, then $\overline{\beta}$ is a $\overline{\sigma}$-sesquilinear form which induces $\pi$ and $\overline{\beta(u,v)}=\overline{\sigma}\overline{\beta(v,u)}$. If $\overline{\sigma}=I$, then $\overline{\beta}$ is symmetric, otherwise
 $$\overline{\beta(u,v)}=\overline{\sigma}(\overline{\sigma}(\overline{\beta(u,v)}))=\overline{\sigma}^{2}(\overline{\beta(u,v)}),$$
 $\overline{\sigma}^{2}=I$ and $\overline{\beta}$ is Hermitian.
\end{proof}
Note that if $q$ is even, then any alternating form is also symmetric, but not conversely. We may distinguish the different types of polarity according to the condition on the non-degenerate reflexive sesquilinear form $\beta$: for $q$ odd, we have three types of polarities:
\begin{enumerate}
 \item if $\beta$ is an alternating form and $n$ is odd, the polarity is \textit{symplectic};
 \item if $\beta$ is a symmetric form, the polarity is \textit{orthogonal};
 \item if $\beta$ is a Hermitian form, the polarity is \textit{unitary}.
\end{enumerate}
For $q$ even, we have also three types of polarities:
\begin{enumerate}
 \item if $\beta$ is an alternating form and $n$ is odd, the polarity is \textit{symplectic};
 \item if $\beta$ is a symmetric and not alternating form, the polarity is a \textit{pseudo-polarity};
 \item if $\beta$ is a Hermitian form, the polarity is \textit{unitary}.
\end{enumerate}
\begin{defn}
 Let $\beta$ be a (non-degenerate) reflexive $\sigma$-sesquilinear form on the vector space $V$. In $PG(V)$, the set of totally isotropic subspaces with respect to $\beta$ is called a (non-degenerate) \textit{finite classical polar space}.
\end{defn}

 \begin{defn}
  Let $\beta$ be a (non-degenerate) reflexive $\sigma$-sesquilinear form on the vector space $V$. A linear transformation $f$ of $V$ is said to be:
  \begin{itemize}
   \item \textit{isometry} of $\beta$ if $\beta(f(x),f(y))=\beta(x,y)$, $\forall x,y\in V$;
   \item \textit{similarity} of $\beta$ if $\exists\lambda\in F_{q}\setminus\{0\}$ such that $\beta(f(x),f(y))=\lambda\beta(x,y)$, $\forall x,y\in V$.
  \end{itemize}
  A semi-linear transformation $\varphi_{A,\sigma}$ of $V$ is said to be a \textit{semi-similarity} of $\beta$ if $\exists\lambda\in F_{q}\setminus\{0\}$ such that $\beta(\varphi_{A,\sigma}(x),\varphi_{A,\sigma}(y))=\lambda\beta(x,y)^{\sigma}$, $\forall x,y\in V$.\\
  Isometries and similarities form two subgroups of $GL(V)$, while semi-similarities form a subgroup of $\Gamma L(V)$.
 \end{defn}
\begin{defn}
 \begin{itemize}
  \item When $n=2d-1$ is odd and $\beta$ is alternating, we may assume for $x=(x_{0},x_{1},\ldots,x_{2d-1})$ and $y=(y_{0},y_{1},\ldots,y_{2d-1})$
  \begin{equation}
   \beta(x,y)=x_{0}y_{1}-x_{1}y_{0}+\ldots+x_{2d-2}y_{2d-1}-x_{2d-1}y_{2d-2},
  \end{equation}
 and the totally isotropic space is the \textit{symplectic polar space} $W(2d-1,q)$. The \textit{symplectic groups} of semi-similarities, similarities and isometries are denoted by $\Gamma Sp(2d,q)$, $CSp(2d,q)$ and $Sp(2d,q)$, respectively and the related projective groups are $P\Gamma Sp(2d,q)$, $PCSp(2d,q)$ and $PSp(2d,q)$, where
 \begin{equation}
  PSp(2d,q)\cong\frac{Sp(2d,q)}{Sp(2d,q)\cap Z(GL(2d,q))}.
 \end{equation}
 and $|Sp(2d,q)\cap Z(GL(2d,q))|=2$.
 \item If $q$ is odd and $\beta$ is symmetric, then the totally isotropic space gives rise either to a \textit{hyperbolic} or \textit{elliptic quadric} if $n$ is odd or to a \textit{parabolic quadric} if $n=2d$ is even. They have the following canonical forms:
 \begin{enumerate}
  \item The elliptic quadric $Q^{-}(2d+1,q)$, $d\geq1$, of $PG(2d+1,q)$ is given by the equation
   \begin{equation}
    X_{0}X_{1}+\ldots+X_{2d-2}X_{2d-1}+f(X_{2d}X_{2d+1})=0,
   \end{equation}
   where f is a homogeneous irreducible polynomial of degree two over $F_{q}$.
  \item The hyperbolic quadric $Q^{+}(2d-1,q)$, $d\geq1$, of $PG(2d-1,q)$ is given by the equation
  \begin{equation}
    X_{0}X_{1}+\ldots+X_{2d-2}X_{2d-1}=0.
   \end{equation}
  \item The parabolic quadric $Q(2d,q)$, $d\geq1$, of $PG(2d,q)$ is given by the equation
  \begin{equation}
    X_{0}X_{1}+\ldots+X_{2d-2}X_{2d-1}+X^{2}_{2d}=0.
  \end{equation}
 \end{enumerate}
 Let $\varepsilon=+,-,0$ if the quadric is hyperbolic, elliptic or parabolic, respectively. The \textit{orthogonal groups} of semi-similarities, similarities and isometries are denoted by $\Gamma GO^{\varepsilon}(n+1,q)$, $CGO^{\varepsilon}(2d,q)$, $GO^{\varepsilon}(2d,q)$ and $SO^{\varepsilon}(2d,q)$, respectively and the related projective groups are $P\Gamma GO^{\varepsilon}(2d,q)$, $PCGO^{\varepsilon}(2d,q)$ and $PGO^{\varepsilon}(2d,q)$ and $PSO^{\varepsilon}(2d,q)$, where
 \begin{equation}
  PGO^{\varepsilon}(2d,q)\cong\frac{GO^{\varepsilon}(2d,q)}{GO^{\varepsilon}(2d,q)\cap Z(GL(2d,q))}.
 \end{equation}
 and $|GO^{\varepsilon}(2d,q)\cap Z(GL(2d,q))|=2$.
\item If the field has a square number $q^{2}$ of elements and $\beta$ is a Hermitian form, we may assume for $x=(x_{0},x_{1},\ldots,x_{2d-1})$ and $y=(y_{0},y_{1},\ldots,y_{2d-1})$
  \begin{equation}
   \beta(x,y)=x_{0}y_{0}^{q}+x_{1}y_{1}^{q}+\ldots+x_{2d-1}y_{2d-1}^{q},
  \end{equation}
  and the totally isotropic space is the \textit{Hermitian polar space} $H(2d-1,q^{2})$ of $PG(2d-1,q^{2})$ that is given by
  \begin{equation}
   X_{0}^{q+1}+X_{1}^{q+1}+\ldots+X_{2d-1}^{q+1}=0.
  \end{equation}
  The \textit{unitary groups} of semi-similarities, similarities and isometries are denoted by $\Gamma GU(2d,q)$, $CGU(2d,q)$, $GU(2d,q)$ and $SU(2d,q)$ respectively and the related projective groups are $P\Gamma GU(2d,q)$, $PCGU(2d,q)$, $PGU(2d,q)$ and $PSU(2d,q)$ where
  \begin{equation}
   PGU(2d,q)\cong\frac{GU(2d,q)}{GU(2d,q)\cap Z(GL(2d,q^{2}))}.
  \end{equation}
  and $|GU(2d,q)\cap Z(GL(2d,q))|=q+1$.
 \end{itemize}
\end{defn}

\begin{defn}
 In the even characteristic case, the results can be obtained by introducing quadratic forms. A \textit{quadratic form} on $V$
is a function $Q:V\rightarrow F_{q}$ such that:
 \begin{enumerate}
  \item $Q(ax)=a^{2}Q(x)$, $\forall x\in V$, $\forall a\in F_{q}$;
  \item $\beta(x,y)=Q(x+y)+Q(x)+Q(y)$ is a bilinear form.
 \end{enumerate}
\end{defn}
 In this case $\beta$ is called the polar form of $Q$. The quadratic form $Q$ is non-degenerate if $\beta(x,y)=Q(x)=0$, $\forall y\in V$ implies $x=0$. When $q$ is odd, $Q$ and $\beta$ determine each other, and $\beta$ is symmetric. If $q$ is even, then $\beta$ is alternating (and hence it is degenerate if $n$ is even), but $Q$ cannot be recovered from $\beta$. A subspace $\Pi$ of $PG(n,q)$ is \textit{totally singular} with respect to $Q$ if $Q(x)=0$, $\forall x\in\Pi$.
\section{Witt's theorems and rank}\label{sec15}
 The following theorems, firstly introduced by E. Witt in \cite{Witt} on quadratic forms, allow us to induce properties related to totally isotropic subspaces of polar spaces.
 \begin{thm}[\textbf{Witt's Extension Theorem}] Let $\beta$ be a non-degenerate reflexive $\sigma$-sesquilinear form of $V$, let $U,W$ be subspaces of $V$, and let $f$ be such that $\beta(f(x),f(y))=\beta(x,y)$, $\forall x,y\in U$. Then there is an isometry $g$ such that $g_{|U}=f$.
 \end{thm}
 \begin{thm}[\textbf{Witt's Cancellation Theorem}] Let $\beta$ be a non-degenerate reflexive $\sigma$-sesquilinear form of $V$, consider the polar space $(PG(V),\pi)$, and let $U,W$ be subspaces of $V$. Consider an isometry $f:U\rightarrow W$, then $U^{\pi}$ and $W^{\pi}$ are also isometric.
 \end{thm}
 From Witt's results we deduce that maximal totally isotropic spaces of a polar space have the same dimension, so it makes sense to consider their dimension as an isometric invariant.
 \begin{cor}
  Any two maximal totally isotropic subspaces of a polar space are isometric, and we call them \textit{generators}. The vector dimension of them is called \textit{Witt index}, or \textit{rank}, of $V$.
 \end{cor}
 By the rank and the classification of Theorem \ref{BVN}, we may classify all polar spaces $\mathcal{P}_{d,e}$ of rank $d$ among the families $Q^{+}(2d-1,q)$, $W(2d-1,q)$, $Q(2d,q)$, $Q^{-}(2d+1,q)$, $H(2d-1,q)$, $H(2d,q)$, when $e$ is equal, respectively, to $0$, $1$, $1$, $2$, $\frac{1}{2}$, $\frac{3}{2}$.

\begin{table}[h!]
\footnotesize{
\begin{tabular}{|c|c|c|c|c|c|c|}
\hline
$\mathcal{P}_{d, e}$ & $Q^+(2d-1,q)$ & $W(2d-1, q)$ & $Q(2d, q)$ & $Q^-(2d+1,q)$ & $H(2d-1,q)$ & $H(2d, q)$ \\
\hline
$e$  &  $0$ & $1$ & $1$ & $2$ & $\frac{1}{2}$ & $\frac{3}{2}$\\
\hline
\end{tabular}}
\caption{\footnotesize{Classification of finite classical polar spaces $\mathcal{P}_{d, e}$.}}
\end{table}

\normalsize
 $\mathcal{M}_{\mathcal{P}_{d,e}}$ will denote the set of generators of $\mathcal{P}_{d,e}$ and the term $k$-\textit{space} of $\mathcal{P}_{d,e}$ will be used to denote a totally singular $k$-dimensional projective space of $\mathcal{P}_{d,e}$. The polar space $\mathcal{P}_{d,e}$ has $\mathcal{O}_d\theta_{d}$ points, where $\mathcal{O}_{d}$ denotes the so called \textit{ovoid number} of $\mathcal{P}_{d,e}$ and its value is $q^{d+e-1}+1$.
 \begin{prop}
  The number of $(k-1)$-spaces of $\mathcal{P}_{d, e}$ is given by
 $$\begin{bmatrix}
         d \\
         k \\
        \end{bmatrix}_{q}\prod_{i=1}^{k}(q^{d+e-i}+1).
$$
 \end{prop}
 \begin{cor}
  The polar space $\mathcal{P}_{d, e}$ has $\theta_{d}(q^{d+e-1}+1)$ points.
 \end{cor}
 In order to give all the possible sections of a polar space we need the definition of projective cone.
 \begin{defn}
  The \textit{projective cone} with \textit{vertex} a $k$-space $\Pi_{k}$ and \textit{basis} a polar space $\mathcal{P}$ is the union of all lines joining $\Pi_{k}$ and $\mathcal{P}$, and we denote it as $\Pi_{k}\mathcal{P}$.
 \end{defn}
 \begin{prop}
  Let $\Sigma$ be a $k$-space of $PG(n,q)$. Then the intersection with a polar space is one of the following:
  \begin{itemize}
   \item $\Sigma\cap Q^{+}(n,q)=\Pi_{k-t-1}Q^{+}(t,q)$, or $\Sigma\cap Q^{+}(n,q)=\Pi_{k-t-1}Q^{-}(t,q)$, $n$ odd and $t$ odd;
   \item $\Sigma\cap Q^{+}(n,q)=\Pi_{k-t-1}Q(t,q)$, $n$ odd and $t$ even;
   \item $\Sigma\cap W(n,q)=\Pi_{k-t-1}W(t,q)$, $n$ odd and $t$ odd;
   \item $\Sigma\cap Q(n,q)=\Pi_{k-t-1}Q^{+}(t,q)$, or $\Sigma\cap Q(n,q)=\Pi_{k-t-1}Q^{-}(t,q)$, $q$ odd, $n$ even and $t$ odd;
   \item $\Sigma\cap Q(n,q)=\Pi_{k-t-1}Q(t,q)$, $q$ odd, $n$ even and $t$ even;
   \item $\Sigma\cap Q^{-}(n,q)=\Pi_{k-t-1}Q^{+}(t,q)$, or $\Sigma\cap Q^{-}(n,q)=\Pi_{k-t-1}Q^{-}(t,q)$, $n$ odd and $t$ odd;
   \item $\Sigma\cap Q^{-}(n,q)=\Pi_{k-t-1}Q(t,q)$, $n$ odd and $t$ even;
   \item $\Sigma\cap H(n,q)=\Pi_{k-t-1}H(t,q)$, $q$ square.
  \end{itemize}
 \end{prop}
\section{Small rank examples of polar spaces}\label{sec16}
We now give some more details about small rank examples of polar spaces, that will be useful in the following chapters.
\subsection{The Hermitian curve $H(2,q^{2})$}\label{subsec161}
Let $PG(2, q^{2})$ be the projective plane over the finite field $F_{q^{2}}$, with homogeneous projective coordinates $(X_{0}, X_{1}, X_{2})$. The unitary polarity of $PG(2, q^{2})$ is induced by the sesquilinear non-degenerate Hermitian form on the vector space $V(3,q^{2})$. A non-degenerate Hermitian curve $H(2, q^{2})$ consists of the isotropic points of a unitary polarity. A canonical form of the Hermitian curve $H(2,q^{2})$ is
  $$X_0^{q+1}+X_1^{q+1}+X_2^{q+1}=0,$$
  and the semilinear group stabilizing $H(2,q^{2})$ is the projective unitary group $P\Gamma U(3,q)$, acting on
  the points of the curve as 2-transitive permutation group, as it was shown in \cite{Hirschfeld1}.
 The Hermitian curve $H(2,q^{2})$ has $q^{3}+1$ points and does not contain isotropic lines, i.e. line whose all points are isotropic. Lines of $PG(2,q^{2})$ are called either \textit{tangent} or \textit{secant} when they meet the curve in 1 or $q+1$ points respectively.

\subsection{The Hermitian surface $H(3,q^{2})$}\label{subsec162}
Let $PG(3, q^{2})$ be the projective $3$-space over the finite field $F_{q^{2}}$, with homogeneous projective coordinates $(X_{0}, X_{1}, X_{2},X_{3})$. The Hermitian surface $H(3,q^2)$ of $PG(3,q^2)$ is the set of isotropic points of a non-degenerate unitary polarity of $PG(3,q^2)$. Generators of $H(3,q^2)$ are totally isotropic subspaces of maximal dimension. A generator of $H(3,q^2)$ is a totally isotropic line of $PG(3,q^2)$. The total number of generators of $H(3,q^2)$ is $(q^3+1)(q+1)$ and through any point $P\in H(3,q^2)$ there exist exactly $q+1$ generators and they are the intersection of $H(3,q^2)$ with its tangent plane at $P$. A canonical form of $H(3,q^2)$ is
$$X_0^{q+1}+X_1^{q+1}+X_2^{q+1}+X_3^{q+1}=0.$$
The semilinear group stabilizing $H(3,q^{2})$ is the projective unitary group $P\Gamma U(4,q)$, while the group of projectivities is the projective unitary group $PGU(4,q)$ and it acts on the points of $H(3,q^2)$ as a permutation group, see \cite{HKT}. The number of points of $H(3,q^2)$ is $(q^{3}+1)(q^{2}+1)$. Up to a change of the projective frame in $PG(3, q^2) $, the equation of $H(3,q^2)$ may also be written in the form
$$X_1^{q+1}+2X_2^{q+1}-X_3^qX_0-X_3X_0^q=0.$$

\subsection{The Klein Quadric $Q^{+}(5,q)$}\label{subsec163}
We use $L$ to denote the set of all lines of $PG(3,q)$. For any $\ell=\langle x,y\rangle\in L$, with $x=(x_{0}, x_{1}, x_{2}, x_{3})$, $y=(y_{0}, y_{1}, y_{2}, y_{3})$, set
$$p_{ij}=\begin{vmatrix} x_i & x_j\\
y_i & y_j \end{vmatrix},$$
$i, j=0, 1, 2, 3$. Then, $p_{ii}=0$ and $p_{ji}=-p_{ij}$. If $x'=(x'_{0}, x'_{1}, x'_{2}, x'_{3})=ax+by$ and $y'=(y'_{0}, y'_{1}, y'_{2}, y'_{3})=cy+dx$, with $A=\left(\begin{smallmatrix} a & b \\ c & d\end{smallmatrix}\right)\in GL(2,q)$, then $p'_{ij}=x'_{i}y'_{j}-x'_{j}y'_{i}=det(A)p_{ij}.$ This implies that the map
\begin{equation*}
\begin{array}{lccc}
\mathcal{K}: & L & \longrightarrow & PG(5,q) \\
 & \langle x,y\rangle & \mapsto & \langle(p_{12},p_{13},p_{14},p_{23},p_{24},p_{34})\rangle\\
\end{array}
\end{equation*}
is well-defined, i.e. up to scalar multiples the so called \textit{Pl\"{u}cker coordinates} $(p_{01},p_{02},p_{03},p_{12}, p_{13},p_{23})$ do not depend on the choice of the two distinct points $x,y\in\ell$ and to each line of $PG(3,q)$ there corresponds a point of $PG(5,q)$. The map $\mathcal{K}$ is called \textit{Klein correspondence}.
\begin{prop}
 All points of $PG(5,q)$ arising from lines of $PG(3,q)$ via the Klein correspondence, satisfy the equation
 $$X_{0}X_{5}-X_{1}X_{4}+X_{2}X_{3}=0,$$
 of a hyperbolic quadric $Q^{+}(5,q)$.
\end{prop}
\begin{proof}
 The proof is quite straightforward, by giving the explicit values of $p_{ij}$ in the equation
 $$p_{01}p_{23}-p_{02}p_{13}+p_{03}p_{12}=0.$$
\end{proof}
This result means that $\mathcal{K}$ is a bijection from the $(q^{2}+q+1)(q^{2}+1)$ lines of $PG(3,q)$ to the $(q^{2}+q+1)(q^{2}+1)$ points of the \textit{Klein Quadric} $Q^{+}(5,q):X_{0}X_{5}-X_{1}X_{4}+X_{2}X_{3}=0.$ A hyperbolic quadric of rank 3 has $(q^{2}+q+1)(q^{2}+1)$ points and the generators are the $2(q^{2}+1)(q+1)$ totally isotropic projective planes. Generators are partitioned into two classes $\mathcal{L}$ and $\mathcal{G}$, called respectively \textit{Latin planes} and \textit{Greek planes}. We now list some known facts about the Klein correspondence, for more details see \cite{Hirschfeld2}.
\begin{prop}
\label{KleinQ}
 Under the Klein correspondence:
 \begin{table}[h!]
 \footnotesize{\begin{tabular}{|c|c|}
\hline
\textbf{Lines of $PG(3,q)$} & \textbf{Points of $Q^+(5,q)$}\\
\hline
Two skew lines & Two non-orthogonal points\\
\hline
Two intersecting lines & Two orthogonal points\\
\hline
$q+1$ lines on a plane-pencil  &  $q+1$ points on a line\\
\hline
$q^{2}+q+1$ lines through a point & Latin plane $\mathcal{L}$\\
\hline
$q^{2}+q+1$ lines on a plane & Greek plane $\mathcal{G}$\\
\hline
Unique line through two points & Two Latin planes meet in a point\\
\hline
Unique line in the intersection of two planes & Two Greek planes meet in a point\\
\hline
Lines through a point and on a plane form & The intersection of a Latin plane and\\
a pencil or have empty intersection& a Greek plane is either a line or empty\\
 as the point does or does not lie on the plane  & \\
\hline
$q+1$ lines on a regulus & Non-degenerate conic $\mathcal{C}$\\
\hline
$2(q+1)$ isotropic lines of $Q^{+}(3,q)$ & Two non-degenerate conics $\mathcal{C}$ and $\mathcal{C}'$\\
 & lying on two planes $\pi$ and $\pi^{\perp}$\\
\hline
$(q^{2}+1)(q+1)$ isotropic lines of $W(3,q)$ & Parabolic quadric $Q(4,q)\subseteq Q^{+}(5,q)$\\
\hline
$(q^{3}+1)(q+1)$ isotropic lines of $H(3,q^{2})$ & Elliptic quadric $Q^{-}(5,q)$, intersection of\\
 & $Q^{+}(5,q^{2})$ with a subgeometry $PG(5,q)$\\
\hline
\end{tabular}}
\caption{\footnotesize{Images of structures of $PG(3,q)$ under the action of the Klein map $\mathcal{K}$.}}
\end{table}
\end{prop}
\normalsize
\subsection{$Q(4,q)$ and $W(3,q)$}\label{subsec164}
\begin{defn}
 A \textit{generalized quadrangle} $GQ(s,t)$ is an incidence structure $(\mathcal{P},\mathcal{B},\mathcal{I})$ of points and lines, with an incidence relation $\mathcal{I}\subseteq\mathcal{P}\times\mathcal{B}$ such that:
 \begin{enumerate}
  \item Each line contains exactly $s+1$ points, and there is at most one point on two distinct lines.
  \item Each point lies on exactly $t+1$ lines, and there is at most one line through two distinct points.
  \item For every point $P$ and every line $\ell$, such that $P\notin\ell$, there is an unique line $r$ and an unique point $Q$, such that $P\in r$, and $Q\in r$.
 \end{enumerate}
\end{defn}
\begin{prop}
 \begin{itemize}
  \item If the incidence structure $(\mathcal{P},\mathcal{B},\mathcal{I})$ defines a generalized quadrangle, then  $(\mathcal{B},\mathcal{P},\mathcal{I}^{-1})$, with $\mathcal{I}^{-1}$ inverse incidence relation, defines also a generalized quadrangle called dual.
  \item A $GQ(s,t)$ has $(s+1)(st+1)$ points and $(t+1)(st+1)$ lines.
 \end{itemize}
\end{prop}
\begin{proof}
 By interchanging the role of points and lines in a $GQ(s,t)$, we obtain the $GQ(t,s)$. Fix now a line $\ell\in\mathcal{P}$, which contains the $s+1$ points $P_{0},P_{1},\ldots,P_{s}$. Through each of the $P_{i}$, $i=0,1,\ldots,s$, there pass exactly other $t$ lines, and from the axioms of the generalized quadrangles we can say that those $t(s+1)$ lines cover exactly one time each point of $GQ(s,t)\setminus\ell$, so the total number of points is
 $$(s+1)+t(s+1)s=(s+1)(st+1).$$
 In an analogous way we find the result for the lines.
\end{proof}
\begin{exmp}
 \begin{itemize}
  \item The \textit{classical quadrangle} is a $GQ(1,1)$, since each point lies on 2 lines and each line contains 2 points;
  \item $GQ(s,1)$ is a grid, and if $s=q$ prime power, it corresponds to a hyperbolic quadric $Q^{+}(3,q)$;
  \item $GQ(t,1)$ is a dual grid.
 \end{itemize}
\begin{figure}[h]
  \centering
  \includegraphics[scale=0.5]{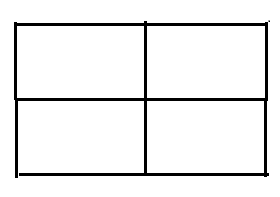}
  \caption{\footnotesize{$GQ(2,1)\cong Q^{+}(3,2)$}}
 \end{figure}
\end{exmp}
The following theorem gives a condition for the parameters $s$ and $t$.
\begin{thm}[\textbf{Higman's inequality}] For a $GQ(s,t)$, with $s,t>1$.\\
 $s\leq t^{2}$ and $t\leq s^{2}$.
\end{thm}
Actually they are known examples for $(s,t)$, or dually $(t,s)$, in \\ $\{(q,q),(q,q^{2}),(q^{2},q^{3}),(q+1,q-1)\}$.
\begin{prop}
 The polar space $\mathcal{P}_{2,e}$ is a $GQ(q,q^{e})$.
\end{prop}
In particular way, $W(3,q)$ and $Q(4,q)$ are both a $GQ(q,q)$.
\begin{thm}\cite[Theorem 3.2.1]{GQs}
 The parabolic quadric $Q(4,q)$ is isomorphic to the dual of $W(3,q)$. Moreover, if $q$ is even, $Q(4,q)$ (or $W(3,q)$) is self-dual.
\end{thm}
\begin{proof}
 From Proposition \ref{KleinQ}, the polar spaces $Q(4,q)$ and $W(3,q)$ are dual, as we can refer to both as a $GQ(q,q)$, interchanging the role of points and lines. Now let $q$ be even. The tangent 3-spaces of $Q(4,q)$ all meet in one point $N$, called \textit{nucleus} of $Q(4,q)$, see \cite{Hirschfeld1} for more details. From $N$ we project the quadric onto a $PG(3,q)$ not containing $N$. This gives a bijection of the $Q(4,q)$ onto $PG(3,q)$, mapping the $(q^{2}+1)(q+1)$ lines of $Q(4,q)$ in $(q^{2}+1)(q+1)$ lines of $PG(3,q)$. Since the $q+1$ lines incident with a given point in the quadric are contained in a tangent 3-space, they are mapped onto elements of a plane-pencil of lines of $PG(3,q)$. Hence the images of the lines of $Q(4,q)$ constitute a linear complex of lines, i.e. they are the totally isotropic lines with respect to a symplectic polarity of $PG(3,q)$, see \cite{SegreNew}. It follows that $Q(4,q)\cong W(3,q)$.
\end{proof}

\chapter{Regular systems of finite classical polar spaces}\label{ch2}
A $(r-1)$-spread of a projective space $PG(n-1,q)$ is a partition of the pointset in $(r-1)$-spaces. A necessary and sufficient condition for the existence of $(r-1)$-spreads was given by B. Segre in \cite{Segre0}. Successively, the analog concept for polar spaces was introduced, defining a spread of a polar space to be a partition of the pointset in generators. Since spreads of polar spaces do not always exist, a more general notation was also investigated, namely the $m$-regular systems which are sets of generators of the polar space such that each point (or $(k-1)$-space) lies in exactly $m$ of generators of the system. In this chapter we deal with regular systems of classical polar spaces. New result presented in Sections \ref{sec22}-\ref{sec25} and \ref{sec27}-\ref{sec28} come from the papers \cite{VS2, VS4}.
\section{Spreads and field reduction}\label{sec21}
\begin{defn}
 A \textit{$(r-1)$-spread} in $PG(n-1,q)$ is a set of $(r-1)$-spaces partitioning the set of points in $PG(n-1,q)$.
\end{defn}
Segre's necessary and sufficient condition for the existence of $(r-1)$-spread is the following
\begin{thm}\label{segrespread}
A projective space $PG(n-1,q)$ has a $(r-1)$-spread if and only if $r$ divides $n$.
\end{thm}
 Segre showed the necessary part in Theorem \ref{segrespread} by a counting argument on the number of points of a projective space and using the following lemma.
\begin{lem}\cite[Section 22]{Segre0}
 Let $r,n\geq1$, $q\geq2$. $r$ divides $n$ if and only if $q^{r}-1$ divides $q^{n}-1$.
\end{lem}
\begin{proof}
 The assertion follows from
 \begin{equation}\label{AAA}
  gcd(q^{r}-1,q^{n}-1)=q^{gcd(r,n)}-1.
 \end{equation}
 To prove (\ref{AAA}), we may assume w.l.o.g. $n>r$. Then $n=rs+t$, $0\leq t<r$. Therefore,
 $$q^{n}-1=q^{rs+t}-1=q^{t}q^{rs}+q^{t}-q^{t}-1=q^{t}(q^{rs}-1)+(q^{t}-1),$$
 with $0\leq q^{t}-1<q^{r}-1$ where equality holds if and only if $t=0$. Then $q^{t}-1$ is equal to $q^{n}-1 \pmod {q^{r}-1}$. Now, the \textit{Euclidean algorithm for greatest common divisor} applied to $(n,r)$ and $(q^{n}-1,q^{r}-1)$ ends up after the same number of steps, and hence the last non zero remainders give Equation \eqref{AAA}.
\end{proof}
The sufficient condition in Theorem \ref{segrespread} can be proven by using particular spreads named \textit{Desarguesian spreads}. Let $V$ be a vector space of dimension $N$ over the field $F_{q^n}$. Since $F_{q}$ is a subfield of $F_{q^n}$, the vector space $V$ can be viewed as a vector space $V'$ of dimension $nN$ over the field $F_q$. Therefore, there exists a bijection $\phi$ between the vectors of $V$ and $V'$ which maps $k$-dimensional subspaces of $V$ to certain $kn$-dimensional subspaces of $V'$. Let $\pi=PG(U)$ be a $(k-1)$-subspace of $PG(n-1,q^{n})$. Then $\pi$ corresponds to a $(kn-1)$-subspace $K(\pi)$ of $PG(rn-1,q)$ defined by vectors of the $rn$-vector space spanned by the vectors of $U$ seen as $n$-dimensional spaces in $PG(rn-1,q)$. The map
\begin{equation*}
   \begin{array}{lccc}
    \phi: & PG(r-1,q^{n}) & \longrightarrow & PG(rn-1,q) \\
    & \pi & \mapsto & K(\pi),\\
   \end{array}
\end{equation*}
is called a \textit{field reduction map} and it has been widely studied by many authors \cite{Gill, LV, V}.

\begin{lem}\cite[Lemma 2.2]{LV}
 \label{desaspread} Let $\mathcal{P}$ denote the pointset of $PG(r-1,q^{n})$.
 Then the field reduction map $\phi$ has the following properties:
 \begin{enumerate}
  \item $\phi$ is injective.
  \item Any two distinct element in $\phi[\mathcal{P}]$ are disjoint.
  \item Each point in $PG(rn-1,q)$ is contained in an element of $\phi[\mathcal{P}]$.
  \item $|\phi[\mathcal{P}]|=\frac{q^{nr}-1}{q^{n}-1}.$
 \end{enumerate}
\end{lem}
\begin{proof}
 \begin{enumerate}
  \item If $\phi(\pi_{1})=\phi(\pi_{2})$ with $\pi_{1}=PG(U_{1})$ and $\pi_{2}=PG(U_{2})$, vectors of $U_{1}$ and $U_{2}$ span the same $rn$-vector space in $PG(rn-1,q)$. Then $U_{1}=U_{2}$ and $\pi_{1}=\pi_{2}$.
  \item Consider $\phi(P_{1})$ and $\phi(P_{2})$, where $P_{1}=PG(\langle u\rangle)$ and $P_{2}=PG(\langle v\rangle)$ over $F_{q^{n}}$. If the spaces spanned by $u$ and $v$ have a point in common, then $\alpha u=\beta v$, and that means that they span the same $F_q$-space, i.e. $\phi(P_{1})=\phi(P_{2})$.
  \item An $F_q$-point $P=PG(\langle w\rangle)$ belongs to $\phi(w)$.
  \item It arises naturally from 1 and 2.
 \end{enumerate}
\end{proof}
From Lemma \ref{desaspread} the image via field reduction of $PG(r-1,q^{n})$ in $PG(rn-1,q)$ is always a spread, called \textit{Desarguesian spread}. It provides the sufficient condition in Theorem \ref{segrespread}. Being solved the problem of spreads in projective spaces, next step will be to consider partitions of polar spaces.
\begin{defn}
 Let $\mathcal{P}$ be a finite classical polar space, whose generators are the projective $(d-1)$-subspaces. A \textit{spread} $\mathcal{S}$ of a polar space $\mathcal{P}$ is a set of generators of $\mathcal{P}$ such that every point lies on exactly one generator of $\mathcal{S}$.
\end{defn}
 Table below reports the known results on spread of polar spaces.
 The reference is \cite[Table 7.4]{Hirschfeld3}:
\begin{table}[h!]
\label{tabHirschfeld}
\centering
\begin{center}
\begin{tabular}{|c|c|}
\hline
Polar space $\mathcal{P}$ & Existence of $\mathcal{S}$   \\ \hline
$Q^{+}(3,q)$ & Yes  \\ \hline
$Q^{+}(7,p)$, $p>2$ prime  & Yes \\ \hline
$Q^{+}(7,q)$, $q$ odd, $q\equiv 0,2 \pmod3$ & Yes \\ \hline
$Q^{+}(4n+3,q)$, $q$ even & Yes  \\ \hline
$Q^{+}(4n+1,q)$ & No \\ \hline
$W(2n+1,q)$, $n\geq1$ & Yes  \\ \hline
$Q(2n,q)$, $n\geq2$, $q$ even & Yes  \\ \hline
$Q(6,p)$, $p>2$ prime & Yes \\ \hline
$Q(6,q)$, $q$ odd, $q\equiv 0,2 \pmod3$ & Yes \\ \hline
$Q(4n,q)$, $q$ odd & No \\ \hline
$Q^{-}(5,q)$ & Yes \\ \hline
$Q^{-}(2n+1,q)$, $n>2$, $q$ even & Yes \\ \hline
$H(4,4)$ & No \\ \hline
$H(2n+1,q^{2})$ & No\\ \hline
\end{tabular}
\caption{\footnotesize{Some existence and non existence results for spreads in finite classical polar spaces.}}
\end{center}
\end{table}

\newpage
\section{Regular systems of polar spaces}\label{sec22}
Since spreads do not exist in many examples of finite polar spaces, B. Segre in \cite{Segre} gave the more general definition of \textit{$m$-regular systems} of generators, in which the number of generators of the family through a point is fixed. The definition of $m$-regular system may be also extended starting from $(k-1)$-spaces of $\mathcal{P}_{d, e}$, $1 \leq k \leq d-1$:
\begin{defn}
An \textit{$m$-regular system} (or simply \textit{regular system}) with respect to $(k-1)$-spaces of $\mathcal{P}_{d, e}$, $1 \leq k \leq d-1$, is a set $\mathcal{R}$ of generators of $\mathcal{P}_{d, e}$ with the property that every $(k-1)$-space of $\mathcal{P}_{d, e}$ lies on exactly $m$ generators of $\mathcal{R}$, $0 \leq m \leq |\mathcal{M}_{\mathcal{P}_{d-k, e}}|$, where $|\mathcal{M}_{\mathcal{P}_{d-k, e}}|$ is the number of generators of $\mathcal{P}_{d, e}$ passing through one of its $(k-1)$-spaces. A regular system with respect to $(k-1)$-spaces of $\mathcal{P}_{d, e}$ having the same size as its complement (in $\mathcal{M}_{mathcal{P}_{d, e}}$) is said to be a \textit{hemisystem}.
\end{defn}
Not much is known about regular systems w.r.t. $(k-1)$-spaces of a polar space. If $k \geq 2$ and the polar space is not a hyperbolic quadric only a few examples are known and they are hemisystems w.r.t. lines of $Q(6, q)$, $q\in\{3, 5,7,11\}$, see \cite{BLL,LJ}. If $k = 1$ and the rank $d$ of the polar space is at least three, then some constructions are known on the parabolic quadric $Q(2d, q)$ \cite{CPa0, LN} and on the Hermitian variey $H(2d-1, q)$ \cite{LB, CP1}. Finally since polar spaces of rank two are generalized quadrangles, several examples of regular systems arise from by using duality, see \cite{BDS, bamberg2010every, BLP, bamberg2018new, BW, CCEM, cossidente2017intriguing, cossidente2005hemisystems, FMX, FT, KNS_Hemi, Pa, PW}. However, in the case of generalized quadrangles, many questions are still unsolved. In the following sections we deal with regular systems of polar spaces, giving results appeared in the paper \cite{VS2}.

We describe below some first properties of $m$-regular system with respect to $(k-1)$-spaces of a polar space. Here we are mainly concerned with constructions of regular systems w.r.t. points of various polar spaces. In particular three different methods are presented: by partitioning the generators of an elliptic quadric $Q^-(2d+1, q)$ into generators of hyperbolic quadrics $Q^+(2d-1, q)$ embedded in it; regular systems w.r.t. points of a polar space $\mathcal{P}$ obtained by means of a $k$-system of $\mathcal{P}$; regular systems w.r.t. points arising from field reduction.

We summarize some straightforward properties regarding regular systems of a polar space $\mathcal{P}_{d, e}$.

\begin{lem}\label{properties}
 Let $\mathcal{A}$ and $\mathcal{B}$ be an $m$-regular system and an $m'$-regular system w.r.t. $(k-1)$-spaces of $\mathcal{P}_{d, e}$, respectively, then:
 \begin{enumerate}
  \item $|\mathcal{A} | = m \prod_{i = 1}^{k}(q^{d + e - i} + 1)$, $| \mathcal{B} | = m' \prod_{i = 1}^{k}(q^{d + e - i} + 1)$;
  \item a regular system of $\mathcal{P}_{d, e}$ w.r.t. $(k-1)$-spaces is also a regular system w.r.t. $(k'-1)$-spaces, $1 \leq k' \leq k-1$;
  \item if $\mathcal{A} \subseteq \mathcal{B}$, then $\mathcal{B} \setminus \mathcal{A}$ is an $(m'-m)$-regular system w.r.t. $(k-1)$-spaces of $\mathcal{P}_{d, e}$;
  \item if $\mathcal{A}$ and $\mathcal{B}$ are disjoint, then $\mathcal{A} \cup \mathcal{B}$ is an $(m+m')$-regular system w.r.t. $(k-1)$-spaces of $\mathcal{P}_{d, e}$;
  \item the empty set and $\mathcal{M}_{\mathcal{P}_{d, e}}$ are trivial examples of $(k-1)$-regular systems of $\mathcal{P}_{d, e}$, for each $1\leq k\leq d-1$, and $m=0,|\mathcal{M}_{\mathcal{P}_{d-k, e}}|$, respectively.
 \end{enumerate}
\end{lem}
\section{Hemisystems of elliptic quadrics}\label{sec23}
In \cite[Section 6]{BLP}, the authors construct hemisystems of $Q(4, q)$, $q$ odd, by partitioning the generators of $Q(4, q)$ into generators of hyperbolic quadrics $Q^{+}(3, q)$ embedded in $Q(4, q)$. In this section we investigate and generalize their idea, providing a construction of hemisystems of $Q^-(5, q)$, $q$ odd. Let $\perp$ be the polarity of $PG(2n+1, q)$ associated with $Q^-(2n+1, q)$.
\begin{thm}\label{th1}
 Let $P$ be a partition of the generators of the elliptic quadric $Q^-(2n+1, q)$, $n \geq 2$, into generators of hyperbolic quadrics $Q^+(2n-1, q)$ embedded in $Q^-(2n+1, q)$. Then $q$ is odd and $2^\frac{(q^n+1)(q^{n+1}+1)}{2(q+1)}$ hemisystems w.r.t. points of $Q^-(2n+1,q)$ arise, by taking one family from each of the Latin and Greek generator families in $P$, and then considering the union of these generators.
\end{thm}
\begin{proof}
 Since the number of generators of $Q^-(2n +1, q)$ is $\prod_{i = 2}^{n+1}(q^i+1)$ and the number of generators of $Q^+(2n-1, q)$ is $\prod_{i = 0}^{n-1}(q^i+1)$, we have that the number of members in $P$ is $\frac{(q^n+1)(q^{n+1}+1)}{2(q+1)}$ and hence $q$ has to be odd. Similarly through a point $P$ of $Q^-(2n +1, q)$ there pass $\prod_{i = 2}^{n}(q^i +1)$ generators of $Q^-(2n+1, q)$ and $\prod_{i = 0}^{n-2}(q^i+1)$ generators of $Q^+(2n-1, q)$, therefore every point of $Q^-(2n+1, q)$ is contained in a constant number of elements of $P$. On the other hand, since the number of generators of a Latin family through $P$ equals the number of generators of a Greek family through $P$, by selecting one family from each Latin and Greek pair in $P$ amounts to choosing exactly half of the generators through $P$. Finally note that the number of hemisystems obtained by selecting one family from each of the elements of $P$ equals $2^\frac{(q^n+1)(q^{n+1}+1)}{2(q+1)}$.
\end{proof}

\begin{prop}
 Let $\mathcal{L}$ be a set of $\frac{(q^n+1)(q^{n+1}+1)}{2(q+1)}$ lines external to $Q^-(2n+1, q)$ such that
$$|\langle r, r' \rangle \cap Q^-(2n+1, q)| \neq
\begin{cases}
 1 & \mbox{ if } \;\; |r \cap r'| = 1, \\
 q+1 & \mbox{ if } \;\; |r \cap r'| = 0 ,
\end{cases}$$
for each $r, r' \in \mathcal{L}$, $r \neq r'$. Then there exists a partition of the generators of $Q^-(2n+1, q)$ into generators of a $Q^+(2n-1, q)$.
\end{prop}
\begin{proof}
 The generators of the hyperbolic quadrics $r^\perp \cap Q^-(2n+1, q)$, with $r \in \mathcal{L}$, form the partition.
\end{proof}
Below we show that there exists such a partition of the generators of $Q^-(5, q)$.
\begin{cons}
 Let $\Pi$ be a solid of $PG(5, q)$ meeting $Q^-(5, q)$ in a hyperbolic quadric $Q^+(3, q)$ and let $\ell = \Pi^\perp$. Let $\mathcal{X}$ be the set of lines of $\Pi$ that are external to $Q^-(5, q)$. Let $\ell_1$ be a line of $\mathcal{X}$ and let $\ell_2 = \ell_1^\perp \cap \Pi$. Denote by $\Pi_i = \ell_i^\perp$, $i = 1,2$. Since $\ell_i \in \mathcal{X}$, $\Pi_i \cap Q^-(5, q)$ is a three-dimensional hyperbolic quadric. Let $\mathcal{X}_1$ be the set of lines of $\Pi_1$ external to $Q^-(5, q)$ and intersecting $\ell$ in at least a point; let $\mathcal{X}_2$ be the set of lines of $\Pi_2$ external to $Q^-(5, q)$ and meeting both $\ell$ and $\ell_1$ in exactly one point. Note that in $PG(3, q)$ there pass $\frac{q(q-1)}{2}$ lines external to a $Q^+(3, q)$ through a point off the quadric. Thus $|\mathcal{X}| = \frac{q^2(q-1)^2}{2}$, $|\mathcal{X}_1| = \frac{(q-2)(q+1)^2}{2}+1$ and $|\mathcal{X}_2| = \frac{(q+1)^2}{2}$ and the set $\mathcal{X} \cup \mathcal{X}_1 \cup \mathcal{X}_2$ consists of $ \frac{(q^2-q+1)(q^2+1)}{2}$ lines external to $Q^-(5, q)$.
\end{cons}
\begin{remark}
 Let $P$ be a point of $\ell_1 \cup \ell_2$ and let us denote by $\pi_P$ the plane spanned by $\ell$ and $P$. Note that $\pi_P \cap Q^-(5, q)$ is a conic of which $P$ is an internal point. Hence the lines intersecting both $\ell$ and $\ell_i$, in exactly one point are either external or secant and they are equal in number, for each $i \in\{1,2\}$.
\end{remark}
\begin{thm}
 For each $r, r' \in \mathcal{X} \cup \mathcal{X}_1 \cup \mathcal{X}_2$, $r \neq r'$, we have
 $$|\langle r, r' \rangle \cap Q^-(5, q)| \neq
 \begin{cases}
  1 & \mbox{ if } \;\; |r \cap r'| = 1, \\
  q+1 & \mbox{ if } \;\; |r \cap r'| = 0 .
\end{cases}$$
\end{thm}
\begin{proof}
 Let $r, r'$ be two distinct lines both lying in $\mathcal{X}$ (or in $\mathcal{X}_1$, or in $\mathcal{X}_2$). Then $|\langle r, r' \rangle \cap Q^-(5, q)|$ equals $(q+1)^2$ or $q+1$, according to whether $|r \cap r'| = 0$ or $|r \cap r'| = 1$, respectively. Assume that $r \in \mathcal{X}$ and $r' \in \mathcal{X}_2$. Let $P$ be the point $r' \cap \ell_1$. We have two possibilities: either $P \in r$ or $P \notin r$. In the former case since $|r \cap r'| = 1$, $\langle r', r \rangle$ is a plane contained in the three-space $\langle \ell, r \rangle$ and $\langle \ell, r \rangle^\perp = r^\perp \cap \ell^\perp = r^\perp \cap \Pi$ is a line of $\Pi$ external to $Q^-(5, q)$. Hence $|\langle r, r' \rangle \cap Q^-(5, q)| = q+1$. In the latter case, let $\Gamma$ be the hyperplane $\langle r, r', \ell \rangle = \langle r, P, \ell \rangle$. Then $\Gamma^\perp = \langle r, P \rangle^\perp \cap \Pi$ is a point of $\Pi$ off the quadric $Q^-(5, q)$ since $\langle r, P \rangle \cap Q^-(5, q)$ is a conic. Therefore $\Gamma \cap Q^-(5, q)$ is a parabolic quadric $Q(4, q)$ and $|\langle r, r' \rangle \cap Q^-(5, q)| \in \{q^2+q+1, (q+1)^2, q^2+1\}$. Assume that $r \in \mathcal{X}$ and $r' \in \mathcal{X}_1$. If $r' = \ell$, then $\langle r, r' \rangle^\perp \in \mathcal{X}$ and hence $|\langle r, r' \rangle \cap Q^-(5, q)| = (q+1)^2$. If $r'$ meets $\ell_2$ in a point, repeating the same argument used before we are done. Let $r'$ be a line meeting $\ell$ in a point and disjoint from $\ell_2$. Then the plane $\langle r', \ell \rangle$ intersects $\Pi$ in a point, say $T$. Let $\Gamma$ be the projective space $\langle r, r', \ell \rangle$. If $T \in r$, then $\Gamma = \langle r, r' \rangle$ is a solid, $\Gamma^\perp \in \mathcal{X}$ and $|\langle r, r' \rangle \cap Q^-(5, q)| = (q+1)^2$. If $T \notin r$, then $\Gamma$ is the hyperplane $\langle r, T, \ell \rangle$ and as before $\Gamma^\perp = \langle r, T \rangle^\perp \cap \Pi$ is a point of $\Pi$ off the quadric $Q^-(5, q)$ since $\langle r, T \rangle \cap Q^-(5, q)$ is a conic. Therefore $\Gamma \cap Q^-(5, q)$ is a parabolic quadric $Q(4, q)$ and $|\langle r, r' \rangle \cap Q^-(5, q)| \in \{q^2+q+1, (q+1)^2, q^2+1\}$. Assume that $r \in \mathcal{X}_1$ and $r' \in \mathcal{X}_2$. Let $|r \cap r'| = 1$. If $r = \ell$, then $|\langle r, r' \rangle \cap Q^-(5, q)| = q+1$. If $r \neq \ell$, then the plane $\langle r, r' \rangle$ is contained in the solid $\langle \ell_1, r \rangle$, where $|\langle \ell_1, r \rangle \cap Q^-(5, q)| = (q+1)^2$, since $\langle \ell_1, r \rangle^\perp = r^\perp \cap \Pi_1$ is an external line of $\Pi_1$. Therefore, we have again that $|\langle r, r' \rangle \cap Q^-(5, q)| = q+1$. Let $|r \cap r'| = 0$ and let $\Gamma$ be the projective space spanned by $\langle r, r', \ell \rangle$. Since $r$ is not contained in $\Pi_2$, then $\Gamma$ is a solid and $\ell_1 \not\subseteq \Gamma$. Hence $\langle \Gamma, \ell_1 \rangle$ is a hyperplane and $\langle \Gamma, \ell_1 \rangle^\perp \in \ell_2$. Therefore $\langle \Gamma, \ell_1 \rangle \cap Q^-(5, q)$ is a parabolic quadric $Q(4, q)$ and $|\langle r, r' \rangle \cap Q^-(5, q)| \in \{q^2+q+1, (q+1)^2, q^2+1\}$.
\end{proof}

\section{Regular systems arising from $k$-systems of polar spaces}\label{sec24}
 In this section we explore a class of regular systems w.r.t. points of a polar space $\mathcal{P}$ that can be obtained by means of a $k$-system of $\mathcal{P}$. The notion of $k$--system was introduced by E. Shult and J. A. Thas in \cite{ST}.
\begin{defn}
 A $k$-system of a polar space $\mathcal{P}_{d, e}$, $1 \leq k \leq d - 2$, is a set $\{\Pi_1, \dots, \Pi_{q^{d+e-1} + 1}\}$ of $k$-spaces of $\mathcal{P}_{d, e}$, such that a generator of $\mathcal{P}_{d, e}$ containing $\Pi_i$, is disjoint from $\bigcup_{j = 1, j \neq i}^{q^{d + e - 1} + 1} \Pi_j$.
\end{defn}
Let $\mathcal{S}$ be a $k$-system of $\mathcal{P}_{d, e}$ and let $\mathcal{G}$ the set of generators of $\mathcal{P}_{d, e}$ containing one element of $\mathcal{S}$. It is not difficult to see that $\mathcal{G}$ is a regular system of $\mathcal{P}_{d, e}$ w.r.t. points.
\begin{lem}\label{k-system}
 The set $\mathcal{G}$ is a $|\mathcal{M}_{\mathcal{}P_{d - k - 1, e}}|$-regular system of $\mathcal{P}_{d, e}$ w.r.t. points.
\end{lem}
\begin{proof}
 Let $P$ be a point of $\mathcal{P}_{d, e}$. If $P \in \Pi_i$, for some $i$, then there are $|\mathcal{M}_{\mathcal{P}_{d - k - 1, e}}|$ members of $\mathcal{G}$ containing $P$. If $P$ is not contained in an element of $\mathcal{S}$, then from \cite[Theorem 5]{ST}, there are $q^{d - k - 2 + e} + 1$ $(k+1)$-spaces of $\mathcal{P}_{d, e}$ containing $P$ and a member of $\mathcal{S}$. Hence the point $P$ belongs to $|\mathcal{M}_{\mathcal{P}_{d - k - 2, e}}| (q^{d - k  - 2 + e} + 1) = |\mathcal{M}_{\mathcal{P}_{d - k - 1, e}}|$ generators of $\mathcal{G}$.
\end{proof}
\subsection{Regular systems of $Q^+(5,q)$, $q$ odd}\label{subsec241}
Here we turn our attention to the Klein quadric $Q^+(5,q)$ of $PG(5,q)$, $q$ odd. A $1$-system of $Q^+(5, q)$ is a set $\mathcal{S}$ of $q^2+1$ lines of $Q^+(5, q)$ such that no plane of $Q^+(5, q)$ through a line of $\mathcal{S}$ has a point in common with the remaining lines of $\mathcal{S}$. From \cite[Theorem 15]{ST}, the quadric $Q^+(5,q)$, $q$ odd, has a unique $1$-system $\mathcal{S}$ and the points covered by the $q^2+1$ lines of $\mathcal{S}$ correspond, under the Klein correspondence, to the lines that are tangent to an elliptic quadric of $PG(3, q)$. In \cite{E}, G. Ebert provided a partition $\mathcal{F}$ of $PG(3,q)$ into $q+1$ disjoint three-dimensional elliptic quadrics, say $\mathcal{E}_0, \ldots, \mathcal{E}_q$. This result is achieved by considering a subgroup of order $q^2+1$ of a Singer cyclic group of $PGL(4,q)$. In particular, if $q$ is odd, we have the following result.
\begin{lem} \label{ebert} \cite[Corollary 1, Theorem 5]{E}
 The lines of $PG(3,q)$, $q$ odd, are partitioned as follows:
 \begin{enumerate}
  \item $\frac{(q^2+1)^2}{2}$ lines tangent to no elliptic quadrics of $\mathcal{F}$ and hence secant to $\frac{(q+1)}{2}$ of them;
  \item $\frac{(q+1)^2(q^2+1)}{2}$ lines tangent to two elliptic quadrics of $\mathcal{F}$ and hence secant to $\frac{(q-1)}{2}$ of them,
  \item no plane of $PG(3, q)$ is tangent to two distinct elliptic quadrics of $\mathcal{F}$;
  \item there exists a line of $PG(3,q)$ that is tangent to $\mathcal{E}_i$ and $\mathcal{E}_j$, with $0 \leq i, j \leq q$, $i \neq j$, if and only if $i+j$ is odd.
 \end{enumerate}
\end{lem}
Let $\mathcal{L}_i$ be the set consisting of the $(q+1)(q^2+1)$ lines that are tangent to the elliptic quadric $\mathcal{E}_i$ and let $\mathcal{S}_i$ be the corresponding $1$-system of $Q^+(5,q)$, obtained by applying the Klein correspondence to the lines of $\mathcal{L}_i$, $0 \leq i \leq q$. By using Lemma \ref{ebert}, we are able to prove the following result.
\begin{thm}\label{klein}
 There is no plane of $Q^+(5,q)$, $q$ odd, containing two lines of $\bigcup_{i = 0}^q \mathcal{S}_i$.
\end{thm}
\begin{proof}
 Assume for a contradiction that a generator of $Q^+(5, q)$, say $\pi$, contains two lines, say $ \ell_1 \in \mathcal{S}_i$ and $\ell_2 \in \mathcal{S}_j$. Note that necessarily $i \neq j$. Under the inverse of the Klein correspondence, the points of $\ell_1$ are mapped to the lines that are tangent to $\mathcal{E}_i$ at a point, say $P_1$, and the points of $\ell_2$ are mapped to the lines that are tangent to $\mathcal{E}_j$ at a point, say $P_2$. If $\pi$ is a Latin plane, then $P_1 = P_2$, contradicting the fact that $|\mathcal{E}_i \cap \mathcal{E}_j| = 0$. If $\pi$ is a Greek plane, then there would exist a plane of $PG(3, q)$ tangent to two distinct elliptic quadrics of $\mathcal{F}$, contradicting 3 of Lemma \ref{properties}.
 \end{proof}
From Lemma \ref{k-system} and 4 of Lemma \ref{properties}, we have the following result.
\begin{cor}
 $Q^+(5,q)$, $q$ odd, has a $2m$-regular system w.r.t. points for all $m$, $1 \leq m \leq (q+1)$.
\end{cor}
\begin{remark}
 From 4 of Lemma \ref{ebert}, it follows that the $\frac{(q+1)(q^2+1)}{2}$ lines of $\bigcup_{i = 0}^{\frac{q-1}{2}} \mathcal{S}_{2i}$ are pairwise disjoint. Similarly for the lines of $\bigcup_{i = 0}^{\frac{q-1}{2}} \mathcal{S}_{2i + 1}$. Moreover these two sets of lines cover the same points.
\end{remark}

\subsection{Regular systems of $Q(6, 3)$}\label{subsec242}

Let $Q(6, q)$ be a parabolic quadric of $PG(6, q)$ and let $PGO(7, q)$ be the group of projectivities of $PG(6, q)$ leaving invariant $Q(6, q)$. A $1$-system of $Q(6, q)$ is a set $\mathcal{S}$ of $q^3+1$ lines of $Q(6, q)$ such that no plane of $Q(6, q)$ through a line of $\mathcal{S}$ has a point in common with the remaining lines of $\mathcal{S}$. From Lemma \ref{k-system}, the planes of $Q(6, q)$ containing a line of $\mathcal{S}$ form a $(q+1)$-regular system w.r.t. points. If $\mathcal{S}$ is the union of $q^2-q+1$ reguli, let $\mathcal{S}^o$ be the $1$-system obtained by taking the $q^2-q+1$ opposite reguli of $\mathcal{S}$.
\begin{prop}
 If $\mathcal{S}$ is a $1$-system of $Q(6, q)$ obtained from the union of $q^2-q+1$ reguli, then the set of planes of $Q(6, q)$ containing a line of $\mathcal{S} \cup \mathcal{S}^o$ is a $2(q+1)$-regular system w.r.t. points of $Q(6, q)$.
\end{prop}
\begin{proof}
 By construction $\mathcal{S}^o$ is a $1$-system of $Q(6, q)$ and $|\mathcal{S} \cap \mathcal{S}^o| = 0$. Moreover a generator of $Q(6, q)$ containing a line of $\mathcal{S}$ contains no line of $\mathcal{S}^o$, otherwise there would be two lines of $\mathcal{S}$ meeting in one point, a contradiction. Therefore from 4 of Lemma \ref{properties}, the set of generators of $Q(6, q)$ containing a line of $\mathcal{S} \cup \mathcal{S}^o$ is a $2(q+1)$-regular system w.r.t. points of $Q(6, q)$.
\end{proof}
Below we give a construction of a $1$-system of $Q(6, 3)$ that is the union of $7$ reguli and it is not contained in an elliptic hyperplane of $Q(6,3)$. The group $PGO(7, 3)$ has two orbits on points of $PG(6, 3)$ not on $Q(6, 3)$, i.e., the set $I$ of \textit{internal points} and the set $E$ of \textit{external points}. Let $\perp$ be the orthogonal polarity of $PG(6, 3)$ defining $Q(6, 3)$. If $P \in PG(6, 3) \setminus Q(6, 3)$, we have that $P^\perp \cap Q(6, 3)$ is an elliptic quadric $Q^-(5, 3)$ or a hyperbolic quadric $Q^+(5, 3)$, according as $P$ belongs to $I$ or $E$, respectively.
\begin{lem}
 Consider the parabolic quadric
 $$Q(6,3):X_{0}^{2}+X{1}^{2}+X_{2}^{2}+X_{3}^{2}+X_{4}^{2}+X_{5}^{2}+X_{6}^{2}=0.$$
 The points $P_{i}$, $0\leq i\leq6$, are internal and they form a self-polar simplex $\mathcal{X}$ of $Q(6, 3)$, i.e., $P_i^\perp= \langle P_j: \;\; 0 \leq j \leq 6,  i \neq j\rangle$, with $0 \leq i \leq 6$.
\end{lem}
\begin{proof}
 Consider $P_{0}=\langle(1,0,0,0,0,0,0)\rangle$. The hyperplane $P_{0}^{\perp}: X_{0}=0$ cut on $Q(6,3)$ a section
  \begin{equation}
  \label{sectionQ63}
  \begin{cases}
   X_{0}=0\\
   X_{1}^{2}+X_{2}^{2}+X_{3}^{2}+X_{4}^{2}+X_{5}^{2}+X_{6}^{2}=0.
  \end{cases}
  \end{equation}
  Since $q\equiv -1\pmod4$ Equation \eqref{sectionQ63} gives an elliptic quadric $Q^{-}(5,3)$. An alternative way of proving consists on counting solutions of Equation \eqref{sectionQ63}. All such points have $X_{0}=0$. Moreover, since $1^{2}=(-1)^{2}=1$ in $F_{3}$, 3 divides the number of non zero coordinates.  We have $\frac{2^{6}}{2}=32$ projective points with 6 non zero entries. The number of sets of 3 non zero entries in $F_{3}$ is $\frac{2^{3}}{2}=4$, being multiplied by all the $\frac{6!}{3!\cdot3!}=20$ permutations on 6 elements, 3 zeros and 3 non zeros, represented by the anagrams of the word 000111. The total number of points is $32+4\cdot20=112$, i.e. the number of points of a non-degenerate elliptic quadric in $PG(5,3)$, so $P_{0}$ is an internal point. Moreover $P_{0}^{\perp}: X_{0}=0$ contains all the other 6 linearly independent points of $\mathcal{X}$, which generate the hyperplane. In an analogue way we prove the thesis for all points $P_{i}$.
 \end{proof}
\begin{cons}\label{sys_para}
  Consider the self-polar simplex $\mathcal{X}$ of $Q(6, 3)$. A line containing two points of $\mathcal{X}$ is external to $Q(6, 3)$, a plane $\sigma$ spanned by three points of $\mathcal{X}$ meets $Q(6, 3)$ in a non-degenerate conic and $\sigma^\perp \cap Q(6, 3)$ is a hyperbolic quadric $Q^+(3,3)$. Let $\pi = \langle P_1, P_2, P_3 \rangle$ and consider the following $7$ lines: $r_1 = \langle P_1, P_2 \rangle$, $r_2 = \langle P_2, P_3 \rangle$, $r_3 = \langle P_3, P_1 \rangle$, $\ell_1 = \langle P_4, P_5 \rangle$, $\ell_1' = \langle P_6, P_7 \rangle$, $\ell_2 = \langle P_4, P_7 \rangle$, $\ell_2' = \langle P_5, P_6 \rangle$, $\ell_3 = \langle P_4, P_6 \rangle$, $\ell_3' = \langle P_5, P_7 \rangle$. Let $\varphi$ be a permutation of $\{1, 2, 3\}$. Let $\mathcal{R}_i$ be one of the two reguli of lines of the hyperbolic quadric $\langle r_i, \ell_{\varphi(i)} \rangle \cap Q(6, 3)$, $\mathcal{R}_i'$ one of the two reguli of $\langle r_i, \ell_{\varphi(i)}' \rangle \cap Q(6, 3)$ and $\mathcal{R}$ one of the two reguli of $\pi^\perp \cap Q(6, 3)$. Set $\mathcal{S} =(\bigcup_{i = 1}^{3} \mathcal{R}_i \cup \mathcal{R}_i') \cup \mathcal{R}$.
\end{cons}
\begin{prop}
 The set $\mathcal{S}$ is a $1$-system of $Q(6, 3)$.
\end{prop}
\begin{proof}
 It is enough to show that the $5$-space generated by two $3$-spaces $T_1$, $T_2$, containing two of the seven reguli of $\mathcal{S}$ meets $Q(6,3)$ in an elliptic quadric $Q^-(5, 3)$. By construction, $T_{1}$ and $T_{2}$ meet in a line containing two points of $\mathcal{X}$. Therefore, $\langle T_{1}, T_{2} \rangle$ is a $5$-space, containing six points of $\mathcal{X}$ and hence $\langle T_{1}, T_{2} \rangle^{\perp}$ is the remaining point of $\mathcal{X}$, that is an internal point of $Q(6, 3)$. It follows that $\langle T_{1},T_{2} \rangle \cap Q(6, 3)$ is an elliptic quadric $Q^{-}(5, 3)$.
\end{proof}
\begin{cor}
 Let $\mathcal{S}^o = (\bigcup_{i = 1}^{3} \mathcal{R}_i^o \cup \mathcal{R}_i'^o) \cup \mathcal{R}^o$, where $\mathcal{R}_i^o$, $\mathcal{R}_i'^o$ and $\mathcal{R}^o$ are the opposite regulus of $\mathcal{R}_i$, $\mathcal{R}_i'$ and $\mathcal{R}$, respectively. Then the set of generators of $Q(6,3)$ containing a line of $\mathcal{S} \cup \mathcal{S}^o$ is an $8$-regular system of $Q(6, 3)$ w.r.t. points.
\end{cor}
Some computations performed with \textit{MAGMA} \cite{magma} show that, up to projectivities, there are two distinct $1$-systems of $Q(6, 3)$ arising from Construction \ref{sys_para}. The former is stabilized by a subgroup of $PGO(7, 3)$ of order $1344$ isomorphic to $2^3 \cdot PSL(3, 2)$ and it acts transitively on members of $\mathcal{S}$. Hence in this case $\mathcal{S}$ is the $1$-system described in \cite[Section 3.4]{BP}, \cite{DHV}. The latter is stabilized by a subgroup of $PGO(7, 3)$ of order $192$ that is not transitive on lines of $\mathcal{S}$. To the best of our knowledge it is new.

\section{Regular systems arising from field reduction}\label{sec25}
Field reduction map, see Section \ref{sec21}, is a tool which allows one to construct Desarguesian $m$-spreads of $PG(V')$ \cite{Segre0} or $\mathcal{P}'$ \cite{D1,D2,D3,D,D4,D5} and \textit{classical} $m$-systems of $\mathcal{P}'$, see \cite{ST, ST1}. In our setting if $U$ denotes the set of singular vectors (totally singular subspaces) of a non-degenerate quadratic form $Q$ on $V$ or isotropic vectors (totally isotropic subspaces) of a non-degenerate reflexive sesquilinear form $\beta$ on $V$, then $\phi(U) = \{\phi(u)|u \in U\}$ are vectors (subspaces) of $V'$ which are either singular (totally singular) with respect to a quadratic form of $V'$ or isotropic (totally isotropic) with respect to a reflexive sesquilinear form on $V'$. Moreover, if $f$ is a semilinear map of $V$, then there exists a unique semilinear map $f'$ of $V'$ such that $\phi f = f' \phi$. Analagously, a semi-similarity of $Q$ or $\beta$ gives rise to a semi-similarity of the corresponding form of $V'$. Hence there is a natural embedding of a subgroup of $\Gamma L(V)$ in $\Gamma L(V')$ and of a subgroup of semi-similarities of $Q$ or $\beta$ in the group of semi-similarities of the corresponding form of $V'$. The subgroups of $\Gamma L(V')$ arising in this way belong to the third-class described by Aschbacher \cite{A}. From a projective point of view $\phi$ maps $(k-1)$-spaces of $PG(V)$ to certain $(kn-1)$-spaces of $PG(V')$ and if $\mathcal{P}$ denotes the non-degenerate polar space of $PG(V)$ associated with $Q$ or $\beta$, then $\phi$ sends $(k-1)$-spaces of $\mathcal{P}$ to certain $(kr-1)$-spaces of a polar space $\mathcal{P}'$ of $PG(V')$.The content of this section is based on the following observation. If $G$ is a group of collineations of $\mathcal{P}'$ acting transitively on points of $\mathcal{P}'$ and $\mathcal{O}$ is an orbit on the generators of $\mathcal{P}'$ under the action of $G$, then through each point of $\mathcal{P}'$ there will be a constant number of elements of $\mathcal{O}$, i.e., $\mathcal{O}$ is a regular system w.r.t. points of $\mathcal{P}'$. See also \cite{BE}. On the other hand, if $\mathcal{D}$ is a Desarguesian $m$-spread of $\mathcal{P}'$, by Witt's Theorems it follows that there exists a group of projectivities of $\mathcal{P}'$, say $G_{\mathcal{D}}$, preserving $\mathcal{D}$ and acting transitively on the points of $\mathcal{P}'$. Therefore every orbit of $G_{\mathcal{D}}$ on the generators of $\mathcal{P}'$ is a regular system w.r.t. points of $\mathcal{P}'$. In particular in order to determine regular systems of $\mathcal{P}'$ w.r.t. points arising in this way, a comprehensive study of the action of $G_{\mathcal{D}}$ on the generators of $\mathcal{P}'$ is needed.
\subsection{Subgeometries embedded in $H(2n-1,q^2)$}\label{subsec251}
Here we focus on the case when $\mathcal{P}$ is a non-degenerate Hermitian variety $H(N-1, q^2)$ of $PG(N-1, q^2)$ and $\mathcal{P}'$ is a non-degenerate quadric $Q$ of $PG(2N-1, q)$ with a Desarguesian line-spread $\mathcal{L}_1$. In the projective orthogonal group of $Q$ there are two subgroups isomorphic to $\frac{SU(N, q)}{\langle -I \rangle}$ and $\frac{GU(N, q)}{\langle -I \rangle}$ preserving $\mathcal{L}_1$. The action of these two subgroups on generators of $Q$ is determined. Some preliminary results of the following sections overlap the content of other papers (see for instance \cite{D}); however we prefer to include here a proof for the convenience of the reader. Let $H(2n-1, q^2)$ be the non-degenerate Hermitian variety of $PG(2n-1, q^2)$ associated with the Hermitian form on $V(2n, q^2)$ given by
$$(X_0, \ldots, X_{2n-1})J(Y_0^q, \ldots, Y_{2n-1}^q )^T,$$
where $J^q = J^T$ and let $\perp$ be the corresponding Hermitian polarity of $PG(2n-1, q^2)$. Consider the following groups of unitary isometries
\begin{align*}
 GU(2n, q) = \{M \in GL(2n, q^2)| M^T J M^q = J\}, \\
 SU(2n, q) = \{M \in SL(2n, q^2)| M^T J M^q = J\},
\end{align*}
and the groups of projectivities induced by them, namely $PGU(2n, q)$ and $PSU(2n, q)$, respectively. Let $\Sigma$ be a Baer subgeometry of $PG(2n-1, q^2)$, i.e. a subgeometry of $PG(2n-1,q^2)$ isomorphic to $PG(2n-1,q)$ and let $\tau$ be the Baer involution of $PG(2n-1, q^2)$ fixing pointwise $\Sigma$. We say that $H(2n-1, q^2)$ induces a polarity on $\Sigma$ if $\perp$, when restricted to $\Sigma$, defines a polarity of $\Sigma$, or equivalently if $P^\perp \cap \Sigma$ is a hyperplane of $\Sigma$, for every point $P \in \Sigma$.
\begin{lem}
 $H(2n-1, q^2)$ induces a polarity on $\Sigma$ if and only if $\tau$ fixes $H(2n-1, q^2)$.
\end{lem}
\begin{proof}
 Suppose first that $H(2n-1,q^2)$ induces a polarity on $\Sigma$. Since $PGL(2n-1, q^2)$ acts transitively on the Baer subgeometries of $PG(2n-1, q^2)$, we may assume w.l.o.g. that $\Sigma$ is the canonical subgeometry of $PG(2n-1,q^2)$. Hence $H(2n-1, q^2)$ is associated with the Hermitian form $(X_0, \ldots, X_{2n-1})J(Y_0^q, \ldots, Y_{2n-1}^q )^T$, whereas the polarity of $\Sigma$ is defined by the form $(X_0, \ldots, X_{2n-1})J(Y_0, \ldots, Y_{2n-1} )^T$. Thus $J^q = J^T = \pm J$. Since $I^T J I = \pm J^q$, $\tau$ fixes $H(2n-1,q^2)$. On the other hand, let $\tau$ be the composition of the projectivity represented by a matrix $M \in GL(2n,q^2)$ with the non-linear involution sending the point $P$ to the point $P^q$. Since $\tau$ fixes $H(2n-1,q^2)$, we have that $M^T J M^q = \lambda J^q$, for some $\lambda \in F_q \setminus \{0\}$. Replacing $M$ with $a M$, where $a^{q+1} = \lambda$, we have that $M^T J M^q = J^q$. We claim that $J M^q = M^{-T} J^q$ defines a polarity of $\Sigma$. Indeed, since $\tau$ is an involution, then $M M^q = \sigma I$, for some $\sigma \in F_{q^2} \setminus \{0\}$ and
 $$\sigma J (P^q)^T = J (M (P^q)^T)^q = J M^q P^T = M^{-T} J^q  {P}^T = M^{-T} (J (P^q)^T)^q.$$
 This means that $P^\perp = (P^\tau)^\perp = (P^\perp)^\tau$. In particular $P^\perp$ is fixed by $\tau$ and hence $P^\perp$ is a hyperplane of $\Sigma$, as required.
\end{proof}

Note that if $H(2n-1, q^2)$ induces a polarity on $\Sigma$, then $\perp|_{\Sigma}$ is either symplectic or pseudo-symplectic if $q$ is even and is either symplectic or orthogonal if $q$ is odd. Indeed, with the same notation used above, if $M M^q = \sigma I$ and $M^T J M^q = J^q$, then $M M^q = M J^{-1} M^{-T} J^q = \sigma I$ and hence $J M^{-1} J^{-T} M^T = \sigma I$. Thus $M J^{-1} = \sigma (J^{q})^{-1} M^T$ and $J^{-T} M^T = \sigma M J^{-1}$. Therefore $M J^{-1} = \sigma (J^{q})^{-1} M^T = \sigma J^{-T} M^T = \sigma (\sigma M J^{-1}) = \sigma^2 M J^{-1}$ and $\sigma^2 = \pm 1$. The polarity $\perp|_{\Sigma}$ of $\Sigma$, when extended in $PG(2n-1, q^2)$ gives rise to a polarity $\perp'$ of $PG(2n-1, q^2)$ associated with the non-degenerate bilinear form given by $(X_0, \ldots, X_{2n-1})J(Y_0, \ldots, Y_{2n-1} )^T$ such that
\begin{enumerate}
 \item $\perp'$ is of the same type as $\perp|_{\Sigma}$,
 \item $\perp \perp' = \perp' \perp = \tau$.
\end{enumerate}
The polarity $\perp'$ is said to be \textit{permutable} with $\perp$, see also \cite{Segre}. We say that $\Sigma$ is embedded in $H(2n-1, q^2)$ if $P \in H(2n-1, q^2)$, for every point $P \in \Sigma$.
\begin{lem}
 $\Sigma$ is embedded in $H(2n-1, q^2)$ if and only if $H(2n-1, q^2)$ induces a symplectic polarity on $\Sigma$.
\end{lem}
\begin{proof}
 If $\Sigma$ is embedded in $H(2n-1, q^2)$ and $P$ is a point of $H(2n-1, q^2)$ with $P \notin \Sigma$, then there exists a unique line $\ell$ of $PG(2n-1, q^2)$ such that $P \in \ell$ and $|\ell \cap \Sigma| = q+1$. Moreover $P^\tau \in \ell$ and $P^\tau \in H(2n-1, q^2)$, since $\ell$ is a line of $H(2n-1, q^2)$. Therefore $\tau$ fixes $H(2n-1, q^2)$, $H(2n-1, q^2)$ induces a polarity on $\Sigma$ and such a polarity has to be symplectic since $\Sigma \subset H(2n-1, q^2)$. Vice versa if $H(2n-1, q^2)$ induces a symplectic polarity on $\Sigma$, then $\Sigma$ is embedded in $H(2n-1, q^2)$.
\end{proof}
\begin{prop}
 There are $q^{n^2-n} \prod_{i = 2}^{n} (q^{2i-1} + 1)$ Baer subgeometries embedded in $H(2n-1, q^2)$.
\end{prop}
\begin{proof}
 First observe that if $\ell$ is a line containing $q+1$ points of $H(2n-1, q^2)$, then $\ell^\perp \cap H(2n-1, q^2) = H(2n-3, q^2)$ and through a Baer subgeometry $\Sigma$ isomorphic to $PG(2n-3, q)$ embedded in $H(2n-3, q^2)$ there pass exactly $q+1$ Baer subgeometries isomorphic to $PG(2n-1,q)$ and embedded in $H(2n-1, q^2)$. Indeed, if $\Sigma'$ is a Baer subgeometry embedded in $H(2n-1, q^2)$, with $\Sigma \subset \Sigma'$, then $\ell \cap H(2n-1, q^2) = \ell \cap \Sigma'$ and every line joining a point $P' = \ell \cap \Sigma'$ and a point $P \in \Sigma$ meets $\Sigma'$ in a Baer subline. On the other hand, on the line $P P'$ there are exactly $q+1$ Baer sublines containing both $P$ and $P'$ and if $s$ is one such Bear subline, then there is a unique Baer subgeometry $\Sigma'$ containing $\Sigma$ and $s$. Moreover since every line joining a point of $\ell \cap \Sigma'$ and a point of $\Sigma$ is a line of $H(2n-1, q^2)$ we have that these subgeometries are embedded in $H(2n-1, q^2)$. Let $x_n$ be the number of Baer subgeometries embedded in $H(2n-1, q^2)$. Since the number of $H(2n-3, q^2)$ in $H(2n-1, q^2)$ equals $\frac{q^{4n-4}(q^{2n-1}+1)(q^{2n}-1)}{(q+1)(q^2-1)}$ and the number of $W(2n-3, q)$ in $W(2n-1, q)$ equals $\frac{q^{2n-2}(q^{2n}-1)}{q^2-1}$, a standard double counting argument on couples $(\Sigma, \Sigma')$, where $\Sigma$ is a Baer subgeometry isomorphic to $PG(2n-3, q)$ embedded in some $H(2n-3, q^2) \subset H(2n-1, q^2)$, $\Sigma'$ is a Baer subgeometry isomorphic to $PG(2n-1, q)$ and embedded in $H(2n-1, q^2)$, with $\Sigma \subset \Sigma'$, gives $x_{n} = q^{2n-2}(q^{2n-1}+1) x_{n-1}$. The result follows from the fact that $x_1 = 1$.
\end{proof}
\begin{prop}\label{stabilizer1}
 The group $PGU(2n, q)$ acts transitively on Baer subgeometries embedded in $H(2n-1, q^2)$.
\end{prop}
\begin{proof}
 Since $PGL(2n-1, q^2)$ is transitive on non-degenerate Hermitian varieties of $PG(2n-1, q^2)$, we may assume w.l.o.g. that $J$ is the matrix
 \begin{equation*}\label{hermitian_matrix}
 \left(
     \begin{array}{ccccc}
       0 & \xi & \cdots & 0 & 0 \\
       \xi^{q} & 0 & \cdots & 0 & 0 \\
       \vdots & \vdots & \ddots & \vdots & \vdots \\
       0 & 0 & \cdots & 0 & \xi \\
      0 & 0 & \cdots & \xi^{q} & 0 \\
     \end{array}
   \right)
 \end{equation*}
 where $0 \neq \xi \in F_{q^2}$ is such that $\xi^q = -\xi$. Let $\Sigma$ be the canonical Baer subgeometry of $PG(2n-1, q^2)$. Then $\Sigma$ is embedded in $H(2n-1, q^2)$ and a symplectic polarity is induced on $\Sigma$. We denote by $\tilde\Sigma$ the set of all non-zero vectors of $V(2n, q^2)$ representing the points of $\Sigma$. If $S \in GU(2n, q)$ and $S^q = S$, then $S$ fixes $\tilde\Sigma$ and hence $S\in Sp(2n, q)$. The group $Sp(2n,q)$ is a subgroup of $GU(2n, q)$ stabilizing $\tilde\Sigma$. Consider $N \in GU(2n, q)$ such that $N$ stabilizes $\tilde\Sigma$ and let $u_i$ be the vector having $1$ in the $i$-th position and $0$ elsewhere. Then $N u_i^T = \rho_i x_i^T$, $1 \leq i \leq 2n$, where $\rho_i \in F_{q^2} \setminus \{0\}$ and $x_i^q = x_i$. Since there exists a matrix $S \in Sp(2n, q)$ such that $S^{-1} u_i^T = x_i^T$, we have that $SN u_i^T= \rho_i u_i^T$. Moreover $SN$ stabilizes $\tilde\Sigma$. Therefore $\rho_i = \lambda_i \rho_1$, $2 \leq i \leq 2n$, where $\lambda_i \in F_q \neq \{0\}$ and $SN = \rho_1 diag(1, \lambda_2, \dots, \lambda_{2n})$. Straightforward calculations show that $(SN)^T J (SN)^q = J$ if and only if $\rho_1 = \lambda_2^{-1}$ and $\lambda_{2j} = \lambda_2 \lambda_{2j-1}^{-1}$, $2 \leq j \leq n$, whereas $diag(1, \lambda, \lambda_3, \lambda_2 \lambda_3^{-1}, \ldots, \lambda_{2n-1}, \lambda_2 \lambda_{2n-1}^{-1}) \in Sp(2n, q^2)$ if and only if $\lambda_2 = 1$. It follows that
 $$Stab_{GU(2n, q^2)}(\tilde\Sigma) = \langle \Delta, Sp(2n, q) \rangle,$$
 where $\Delta = \{\Delta_{\rho} = diag(\rho, \rho^{-q}, \dots, \rho, \rho^{-q}) | \rho \in F_{q^2} \setminus \{0\}\}$. Since $\Delta_{\rho} \in Sp(2n, q)$ if and only if $\rho \in F_q \setminus \{0\}$, we have that $|Stab_{GU(2n, q)}(\tilde\Sigma)| = (q+1) |Sp(2n, q)| = q^{n^2} (q+1) \prod_{i = 1}^{n} (q^{2i}-1)$. The centre of $\langle \Delta, Sp(2n, q) \rangle$ consists of $\{\Delta_{\rho} | \rho^{q+1} = 1\}$. Hence $|Stab_{PGU(2n, q)}(\Sigma)| = q^{n^2} \prod_{i = 1}^{n} (q^{2i}-1)$. The result follows from the fact that $\frac{|PGU(2n, q)|}{|Stab_{PGU(2n, q)}(\Sigma)|}$ equals the number of Baer subgeometries embedded in $H(2n-1, q^2)$.
\end{proof}
\begin{prop}\label{stabilizer2}
 The group $PSU(2n, q)$ has $gcd(q+1, n)$ equally sized orbits on Baer subgeometries embedded in $H(2n-1, q^2)$.
\end{prop}
\begin{proof}
 With the same notation used in Proposition \ref{stabilizer1}, since $Sp(2n, q) \leq SU(2n, q)$ and $\Delta \cap SU(2n, q) = \overline{\Delta} = \{\Delta_{\rho} |\rho^{n(q-1)} = 1\}$, we have that $Stab_{SU(2n, q)}(\tilde\Sigma) = \langle \overline{\Delta}, Sp(2n, q) \rangle$. Moreover $|Stab_{SU(2n, q)}(\tilde\Sigma)|= gcd(q+1, n) q^{n^2} $\\
 $\prod_{i = 1}^{n}(q^{2i}-1)$, since $|\overline{\Delta}| = (q-1) gcd(q+1, n)$ and $|\overline{\Delta} \cap Sp(2n, q)| = q-1$. The centre of $Stab_{SU(2n, q)}(\tilde\Sigma)$ consists of $\{\Delta_{\rho}  | \rho^{n(q-1)} = \rho^{q+1} = 1\}$ and hence has size $gcd(q+1, 2n)$. It follows that $|Stab_{PSU(2n, q)}(\Sigma)| = \frac{gcd(q+1, n) q^{n^2} \prod_{i = 1}^{n} (q^{2i}-1)}{gcd(q+1, 2n)}$ and $|\Sigma^{PSU(2n, q)}| =  \frac{q^{n^2-n} \prod_{i = 2}^{n} (q^{2i-1} + 1)}{gcd(q+1, n)}$. If $\Sigma'$ is a subgeometry embedded in $H(2n-1, q^2)$, then $Stab_{SU(2n, q)}(\tilde\Sigma')$ is conjugate to $Stab_{SU(2n, q)}(\tilde\Sigma)$ in $GU(2n, q)$ and $Stab_{PSU(2n, q)}(\Sigma')$ is conjugate to $Stab_{PSU(2n, q)}(\Sigma)$ in $PGU(2n, q)$. Moreover $SU(2n, q) \trianglelefteq GU(2n, q)$ and $PSU(2n, q) \trianglelefteq PGU(2n, q)$. Hence every orbit of $PSU(2n, q)$ on Baer subgeometries embedded in $H(2n-1, q^2)$ has the same size.
\end{proof}

\subsection{The Desarguesian line-spread of a quadric in odd dimension}\label{subsec252}
Let $\alpha \in F_q$ such that $X^2 - X - \alpha$ is irreducible over $F_q$. Let $w \in F_{q^2}$ such that $w^2 = w + \alpha$. Then every element $x \in F_{q^2}$ can be uniquely written as $x = y + w z$, where $y, z \in F_q$. Moreover $w + w^q = 1$ and $w^{q+1} = - \alpha$. Consider the \textit{field reduction map}:
$$\phi: x = (x_0, \ldots, x_{N-1}) \in V(N, q^2) \longmapsto y = (y_0, z_0, \ldots, y_{N-1}, z_{N-1}) \in V(2N, q),$$
see \cite{V}.
Thus
$$\phi ( \langle (x_0, \ldots, x_{N-1}) \rangle_{q^2}) = \langle (y_0, z_0, \ldots, y_{N-1}, z_{N-1}) \rangle_{q}.$$
Moreover
\footnotesize
$$\phi (\langle (x_0, \ldots, x_{N-1}) \rangle_{q^2} ) = \langle y= (y_0, z_0, \dots, y_{N-1}, z_{N-1}), \tilde{y} = (\alpha z_0, y_0 + z_0, \ldots, \alpha z_{N-1}, y_{N-1} + z_{N-1}) \rangle_{q}$$
\normalsize
and
$$\mathcal{L} = \{ \phi ( \langle x \rangle_{q^2}) | 0 \neq x \in V(N, q^2) \}$$
is a Desarguesian line-spread of $PG(2N-1, q^2)$. Let $\xi = 2w - 1$. Hence $\xi^q = -\xi$ and $\xi^2 = 4 \alpha + 1$.

\subsection{The hyperbolic case}\label{subsec253}
The following very detailed example gives the image via field reduction of the Hermitian polarity $H(1,q^{2})$ on the \textit{Galois line} $PG(1,q^{2})$, i.e. a Baer subline.
\begin{exmp}
 Let $H(1,q^{2})$ be the Hermitian variety of equation
 \begin{equation}
  X_{0}X_{1}^{q}-X_{0}^{q}X_{1}=0.
 \end{equation}
 Now $X_{0}=x_{1}+\omega x_{2}$, $X_{1}=x_{3}+\omega x_{4}$, and so
 $$0=X_{0}X_{1}^{q}-X_{0}^{q}X_{1}=(x_{1}+\omega x_{2})(x_{3}+\omega^{q}x_{4})-(x_{1}+\omega^{q}x_{2})(x_{3}+\omega x_{4})=$$
 $$=x_{1}x_{3}+\omega x_{2}x_{3}+\omega^{q}x_{1}x_{4}+\omega^{q+1}x_{2}x_{4}-x_{1}x_{3}-\omega^{q}x_{2}x_{3}-\omega x_{1}x_{4}-\omega^{q+1}x_{2}x_{4}=$$
 $$=(\omega^{q}-\omega)(x_{1}x_{4}-x_{2}x_{3}).$$
 By the coefficient $\omega^{q}-\omega$ we get a hyperbolic quadric with equation
 \begin{equation}
  Q^{+}(3,q):x_{1}x_{4}-x_{2}x_{3}=0.
 \end{equation}
 The points $U_{1}=(1,0,0,0)$, $U_{2}=(0,1,0,0)$, $U_{3}=(0,0,1,0)$ and $U_{4}=(0,0,0,1)$ belong to the quadric $Q^{+}(3,q)$. That hyperbolic quadric has $(q+1)^{2}$ points and $2(q+1)$ lines, divided into two reguli of $q+1$ lines, and every regulus form a spread of the quadric. Let $U_{1}U_{2}$ be the line of equation:
 \begin{equation}
  \begin{cases}
   x_{3}=0\\
   x_{4}=0.
  \end{cases}
 \end{equation}
 We have a parametrization of the line in the $q$ points $(1,\lambda,0,0)$, together with the \textit{point at infinity} $U_{2}$.
 In an analogue way we take another line of the regulus, $U_{3}U_{4}$ of equation:
 \begin{equation}
  \label{111}
  \begin{cases}
   x_{1}=0\\
   x_{2}=0.
  \end{cases}
 \end{equation}
 The parametrization consists of $q$ points $(0,0,1,\lambda)$, for $\lambda\in F_{q}$, together with the point at infinity $U_{4}$.
 It can be easily seen that the lines $U_{1}U_{2}$, $U_{1}U_{3}$, $U_{2}U_{4}$ and $U_{3}U_{4}$ completely lie on the quadric, while $U_{1}U_{4}$ and $U_{2}U_{3}$ are two secant lines. Obviously $U_{1}U_{2}$ and $U_{3}U_{4}$ belong to one regulus; $U_{1}U_{3}$ and $U_{2}U_{4}$ to the opposite one.

 The first $q$ lines of the first regulus have equation ($\mu\in F_q$):
 \begin{equation}
  \label{112}
  \begin{cases}
   x_{3}=\mu x_{1}\\
   x_{4}=\mu x_{2}.
  \end{cases}
 \end{equation}
 Every line has a parametrization that consists of $q$ points $(1,\lambda,\mu,\mu\lambda)$,  together with the point at infinity $(0,1,0,\mu)$. The $(q+1)$-st line is $U_{3}U_{4}$. In an analogue way we find the opposite regulus: the first $q$ lines have equation:
 \begin{equation}
  \begin{cases}
   x_{2}=\mu x_{1}\\
   x_{4}=\mu x_{3}.
  \end{cases}
 \end{equation}
 The parametrization consists of $q$ points $(1,\mu,\lambda,\mu\lambda)$, together with the point at infinity $(0,0,1,\mu)$. The $(q+1)$-st line is $U_{2}U_{4}$. The $q+1$ projective points lying on $H(1,q^{2})$ are the $q$ points $(1,a)$ together with the point $(0,1)$, so the automorphism group of $H$ consists of $2\times2$ invertible matrices in $GL(2,q)$. We can now represent the lifted matrices belonging by field reduction to $\alpha M$: the action on a point $(\mu, \lambda)$ is the following:
  \begin{equation*}\left(
     \begin{array}{cc}
       \alpha a & \alpha b \\
       \alpha c & \alpha d \\
     \end{array}
   \right)\left(
     \begin{array}{c}
       \mu  \\
       \lambda \\
     \end{array}
   \right)=\left(
     \begin{array}{c}
       \alpha (a\mu+b\lambda) \\
       \alpha (c\mu+d\lambda) \\
     \end{array}
   \right).\end{equation*}
  Let $\alpha=\alpha_{0}+\omega\alpha_{1}$, $\mu=x_{1}+\omega x_{2}$ and $\lambda=x_{3}+\omega x_{4}$, with $\alpha_{0}, \alpha_{1}, x_{1}, x_{2}, x_{3}, x_{4}\in F_q$ and take $\omega^{2}\in F_q^{2}$ as $A+B\omega$, with $A, B\in F_q$.  We write the previous as:\\
 \begin{equation*}\left(
   \begin{array}{cc}
    \alpha_{0}a+\omega\alpha_{1}a & \alpha_{0}b+\omega\alpha_{1}b \\
    \alpha_{0}c+\omega\alpha_{1}c & \alpha_{0}d+\omega\alpha_{1}d \\
   \end{array}
  \right)\left(
   \begin{array}{c}
    x_{1}+\omega x_{2} \\
    x_{3}+\omega x_{4} \\
   \end{array}
  \right)=\end{equation*}
  \begin{equation*}
  =\left(
   \begin{array}{c}
    (\alpha_{0}a+\omega\alpha_{1}a)(x_{1}+\omega x_{2})+(\alpha_{0}b+\omega\alpha_{1}b)(x_{3}+\omega x_{4}) \\
    (\alpha_{0}c+\omega\alpha_{1}c)(x_{1}+\omega x_{2})+(\alpha_{0}d+\omega\alpha_{1}d)(x_{3}+\omega x_{4}) \\
   \end{array}
  \right).\end{equation*}

 The point $(\mu, \lambda)$ is sent in
 \footnotesize
 \begin{equation*}\left(
    \begin{array}{c}
     (\alpha_{0}ax_{1}+\omega\alpha_{0}ax_{2}+\omega\alpha_{1}ax_{1}+\omega^{2}\alpha_{1}ax_{2})+(\alpha_{0}bx_{3}+\omega\alpha_{0}bx_{4}+\omega\alpha_{1}bx_{3}+\omega^{2}\alpha_{1}bx_{4}) \\
     (\alpha_{0}cx_{1}+\omega\alpha_{0}cx_{2}+\omega\alpha_{1}cx_{1}+\omega^{2}\alpha_{1}cx_{2})+(\alpha_{0}dx_{3}+\omega\alpha_{0}dx_{4}+\omega\alpha_{1}dx_{3}+\omega^{2}\alpha_{1}dx_{4}) \\
    \end{array}
   \right),\end{equation*}
   \normalsize
 that, for $\omega^{2}=A+B\omega$, with some further calculation, becomes:
 \footnotesize
 \begin{equation*}\left(
   \begin{array}{c}
    (\alpha_{0}ax_{1}+\alpha_{1}aAx_{2}+\alpha_{0}bx_{3}+\alpha_{1}bAx_{4})+\omega(\alpha_{1}ax_{1}+(\alpha_{0}a+\alpha_{1}aB)x_{2}+\alpha_{1}bx_{3}+(\alpha_{0}b+\alpha_{1}bB)x_{4})\\
    (\alpha_{0}cx_{1}+\alpha_{1}cAx_{2}+\alpha_{0}dx_{3}+\alpha_{1}dAx_{4})+\omega(\alpha_{1}cx_{1}+(\alpha_{0}c+\alpha_{1}cB)x_{2}+\alpha_{1}dx_{3}+(\alpha_{0}d+\alpha_{1}dB)x_{4})\\
    \end{array}
   \right),\end{equation*}
 \normalsize
 we can so explain the action of the matrix, taking the point $(x_{1},x_{2},x_{3},x_{4})$ as the image of $(\mu,\lambda)$ in $F_q$. The action is:\footnotesize
 \begin{equation*}\left(
     \begin{array}{cccc}
       a_{11} & a_{12} & a_{13} & a_{14} \\
       a_{21} & a_{22} & a_{23} & a_{24} \\
       a_{31} & a_{32} & a_{33} & a_{34} \\
       a_{41} & a_{42} & a_{43} & a_{44} \\
     \end{array}
   \right)\left(
     \begin{array}{c}
       x_{1}  \\
       x_{2} \\
       x_{3}\\
       x_{4} \\
     \end{array}
   \right)=\left(
     \begin{array}{c}
       \alpha_{0}ax_{1}+\alpha_{1}aAx_{2}+\alpha_{0}bx_{3}+\alpha_{1}bAx_{4} \\
       \alpha_{1}ax_{1}+(\alpha_{0}a+\alpha_{1}aB)x_{2}+\alpha_{1}bx_{3}+(\alpha_{0}b+\alpha_{1}bB)x_{4}\\
       \alpha_{0}cx_{1}+\alpha_{1}cAx_{2}+\alpha_{0}dx_{3}+\alpha_{1}dAx_{4}\\
       \alpha_{1}cx_{1}+(\alpha_{0}c+\alpha_{1}cB)x_{2}+\alpha_{1}dx_{3}+(\alpha_{0}d+\alpha_{1}dB)x_{4}\\
     \end{array}
   \right).\end{equation*}
 \normalsize
 The final matrix is
 \begin{equation*}M=\left(
     \begin{array}{cccc}
       \alpha_{0}a & \alpha_{1}aA & \alpha_{0}b & \alpha_{1}bA \\
       \alpha_{1}a & \alpha_{0}a+\alpha_{1}aB & \alpha_{1}b & \alpha_{0}b+\alpha_{1}bB\\
       \alpha_{0}c & \alpha_{1}cA & \alpha_{0}d & \alpha_{1}dA\\
       \alpha_{1}c & \alpha_{0}c+\alpha_{1}cB & \alpha_{1}d & \alpha_{0}d+\alpha_{1}dB\\
     \end{array}
   \right).\end{equation*}
 We focus on the action of the stabiliser subgroup of $H(1,q^{2})$, considering the $q^{2}-1$ matrices of the form:
 \begin{equation*}M=\left(
     \begin{array}{cc}
       \alpha & 0 \\
       0 & \alpha  \\
     \end{array}
   \right),\end{equation*}
 with $\alpha\in F_{q^2}\setminus\{0\}.$\\
 In the previous representation, $a=d=1$ and $b=c=0$, so the matrix is
 \begin{equation*}M=\left(
     \begin{array}{cccc}
       \alpha_{0} & \alpha_{1}A & 0 & 0 \\
       \alpha_{1} & \alpha_{0}+\alpha_{1}B & 0 & 0\\
       0 & 0 & \alpha_{0} & \alpha_{1}A\\
       0 & 0 & \alpha_{1} & \alpha_{0}+\alpha_{1}B\\
     \end{array}
   \right).\end{equation*}
 We already considered the $q+1$ points of $H$ as the $q$ points $(1,a)$ together with $(0,1)$. The images of the isotropic points of $H$ make a Desarguesian spread of the quadric, so they are sent in one of the two reguli of $Q^{+}(3,q)$. The field reduction maps the point $(0,1)$ into the line $U_{3}U_{4}$ of equation (\ref{111}), and the points $(1,\mu)$ into the lines of equation (\ref{112}). The subgroup of the matrices belonging to the stabiliser subgroup of $H$, acts on the lines of the first regulus (so the lines $U_{1}U_{2}$, $U_{3}U_{4}$, $\ldots$) \textit{setwise} fixing them. Now we show that the subgroup acts transitively on the opposite regulus. Here we have the $q$ lines with the parametrization $(1,\mu,\lambda,\mu\lambda)$, together with the point at infinity $(0,0,1,\mu)$, and the $(q+1)$-st line with parametrization $(0,1,0,\lambda)$ together with the point at infinity $U_{4}$.\\
 Let us verify the transitivity of the action on these lines:
 \begin{equation*}\left(
     \begin{array}{cccc}
       \alpha_{0} & \alpha_{1}A & 0 & 0 \\
       \alpha_{1} & \alpha_{0}+\alpha_{1}B & 0 & 0\\
       0 & 0 & \alpha_{0} & \alpha_{1}A\\
       0 & 0 & \alpha_{1} & \alpha_{0}+\alpha_{1}B\\
     \end{array}
   \right)\left(
     \begin{array}{c}
       1  \\
       \mu \\
       \lambda \\
       \mu\lambda \\
     \end{array}
   \right)=\end{equation*}
   \begin{equation*}=\left(
     \begin{array}{c}
       \alpha_{0}+\mu\alpha_{1}A \\
       \alpha_{1}+\mu(\alpha_{0}+\alpha_{1}B)\\
       \lambda(\alpha_{0}+\mu\alpha_{1}A) \\
       \lambda(\alpha_{1}+\mu(\alpha_{0}+\alpha_{1}B))\\
     \end{array}
   \right)=\left(
     \begin{array}{c}
       1\\
       \mu_{1}\\
       \lambda \\
       \lambda\mu_{1}\\
     \end{array}
   \right).\end{equation*}
   So we get another line with $\mu_{1}=\frac{\alpha_{1}+\mu(\alpha_{0}+\alpha_{1}B)}{\alpha_{0}+\mu\alpha_{1}A}$, else if $\alpha_{0}+\mu\alpha_{1}A=0$, we get the $(q+1)$-st line:
   \begin{equation*}\left(
     \begin{array}{c}
       0 \\
       \alpha_{1}+\mu(\alpha_{0}+\alpha_{1}B)\\
       0 \\
       \lambda(\alpha_{1}+\mu(\alpha_{0}+\alpha_{1}B))\\
     \end{array}
   \right)=\left(
     \begin{array}{c}
       0\\
       1\\
       0 \\
       \lambda\\
     \end{array}
   \right).\end{equation*}

 The action on the \textit{point at infinity} is:
 \begin{equation*}\left(
     \begin{array}{cccc}
       \alpha_{0} & \alpha_{1}A & 0 & 0 \\
       \alpha_{1} & \alpha_{0}+\alpha_{1}B & 0 & 0\\
       0 & 0 & \alpha_{0} & \alpha_{1}A\\
       0 & 0 & \alpha_{1} & \alpha_{0}+\alpha_{1}B\\
     \end{array}
   \right)\left(
     \begin{array}{c}
       0  \\
       0 \\
       1 \\
       \mu \\
     \end{array}
   \right)=\left(
     \begin{array}{c}
       0 \\
       0\\
       \alpha_{0}+\mu\alpha_{1}A \\
       \alpha_{1}+\mu(\alpha_{0}+\alpha_{1}B)\\
     \end{array}
   \right)=\left(
     \begin{array}{c}
       0\\
       0\\
       1 \\
       \mu_{1}\\
     \end{array}
   \right),\end{equation*}
   if $\alpha_{0}+\mu\alpha_{1}A=0$, we get $U_{4}$ of the $(q+1)$-st line:
   \begin{equation*}\left(
     \begin{array}{c}
       0 \\
       0\\
       0 \\
       \alpha_{1}+\mu(\alpha_{0}+\alpha_{1}B)\\
     \end{array}
   \right)=\left(
     \begin{array}{c}
       0\\
       0\\
       0 \\
       1\\
     \end{array}
   \right).\end{equation*}
 Now we focus on the $(q+1)$-st line:
 \begin{equation*}\left(
     \begin{array}{cccc}
       \alpha_{0} & \alpha_{1}A & 0 & 0 \\
       \alpha_{1} & \alpha_{0}+\alpha_{1}B & 0 & 0\\
       0 & 0 & \alpha_{0} & \alpha_{1}A\\
       0 & 0 & \alpha_{1} & \alpha_{0}+\alpha_{1}B\\
     \end{array}
   \right)\left(
     \begin{array}{c}
       0  \\
       1 \\
       0 \\
       \lambda \\
     \end{array}
   \right)=\left(
     \begin{array}{c}
       \alpha_{1}A \\
       \alpha_{0}+\alpha_{1}B\\
       \lambda(\alpha_{1}A) \\
       \lambda(\alpha_{0}+\alpha_{1}B)\\
     \end{array}
   \right)=\left(
     \begin{array}{c}
       1\\
       \mu_{2}\\
       \lambda \\
       \lambda\mu_{2}\\
     \end{array}
   \right).\end{equation*}
   So we get another line with $\mu_{2}=\frac{\alpha_{0}+\alpha_{1}B}{\alpha_{1}A}$, else if $\alpha_{1}A=0$, we get the $(q+1)$-st line:
   \begin{equation*}\left(
     \begin{array}{c}
       0 \\
       \alpha_{0}+\alpha_{1}B\\
       0 \\
       \lambda(\alpha_{0}+\alpha_{1}B)\\
     \end{array}
   \right)=\left(
     \begin{array}{c}
       0\\
       1\\
       0 \\
       \lambda\\
     \end{array}
   \right).\end{equation*}
 The action on the point at infinity $U_{4}$ is:
 \begin{equation*}\left(
     \begin{array}{cccc}
       \alpha_{0} & \alpha_{1}A & 0 & 0 \\
       \alpha_{1} & \alpha_{0}+\alpha_{1}B & 0 & 0\\
       0 & 0 & \alpha_{0} & \alpha_{1}A\\
       0 & 0 & \alpha_{1} & \alpha_{0}+\alpha_{1}B\\
     \end{array}
   \right)\left(
     \begin{array}{c}
       0  \\
       0 \\
       0 \\
       1 \\
     \end{array}
   \right)=\left(
     \begin{array}{c}
       0 \\
       0\\
       \alpha_{1}A \\
       \alpha_{0}+\alpha_{1}B\\
     \end{array}
   \right)=\left(
     \begin{array}{c}
       0\\
       0\\
       1 \\
       \mu_{2}\\
     \end{array}
   \right),\end{equation*}
   if $\alpha_{1}A=0$, we fix the point $U_{4}$:
   \begin{equation*}\left(
     \begin{array}{c}
       0 \\
       0\\
       0 \\
       \alpha_{0}+\alpha_{1}B\\
     \end{array}
   \right)=\left(
     \begin{array}{c}
       0\\
       0\\
       0 \\
       1\\
     \end{array}
   \right).\end{equation*}
 We can also prove the transitivity on the second regulus with the \textit{Orbit-stabiliser theorem}.
 Let us take the $(q+1)$-st line with parametrization \\$(0,1,0,\lambda)\cap\{(0,0,0,1)\}$. The action is:
 \begin{equation*}\left(
     \begin{array}{cccc}
       \alpha_{0} & \alpha_{1}A & 0 & 0 \\
       \alpha_{1} & \alpha_{0}+\alpha_{1}B & 0 & 0\\
       0 & 0 & \alpha_{0} & \alpha_{1}A\\
       0 & 0 & \alpha_{1} & \alpha_{0}+\alpha_{1}B\\
     \end{array}
   \right)\left(
     \begin{array}{c}
       0  \\
       1 \\
       0 \\
       \lambda \\
     \end{array}
   \right)=\left(
     \begin{array}{c}
       \alpha_{1}A \\
       \alpha_{0}+\alpha_{1}B\\
       \lambda(\alpha_{1}A) \\
       \lambda(\alpha_{0}+\alpha_{1}B)\\
     \end{array}
   \right)=\left(
     \begin{array}{c}
       1\\
       \mu_{2}\\
       \lambda \\
       \lambda\mu_{2}\\
     \end{array}
   \right).\end{equation*}
 We get the $(q+1)$-st line if $\alpha_{1}A=0$ :
 \begin{equation*}\left(
   \begin{array}{c}
    0 \\
    \alpha_{0}+\alpha_{1}B\\
    0 \\
    \lambda(\alpha_{0}+\alpha_{1}B)\\
   \end{array}
   \right)=\left(
     \begin{array}{c}
       0\\
       1\\
       0 \\
       \lambda\\
     \end{array}
   \right),\end{equation*}
   and in analogue way we proceed on the infinity $U_{4}$.
   Now let $G$ be the group of the $q^{2}-1$ matrices $M$, stabilizing the first regulus. The stabilizing subgroup of the $(q+1)$-st line $U_{2}U_{4}$ needs the condition $\alpha_{1}A=0$, so $\alpha_{1}=0$. The described matrices are:
   \begin{equation*}\left(
     \begin{array}{cccc}
       \alpha_{0} & 0 & 0 & 0 \\
       0 & \alpha_{0} & 0 & 0\\
       0 & 0 & \alpha_{0} & 0\\
       0 & 0 & 0 & \alpha_{0}\\
     \end{array}
   \right)\end{equation*} with $\alpha_{0}\in K\setminus\{0\}$,\\ so $|Stab_{G}(\{U_{2}U_{4}\})|=q-1$. From the \textit{Orbit-stabiliser theorem}
   \begin{equation}
   \label{OS}
    |O_{G}(\{U_{2}U_{4}\})|=\frac{|G|}{|Stab_{G}(\{U_{2}U_{4}\})|}=\frac{q^{2}-1}{q-1}=q+1,
   \end{equation}
   so there is a unique orbit of $q+1$ lines, so $G$ acts transitively on the second regulus.\\
   Indeed in (\ref{OS}) we should divide both numerator and denominator by $q-1$, because the matrices representing the automorphism group of $H(1,q^{2})$ are defined within a non-zero element of $F_q$, giving also the embedding $PGU(2,q)\cong PGL(2,q)$.
\end{exmp}

We are now ready to give the general argument. Assume that $N = 2n$.
Consider the following non-degenerate Hermitian form $H$ on $V(2n, q^2)$
$$H(x, x') = x J_n ({x'}^q)^T = ( x_0, \ldots, x_{2n-1})J_n (x_0'^q, \dots, x_{2n-1}'^q )^T ,$$
where $J_n$ is given by the \eqref{hermitian_matrix}
$$ \left(
     \begin{array}{ccccc}
       0 & \xi & \cdots & 0 & 0 \\
       \xi^{q} & 0 & \cdots & 0 & 0 \\
       \vdots & \vdots & \ddots & \vdots & \vdots \\
       0 & 0 & \cdots & 0 & \xi \\
      0 & 0 & \cdots & \xi^{q} & 0 \\
     \end{array}
   \right).$$
Denote by $J'_{n}$ the skew symmetric matrix $\xi^{-1} J_n$ and consider the following groups of unitary isometries or symplectic similarities
\begin{align*}
 GU(2n, q) = \{M \in GL(2n, q^2)  | M^T J_n M^q = J_n\}, \\
 SU(2n, q) = \{M \in SL(2n, q^2)  | M^T J_n M^q = J_n\}, \\
 Sp(2n, q) = \{M \in GL(2n, q) | M^T J_n' M = J_n'\}.
\end{align*}
We have that
\begin{align*}
 H(x, x) & = \xi (x_0 x_1^q - x_0^q x_1 + \ldots + x_{2n-2} x_{2n-1}^q - x_{2n-2}^q x_{2n-1}) \\
 & = - \xi^2 (y_0 z_1 - z_0 y_1 + \ldots + y_{2n-2} z_{2n-1} - z_{2n-2} y_{2n-1}) = Q(y).
\end{align*}
$Q$ is a non-degenerate quadratic form on $V(4n, q)$ of hyperbolic type. Let $H(2n-1, q^2)$ be the Hermitian variety of $PG(2n-1, q^2)$ determined by $H$. Denote by $Q^+(4n-1, q)$ the hyperbolic quadric of $PG(4n-1, q)$ defined by $Q$ and let $PGO^+(4n, q)$ denote the group of projectivities of $PG(4n-1, q)$ stabilizing $Q^+(4n-1, q)$. It follows that
$$\mathcal{L}_1 = \{ \phi ( \langle x \rangle_{q^2} )  | 0 \neq x \in V(N, q^2), H(x,x) = 0\} \subset \mathcal{L} $$
is a line-spread of $Q^+(4n-1, q)$, see also \cite{D}.
\begin{prop}
 There exists a group $G_n \leq PGO^+(4n, q)$ isomorphic to $\frac{GU(2n, q)}{\langle - I_{2n} \rangle}$ stabilizing $\mathcal{L}$.
\end{prop}
\begin{proof}
 Let $M = (a_{i j}) \in GU(2n, q)$ and hence $H(x M^T, x M^T) = H(x, x)$. Denote by $\overline{M}$ the unique element of $GL(4n, q)$ such that $\phi (x M^T) = \phi(x) \overline{M}^T$, for all $x \in V(2n, q^2)$. Let $G_n$ be the group of projectivities of $PG(4n-1, q)$ associated with the matrices $\{\overline{M} |  M \in GU(2n, q)\}$. Let $g \in G_n$, then $g$ stabilizes $\mathcal{L}$. Moreover $Q(\phi(x)\overline{M}^T) = Q(\phi (x M^T)) = H(x M^T, x M^T) = H(x, x) = Q(\phi(x))$ and hence $g$ fixes $Q^+(4n-1, q)$. Some calculations show that if $M = (a_{i j})$, where $a_{i j} = b_{i j} + w c_{i j}$, then
 $$\overline{M} = ( A_{i j} ), \mbox{ where } A_{i j} =\begin{pmatrix}
 b_{i j} & c_{i j} \alpha \\
 c_{i j} & b_{i j} + c_{i j} \\
 \end{pmatrix}.$$
 It follows that if $\overline{M}$ induces the identity projectivity, then $A_{i j} = 0$, $i \neq j$, and $A_{i i} = \lambda I_{2}$, $1 \leq i \leq 2n$, for some $\lambda \in F_q \setminus \{0\}$. Hence $M = \lambda I_{2n}$. On the other hand, since $M \in GU(2n, q)$, we have that $\lambda = \pm 1$.
\end{proof}
Let $K_n$ be the subgroup of $G_n$ consisting of projectivities of $PG(4n-1, q)$ associated with the set of matrices $\{\overline{M} |M \in SU(2n,q)\}$. We have that $K_n \simeq \frac{SU(2n, q)}{\langle - I_{2n} \rangle}$.
\begin{prop}\label{points}
 The groups $K_n$ and $G_n$ act transitively on points of $Q^+(4n-1, q)$ and lines of $\mathcal{L}_1$.
\end{prop}
\begin{proof}
 By Witt's Theorems the group $GU(2n, q)$ acts transitively on \\$1$-dimensional subspaces of $V(2n, q^2)$ that are totally isotropic with respect to $H$. Hence $G_n$ acts transitively on lines of $\mathcal{L}_1$. Alternatively, let $u = (0, \ldots, 1) \in V(2n, q^2)$ and consider a matrix $M \in GU(2n, q)$. Some straightforward calculations show that $M$ stabilizes $\langle u \rangle_{q^2}$ if and only if it has the following form
 \begin{equation}\label{matrix}
  \begin{pmatrix}
     &      & & a_{1} & 0 \\
     & M' & & \vdots & \vdots \\
     &      & & a_{2n-2} & 0 \\
   0 & \cdots & 0 & a_{2n-1} & 0 \\
   a_1'& \cdots & a_{2n-2}' & a_{2n-1}' & a_{2n-1}^{-q} \\
  \end{pmatrix},
 \end{equation}
 where $M'^T J_{n-1} M'^q = J_{n-1}$, $\det(M') \neq 0$, $a_1, \ldots, a_{2n-2} \in F_{q^2}$, $a_{2n-1} \in F_{q^2} \setminus \{0\}$, $(a_{1}', \ldots, a_{2n-2}')$ satisfies
 $$(a_{1}', \dots, a_{2n-2}') J_{n-1} M'^q = - \xi^q a_{2n-1}^{-q} (a_1^q, \dots, a_{2n-2}^q),$$
 and $a_{2n-1}'$ is such that $\xi a_1' a_2'^q + \xi^q a_1'^q a_2' + \ldots + \xi a_{2n-1}' a_{2n}'^q + \xi^q a_{2n-1}'^q a_{2n}' = 0$. Then $|Stab_{GU(2n, q)}(\langle u \rangle_{q^2})| = q^{4n-3} (q^2-1) |GU(2n-2, q)|$ and $G_n$ is transitive on lines of $\mathcal{L}_1$. In a similar way $M$ stabilizes $\langle u \rangle_{q}$ if and only if $M$ has the form \eqref{matrix}, where $M'^T J_{n-1} M'^q = J_{n-1}$, $det(M') \neq 0$, $a_1, \ldots, a_{2n-2} \in F_{q^2}$, $a_{2n-1} \in F_q \setminus \{0\}$, $(a_{1}', \ldots, a_{2n-2}')$ satisfies
 $$(a_{1}', \ldots, a_{2n-2}') J_{n-1} M'^q = - \xi^q a_{2n-1}^{-q} (a_1^q, \ldots, a_{2n-2}^q),$$
 and $a_{2n-1}'$ is such that $\xi a_1' a_2'^q + \xi^q a_1'^q a_2' + \ldots + \xi a_{2n-1}' a_{2n}'^q + \xi^q a_{2n-1}'^q a_{2n}' = 0$. Then $|Stab_{GU(2n, q)}(\langle u \rangle_{q})| = q^{4n-3} (q-1) |GU(2n-2, q)|$ and $G_n$ is transitive on points of $Q^+(4n-1, q)$. On the other hand, $M \in SU(2n, q)$ stabilizes $\langle u \rangle_{q^2}$ if and only if $M$ has the form \eqref{matrix}, where $M'^T J_{n-1} M'^q = J_{n-1}$, $det(M') \neq 0$, $a_1, \ldots, a_{2n-2} \in F_{q^2}$, $a_{2n-1}^{q-1} = det(M')$, $(a_{1}', \ldots, a_{2n-2}')$ satisfies
 $$(a_{1}', \ldots, a_{2n-2}') J_{n-1} M'^q = - \xi^q a_{2n-1}^{-q} (a_1^q, \dots, a_{2n-2}^q),$$
 and $a_{2n-1}'$ is such that $\xi a_1' a_2'^q + \xi^q a_1'^q a_2' + \ldots + \xi a_{2n-1}' a_{2n}'^q + \xi^q a_{2n-1}'^q a_{2n}' = 0$. Then $|Stab_{SU(2n, q)}(\langle u \rangle_{q^2})| = q^{4n-3} (q-1) |GU(2n-2, q)|$ and $G_n$ is transitive on lines of $\mathcal{L}_1$. In the same fashion $M \in SU(2n, q)$ stabilizes $\langle u \rangle_{q}$ if and only if $M$ has the form \eqref{matrix}, where $M'^T J_{n-1} M'^q = J_{n-1}$, $det(M') = 1$, $a_1, \dots, a_{2n-2} \in F_{q^2}$, $a_{2n-1} \in F_q \setminus \{0\}$, $(a_{1}', \ldots, a_{2n-2}')$ satisfies
 $$(a_{1}', \ldots, a_{2n-2}') J_{n-1} M'^q = - \xi^q a_{2n-1}^{-q} (a_1^q, \ldots, a_{2n-2}^q),$$
 and $a_{2n-1}'$ is such that $\xi a_1' a_2'^q + \xi^q a_1'^q a_2' + \ldots + \xi a_{2n-1}' a_{2n}'^q + \xi^q a_{2n-1}'^q a_{2n}' = 0$. Then $|Stab_{SU(2n, q)}(\langle u \rangle_{q})| = q^{4n-3} (q-1) |SU(2n-2, q)|$ and $G_n$ is transitive on points of $Q^+(4n-1, q)$.
\end{proof}
\begin{thm}\label{hyp}
 The group $G_n$ has $n+1$ orbits, say $\mathcal{O}_{n, i}$, $0 \leq i \leq n$, on generators of $Q^+(4n-1, q)$, where
 $$|\mathcal{O}_{n, 0}| = q^{n^2-n} \prod_{j = 1}^{n} (q^{2j-1} +1), \quad |\mathcal{O}_{n, n}| = \prod_{j = 1}^{n} (q^{2j-1} + 1),$$
 $$|\mathcal{O}_{n, i}| = q^{(n-i)(n-i-1)} \frac{\prod_{j = n-i+1}^{n} (q^{2j} - 1)}{\prod_{j = 1}^{i} (q^{2j} - 1)} \prod_{j = 1}^{n} (q^{2j-1} +1), \quad 1 \leq i \leq n-1,$$
 and a member of $\mathcal{O}_{n, i}$ contains exactly $\frac{q^{2i}-1}{q^2-1}$ lines of $\mathcal{L}_1$. Moreover $\mathcal{O}_{n, i}$ is a $K_n$-orbit if $1 \leq i \leq n$, whereas $\mathcal{O}_{n, 0}$ splits into $q+1$ equally sized orbits under the action of $K_n$.
\end{thm}
\begin{proof}
 By induction on $n$. The statement holds true for $n = 1$. Indeed $G_1$ has two orbits on generators of $Q^+(3, q)$. Let $\Pi$ be a generator of $Q^+(4n-1, q)$ containing at least one line of $\mathcal{L}_1$. Since $G_n$ is transitive on lines of $\mathcal{L}_1$, we may assume that $\Pi$ contains the line $\ell = \phi(\langle u \rangle_{q^2}) \in \mathcal{L}_1$, where $u = (0, \ldots, 0,1)$. Consider the $(4n-5)$-space $\Lambda$ of $PG(4n-1, q)$ given by $y_{2n-1} = z_{2n-1} = y_{2n} = z_{2n} = 0$. Then $\Lambda$ meets $Q^+(4n-1, q)$ in a $Q^+(4n-5, q)$ and $\Pi \cap \Lambda$ is a generator of $Q^+(4n-5, q)$. From the proof of Proposition \ref{points}, we have that the stabilizer of $\ell$ in $G_n$, when restricted to $\Lambda$, coincides with $G_{n-1}$. By the induction hypothesis the group $G_{n-1}$ has $n$ orbits on generators of $Q^+(4n-5, q)$, namely $\mathcal{O}_{n-1, i}$, $0 \leq i \leq n-1$. Let $\Pi$ be a member of $\mathcal{O}_{n, i}$ whenever $\Pi \cap \Lambda$ belongs to $\mathcal{O}_{n-1, i-1}$, where $1 \leq i \leq n$. Note that $\Pi$ contains $\frac{q^{2i}-1}{q^2-1}$ lines of $\mathcal{L}_1$ if and only if $\Pi \cap \Lambda$ contains $\frac{q^{2i-2}-1}{q^2-1}$ lines of $\mathcal{L}_1$. In order to calculate the size of $\mathcal{O}_{n, i}$, $1 \leq i \leq n$, set $\mathcal{O}_{0, 0} = 1$ and observe that
 $$|\mathcal{O}_{n, i}| =  |\mathcal{O}_{n-i, 0}| \cdot \mbox{\# of $(i-1)$-spaces of } H(2n-1, q^2).$$
 It follows that the number of generators of $Q^+(4n-1, q)$ not contained in $\cup_{i = 1}^{n} \mathcal{O}_{n, i}$ or equivalently containing no line of $\mathcal{L}_1$ equals $q^{n^2-n} \prod_{j = 1}^{n} (q^{2j-1} +1)$. Let us consider the following vector space over $F_q$: $\{\langle x \rangle_{q} | 0 \neq x \in V(2n, q^2), x^q = x\}$ and assume that $\Pi = \{\phi(\langle x \rangle_{q})| 0 \neq x \in V(2n, q^2), x^q = x\}$. Then $\Pi$ is the generator of $Q^+(4n-1, q)$ given by $z_1 = \ldots = z_{2n} = 0$ and no line of $\mathcal{L}_1$ is contained in $\Pi$. Let $g \in G_n$ such that $\Pi^g = \Pi$, where $g$ is represented by the matrix $\overline{M}$. Thus $M = (a_{i j}) \in GU(2n, q)$ is such that $M$ fixes $\Sigma$. Therefore $a_{i j}^q = a_{i j}$ and $M^T J'_n M = J'_n$, that is $M \in Sp(2n, q)$. Vice versa, if $M \in Sp(2n, q)$, then the projectivity of $G_n$ represented by $\overline{M}$ stabilizes $\Pi$. Hence $Stab_{G_n}(\Pi) = \frac{Sp(2n, q)}{\langle -I_{2n} \rangle}$. It follows that $|\Pi^{G_n}| = q^{n^2-n} \prod_{j = 1}^{n} (q^{2j-1} +1)$ and $G_n$ is transitive on generators of $Q^+(4n-1, q)$ containing no line of $\mathcal{L}_1$. In order to study the action of the group $K_n$ on generators of $Q^+(4n-1, q)$, note that, from the proof of Proposition \ref{points}, the stabilizer of $\ell$ in $K_n$, when restricted on $\Lambda$, coincides with $G_{n-1}$. Hence an analogous inductive argument to that used in the previous case gives that $\mathcal{O}_{n, i}$, $1 \leq i \leq n$, is a $K_n$-orbit. Let $\Pi$ be a generator of $Q^+(4n-1, q)$ containing no line of $\mathcal{L}_1$, i.e., $\Pi \in \mathcal{O}_{n, 0}$. The lines of $\mathcal{L}_1$ meeting $\Pi$ in one point are the images under $\phi$ of the points of a Baer subgeometry $\Sigma$ embedded in $H(2n-1, q^2)$. Hence if $\overline{M} \in Stab_{K_n}(\Pi)$, then $M \in Stab_{SU(2n, q)}(\Sigma)$, that is conjugate in $GU(2n, q)$ to $\langle \overline{\Delta}, Sp(2n, q) \rangle$ by Proposition \ref{stabilizer2}. On the other hand if $\overline{M}$ is conjugate to an element of $\langle \overline{\Delta}, Sp(2n, q) \rangle$ and it induces a projectivity that fixes $\Pi$, then $M \in Sp(2n, q)$. Therefore $|Stab_{K_n}(\Pi)| = |\frac{Sp(2n, q)}{\langle -I_{2n} \rangle}|$ and $|\Pi^{K_n}| = q^{n^2-n} \prod_{j = 2}^{n} (q^{2j-1} +1)$.
\end{proof}

\subsection{The elliptic case}\label{subsec254}
Even in the elliptic case we start with a detailed example
\begin{exmp}
  Conside an hermitian form $H(2,q^{2})$
 \begin{equation}
  X_{0}X_{1}^{q}-X_{0}^{q}X_{1}+zX_{2}^{q+1}=0,
 \end{equation}
 where $z$ is an element of $F_{q^{2}}$ such that $z^{q}=-z$. We choose $z=\omega^{q}-\omega$ \\ $(z^{q}=(\omega^{q}-\omega)^{q}=\omega^{q^{2}}-\omega^{q}=\omega-\omega^{q}=-z)$. Now $X_{0}=x_{1}+\omega x_{2}$, $X_{1}=x_{3}+\omega x_{4}$ and $X_{2}=x_{5}+\omega x_{6}$, we get
 $$0=X_{0}X_{1}^{q}-X_{0}^{q}X_{1}+(\omega^{q}-\omega)X_{2}^{q+1}=(\omega^{q}-\omega)(x_{1}x_{4}-x_{2}x_{3})+(\omega^{q}-\omega)(x_{5}+\omega x_{6})^{q+1}=$$
 $$=(\omega^{q}-\omega)(x_{1}x_{4}-x_{2}x_{3})+(\omega^{q}-\omega)(x_{5}+\omega x_{6})(x_{5}+\omega x_{6})^{q}=$$
 $$=(\omega^{q}-\omega)(x_{1}x_{4}-x_{2}x_{3}+x_{5}^{2}+\omega x_{5}x_{6}+\omega^{q}x_{5}x_{6}+\omega^{q+1}x_{6}^{2})=$$
 $$=(\omega^{q}-\omega)(x_{1}x_{4}-x_{2}x_{3}+x_{5}^{2}+a x_{5}x_{6}+b x_{6}^{2}),$$
 with $a=\omega+\omega^{q}$ e $b=\omega^{q+1}$.
 We see that $a,b\in F_q$, since $a^{q}=(\omega+\omega^{q})^{q}=\omega^{q}+\omega^{q^{2}}=\omega^{q}+\omega=a$ and $b^{q}=(\omega^{q+1})^{q}=\omega^{q^{2}+q}=\omega^{q^{2}}\omega^{q}=\omega\omega^{q}=\omega^{q+1}=b$.

 By the coefficient $\omega^{q}-\omega$ we get an elliptic quadric
 \begin{equation}
  Q^{-}(5,q):x_{1}x_{4}-x_{2}x_{3}+x_{5}^{2}+ax_{5}x_{6}+bx_{6}^{2}=0,
 \end{equation}
 with $x^{2}+axy+by^{2}=0$ irreducible on $F_q[x,y]$. An elliptic quadric on $PG(5,q)$ has $(q^{3}+1)(q+1)$ points and $(q^{3}+1)(q^{2}+1)$ lines. The field reduction gives us, from the $q^{3}+1$ points of $H(2,q^{2})$, as many lines of $Q^{-}(5,q)$, that form a Desarguesian spread. Now we consider the lines that meet the hermitian curve in $q+1$ points. The section on the $H(2,q^{2})$ is a Baer subline isomorphic to an $H(1,q^{2})$. Among the $q^{4}+q^{2}+1$ lines in $PG(2,q^{2})$,  $q^{3}+1$ are tangent to the Hermitian curve (one for each isotropic point) and the other
 $q^{2}(q^{2}-q+1)$ are secant. Every $(q+1)$-section, through the field reduction, become an hyperbolic quadric $Q^{+}(3,q)$ section of $Q^{-}(5,q)$, so in every one there is a regulus of $q+1$ lines (part of the $q^{3}+1$ of the Desarguesian spread of $Q^{-}(5,q)$), and an opposite regulus of $q+1$ lines. Considering every secant we get $q^{2}(q^{2}-q+1)(q+1)=q^{2}(q^{3}+1)$ lines, and adding the $q^{3}+1$ lines of the spread (belonging to the points of $H$) we have all the $(q^{3}+1)(q^{2}+1)$ lines of $Q^{-}(5,q)$. Now we take, on the hermitian curve, the secant $x_{2}=0$, that meet the curve in the $q$ points $(1,a,0)$, $(a\in F_{q})$ together with the point $(0,1,0)$. The polar space of the line is the point $P=(0,0,1)$, external to $H(2,q^{2})$. Through $P$ we have $q+1$ \textit{homologies} of equation
 \begin{equation}
  \begin{cases}
   x_{0}'=x_{0}\\
   x_{1}'=x_{1}\\
   x_{2}'=ax_{2}\\
  \end{cases}
 \end{equation}
 with $a$ $(q+1)$-st root of 1. Those are automorphisms of center $(0,0,1)$ and axis $x_{2}=0$. Moreover, taking all the $q^{2}(q^{2}-q+1)$ non-isotropic points $P$, the homology group of $H(2,q^{2})$ through $P$ is a cyclic group of order $q+1$. Other automorphisms of $H(2,q^{2})$, that stabilize pointwise that $q+1$-section, are represented by the matrices of $GL(3,q)$ of the form:
 \begin{equation*}M=\left(
     \begin{array}{ccc}
       a & b & 0\\
       c & d & 0 \\
       0 & 0 & 1\\
     \end{array}
   \right),\end{equation*}
 with $a, b, c, d\in F_{q}$ and $ad-bc\neq0$.\\
 As in the case $n=2$ these matrices, embedded on the automorphism group of $Q^{+}(3,q)$, act transitively on the lines belonging to the second regulus, while the first one rises from the $q+1$ points of $H(1,q^{2})$. We repeat the procedure on all the $q^{2}(q^{2}-q+1)$ secant lines. The group made by the matrices $M$ is isomorphic to $SL(2,q)$, in fact:
 \begin{equation*}\left(
   \begin{array}{ccc}
    a & b & 0 \\
    c & c & 0\\
    0 & 0 & 1\\
     \end{array}
   \right)\left(
   \begin{array}{c}
       x_{0}  \\
       x_{1} \\
       x_{2} \\
     \end{array}
   \right)=\left(
   \begin{array}{c}
    ax_{0}+bx_{1} \\
    cx_{0}+dx_{1} \\
    x_{2}\\
    \end{array}
   \right),\end{equation*}

 to stabilize $H(2,q^{2})$:
 $$(ax_{0}+bx_{1})(cx_{0}+dx_{1})^{q}-(ax_{0}+bx_{1})^{q}(cx_{0}+dx_{1})+zx_{2}^{q+1}=0.$$
 But, for $(x_{0},x_{1},x_{2})$ on the hermitian curve,
 $$(ax_{0}+bx_{1})(cx_{0}+dx_{1})^{q}-(ax_{0}+bx_{1})^{q}(cx_{0}+dx_{1})+zx_{2}^{q+1}=x_{0}^{q}x_{1}-x_{0}x_{1}^{q}+zx_{2}^{q+1},$$
 $$acx_{0}^{q+1}+adx_{0}x_{1}^{q}+bcx_{1}x_{0}^{q}+bdx_{0}^{q+1}-acx_{0}^{q+1}-adx_{0}^{q}x_{2}-bcx_{1}^{q}x_{0}-bdx_{0}^{q+1}=x_{0}^{q}x_{1}-x_{0}x_{1}^{q},$$
 $$(ad-bc)(x_{0}x_{1}^{q}-x_{0}^{q}x_{1})=x_{0}^{q}x_{1}-x_{0}x_{1}^{q},$$
 and so $ad-bc=1$.\\
 $PGU(3,q)$ has a unique orbit $q^{2}(q^{2}-q+1)$ on the external points, so the stabilizing subgroup of the polar line of an external point (a secant line) has order (using Orbit-stabiliser theorem):
 $$\frac{|PGU(3,q)|}{q^{2}(q^{2}-q+1)}=\frac{q^{3}(q^{3}+1)(q^{2}-1)}{q^{2}(q^{2}-q+1)}=q(q^{2}-1)(q+1)=(q+1)|SL(2,q)|,$$
 i.e. is exactly the group generated by the $q+1$ homologies and the matrices $M$. Now to prove the transitivity on all the $q^{2}(q^{3}+1)$ lines out of the Desarguesian spread of $Q^{-}(5,q)$, we need only to show that $PGU(3,q)$ is transitive on the $(q+1)$-secant lines. But we know that the automorphism group is $2$-transitive on the curve's points, and it is linear, so we get the thesis.
\end{exmp}

We are now ready to give the general argument.
Assume that $N = 2n+1$. Consider the following non--degenerate Hermitian form $H$ on $V(2n + 1, q^2)$
$$H(x, x') = x \tilde{J_n} ({x'}^q)^T = ( x_0, \ldots, x_{2n})\tilde{J_n}(x_0'^q, \ldots, x_{2n}'^q )^T ,$$
where $\tilde{J_n}$ is given by
\begin{equation*}\label{hermitian_matrix_1}
 \left(
     \begin{array}{cccccc}
       0 & \xi & \cdots & 0 & 0 & 0\\
       \xi^{q} & 0 & \cdots & 0 & 0 & 0 \\
       \vdots & \vdots & \ddots & \vdots & \vdots & \vdots \\
       0 & 0 & \cdots & 0 & \xi & 0 \\
       0 & 0 & \cdots & \xi^{q} & 0 & 0\\
       0 & 0 & \cdots & 0 & 0 & -\xi^{2}
     \end{array}
   \right)
 \end{equation*}
Consider the following groups of isometries
\begin{align*}
 GU(2n+1, q) = \{M \in GL(2n+1, q^2) |  M^T \tilde{J_n} M^q = \tilde{J_n}\}, \\
 SU(2n+1, q) = \{M \in SL(2n+1, q^2) |  M^T \tilde{J_n} M^q = \tilde{J_n}\}.\\
 Sp(2n, q) = \{M \in GL(2n, q)  | M^T J_n' M = J_n'\}.
\end{align*}
We have that
$$H(x, x) = \xi (x_0 x_1^q - x_0^q x_1 + \ldots + x_{2n-2} x_{2n-1}^q - x_{2n-2}^q x_{2n-1} - \xi x_{2n}^{q+1})= $$
$$= - \xi^2 (y_0 z_1 - z_0 y_1 + \ldots + y_{2n-2} z_{2n-1} - z_{2n-2} y_{2n-1} + y_{2n}^2 + y_{2n} z_{2n} - \alpha z_{2n}^2) = Q(y).$$
Note that $Q$ is a non-degenerate quadratic form on $V(4n+2, q)$ of elliptic type. Let $H(2n, q^2)$ be the Hermitian variety of $PG(2n, q^2)$ determined by $H$. Denote by $Q^-(4n+1, q)$ the elliptic quadric of $PG(4n+1, q)$ defined by $Q$ and let $PGO^-(4n+2, q)$ denote the group of projectivities of $PG(4n+1, q)$ stabilizing $Q^-(4n+1, q)$. It follows that
$$\mathcal{L}_1 = \{ \phi ( \langle x \rangle_{q^2} )  | 0 \neq x \in V(N, q^2), H(x, x) = 0\} \subset \mathcal{L}$$
is a line-spread of $Q^-(4n+1, q)$, see also \cite{D}. Similarly to the previous case, the following holds.
\begin{prop}
 There exists a group $\tilde{G}_n \leq PGO^-(4n+2, q)$ isomorphic to $\frac{GU(2n+1, q)}{\langle - I_{2n+1} \rangle}$ stabilizing $\mathcal{L}$.
\end{prop}
\begin{proof}
 Let $M = (a_{i j}) \in GU(2n, q)$ and hence $H(x M^T, x M^T) = H(x, x)$. Denote by $\overline{M}$ the unique element of $GL(4n, q)$ such that $\phi (x M^T) = \phi(x) \overline{M}^T$, forall $x \in V(2n, q^2)$. Let $G_n$ be the group of projectivities of $PG(4n-1, q)$ associated with the matrices $\{\overline{M}  | M \in GU(2n, q)\}$. Let $g \in G_n$, then $g$ stabilizes $\mathcal{L}$. Moreover $Q(\phi(x)\overline{M}^T) = Q(\phi (x M^T)) = H(x M^T, x M^T) = H(x, x) = Q(\phi(x))$ and hence $g$ fixes $Q^+(4n-1, q)$. Some calculations show that if $M = (a_{i j})$, where $a_{i j} = b_{i j} + w c_{i j}$, then
 $$\overline{M} = ( A_{i j} ), \mbox{ where }A_{i j} =\begin{pmatrix}
 b_{i j} & c_{i j} \alpha \\
 c_{i j} & b_{i j} + c_{i j} \\
 \end{pmatrix}.$$
 It follows that if $\overline{M}$ induces the identity projectivity, then $A_{i j} = 0$, $i \neq j$, and $A_{i i} = \lambda I_{2}$, $1 \leq i \leq 2n$, for some $\lambda \in F_q \setminus \{0\}$. Hence $M = \lambda I_{2n}$. On the other hand, since $M \in GU(2n, q)$, we have that $\lambda = \pm 1$.
\end{proof}
Let $\tilde{K}_n$ be the subgroup of $\tilde{G}_n$ consisting of projectivities of $PG(4n+1, q)$ associated with the matrices $\{\overline{M} |M \in SU(2n+1, q)\}$. We have that $\tilde{K}_n \simeq \frac{SU(2n+1, q)}{\langle - I_{2n+1} \rangle}$.
\begin{prop}\label{points_1}
 The groups $\tilde{K}_n$ and $\tilde{G}_n$ act transitively on lines of $\mathcal{L} \setminus \mathcal{L}_1$ and for $\ell \in \mathcal{L} \setminus \mathcal{L}_1$, $Stab_{\tilde{K}_n}(\ell) \simeq G_n$ and $Stab_{\tilde{G}_{n}}(\ell) \simeq (q+1) \times G_n$.
\end{prop}
\begin{proof}
 Let $u = (0, \ldots, 1) \in V(2n+1, q^2)$ and consider a matrix $M \in GU(2n+1, q)$. Some straightforward calculations show that $M$ stabilizes $\langle u \rangle_{q^2}$ if and only if it has the form
 \begin{equation}\label{matrix_1}
  \begin{pmatrix}
    &          & & 0 \\
    &  M'    & & \vdots \\
    &          & & 0 \\
   0 & \cdots & 0 & a \\
  \end{pmatrix},
 \end{equation}
 here $M'^T J_{n} M'^q = J_{n}$, $det(M') \neq 0$, $a \in F_{q^2}$, $a^{q+1} = 1$. Then,\\ $|Stab_{GU(2n+1, q)}(\langle u \rangle_{q^2})| = (q+1) |GU(2n, q)|$. Hence $\tilde{G}_n$ is transitive on lines of $\mathcal{L} \setminus \mathcal{L}_1$, since $|\frac{GU(2n+1, q)}{Stab_{GU(2n+1, q)}(\langle u \rangle_{q^2})}| = |\mathcal{L} \setminus \mathcal{L}_1|$. On the other hand, $M \in SU(2n, q)$ stabilizes $\langle u \rangle_{q^2}$ if and only if $M$ is of type \eqref{matrix_1}, where $M'^T J_{n} M'^q = J_{n}$, $det(M') \neq 0$ and $a = det(M')^{-1}$. Then $|Stab_{SU(2n+1, q)}(\langle u \rangle_{q^2})| = |GU(2n, q)|$. Therefore $|\frac{SU(2n+1, q)}{Stab_{SU(2n+1, q)}(\langle u \rangle_{q^2})}| = |\mathcal{L} \setminus \mathcal{L}_1|$ and $\tilde{K}_n$ is transitive on lines of $\mathcal{L} \setminus \mathcal{L}_1$.
\end{proof}
\begin{thm}
 The group $\tilde{G}_n$ has $n+1$ orbits, say $\tilde{\mathcal{O}}_{n, i}$, $0 \leq i \leq n$, on generators of $Q^-(4n+1, q)$, where
 $$|\tilde{\mathcal{O}}_{n, 0}| = q^{n^2+n} \prod_{j = 2}^{n+1} (q^{2j-1} +1), \quad |\tilde{\mathcal{O}}_{n, n}| = \prod_{j = 2}^{n+1} (q^{2j-1} + 1),$$
 $$|\tilde{\mathcal{O}}_{n, i}| = q^{(n-i)(n-i+1)} \frac{\prod_{j = n-i+1}^{n} (q^{2j} - 1)}{\prod_{j = 1}^{i} (q^{2j} - 1)} \prod_{j = 2}^{n+1} (q^{2j-1} +1), \quad 1 \leq i \leq n-1,$$
 and a member of $\tilde{\mathcal{O}}_{n, i}$ contains exactly $\frac{q^{2i}-1}{q^2-1}$ lines of $\mathcal{L}_1$.
\end{thm}
\begin{proof}
 Let $\ell$ be a line of $\mathcal{L} \setminus \mathcal{L}_1$. Then $\ell^\perp \cap Q^-(4n+1, q) = Q^+(4n-1, q)$ and $Stab_{\tilde{G}_{n}}(\ell) \simeq (q+1) \times G_n$. Hence, from Theorem \ref{hyp}, the group $Stab_{\tilde{G}_n}(\ell)$ has $n+1$ orbits, say $\mathcal{O}_{n, i}$, $0 \leq i \leq n$, on generators of $Q^-(4n+1, q)$ contained in $\ell^\perp$. Moreover $\tilde{G}_n$ is transitive on lines of $\mathcal{L} \setminus \mathcal{L}_1$. The result follows by counting in two ways the couples $(\ell, g)$, where $\ell \in \mathcal{L} \setminus \mathcal{L}_1$, $g$ is a generator of $Q^-(4n+1, q)$ containing exactly $\frac{q^{2i}-1}{q^2-1}$ lines of $\mathcal{L}_1$ and $\ell \in g^\perp$. Indeed, on the one hand we have that this number equals $|\mathcal{L} \setminus \mathcal{L}_1| \cdot |\mathcal{O}_{n, i}| = \frac{q^{2n} (q^{2n + 1} + 1)}{q+1} \cdot |\mathcal{O}_{n, i}|$, whereas on the other hand it coincides with $(\frac{(q^2-1)(q^{2i} - 1)}{q^2-1} + 1) \cdot |\tilde{\mathcal{O}}_{n, i}| = q^{2i} \cdot |\tilde{\mathcal{O}}_{n, i}|$.

\end{proof}
\section{Segre's hemisystem}\label{sec26}
In the last row of Table \ref{tabHirschfeld} it  is stated the non existence of spreads in the Hermitian polar space $H(2n+1,q^{2})$. In this section we focus on case $n=1$, the Hermitian surface $H(3,q^{2})$, whose study was initiated by B. Segre in \cite{Segre}. The Hermitian surface $H(3,q^{2})$ has $(q^{3}+1)(q^{2}+1)$ points and $(q^{3}+1)(q+1)$ generators, i.e. lines, and each point lies on $n=q+1$ generators. For an integer $m$ with $0\leq m\leq n$, an $m$-regular system of $H(3,q^{2})$ consists of generators such that each point lies on exactly $m$ of them.

\begin{thm}\cite[Proposition 91]{Segre}
\label{SegreH}
 For $q$ odd, then $H(3,q^2)$ admits non-trivial $m$-regular systems if and only if $m=\frac{n}{2}$, i.e. it admits only hemisystems.
\end{thm}

\begin{proof}
 The theorem may also be proven using strongly regular graphs. We refer the reader to Section \ref{sec51}.
\end{proof}

Moreover in \cite{Segre} the following hemisystem in case $q=3$ is constructed.
\begin{cons}\cite[Propositions 95-101]{Segre}
 \label{SegreCons} The Hermitian surface $H(3,9)$ has $280$ points and $112$ lines, $4$ through each point. A hemisystem is a set of $56$ lines such that each point lies on $2$ lines.
 An equation for $H(3,9)$ is
 $$X_0X_2(X_0^{2}+X_2^{2})=X_1X_3(X_1^{2}+X_3^{2}).$$
 In $H(3,9)$ there are $2^2\cdot3^6\cdot5\cdot7=102060$ skew quadrangles. Consider the quadrangle $\theta=(U_1,U_2,U_3,U_4)$. The diagonals of $\theta$ are the lines $X_0=X_2=0$ and $X_1=X_3=0$, and the \textit{opposite quadrangle} $\theta'$ is defined by the points
 $A_1=(1,0,j,0)$, $A_2=(0,1,0,j)$, $A_3=(j,0,1,0)$, $A_4=(0,j,0,1)$. The $4$ sides of $\theta$ have equations $X_0=X_1=0$, $X_1=X_2=0$, $X_2=X_3=0$, $X_3=X_0=0$. Lines of $H(3,9)$ belong to 5 sets
 \begin{itemize}
  \item[Type 1] $4$ sides of $\theta$;
  \item[Type 2] $8$ other lines through each vertex $U_i$ of $\theta$, $2$ per each vertex;
  \item[Type 3] $16$ other lines cutting each pair of opposite sides of $\theta$, $10-2=8$ lines per each pair of sides;
  \item[Type 4] $64$ lines meeting the quadrangle in one point (out of the vertices), $2$ per each point;
  \item[Type 5] $112-4-8-16-64=20$ other lines.
 \end{itemize}
 Lines of Type 5 have equations
 \begin{equation*}
  \begin{cases}
   X_1=cX_0-ic^{-3}X_2\\
   X_3=cX_0-ic^{-3}X_2,
  \end{cases}
 \end{equation*}
 when $c^8=1$ and $i^2=-1$. It can be proven that these 20 lines form 5 skew quadrangles, which form with $\theta$ a system $\mathcal{Q}$ of 6 skew quadrangles. A hemisystem containing a quadrangle of the system $\mathcal{Q}$, must contain all the 6 quadrangles in $\mathcal{Q}$. All other $56-24=32$ lines of the hemisystem must be of Type 4. Now, through each of the $32$ non vertices of $\theta$ there pass exactly $2$ lines of Type 4. It can also be proven that all these $64$ lines may be divided in $4$ systems of $16$ lines meeting a fixed pair of opposite sides of $\theta$ and a fixed pair of opposite sides of $\theta'$. Choosing $2$ of the $4$ systems such that we do not get the same pair of sides of $\theta$ and $\theta'$ we obtain the remaining $32$ lines of the hemisystem.
\end{cons}

\section{New hemsystems of the Hermitian surface}\label{sec27}
 Finding new examples of hemisystems is a challenging problem. In \cite{Thas} J. A. Thas conjectured that the Construction \ref{SegreCons} for $q=3$ was the only example of hemisystem on the Hermitian surface $H(3,q^{2})$.

 Later, the conjecture was disproved and the first infinite family was constructed almost 50 years after by A. Cossidente and T. Penttila \cite{cossidente2005hemisystems} who also found a new sporadic example in $H(3,25)$.  Later on, J. Bamberg, M. Giudici and G. Royle \cite{bamberg2010every} and \cite[Section 4.1]{bamberg2013hemisystems} constructed more sporadic examples for $q=7,9,11,17,19,23,27$. In \cite{bamberg2010every} the authors provide a construction method of hemisystems for a class of generalized quadrangles which includes $H(3,q^2)$. A hemisystem obtained by this method is left invariant by an elementary abelian group of order $q^2$ and the Cossidente-Penttila hemisystem can be also obtained in this way. Recently several new infinite families of hemisystems appeared in the literature. J. Bamberg, M. Lee, K. Momihara and Q. Xiang \cite{bamberg2018new} constructed  a new infinite family of hemisystems on $H(3,q^2)$ for every $q\equiv -1\pmod 4$ that generalize one of the previously known sporadic examples. Their construction is based on  cyclotomic classes of $F_{q^6}^*$ and involves results on characters and Gauss sums. A. Cossidente and F. Pavese \cite{cossidente2017intriguing} constructed, for every odd $q$, a hemisystem of $H(3,q^2)$ invariant by a subgroup of $PGU(4,q)$ of order $(q+1)q^2$.

 The approach introduced in \cite{KNS_Hemi} by G. Korchm\'{a}ros, G. P. Nagy and P. Speziali, relies on the Fuhrmann-Torres curve over $q^2$  naturally embedded in $H(3,q^2)$. Here the term curve defined over $q^2$ is used for a (projective, geometrically irreducible, non-singular) algebraic curve $\mathcal{X}$ of $PG(3,q^2)$, see Appendix \ref{apA}. Their construction provided a hemisystem of $H(3,q^{2})$ whenever $q=p$ is a prime of the form $p=1+4a^2$ for an even integer $a$. Here we investigate the analogous construction for $p=1+4a^2$  with an odd integer $a$, and show that it produces a hemisystem , as well, for every such $p$. We mention that a prime number $p$ of the form $p=1+4a^2$ with an integer $a$ is called a \textit{Landau number}. Since the famous Landau's conjecture, dating back to 1904, is still to be proven (or disproved), it is unknown whether there exists an infinite sequence of such primes, and hence whether an infinite family of hemisystems is obtained or not. What is known so far is that 37 primes up to 51000 with this property exist, namely 5, 17, 37, 101, 197, 257, 401, 577, 677, 1297, 1601, 2917, 3137, 4357, 5477, 7057, 8101, 8837, 12101, 13457, 14401, 15377, 15877, 16901, 17957, 21317, 22501, 24337, 25601, 28901, 30977, 32401, 33857, 41617, 42437, 44101, 50177, see \cite{Landau}.

 Our main result is stated in the following theorem.
 \begin{thm}\label{th:1}
  Let $ p $ be a prime number where $ p =1 +4a^2 $ with an odd integer $ a $. Then there exists a hemisystem in the Hermitian surface $H(3,p^2)$ of $PG(3,p^2) $ which is left invariant by a subgroup of $PGU(4, p) $ isomorphic to $PSL(2, p)\times C_\frac{p+1}{2}$.
 \end{thm}

\subsection{The Fuhrmann-Torres construction}\label{sec271}
 \begin{defn}
  In $PG(2,\overline{F_{q}})$ with homogeneous coordinates $(X:Y:Z)$, the \textit{Fuhrmann-Torres curve} is the plane curve $\mathcal{F}^+$ of genus $\frac{1}{4}(q-1)^2$ with equation
  \begin{equation}
   \mathcal{F}^+: Y^q-YZ^{q-1}=X^\frac{q+1}{2}Z^\frac{q-1}{2}.
  \end{equation}
 \end{defn}
 The morphism
 \begin{equation*}
  \begin{array}{lccc}
  \varphi: & \mathcal{F}^+ & \longrightarrow & PG(3,\overline{F_{q}}) \\
 & (X:Y:Z) & \mapsto & (Z^2:XZ:YZ:Y^2 )\\
\end{array}
\end{equation*}
 defines an embedding (called natural embedding) of $\mathcal{F}^+$ which is a $q+1$ degree curve $\mathcal{X}^+$ whose points (including those defined over $\overline{F_{q}}$) are contained in $H(3,q^2)$. In particular, $\mathcal{F}^+$ is an $F_{q^{2}}$-maximal curve. The \textit{twin Fuhrmann-Torres curve} is defined by the equation
 \begin{equation}
   \mathcal{F}^-: Y^q-YZ^{q-1}=-X^\frac{q+1}{2}Z^\frac{q-1}{2}.
 \end{equation}
 and the above claims remain valid with respect to the same morphism. For more details see \cite{FTorr}.

 Some useful properties of the Fuhrmann-Torres curve, also valid for any $F_{q^{2}}$-maximal curve $\mathcal{X}$ naturally embedded in $H(3,q^2)$, can be found in \cite[Sections 2,3,4]{KNS_Hemi}. In particular, $\mathcal{X}^+$ is a $q+1$ degree curve lying in the Hermitian surface $H(3,q^2)$. Furthermore $\mathcal{X}^+(F_{q^{2}})$ is partitioned in $\Omega$ and $\Delta^+$, where $\Omega$ is the set cut out on $\mathcal{X}^+$ by the plane $\pi:X_1=0$. Note that $|\Omega|=q+1$ and $|\Delta^+|=\frac{1}{2}(q^3-q)$.
 Equivalently $\Omega$ is the intersection in $\pi$ of the conic $\mathcal{C}:X_0X_3-X_2^2=0$ and the Hermitian curve $H(2,q^2)$ with equation $X_0^{q}X_3+X_0X_3^q-2X_2^{q+1}=0$. Moreover, the above properties hold true when $^+$ is replaced by $^-$ and $\mathcal{X}^-$ is the natural embedding of the plane curve $\mathcal{F}^-$. The curves $\mathcal{X}^+$ and $\mathcal{X}^-$ are isomorphic over $F_{q^{2}}$ and $\Omega$ is the set of their common points. We use classical terminology regarding maximal curves, see Appendix \ref{apA}.

The key point of the construction below is the following corollary to the \textit{Natural Embedding Theorem}.
\begin{lem}\cite[Lemma 3.4]{KNS_Hemi}
 Let $\mathcal{C}$ be an $F_{q^{2}}$-maximal curve naturally embedded in the Hermitian surface $H(3,q^2)$. Then
  \begin{enumerate}
   \item No two distinct points in $\mathcal{C}(F_{q^{2}})$ are conjugate under the unitary polarity associated with $H(3,q^2)$.
   \item Any imaginary chord of $\mathcal{C}$ is a generator of $H(3,q^2)$ which is disjoint from $\mathcal{C}$.
   \item For any point $P\in H(3,q^2)$ in $PG(3,q^2)$, if $P\notin\mathcal{C}(F_{q^{2}})$ and $\Pi_P$ is the tangent plane to $H(3,q^2)$ at $P$, then $\Pi_P \cap \mathcal{C}$ consists of $q+1$ pairwise distinct points which are in $\mathcal{C}(F_{q^4})$.
  \end{enumerate}
\end{lem}

From now on let $q$ be a prime $p\equiv 1\pmod 4 $ and let $\mathcal{X}$ be an  $F_{q^{2}}$-maximal curve. Denote by $N_{q^2}$ the number of $F_{q^{2}}$-rational points of $\mathcal{X}$. Let $\mathcal{H}$ denote the set of all imaginary chords of $\mathcal{X}$. Furthermore, for a point $P\in PG(3,q^2)$ lying in $H(3,q^2)\setminus\mathcal{X}(F_{q^{2}})$, let $n_P(\mathcal{X})$ denote the number of generators of $H(3,q^2)$ through $P$ which contain an $F_{q^{2}}$-rational point of $\mathcal{X}$.
\begin{defn}
 A set $\mathcal{M}$ of generators of $H(3,q^2)$ is an \textit{half-hemisystem} on $\mathcal{X}$ if the following properties hold:
 \begin{itemize}
  \item[(A)] Each $F_{q^{2}}$-rational points of $\mathcal{X}$ is incident with exactly $\frac{1}{2}(q+1)$  generators in $\mathcal{M}$.
  \item[(B)] For any point $P\in H(3,q^2)\setminus\mathcal{X}(F_{q^{2}})$ lying in $PG(3, q^2)$, $ \mathcal{M} $ has as many as $\frac{1}{2}n_P(\mathcal{X})$ generators through $P$ which contain an $F_{q^{2}}$-rational point of $\mathcal{X}$.
	\end{itemize}
\end{defn}
Note that $\mathcal{M}$ consists of $\frac{q+1}{2}N_{q^2}$ generators and $\mathcal{H}$ of $\frac{(q^2+q)(q^2-q-2g(\mathcal{X}))}{2}$ generators of $H(3,q^2)$. Therefore $\mathcal{M}\cup\mathcal{H}$ has exactly $\frac{(q^3+1)(q+1)}{2}$ generators of $H(3,q^2)$.
\begin{res}\cite[Proposition 4.1]{KNS_Hemi}
 $\mathcal{M} \cup \mathcal{H}$ is a hemisystem of $H(3,q^2)$.
\end{res}
Let $\mathfrak{G}$ be a subgroup of $Aut(\mathcal{X})$ and $o_1,\dots,o_r$ be the $ \mathfrak{G} $-orbits on $\mathcal{X}(F_{q^{2}})$. Let $\mathcal{G}$ be the set of all generators meeting $\mathcal{X}^+$. Moreover, for $1 \leq j \leq r$, let $\mathcal{G}_j$ denote the set of all generators of $H(3,q^2)$ meeting $o_j$. Note that $\mathfrak{G}$ leaves each $\mathcal{G}_j$ invariant.
\begin{res}\cite[Proposition 4.2]{KNS_Hemi} \label{CDE}
 With the above notation, assume that the subgroup $\mathfrak{G}$ fulfills the hypothesis:
 \begin{itemize}
  \item[(C)] $\mathfrak{G}$ has a subgroup $\mathfrak{h}$ of index $2$ such that $\mathfrak{G}$ and $\mathfrak{h}$ have the same orbits $o_1,\dots,o_r$ on $\mathcal{X}(F_{q^{2}})$.
  \item[(D)]For any $1\leq j \leq r$, $ \mathfrak{G} $ acts transitively on $\mathcal{G}_j$ while $\mathfrak{h}$ has two orbits on $\mathcal{G}_j$.
 \end{itemize}
 Let $P\notin \mathcal{X}(F_{q^{2}})$ be a point lying on a generator in $\mathcal{G}$, if
 \begin{itemize}
  \item[(E)] there is an element in $\mathfrak{G}_P$ not in $\mathfrak{h}_P$,
	\end{itemize}
	then $P$ satisfies \emph{(B)}.
\end{res}
From \cite[Lemma 5.1]{KNS_Hemi} $\mathcal{G}$ is  also the set of all generators meeting $\mathcal{X}^-$. In particular, $\mathcal{G}$ splits into two subsets
\begin{equation}\label{rk:G}
 \mathcal{G}=\mathcal{G}_1 \cup \mathcal{G}_2,
\end{equation}
where $\mathcal{G}_2$ is the set of the $(q+1)^2 $ generators meeting $\Omega$, while $\mathcal{G}_1$ is the set of the $\frac{1}{2}(q^3-q)(q+1)$ generators meeting both $\Delta^+$ and $\Delta^-$. Thus, the following characterization of $\mathcal{G}$ is very useful.
\begin{res}\cite[Lemma 5.3]{KNS_Hemi}\label{lm:G1}
 The generator set $\mathcal{G}_1$ consists of all the lines $g_{u,v,s,t}$ spanned by the points $P_{u,v}=(1:u:v:v^2) \in \Delta^+ $ and $Q_{s,t}=(1:s:t:t^2) \in \Delta^-$ such that
 $$\mathcal{F}: F(v,t)=(v+t)^{q+1}-2(vt+(vt)^q)=0,$$
 and
 $$u^{\frac{q+1}{2}}=v^q-v, \quad -s^\frac{q+1}{2}=t^q-t, \quad u^qs = (t-v^q)^2.$$
\end{res}

\begin{res}\cite[Lemma 5.4]{KNS_Hemi}
 $Aut(\mathcal{F}) $ contains a subgroup $\Psi \cong PGL(2,q)$ that acts faithfully on the set $ \mathcal{F}(F_{q^{2}})\setminus \mathcal{F}(F_{q}) $ as a sharply transitive permutation group.
\end{res}

\subsection{Automorphisms preserving $\mathcal{G}$ and $\mathcal{X}^+$}\label{subsec272}
In this section we recall the main results about the group-theoretic properties involving, $\mathcal{X}^+$, $\mathcal{X}^-$ and $\mathcal{G}$; see \cite[Section 5]{KNS_Hemi}. The authors showed that $\Psi$ contains a subgroup $\Gamma$ which acts sharply transitively on $\mathcal{G}_1$. Furthermore, $\Gamma$ has a unique index $2$ subgroup $\Phi$ such that
$$\Phi \cong PSL(2,q) \times C_\frac{q+1}{2}.$$
In particular, $\Phi$ has two orbits on $\mathcal{G}_1$, namely $\mathcal{M}_1$ and $\mathcal{M}_2$. In terms of subgroups of $PGU(4,q)$, the following holds.	
\begin{res}\cite[Lemma 5.7]{KNS_Hemi}
 The group $PGU(4,q)$ has a subgroup $\mathfrak{G}$ with the following properties:
 \begin{enumerate}
  \item $\mathfrak{G}$ is an automorphism group of $\mathcal{X}^+$ and $\mathcal{X}^-$;
  \item $\mathfrak{G}$ preserves $\Delta^+$, $\Delta^-$, $\Omega$ and $\mathcal{G}_1$;
  \item $\mathfrak{G}$ acts faithfully on $\Delta^+$, $\Delta^-$ and $\mathcal{G}_1$;
  \item the collineation group induced by $\mathfrak{G}$ on $\pi$ is $\frac{\mathfrak{G}}{Z(\mathfrak{G})} \cong PGL(2,q)$ with $ Z(\mathfrak{G}) \cong C_\frac{q+1}{2} $;
  \item the permutation representation of $ \mathfrak{G} $ on $\mathcal{G}_1$ is $\Gamma$; in particular $\mathfrak{G} \cong \Gamma$;
  \item $\frac{\mathfrak{G}}{Z(\mathfrak{G})}$ acts on $\Omega$ as $PGL(2,q)$ in its $3$-transitive permutation representation.
 \end{enumerate}
 Furthermore, $\mathfrak{G}$ has an index $2$ subgroup $\mathfrak{h}$ isomorphic to $PSL(2,q) \times C_\frac{q+1}{2}$.
\end{res}
With the above notation, in the isomorphism $\mathfrak{G} \cong \Gamma $, $\mathfrak{h}$ and $\Phi$ correspond.
\begin{res}\cite[Lemma 5.9]{KNS_Hemi}
 The elements of order $2$ in $ \mathfrak{h} $ are skew perspectivities, while those in $ \mathfrak{G} \setminus \mathfrak{h} $ are homologies. Furthermore, the linear collineation $ \mathfrak{w} $, defined by
 $$\textbf{\emph{W}}:=\begin{pmatrix}
 	1&0&0&0\\
	0&-1&0&0\\
	0&0&1&0\\
	0&0&0&1
 \end{pmatrix},$$
 interchanges $\mathcal{X}^+$ with $\mathcal{X}^-$ and the linear group generated by $\mathfrak{G}$ and $\mathfrak{w}$ is the direct product $ \mathfrak{G} \times \left\langle \mathfrak{w}\right\rangle  $.
\end{res}
\begin{res}\cite[Lemma 5.11]{KNS_Hemi}
 $\mathfrak{G} $ acts transitively on $\mathcal{G}_2$ while $\mathfrak{h}$ has two orbits on $ \mathcal{G}_2 $.
\end{res}
From the result of this section, the following theorem follows				
\begin{thm}\cite[Theorem 5.13]{KNS_Hemi}\label{th:CD}
 Conditions \emph{(C)} and \emph{(D)} are fulfilled for $\mathcal{X}=\mathcal{X}^+$, with $\Gamma=\mathfrak{G}$ and $\Phi=\mathfrak{h}$.
\end{thm}
More precisely,	$\mathcal{G}=\mathcal{G}_1 \cup \mathcal{G}_2 $ with $ \mathcal{G}_1=\mathcal{M}_1 \cup \mathcal{M}_1'$ and $\mathcal{G}_2=\mathcal{M}_2 \cup\mathcal{M}_2'$, where $\mathcal{G}_1$ and $\mathcal{G}_2$ are the $\mathfrak{G}$-orbits on $\mathcal{G}$ whereas $\mathcal{M}_1$, $\mathcal{M}_1'$, $\mathcal{M}_2 $, $ \mathcal{M}_2' $ are the $\mathfrak{h}$-orbits on $\mathcal{G}_1$ and $\mathcal{G}_2$ respectively. This notation fits with \cite[Section 5]{KNS_Hemi}.

\subsection{Points satisfying Condition (E)}\label{subsec273}
The plane $\pi\colon X_1=0$ can be seen as the projective plane $PG(2,q^2)$, with homogeneous coordinates $(X_0:X_2:X_3)$. Then $\mathcal{C}$ is the conic of equation $X_0X_3-X_2^2=0$ and $\Omega$ is the set of points of $\mathcal{C}$ lying in the (canonical Baer) subplane $PG(2,q)$. The points in $PG(2,q^2) \setminus PG(2,q)$ are of three types with respect to the lines of $PG(2,q)$, i.e.
\begin{enumerate}
	\item points of a unique line disjoint from $\Omega$ which meets $\mathcal{C}$ in two distinct points both in $PG(2,q^2)\setminus PG(2,q)$;
	\item points of a unique line meeting $\Omega$ in two distinct points;
	\item points of a unique line which is tangent to $\mathcal{C}$ with tangency point on $\Omega$.
\end{enumerate}
Points of type (1)-(2) and points in $PG(2,q)$ satisfy condition (B), as can be readily seen in the next result.
\begin{res}\cite[Theorem 6.1]{KNS_Hemi}\label{th:I-II}
 If the projection of $P\in H(3,q^2)$ on $\pi$ is a point $P'$ of type \emph{(1)}-\emph{(2)} or $P'\in PG(2,q)$, then condition \emph{(E)} is fulfilled for $\mathcal{X}=\mathcal{X}^+$, $ \Gamma=\mathfrak{G} $ and $ \Phi=\mathfrak{h} $.
\end{res}

\section{Condition (B) for case (3) and $p \equiv 5 \pmod 8$}\label{sec28}
Condition (B) is not always satisfied in Case (3), that is, for points $ P $ whose projection from $ X_\infty=(0,1,0,0) $ on $\pi$ is a point $ P' $ lying on a tangent $\ell $ to $ \mathcal{C}$. Our goal is to show that \cite[Theorem 7.1]{KNS_Hemi}, proven for $p \equiv 1 \pmod 8$, remains true for $ p \equiv 5 \pmod 8$, extending their results to the case $p \equiv 1 \pmod 4$. For this reason, from now on, we assume $q$ be a prime $p \equiv 5 \pmod 8$.	
\begin{thm}\label{th:main}
 Condition $ (B) $ for Case $ (3) $ is satisfied if and only if the number $ N_q  $ of $F_q $-rational points of the elliptic curve with affine equation $Y^2 = X^3 - X $ equals either $ q - 1 $, or $ q +3 $.
\end{thm}
We need few steps before to prove Theorem \ref{th:main}. To begin with, we have to prove the following theorem.
\begin{thm}\label{th:III}
 Let $n_q$ be the number of $\xi \in F_q $ for which $ f(\xi) = \xi^4-48\xi^2+64  $ is a square in $F_q $. Condition $ (B) $ for Case $ (3) $ is satisfied if and only if  $ n_q $ equals either $\frac{1}{2} (q+1)$ or $\frac{1}{2}(q-3)$.
\end{thm}
The proof of Theorem \ref{th:III} is carried out by a series of lemmas. Since $q\equiv 5 \pmod 8$, $2$ is not a square in $F_q$. Let $h$ and $-h$ be the square roots of $2$ in $F_{q^{2}}$. In particular we have that $h^q+h =0$ and $(\pm h)^{q+1} = -2$. Moreover $h$ is a non-square in $F_{q^{2}}$. Moreover $h^{\frac{q-1}{2}}=\alpha\in F_q $, with $\alpha^2 =-1$. Thus, $\alpha \notin \square_q$ and $(1+\alpha)(1-\alpha)=2 \notin \square_q$. Since $ \mathfrak{G} $ is transitive on $\Omega$, the point $ O =(1: 0: 0: 0) $ may be assumed to be the tangency point of $ \ell $. Then $\ell$ has equation $ X_1 =X_3 =0 $, and $ P =(a_0: a_1: a_2: 0) $ with  $a_1\neq0$ and $a_1^{q+1}+2a_2^{q+1}=0$. If $ a_0 =0 $ then $ P =(0: d: 1: 0) $ with $ d^{q+1} +2 =0 $ and  his projection to $\pi\colon X_1=0$ is $ P'=(0: 1: 0) $, which is a point in $PG(2, q) $. By Result \ref{th:I-II} the case $ a_0 =0 $ can be dismissed and $ a_0 =1 $ may be assumed. Therefore, after the dehomogenization with respect to $X_0$, consider the affine coordinates $(X,Y,Z)$ for a point in $PG(3,q^2)$. We may limit ourselves to a point $P=(a,b,0)$ such that $a^{q+1}+2 b^{q+1}=0$. The latter equation holds for $a=\pm h^2 $ and $b= h$. Then we may choose $P=(2\varepsilon, h, 0 )$, where $\varepsilon \in \{-1,1\}$, and we can carry out the case $\varepsilon=1$ and $\varepsilon=-1$ simultaneously.
\subsection{Case of $\mathcal{G}_1$}\label{subsec281}
We keep up our notation $P_{u,v}=(u,v,v^2)$ for a point in $\Delta^+$. The following lemmas are analogous to those in \cite[section 7.1]{KNS_Hemi}.
\begin{lem}
 Let $v \in F_{q^{2}} \setminus F_q $. Then there exists $u \in F_{q^{2}}$ such that the line joining $P$ at $P_{u,v}$ is a generator of $H(3,q^2)$ if and only if
 \begin{equation}\label{v}
  (v^2+2hv)^\frac{q+1}{2} = 2\varepsilon(v^q-v).
 \end{equation}
 If \emph{(\ref{v})} holds, then $u$ is uniquely determined by $v$.
\end{lem}
\begin{proof}
 The line $\ell=PP_{u,v}$ is a generator if and only if $P_{u,v}$ lies on the tangent plane to $H(3,q^2)$ at $P$. This implies
 \begin{equation}\label{u}
  u=\frac{v^2+2hv}{2\varepsilon}.
\end{equation}
  and since $P_{u,v} \in \Delta^+$ then $u^\frac{q+1}{2}=v^q-v$ and $\ell$ is a generator. The converse follows from the proof of \cite[Lemma 7.4]{KNS_Hemi}.
\end{proof}
Lemma \ref{v} can be extended to $Q_{s,t} \in \Delta^-$, providing that $u,v$ are replaced by $s,t$ and Equations \eqref{v}, \eqref{u} by
\begin{equation} \label{t}
 (t^2+2ht)^\frac{q+1}{2}= -2\varepsilon(t^q-t),
\end{equation}
and
\begin{equation} \label{s}
 s=\frac{t^2+2ht}{2\varepsilon}.
\end{equation}
Furthermore $ P,P_{u,v} $ and $ Q_{s,t} $ are collinear if and only if
\begin{equation}
 \begin{cases}
  \begin{aligned}
	&2\varepsilon(t^2-v^2) =t^2u-v^2s, \\
    &vt-h(v+t)=0.
  \end{aligned}	
 \end{cases}
\end{equation}
Therefore, the following lemma holds.
\begin{lem}\label{lm:coll}
 Let $ v, t \in F_{q^{2}} \setminus F_q  $ with $ F(v, t) =0 $. If the line through $ P_{u,v} \in \Delta^+ $ and $ Q_{s,t} \in \Delta^- $ is a generator through $P$, then
 \begin{equation}\label{eq:coll}
  vt-h(v+t)=0
 \end{equation}
 holds.
\end{lem}
We now count the number of generators in $\mathcal{G}_1$ which pass through $P$.
\begin{lem} \label{lm:soluz}
 Equation \emph{(\ref{v})} has exactly $ \frac{1}{2}(q+1) $ solutions in $ F_{q^{2}} \setminus F_q $.
\end{lem}
\begin{proof}
 Let $ r=vh^{-1} $. We obtain:
 $$(r^2+2r)^\frac{q+1}{2} = \varepsilon h(r^q+r).$$
Hence,
$$(r^2+2r)^\frac{q^2-1}{2}=-1$$
and then $r^2+2r$ is a non-square of $F_{q^{2}}$. Thus, there exists $z \in F_{q^{2}}^{*}$ such that $r^2+2r=h z^2$.  Now the system is
\begin{equation}\label{z21}
 \begin{cases}
  \begin{aligned}
   &h z^2 = r^2+2r \\
   &\alpha h z^{q+1} = \varepsilon h(r^q + r).
  \end{aligned}
  \end{cases}
\end{equation}
Let $\lambda = zr^{-1}$, so that
\begin{equation}\label{sL}
 \begin{cases}
  \begin{aligned}
   &h \lambda^2r = r+2, \\
   &\alpha(\lambda r)^{q+1} = \varepsilon(r^q + r).
  \end{aligned}
 \end{cases}
\end{equation}
Since $r = \frac{2}{h\lambda^2-1}$ we obtain
\begin{equation}\label{lambda}
 4 \alpha\lambda^{q+1}-2\varepsilon(h\lambda^{2}-1)-2\varepsilon(h\lambda^2-1)^q=0.
\end{equation}
Now if $\lambda=\lambda_1+ h \lambda_2$, with $\lambda_1$, $\lambda_2 \in F_q$, Equation (\ref{lambda}) reads
\begin{equation}\label{eq:l12}
 \alpha\lambda_1^2-2\alpha\lambda_2^2-4\varepsilon \lambda_1\lambda_2+\varepsilon=0.
\end{equation}
Since the determinant of the matrix of the quadratic form associated to \eqref{eq:l12} is $-2\varepsilon$, that quadratic form is the equation of an irreducible conic of $PG(2,q)$. Thus, we have exactly $q+1$ solutions $\lambda$ of \eqref{lambda}: if $\lambda$ is a solution, then $-\lambda$ is too, hence we have $\frac{q+1}{2}$ values for both $r$ and $v$.	Every solution $v$ of (\ref{v}) is in $F_{q^{2}} \setminus F_q$. In fact, if $(hr)^q =hr$ then $r^q=-r$ and  $\lambda r =0$, which contradicts the first equation of (\ref{sL}).
\end{proof}
Note that $\alpha=h^\frac{q-1}{2}$ is a non-square of $F_q$.
\begin{lem}\label{quad}
 For every solution $v=v_1+h v_2$ of \eqref{v},
 $$\varepsilon v_2+\frac{\alpha}{2}(v_1 v_2 + v_1) \notin \square_q.$$ 	
\end{lem}
\begin{proof}
 Consider System \eqref{z21} and let  $z=z_1+hz_2$ and $r=r_1+hr_2$. Then
 \begin{equation}
  \begin{cases}
   \begin{aligned}
    &z_1^2+2 z_2^2=2r_1r_2+2r_2\\
	&\alpha  z_1^2-2 \alpha z_2^2 = 2 \varepsilon r_1.
   \end{aligned}
  \end{cases}
 \end{equation}
 Summing the two equations we have:
 $$\alpha z_1^2 = \alpha r_2 (r_1 +1)+\varepsilon r_1.$$
 Since $\alpha^2 =-1$ and $q \equiv 5 \pmod 8 $, it follows $ \alpha \notin \square_q$ and then
 $$\alpha r_2 (r_1 +1)+\varepsilon 	r_1$$
 is a non-square of $F_q$. With $v_2=r_1$ and $v_1=2 r_2$ we obtain
 $$\varepsilon v_2+\frac{\alpha}{2}(v_1 v_2 + v_1) \notin \square_q.$$
\end{proof}
Our next step is to characterize the generators of $\mathcal{G}_1$ through $P$.\\
To begin with, we need some notions of number theory, which would allow us to simplify the notation we will use. Note that $(2+h)^\frac{q+1}{2}=\lambda h$, where,
\begin{equation}\label{eq:hou}
 \lambda=(2+h)^\frac{q+1}{2}h^{-1}=[(1+h)h]^\frac{q+1}{2}h^{-1}=(1+h)^\frac{q+1}{2}h^\frac{q-1}{2}.
\end{equation}
Since
$$\lambda^2=(1+h)^{q+1}2^\frac{q-1}{2}=(1+h)(1-h)(-1)=1,$$
we have $\lambda=\pm 1$. Applying the Frobenius map to \eqref{eq:hou} gives
$$\lambda=(1-h)^\frac{q+1}{2}(-h)^\frac{q-1}{2}.$$
Hence $\lambda$ is independent of the choice of $h$ as a square root of $2$.
\begin{prop}\label{pr:hou}
 We have
\begin{equation*}
 \lambda=\begin{cases}
 \begin{aligned}
			1&,  \quad q \equiv 13 &\pmod {16} \\
			-1&, \quad q \equiv 5 &\pmod {16}
		\end{aligned}
	\end{cases}
\end{equation*}
\end{prop}
\begin{proof}
We do the computation for $q \equiv 13 \pmod{16}$, the proofs for the other cases being analogous.
\begin{equation*}
    \begin{aligned}
        (1+h)^\frac{q+1}{2}h^\frac{q-1}{2} &\equiv (1+\sqrt{2})^\frac{q+1}{2}(\sqrt{2})^\frac{q-1}{2} \pmod{\mathfrak{b}} \\
        & =(\sqrt{2}+2)^\frac{q+1}{2} \frac{1}{\sqrt{2}}\\
        &=(\zeta_8+\zeta_8^{-1}+2)^\frac{q+1}{2}\frac{1}{\sqrt{2}}\\
        &=(\zeta_{16}+\zeta_{16}^{-1})^{q+1}\frac{1}{\sqrt{2}}\\
        &\equiv (\zeta_{16}+\zeta_{16}^{-1})(\zeta_{16}^{13}+\zeta_{16}^{-13})\frac{1}{\sqrt{2}} \pmod{\mathfrak{b}}\\
       &\equiv (\zeta_{16}+\zeta_{16}^{-1})(\zeta_{16}^{-3}+\zeta_{16}^{3})\frac{1}{\sqrt{2}} \pmod{\mathfrak{b}}\\
       &=(\zeta_{16}^4+\zeta_{16}^{-2}+\zeta_{16}^2+\zeta_{16}^{-4})\frac{1}{\sqrt{2}}\\
       &=(\zeta_8+\zeta_8^{-1})\frac{1}{\sqrt{2}}=1.
    \end{aligned}
\end{equation*}
\end{proof}		
Let
\begin{equation*}
 \chi=\begin{cases}
 -1, \mbox{ if either } \varepsilon = 1 \mbox{ and } q \equiv 13 \pmod{16} \mbox{ or } \varepsilon = -1 \mbox{ and } q \equiv 5 \pmod{16}\\
  1, \mbox{ if either } \varepsilon = 1 \mbox{ and } q \equiv 5 \pmod{16} \mbox{ or } \varepsilon = -1 \mbox{ and } q \equiv 13 \pmod{16}
 \end{cases}
\end{equation*}
According to Proposition \ref{pr:hou}, we have $\lambda \chi = -\varepsilon$ and hence
\begin{equation}\label{nagy}
 (2\pm \chi h)^\frac{q+1}{2}=\pm\chi \lambda h = \mp\varepsilon h.
\end{equation}
Furthermore,
$$v_0= -2(h-2\chi), \quad u_0=\frac{4}{\varepsilon}(2-\chi h)$$
and
$$t_0=-2(h+2\chi), \quad s_0=\frac{4}{\varepsilon}(2+\chi h)$$
Equation \eqref{nagy} implies
\begin{equation}
 v_0^q-v_0= 4 h = u_0^\frac{q+1}{2}.
\end{equation}
and
$$u_0^qs_0=16(2-h \chi)^2=(t_0-v_0^q)^2.$$
Furthermore,
$$(v_0+t_0)^{q+1}=-32=2(t_0v_0+(t_0v_0)^q).$$
Therefore,  $F(v_0,t_0)=0$. Thus, from Result \ref{lm:G1}, the line through $P_{u_0,v_0}$ and $Q_{s_0,t_0}$ is a generator $g_0 \in \mathcal{G}_1$. Moreover the following hold:
$$u_0=\frac{v_0^2+2 h v_0}{2 \varepsilon}, \quad s_0=\frac{t_0^2+2 h t_0}{2 \varepsilon},$$
showing that $g_0$ passes through $P$. We show how each generator $g$ passing through P can be obtained from $ g_0 $. If $g=P_{u,v}Q_{s,t}$ is a line through $P$, then, by Lemma \ref{lm:coll}, $F(v,t)=0$ and $vt=h(v+t)$. Now for $\alpha,\beta,\gamma$ and $\delta \in F_q$, with $\alpha \delta-\beta\gamma \ne 0$, write
\begin{equation*}\label{key}
 v=\frac{\alpha v_0+\beta}{\gamma v_0 + \delta}, \quad t=\frac{\alpha t_0+\beta}{\gamma t_0+\delta}.
\end{equation*}
From $v_0 t_0=-8$ and $v_0+t_0=-4 h$, we may write Equation \eqref{eq:coll} as
\begin{equation}\label{eq:doppie}
 \begin{aligned}
  8\alpha\gamma&=2\alpha\beta+\beta\delta, \\
  \beta^2&=8(\alpha^2-\alpha \delta-\beta\gamma).
 \end{aligned}
\end{equation}
Our aim is to show that these equations hold if and only if $\alpha,\beta,\gamma$ and $\delta$ depend on a unique parameter $\xi \in F_q \cup \{\infty\}$. To begin with, let $\delta\ne 0$. Then $\alpha \ne 0$. The first equation in \eqref{eq:doppie} forces
$$\gamma=\frac{(2\alpha+1)\beta}{8\alpha}.$$
together with the other equation, we have
$$8\alpha^3-3\alpha\beta^2-8\alpha^2-\beta^2=0.$$
Let $\xi=\beta \alpha^{-1}$. This implies $\alpha^2(8\alpha -3\alpha \xi^2-8-\xi^2)=0$. Therefore
$$\alpha=\frac{\xi^2+8}{8-\xi^2},$$
and the assertion follows for $\delta \neq 0$. For $\delta=0$ we may assume $\beta=1$. If $\alpha\neq0$ then $\gamma=\frac{1}{4}$ and  $8\alpha^2=-1$, which is impossible as $-1$ is a square in $F_q$ while $8$ is not. When $\delta=\alpha=0$ and $\beta=1$, then $\gamma=\frac{-1}{8}$. Therefore,
\begin{equation}\label{eq:vxi}
 v = v_\xi =\frac{(\xi^2+8)v_0+(\xi^2+8)\xi}{\frac{\xi}{8}(-\xi^2+24)v_0+8-3\xi^2}, \quad v_\infty=\frac{1}{-\frac{1}{8}v_0}=-2(h+2\chi).
\end{equation}
Let $A_\xi$	and $A_\infty$ be two matrices in $GL(2,q)$ representing the fractional linear maps $v_\xi$ and $v_\infty$. Thus,
\begin{equation}
 det(A_\xi)=\frac{(\xi^2+8)(\xi^4-48 \xi^2+64)}{8}, \quad det(A_\infty)=(8)^{-1}.
\end{equation}
These equations remain true for  $t_0$ and $t$:
\begin{equation}\label{eq:txi}
 t = t_\xi =\frac{(\xi^2+8)t_0+(\xi^2+8)\xi}{\frac{\xi}{8}(-\xi^2+24)t_0+8-3\xi^2}, \quad t_\infty=\frac{1}{-\frac{1}{8}t_0}=-2(h-2\chi).
\end{equation}
Next we show that Lemma \ref{quad} imposes a condition on $\xi$ in  \eqref{eq:vxi}.
\begin{lem}\label{prin}
 $\xi^2+8$ is a square in $F_q$.
\end{lem}
\begin{proof}
 To use Lemma \ref{quad} we rewrite $ \varepsilon v_2+\frac{\alpha}{2}(v_1 v_2 + v_1) $ in terms of $\xi$. This requires a certain amount of straightforward and tedious computations that we omit. From \eqref{eq:vxi}, we have
 \begin{equation}
  v=\frac{4(\xi^2 + 8)}{\chi16 -\chi 2 \xi^2 + h (8 -\chi 8 \xi +\xi^2)}
 \end{equation}
 and
 \begin{equation}
  v_1=\frac{-4 (\chi 16 -\chi 2 \xi^2) (8 + \xi^2)}{k}, \quad v_2 = \frac{-4 (8 + \xi^2) (8 -\chi 8 \xi + \xi^2)}{k}
 \end{equation}
 where $k=128 +\chi  256 \xi - 224 \xi^2 +\chi  32 \xi^3 + 2 \xi^4$. Then,
 \begin{equation}
  \footnotesize{\varepsilon v_2+\frac{\alpha}{2}(v_1 v_2 + v_1)= \frac{2(1-\chi  \varepsilon \alpha) (8 + \xi^2) ((-16 + 16 \alpha) + (8 + 32 \alpha) \xi + (6 + 10 \alpha) \xi^2 +\xi^3)^2}{(64 - 128 \xi - 112 \xi^2 - 16 \xi^3 + \xi^4)^2}}
 \end{equation}
 \normalsize
 Note that $(1+\alpha)(1-\alpha)=2$ and that $1+\alpha \in \square_q$ if and only if $ q \equiv 13 \pmod{16}$. In fact,
 $$1+\alpha = \pm h^\frac{q+3}{4} \in \square_q \iff h^\frac{(q-1)(q+3)}{8}=1$$
 and in this case $1-\alpha $ is a non-square in $F_q$. Since $\chi \varepsilon=1$ when $ q \equiv 5 \pmod {16}$ and $\chi \varepsilon=-1$ when $ q \equiv 13 \pmod {16}$, we get that $1-\chi\varepsilon\alpha$ is always a square in $F_q$.	Hence $\xi^2+8 \in \square_q $.
\end{proof}
To state a corollary of Lemmas \ref{lm:soluz}, \ref{quad} and \ref{prin}, the partition of $F_q \cup \{\infty\}$ into two subsets $\Sigma_1 \cup \{\infty\}$ and $\Sigma_2$ is useful, where $x \in \Sigma_1 \cup \{\infty\}$ or $x \in \Sigma_2$ according as $x^2+8 \in \square_q$ or not.
\begin{prop}
 Let $P=(2\varepsilon,h,0) \in H(3,q^2)$ with $h^2 =2$. Then the generators in $\mathcal{G}_1$ through the point $P$ which meet $\mathcal{X}^+$ are as many as $n_P=\frac{1}{2}(q+1)$. They are precisely the lines $g_\xi$ joining $P$ to $P_{u,v}=(u,v,v^2)$ with $u,v$ as in equation \eqref{u} and \eqref{eq:vxi}, where $\xi$ ranges over the set $\Sigma_1 \cup \{\infty\}$.
\end{prop}
\subsection{Case of $\mathcal{G}_2$}\label{subsec282}
This case requires much less effort. The tangent plane $\pi_P$ at $P=(2\varepsilon:h:0)$ meets $\pi$ in the line $r$ of equation $2h^qY+Z=0$. Since $\mathcal{C}$ has equation $Z=Y^2$ in $\pi$, the only common points of $r$ and $\mathcal{C}$ are $(0:0:0)$ and $Q=(0:2h:8)$, with $Q \notin \Omega$ as $h \notin F_q$. Then we have the following result.
\begin{prop}
 Let $P=(2\varepsilon,h,0) \in H(3,q^2)$, with $h^2 = 2$. Then there is a unique generator through the point $P$ which meets $\Omega$, namely the line $\ell$ through $P$ and the origin $O=(0:0:0)$.
\end{prop}
From now on, we denote with $\ell^+$ and $\ell^-$ the two generators through $P$ when $\varepsilon=1$ and $\varepsilon=-1$ respectively.
\subsection{Choice of $\mathcal{M}_1$ and $\mathcal{M}_2$}\label{subsec283}
In this subsection, we are going to choose $\mathcal{M}_1$ and $\mathcal{M}_2$ such that Condition (B) is fulfilled. We have two different generators $g_0$'s, one for $\varepsilon=1$, the other for $\varepsilon=-1$:
$$g_0^+ \mbox{ passing through } P^+(2:h:0)$$
and
$$g_0^- \mbox{ passing through } P^-(-2:h:0).$$
\begin{lem}
 The generators $ g_0^+ $ and $g_0^-$ are in different orbits of $\Phi$.
\end{lem}	
\begin{proof}
 The linear collineation associated to the matrix $\textbf{W}$ interchanges the two generators.
\end{proof}
Let $ r $ (resp. $ r' $) be the number of generators in $ \mathcal{M}_1 $ (resp. $\mathcal{M}_1'$) through the point $ P^+ $ that meet $\Delta^+$. Note that
\begin{equation}\label{eq:r+r'}
 r+r' = \frac{q+1}{2}.
\end{equation}
Similarly,
\begin{lem}
 The generators $\ell^+$ and $\ell^-$ are in different orbits of $\Phi$.
\end{lem}
\begin{proof}
 We use the same arguments of \cite[Lemma 7.14]{KNS_Hemi}. Indeed, we replace $(\sqrt{-2}b,b,0)$ and $(-\sqrt{-2}b,b,0)$ with $P^+$ and $P^-$ and the proof follows.
\end{proof}
We are ready to choose $\mathcal{M}_1$ and $\mathcal{M}_2$.
\begin{itemize}
 \item $ \mathcal{M}_1 $ is the $\Phi$-orbit containing $g_0^+$.
 \item $\mathcal{M}_2$ is the $\Phi$-orbit containing $\ell^+$ for $r < r'$ and $\ell^-$ for $r > r'$.
\end{itemize}
\begin{remark}\label{rk:f(x)}
 As in \cite[Proposition 7.15]{KNS_Hemi}, $r'$ is obtained counting the squares in the value set of the polynomial $f(\xi)$, defined in Theorem \ref{th:III}. More precisely, we obtain that the number of $\xi \in F_q$ for which $f(\xi) \in \square_q$ equals $2r'-1$.
\end{remark}
Therefore we have the following proposition.
\begin{prop}
 Condition \emph{(B)} for case \emph{(3} holds if and only if
 $$r=\frac{q-1}{4}, \mbox{ and } r'=\frac{q+3}{4},$$
 or
 $$r=\frac{q+3}{4}, \mbox{ and } r'=\frac{q-1}{4}.$$
\end{prop}
\begin{proof}
 Note that $n_P=\frac{q+3}{2}$ and that condition (B) holds if and only if half of them is in $\mathcal{M}_1 \cup \mathcal{M}_2$. The choices of $r$ and $r'$ are readily seen.
\end{proof}
Thus, Theorem \ref{th:III} follows. Since the properties of the plane curve
$$\mathcal{C}_4:Y^2=X^4-24\omega X^2+16\omega^2, \mbox{ with } \omega=2$$
depend only on $q \equiv 1 \pmod 4$, we also get  the proof of Theorem \ref{th:main}, that is Condition (B) in case (3) is satisfied if and only if the curve
$$\mathcal{C}_3:Y^2=X^3-X$$
has $q-1$ or $q+3$ points. For the details, see \cite{KNS_Hemi} at the end of Section 7.
\subsection{Proof of Theorem \ref{th:1}}\label{subsec284}
We are in the position to work out the case $q=p$ when $p \equiv 1 \pmod{4}$. We write $p=\pi \overline{\pi}$, with $\pi \in \mathbb{Z}[i]$. Here, $\pi$ can be chosen such that $\pi= \alpha_1 + i \alpha_2$ and $\alpha_1=1$. From \cite[Section 2.2.2]{3}, $N_p(\mathcal{C}_3)=q+1-2\alpha_1$. This implies that condition (B) in case (3) is satisfied if and only if
$$p=1+4a^2\quad \mbox{ and } \quad N_p(\mathcal{C}_3)=q-1.$$
Therefore, Theorem \ref{th:1} is a corollary of Theorem \ref{th:CD}, Result \ref{th:I-II} and Theorem \ref{th:III}. Further computer-aided investigations in the case $q=5$ showed that the found hemisystem is isomorphic to a sporadic case described in \cite{cossidente2005hemisystems}, whose full automorphism group is $3.A_{7}$. In all other cases, there were not other known examples stabilized by $PSL(2, q)\times C_\frac{q+1}{2}$, so the above mentioned sporadic hemisystem should be the first known example in our putative new family.

\chapter{Partial spreads and partial ovoids}\label{ch3}
As it was pointed out in Subsection \ref{subsec164}, the point line dual of $W(3,q)$ is $Q(4,q)$, the parabolic quadric of $PG(4,q)$. Therefore spreads of $Q(4,q)$ and ovoids of $W(3,q)$ are equivalent objects and spreads of $W(3,q)$ and ovoids of $Q(4,q)$ are equivalent objects. In odd characteristic, currently, infinite families of ovoids of $Q(4,q)$ are known to exist \cite{PW}. The generalized quadrangle $W(3,q)$ is self-dual if and only if $q$ is even. Hence, when $q$ is even, ovoids and spreads of $W(3,q)$ are equivalent objects, and the same holds for $Q(4,q)$.
\begin{defn}
 An \textit{ovoid} $\mathcal{O}$ of a polar space $\mathcal{P}$ is a set of points of $\mathcal{P}$ such that every generator contains exactly one point of $\mathcal{O}$.
\end{defn}

In this chapter we focus on the symplectic polar space $W(2n-1, q)$ and the Hermitian polar space $H(2n, q^2)$, consisting of the absolute projective subspaces with respect to a non-degenerate symplectic polarity of $PG(2n-1, q)$ and a non-degenerate unitary polarity of $PG(2n, q^2)$, respectively. These results were originally contained in the paper \cite{VS4}. As was shown in Table \ref{tabHirschfeld}, it was proven the non-esxistence of spreads in some polar spaces. Take for example the parabolic quadric $Q(4,q)$: in \cite{Tallini}, G. Tallini showed that for $q$ odd, spreads do not exist.
\begin{thm}\cite[Proposition 1.8]{Tallini}
 \label{TalliniSp}
 $Q(4,q)$ has a spread if and only if $q$ is even.
\end{thm}
\begin{proof}
 Let $\mathcal{S}$ be a line spread of $Q(4,q)$. From Proposition \ref{KleinQ}, points of the $Q(4,q)$ embedded in a $Q^{+}(5,q)$ arise under the Klein correspondance from a complex of lines $\mathcal{C}$. A spread $\mathcal{S}$ which defines a $q^{2}+1$-cap in $PG(3,q)$ whose tangent lines are the lines of the complex $\mathcal{C}$, hence $q$ must be even.
\end{proof}
Since spreads of $Q(4,q)$ and ovoids of $W(3,q)$ are equivalent objects Theorem \ref{TalliniSp} can be restated as follows.
\begin{thm}\cite[Proposition 1.8]{Tallini}
 $W(3,q)$ has an ovoid if and only if $q$ is even.
\end{thm}
The non-existence of spreads or ovoids naturally imply the definition of a structure with less strict conditions.

\begin{defn}
 \begin{itemize}
  \item A \textit{partial spread} $\mathcal{S}$ of a polar space $\mathcal{P}$ is a set of generators of $\mathcal{P}$ such that every point lies on at most one generator of $\mathcal{S}$. A partial spread is said to be \textit{maximal} if it is maximal with respect to set-theoretic inclusion.
  \item A \textit{partial ovoid} $\mathcal{O}$ of a polar space $\mathcal{P}$ is a set of points of $\mathcal{P}$ such that every generator contains at most one point of $\mathcal{O}$. A partial ovoid is said to be \textit{maximal} if it is maximal with respect to set-theoretic inclusion.
 \end{itemize}
\end{defn}

It was later proven that $W(2n-1, q)$ has ovoids if and only if $n =2$ and $q$ is even, whereas $H(2n, q^2)$ does not possess ovoids. Hence the question of the largest (maximal) partial ovoids of $W(2n-1, q)$, $(n, q) \neq (2, 2^h)$, and $H(2n, q^2)$ naturally arises. The main goal of the chapter is to provide constructive lower bounds on the sizes of the largest partial ovoids of the symplectic polar spaces $W(3, q)$, $q$ odd square, $q \not\equiv 0 \pmod{3}$, $W(5, q)$ and of the Hermitian polar spaces $H(4, q^2)$, $q$ even or $q$ odd square, $q \not\equiv 0 \pmod{3}$, $H(6, q^2)$, $H(8, q^2)$. We summarize what is known about maximal partial ovoids of symplectic polar spaces $W(2n-1, q)$ and of Hermitian polar spaces $H(2n, q^2)$. Constructions of maximal partial ovoids of $W(3, q)$, $q$ even, of size an integer between about $\frac{q^2}{10}$ and $\frac{9q^2}{10}$ or of size $q^2-hq+1$, $1 \leq h \leq \frac{q}{2}$, have been provided in \cite{RS, Tallini}. The situation is somewhat different for $q$ odd where the lack of examples is transparent. In \cite{Tallini} G. Tallini proved that a partial ovoid of $W(3, q)$, $q$ odd, has size at most $q^2-q+1$ and constructed a maximal partial ovoid of $W(3, q)$ of size $2q+1$. Regarding symplectic polar spaces in higher dimensions an upper bound on the size of the largest partial ovoid has been provided in \cite{BKMS1} and, if $q$ is even, a partial ovoid of an elliptic or hyperbolic quadric is also a partial ovoid of a symplectic polar space. As for $H(2n, q^2)$, an upper bound on the size of the largest partial ovoid can be found in \cite{BKMS}. In particular, a partial ovoid has at most $q^5-q^4+q^3+1$ points if $n = 2$. The largest known example of a maximal partial ovoid of $H(2n,q^2)$, $n = 2, 3$, occurs when $q=3^h$ and has size $q^4+1$ \cite{MPS}. A straightforward check shows that a non-degenerate plane section of $H(2n,q^2)$ is an example of maximal partial ovoid of $H(2n,q^2)$ of size $q^3+1$. Other examples of maximal partial ovoids of $H(4,q^2)$ of size $2q^3+q^2+1$ have been constructed in \cite{CP} and of size $q^3+1$ in \cite{CS, MPS}.

The following tables summarize old and new results regarding large partial ovoids of symplectic and Hermitian polar spaces in small dimensions.

\begin{table}[h!]
\label{Tab1}\footnotesize{
\begin{tabular}{|c|c|c|}
\hline
 & lower bound & upper bound \\
\hline
$W(3, q)$, $q$ even & $q^2+1$ & $q^2+1$ \\
\hline
$W(3, q)$, $q$ odd 	& {\boldmath $\frac{q^{\frac{3}{2}} +3q - q^{\frac{1}{2}}+3}{3}$}, $q = p^{2h}$, $p \neq 3$ & $q^2-q+1$ \cite{Tallini}\\
 & $2q+1$, $q = p^{2h+1}$ or $q = 3^h$ \cite{Tallini} & \\
\hline
$W(5, q)$ & {\boldmath $2q^2-q+1$}, $q$ even & $\frac{q\sqrt{5q^4+6q^3+7q^2+6q+1}-q^3-q^2-q+2}{2}$ \cite{BKMS1}\\
 & {\boldmath $q^2+q+1$}, $q$ odd & \\
	& $7$, $q = 2$ & $7$, $q = 2$ \\
\hline
$W(7, q)$ & $q^3+1$ \cite{C, CK} & $q^4-q^3-q(q^{\frac{1}{2}}-1)(q-q^{\frac{1}{2}}+1)+3$, \\
& &$q > 2$ \cite{BKMS1} \\
	& & $9$, $q = 2$ \\
\hline
\end{tabular}}

\caption{\footnotesize{Large partial ovoids of $W(2n-1, q)$, $n \in \{2, 3, 4\}$.}}
\end{table}

\begin{table}[h!]
\label{Tab2}\footnotesize{
\begin{tabular}{|c|c|c|}
\hline
 & lower bound & upper bound \\
\hline
$H(4, q^{2})$ & {\boldmath $q^4+1$}, $q = 2^h$ or $q = 3^h$ \cite{MPS} & $q^5-q^4+q^3+1$ \cite{BKMS}  \\
 & {\boldmath $\frac{q^{\frac{7}{2}}+3q^3-q^{\frac{5}{2}}+3q^2}{3}$}, $q=p^{2h}$, $p$ odd, $p\neq3$ & \\
 &  $2q^{3}+q^{2}+1$, $q = p^{2h+1}$, $p \neq 2, 3$ \cite{CP} &\\
\hline
$H(6, q^{2})$	& {\boldmath $2q^{4}-q^{3}+1$}, $q$ even & $q^7-q^6+q^5-q^3+2$ \cite{BKMS}\\
 & {\boldmath $q^4+q^3+1$}, $q$ odd & \\
\hline
$H(8, q^2)$ & {\boldmath $q^5+1$} & $q^9-q^8+q^7-q^5-q^3+q^2+1$ \cite{BKMS}\\
\hline
\end{tabular}}
\caption{\footnotesize{Large partial ovoids of $H(2n, q^2)$, $n \in \{2,3,4\}$.}}
\end{table}

\section{Partial ovoids of symplectic polar spaces}\label{sec31}
\subsection{$W(3, q)$, $q$ odd square, $q \not\equiv 0 \pmod{3}$}\label{subsec311}
Let $W(3, q)$, $q \not\equiv 0 \pmod{3}$, be the symplectic polar space consisting of the subspaces of $PG(3, q)$ induced by the totally isotropic subspaces of $V(4,{q})$ with respect to the non-degenerate alternating form $\beta$ given by
\begin{align}
 \beta(x,y)=x_0 y_3 + x_1 y_2 - x_2 y_3 - x_3 y_0. \label{bilinear}
\end{align}
Denote by $\mathfrak{s}$ the symplectic polarity of $PG(3, q)$ defining $W(3, q)$. Let $\mathcal{C}$ be the twisted cubic of $PG(3, q)$ consisting of the $q+1$ points $\{P_t | t \in F_q\} \cup \{(0,0,0,1)\}$, where $P_t = (1, -3 t, t^2, t^3)$. It is well known that a line of $PG(3, q)$ meets $\mathcal{C}$ in at most $2$ points and a plane shares with $\mathcal{C}$ at most $3$ points (i.e., $\mathcal{C}$ is a so called \textit{$(q+1)$-arc}). A line of $PG(3,q)$ joining two distinct points of $\mathcal{C}$ is called a \textit{real chord} and there are $\frac{q(q+1)}{2}$ of them. Let $\bar{\mathcal{C}} = \{P_t | t \in F_{q^2}\} \cup \{(0,0,0,1)\}$ be the twisted cubic of $PG(3, q^2)$ which extends $\mathcal{C}$ over $F_{q^2}$. The line of $PG(3, q^2)$ obtained by joining $P_t$ and $P_{t^q}$, with $t \notin F_q$, meets the canonical Baer subgeometry $PG(3, q)$ in the $q+1$ points of a line skew to $\mathcal{C}$. Such a line is called an \textit{imaginary chord} and they are $\frac{q(q-1)}{2}$ in number. If $r$ is a (real or imaginary) chord, then the line $r^{\mathfrak{s}}$ is called (\textit{real} or \textit{imaginary}) \textit{axis}. Also, for each point $P$ of $\mathcal{C}$, the line $\ell_{P} = \langle P, P' \rangle$, where $P'$ equals $(0, -3, 2t, 3t^2)$ or $U_3$ if $P = P_t$ or $P = U_4$, respectively, is called the \textit{tangent} line to $\mathcal{C}$ at $P$. With each point $P_t$ (respectively $U_4$) of $\mathcal{C}$ there corresponds the \textit{osculating plane} $P_t^{\mathfrak{s}}$ (respectively $U_4^{\mathfrak{s}}$) with equation $t^3 X_0 + t^2 X_1 + 3t X_2 - X_3 = 0$ (respectively $X_0 = 0$), meeting $\mathcal{C}$ only at $P_t$ (respectively $U_4$) and containing the tangent line. Hence the $q+1$ lines tangent  to $\mathcal{C}$ are generators of $W(3, q)$ and they form a regulus $\mathcal{R}$ if $q$ even. Every point of $PG(3, q) \setminus \mathcal{C}$ lies on exactly one chord or a tangent of $\mathcal{C}$. For more properties and results on $\mathcal{C}$ the reader is referred to \cite[Chapter 21]{Hirschfeld2}. Let $G$ be the group of projectivities of $PG(3, q)$ stabilizing $\mathcal{C}$. Then $G \simeq PGL(2, q)$ whenever $q \geq 5$, and elements of $G$ are induced by the matrices
\begin{align*}
 M_{a,b,c,d} = \begin{pmatrix}
 a^3 & -a^2 b & 3 a b^2 & b^3 \\
 -3a^2 c & a^2 d + 2 abc & -3b^2 c -6 abd & -3b^2 d \\
 a c^2 & \frac{-bc^2 - 2 acd}{3} & ad^2 + 2 bcd & b d^2 \\
 c^3 & - c^2 d & 3 c d^2 & d^3 \\
 \end{pmatrix},
\end{align*}
 where $a,b,c,d \in F_q$, $ad-bc \neq 0$. The group $G$ leaves invariant $W(3, q)$ since $M_{a,b,c,d}^T J M_{a,b,c,d} = (ad-bc)^3 J$, where
\begin{align*}
 & J = \begin{pmatrix}
 0 & 0 & 0 & 1 \\
 0 & 0 & 1 & 0 \\
 0 & -1 & 0 & 0 \\
 -1 & 0 & 0 & 0 \\
 \end{pmatrix}
\end{align*}
is the Gram matrix of $\beta$.
Assume $q$ to be an odd square. Here we show the existence of a partial ovoid of $W(3, q)$ obtained by glueing together the twisted cubic $\mathcal{C}$ of $PG(3, q)$ and an orbit of size $\frac{\sqrt{q}(q-1)}{3}$ of a subgroup $G_{\epsilon}$ of $PSp(4, q) \cap G$ isomorphic to $PGL(2, \sqrt{q})$ stabilizing $\mathcal{C}$. In particular such a subgroup fixes a twisted cubic $\mathcal{C}_{\epsilon} = \mathcal{C} \cap \Lambda_{\epsilon}$ of a Baer subgeometry $\Lambda_{\epsilon}$ of $PG(3, q)$.

\fbox{$\sqrt{q} \equiv 1 \pmod{3}$}
Assume $\sqrt{q} \equiv 1 \pmod{3}$. Let $\mathcal{C}_{1} = \{P_t | t \in F_{\sqrt{q}} \} \cup \{U_4\}$. Thus $\mathcal{C}_1 \subset \mathcal{C}$ is a twisted cubic of the canonical Baer subgeometry $\Lambda_{1} = PG(3, \sqrt{q})$ of $PG(3, q)$. The group $G_1$ of projectivities stabilizing $\mathcal{C}_1$ is isomorphic to $PGL(2, \sqrt{q})$ and it is induced by the matrices
\begin{align*}
& M_{a,b,c,d}, a,b,c,d \in F_{\sqrt{q}}, ad-bc \neq 0.
\end{align*}
Let $R = U_1 + x U_4$, for a fixed $x \in F_{q} \setminus F_{\sqrt{q}}$ where $x$ is not a cube in $F_{q}$. Set
\footnotesize{
\begin{align*}
 & \mathcal{O}_1 = R^{G_1} = \{(a^3+xb^3, -3a^2c-3xb^2d, ac^2+xbd^2, c^3+xd^3)|a,b,c,d \in F_{\sqrt{q}}, ad-bc \neq 0\}.
\end{align*}}
\normalsize
\begin{lem}
 The set $\mathcal{O}_1$ is a partial ovoid of $W(3, q)$ of size $\frac{\sqrt{q}(q-1)}{3}$.
\end{lem}
\begin{proof}
 A projectivity of $G_1$ fixing $R$ is induces by $M_{a,b,c,d}$, where $a, b, c, d \in F_{\sqrt{q}}$, $ad-bc \neq 0$ and
 \begin{align}
  & c^3+xd^3 = x(a^3+xb^3), \label{eq:stab1} \\
  & a^2c + xb^2d = 0, \label{eq:stab2} \\
  & ac^2 + xbd^2 = 0. \label{eq:stab3}
 \end{align}
 If $bd \neq 0$, then $x = - \frac{a^2c}{b^2d} \in F_{\sqrt{q}}$, a contradiction. Hence either $b = 0$ or $d = 0$. Taking into account that $ad-bc \neq 0$, if the former case occurs, then $c = 0$ and $d = \xi a$, with $\xi \in F_{\sqrt{q}}$, $\xi^3 = 1$, whereas if the latter possibility arises, then $a = 0$ and $x ^2 = \frac{c^3}{b^3} \in F_{\sqrt{q}}$, a contradiction. Hence the stabilizer of $R$ in $G_1$ has order $3$ and $|G_1| = \sqrt{q}(q-1)$, by applying the Orbit-Stabilizer Theorem it follows that $\mathcal{O}_1$ has the required size. In order to show that $\mathcal{O}_1$ is a partial ovoid of $W(3, q)$, it is enough to see that the line joining $R$ and a further point $R^g$ of $\mathcal{O}_1$ is not a generator of $W(3, q)$. Here $g \in G_1$ is induced by $M_{a,b,c,d}$. Assume by contradiction that this is not the case, then there are $a,b,c,d \in F_{\sqrt{q}}$, with $ad-bc \neq 0$, such that $A(x) = 0$, where
 \begin{align*}
  & A(x) = c^3 + x d^3 -x a^3 -x^2 b^3.
 \end{align*}
 Hence $A(x) + A(x)^{\sqrt{q}} = A(x) - A(x)^{\sqrt{q}} = 0$, that is
 \begin{align*}
  & 2c^3 + (d^3-a^3) (x+x^{\sqrt{q}}) - (x^2+x^{2\sqrt{q}}) b^3 = 0, \\
  & d^3 - a^3 -(x+x^{\sqrt{q}}) b^3 = 0.
 \end{align*}
 If $b = 0$, then $a^3 = d^3$ and $c = 0$. Thus $g$ fixes $R$, i.e., $R = R^g$. If $b \neq 0$, then the previous equations imply
 \begin{align*}
  x^{\sqrt{q}+1} = - \frac{c^3}{b^3},
 \end{align*}
 that is a contradiction since $x^{\sqrt{q}+1}$ is not a cube in $F_{\sqrt{q}}$. Indeed $x$ is not a cube in $F_{q}$ and $(\sqrt{q}+1, 3) = 1$.
\end{proof}

\fbox{$\sqrt{q} \equiv -1 \pmod{3}$}

Assume $\sqrt{q} \equiv -1 \pmod{3}$. Let $\mathcal{C}_{-1} = \{P_t | t \in F_{q}, t^{\sqrt{q}+1} = 1 \}$. Thus $\mathcal{C}_{-1} \subset \mathcal{C}$ is a twisted cubic of the Baer subgeometry
\begin{align*}
 & \Lambda_{-1} = \{(\alpha, -3\beta, \beta^q, \alpha^q) | \alpha, \beta \in F_{q^2}, (\alpha, \beta) \neq (0, 0)\} \simeq PG(3, \sqrt{q})
\end{align*}
of $PG(3, q)$. In this case the group $G_{-1}$ of projectivities stabilizing $\mathcal{C}_{-1}$ is isomorphic to $PGL(2, \sqrt{q})$ and it is induced by the matrices $M_{a,b,c,d},$
\begin{align*}
 & a,b,c,d \in F_{q}, ad-bc \neq 0, ab^{\sqrt{q}} - cd^{\sqrt{q}} = 0, a^{\sqrt{q}+1} + b^{\sqrt{q}+1} - c^{\sqrt{q}+1} - d^{\sqrt{q}+1} = 0.
\end{align*}
Let $S = U_1 + x U_4$, for a fixed $x \in F_{q} \setminus F_{\sqrt{q}}$ where $x$ is not a cube in $F_{q}$ and $x^{\sqrt{q}+1} \ne 1$. Set
\begin{align*}
 & & \mathcal{O}_{-1} = S^{G_{-1}} = \{(a^3+xb^3, -3a^2c-3xb^2d, ac^2+xbd^2, c^3+xd^3) | a,b,c,d \in F_{q},\\
 & & ad-bc \neq 0,ab^{\sqrt{q}} - cd^{\sqrt{q}} = 0, a^{\sqrt{q}+1} + b^{\sqrt{q}+1} - c^{\sqrt{q}+1} - d^{\sqrt{q}+1} = 0 \}.
\end{align*}
\begin{lem}
 The set $\mathcal{O}_{-1}$ is a partial ovoid of $W(3, q)$ of size $\frac{\sqrt{q}(q-1)}{3}$.
\end{lem}
\begin{proof}
 A projectivity of $G_{-1}$ fixing $S$ is induced by $M_{a,b,c,d}$, where $a, b, c, d \in F_{q}, ad-bc \neq 0$, are such that $ab^{\sqrt{q}} - cd^{\sqrt{q}} = a^{\sqrt{q}+1} + b^{\sqrt{q}+1} - c^{\sqrt{q}+1} - d^{\sqrt{q}+1} = 0$ and satisfy  \eqref{eq:stab1}, \eqref{eq:stab2}, \eqref{eq:stab3}. If $bd \neq 0$, by expliciting $x$ from \eqref{eq:stab3} and substituting it in \eqref{eq:stab2} gives $ac = 0$, which implies $x = 0$, a contradiction. Hence either $b = 0$ or $d = 0$. Since $ad-bc \neq 0$, if $d = 0$, then $a = 0$ and $x ^2 = \frac{c^3}{b^3} \in F_q$, contradicting the fact that $x$ is not a cube in $F_q$; whereas if $b = 0$, then $c = 0$ and $d = \xi a$, with $\xi \in F_{q}$, $\xi^3 = 1$. In this case $a^3 = d^3$ and $a^{\sqrt{q}+1} = d^{\sqrt{q} +1}$. Hence the stabilizer of $S$ in $G_{-1}$ has order $3$ and $|G_{-1}| = \sqrt{q}(q-1)$, by applying the Orbit-Stabilizer Theorem it follows that $\mathcal{O}_{-1}$ has the required size. In order to show that $\mathcal{O}_{-1}$ is a partial ovoid of $W(3, q)$, it is enough to see that the line joining $S$ and a further point $S^g$ of $\mathcal{O}_{-1}$ is not a generator of $W(3, q)$. Here $g \in G_{-1}$ is induced by $M_{a,b,c,d}$. Assume by contradiction that this is not the case, then there are $a,b,c,d \in F_{q}$, with $ad-bc \neq 0$, $ab^{\sqrt{q}} - cd^{\sqrt{q}} = 0$, $a^{\sqrt{q}+1} + b^{\sqrt{q}+1} - c^{\sqrt{q}+1} - d^{\sqrt{q}+1} = 0$, such that $A(x) = 0$, where
 \begin{align*}
  & A(x) = c^3 + x d^3 -x a^3 -x^2 b^3.
 \end{align*}
 If $b = 0$, then $c = 0$, since $cd^{\sqrt{q}} = 0$ and $ad \neq 0$. Moreover $a^{\sqrt{q}+1} = d^{\sqrt{q}+1}$ and $a^3 = d^3$, since $A(x) = 0$. Thus $g$ fixes $S$, i.e., $S = S^g$. If $b \neq 0$, we may assume w. l. o. g. that $b = 1$. Then $a = cd^{\sqrt{q}}$, $c(d^{\sqrt{q}+1}-1) \neq 0$ and $(1-c^{\sqrt{q}+1})(1-d^{\sqrt{q}+1}) = 0$. Therefore $c^{\sqrt{q}+1} = 1$ and $d^{\sqrt{q}+1} \neq 1$. Moreover $A(x) = 0$ gives
 \begin{align*}
  & d^{3\sqrt{q}} - \frac{d^3}{c^3} - \frac{c^3 -x^2}{c^3 x} = 0.
 \end{align*}
 By considering the last equation in the unknown $d^3$, by \cite[Theorem 1.9.3]{Hirschfeld1}, it admits solutions in $F_{q}$ if and only
 \begin{align*}
 & 0  = \frac{c^{3\sqrt{q}} - x^{2\sqrt{q}}}{c^{3\sqrt{q}} x^{\sqrt{q}}} + \frac{c^3 -x^2}{c^3 x} \frac{1}{c^{3\sqrt{q}}} = \frac{x (1-x^{\sqrt{q}+1}) (c^{3\sqrt{q}} + x^{\sqrt{q}-1})}{c^{3\sqrt{q}} x^{\sqrt{q}+1}}.
 \end{align*}
 Hence either $x = 0$ or $x^{\sqrt{q}+1} = 1$ or $x^{\sqrt{q}-1}$ is a cube in $F_{q}$. If the last possibility occurs then $x$ is a cube in $F_{q}$ since $\sqrt{q} \equiv -1 \mod{3}$. We infer that none of the three cases arises.
\end{proof}
\begin{thm}\label{partial-symp3}
 Let $q$ be an odd square with $\sqrt{q} \equiv \epsilon \pmod{3}$, $\epsilon \in \{\pm1\}$. Then the set $\mathcal{O}_{\epsilon} \cup \mathcal{C}$ is a partial ovoid of $W(3, q)$, of size $\frac{q^{\frac{3}{2}} +3q - q^{\frac{1}{2}} +3}{3}$.
\end{thm}
\begin{proof}
 It is enough to show that there is no line of $W(3, q)$ through $R$ or $S$ meeting $\mathcal{C}$. Assume by contradiction that the line spanned by $R$ (or $S$) and $P_t \in \mathcal{C}$ belongs to $W(3, q)$. Then $P_t \in R^\mathfrak{s}$ (or $S^\mathfrak{s}$) and hence, by \ref{bilinear}, $x = t^3$. A contradiction, since $x$ is not a cube in $F_q$. Similarly the line joining $R$ (or $S$) with $U_4$ is not a line of $W(3, q)$.
\end{proof}
\begin{prop}\label{non-complete}
 For $\epsilon \in \{\pm1\}$, the partial ovoid $\mathcal{O}_{\epsilon} \cup \mathcal{C}$ of $W(3, q)$ is not maximal.
\end{prop}
\begin{proof}
 There are $\frac{\sqrt{q}(q-1)}{3}$ points of $\Lambda_{\epsilon}$ not lying on an osculating plane of $\mathcal{C}_{\epsilon}$, see \cite[p. 235]{Hirschfeld2}. Let $N$ be one of these points. We claim that $\mathcal{O}_{\epsilon} \cup \mathcal{C }\cup \{N\}$ is a partial ovoid of $W(3, q)$. In order to see this fact let $\ell$ be a line of $W(3, q)$. If $\ell$ passes through a point $P$ of $\mathcal{C}_{\epsilon}$, then $P^\mathfrak{s}$ is the osculating plane of $\mathcal{C}_{\epsilon}$ at $P$. Hence $\ell \subset P^\mathfrak{s}$. Denote by $\tau$ the Baer involution of $PG(3, q)$ fixing pointwise $\Lambda_{\epsilon}$. If $\ell$ contains a point $T$ of $\mathcal{C} \setminus \Lambda_{\epsilon}$, then $T^\tau \in \mathcal{C} \setminus \Lambda_{\epsilon}$ and the line $r$ joining $T$ and $T^\tau$ meets $\Lambda_{\epsilon}$ in a subline skew to $\mathcal{C}_{\epsilon}$. In particular such a subline is an imaginary chord and $r^\mathfrak{s} \cap \Lambda_{\epsilon}$ is an imaginary axis of $\mathcal{C}_{\epsilon}$. Hence $\ell$ lies in the plane spanned by $T$ and $r^\mathfrak{s}$; it follows that $\ell \cap \Lambda_{\epsilon}$ is a point lying on an imaginary axis of $\mathcal{C}_{\epsilon}$. Similarly, since the points of $\mathcal{O}_{\epsilon}$ are on extended (real or imaginary) chords, if $\ell$ contains a point of $\mathcal{O}_{\epsilon}$, then $\ell \cap \Lambda_{\epsilon}$ is a point lying on a (real or imaginary) axis of $\mathcal{C}_{\epsilon}$. On the other hand every axis is contained in an osculating plane of $\mathcal{C}_{\epsilon}$.
\end{proof}
\begin{prob}
 Find a maximal partial ovoid of $W(3, q)$ containing $\mathcal{O}_{\epsilon} \cup \mathcal{C}$.
\end{prob}
\subsection{Maximal partial ovoids of $W(5, q)$ of size $q^2+q+1$}\label{subsec312}

Here we consider a symplectic polar space in $PG(5, q)$, where $PG(5, q)$ is embedded as a subgeometry in $PG(5, q^3)$. We denote by $N: x \in F_{q^3} \mapsto x^{q^2+q+1} \in F_q$ the \textit{norm function} of $F_{q^3}$ over $F_q$ and by $T: x \mapsto x+x^q+x^{q^2}$ the \textit{trace function} of $F_{q^3}$ over $F_q$. Let $W$ be the $6$-dimensional $F_q$-vector subspace of $V(6,{q^3})$ given by
\begin{align*}
 &\{P_{a,b} = (a,a^q, a^{q^2}, b^{q^2}, b^q, b) | a, b \in F_{q^3}\}.
\end{align*}
Then $PG(W)$ is a $q$-order subgeometry of $PG(5, q^3)$. If $W(5, q^3)$ is the symplectic polar space consisting of the subspaces of $PG(5, q^3)$ induced by the totally isotropic subspaces of $V(6,{q^3})$ with respect to the non-degenerate alternating form given by
\begin{align}
 \beta(x,y)=x_0 y_5 + x_1 y_4 + x_2 y_3 - x_3 y_2 - x_4 y_1 - x_5 y_0,
\end{align}
then $W(5, q^3)$ induces in $PG(W)$ a symplectic polar space, say $W(5, q)$. It is straightforward to check that the cyclic group $K$ of order $q^2+q+1$ formed by the projectivities of $PG(W)$ induced by the matrices
\begin{align*}
 & D_x = \begin{pmatrix}
 x & 0 & 0 & 0 & 0 & 0 \\
 0 & x^q & 0 & 0 & 0 & 0 \\
 0 & 0 & x^{q^2} & 0 & 0 & 0\\
 0 & 0 & 0 & x^{-q^2} & 0 & 0 \\
 0 & 0 & 0 & 0 & x^{-q} & 0 \\
 0 & 0 & 0 & 0 & 0 & x^{-1} \\
 \end{pmatrix}, x \in F_{q^3}, N(x) = 1,
\end{align*}
preserves $W(5, q)$. The group $K$ fixes the two planes
\begin{align*}
 & \pi_1 = \{P_{a, 0} | a \in F_{q^3} \setminus \{0\}\},
 & \pi_2 = \{P_{0, b} | b \in F_{q^3} \setminus \{0\}\}.
\end{align*}
\begin{lem}
 There are one or three $K$-orbits on points of $\pi_i$, $i = 1,2$, according as $q \not\equiv 1 \pmod{3}$ or $q \equiv 1 \pmod{3}$. In the latter case their representatives are as follows
 \begin{align*}
 & \pi_1: P_{z^i,0},
 \pi_2: P_{0, z^i}, i = 1,2,3,
 \end{align*}
 for some $z \in F_{q^3}$ such that $z^{q-1} = \xi$, where $\xi$ is a fixed element in $F_q$ such that $\xi^2+\xi+1 = 0$. The group $K$ has  $q^3-1$ orbits on points of $PG(W) \setminus (\pi_1 \cup \pi_2)$. Each of them has size $q^2+q+1$ and their representatives are the points
 \begin{align*}
   P_{1, b}, \; b \in F_{q^3} \setminus \{0\}, & \mbox{ if } q \not\equiv 1 \pmod{3}, \\
   P_{z^i, b^3}, i = 1,2,3, \; b \in F_{q^3} \setminus \{0\}, & \mbox{ if } q \equiv 1 \pmod{3}.
 \end{align*}
\end{lem}
\begin{proof}
 The element of $K$ induced by $D_x$ fixes a point of $\pi_1$ (or of $\pi_2$) if and only if $x^{q-1} = x^{q^2+q+1} = 1$, that is $x^q = x$, ($x \in F_q$) and $x^{q^2+q+1} = x^{q^2} x^q x = x^3 = 1$. Therefore if $q \not\equiv 1 \pmod{3}$, then $K$ permutes in a single orbit the points of $\pi_1$ (or of $\pi_2$). If $q \equiv 1 \pmod{3}$, then the kernel of the action of $K$ on both $\pi_1$ and $\pi_2$ consists of the subgroup induced by $\langle D_\xi \rangle$, where $\xi$ is a fixed element in $F_q$ such that $\xi^2+\xi+1 = 0$. Such a subgroup has order three since $D_\xi^3 = D_{\xi^3} = id$. Let $z \in F_{q^3}$ be such that $z^{q-1} = \xi$. In order to see that the representatives of the $K$-orbits on points of $\pi_1$ (or of $\pi_2$) are $P_{z^i, 0}$ (or $P_{0, z^i}$), $i = 1,2,3$, observe that
 \begin{align*}
  & F_{q^3} \setminus \{0\}= \bigcup_{i = 1}^3 \{x z^i|x \in F_{q^3}, x^{\frac{q^2+q+1}{3}} = 1\}.
 \end{align*}

 Let $R$ be a point of $PG(W) \setminus (\pi_1 \cup \pi_2)$. It is straightforward to check that no non-trivial element of $K$ leaves $R$ invariant, that is $|R^K| = q^2+q+1$. Since the planes $\pi_1$, $\pi_2$ are disjoint, there exists a unique line $\ell$ of $PG(W)$ intersecting both $\pi_1$ and $\pi_2$ and containing $R$. Hence if $|\ell \cap R^K| \geq 2$, then there exists a non-trivial element in $K$ fixing $\ell$ and hence stabilizing both $\ell \cap \pi_1$ and $\ell \cap \pi_2$. From the discussion above the existence of such a non-trivial element implies $q \equiv 1 \pmod{3}$. Therefore, if $q \not\equiv 1 \pmod{3}$, the representatives of the $K$-orbits on points of $PG(W) \setminus (\pi_1 \cup \pi_2)$ can be taken as the points on the $q^2+q+1$ lines through $P_{1,0}$ and meeting $\pi_2$ in a point. Similarly, if $q \equiv 1 \pmod{3}$, then these representatives can be chosen among the points of $PG(W) \setminus (\pi_1 \cup \pi_2)$ lying on the $3(q^2+q+1)$ lines through $P_{z^i,0}$, $i = 1,2,3$, and meeting $\pi_2$ in a point. In particular, in this case, $P_{z^i, b}$ and $P_{z^i, b'}$ lie in the same $K$-orbit if $b' = \xi^i b$, $i = 1,2,3$. The result now follows.
\end{proof}
\begin{thm}\label{partial-symp5}
 If $c \in F_q \setminus \{0\}$, the orbit $\mathcal{O} = P_{1, c}^K$ is a partial ovoid of $W(5, q)$ of size $q^2+q+1$.
\end{thm}
\begin{proof}
 Let $P'$ be a point of $P_{1, c}^K \setminus \{P_{1, c}\}$. Then there exists $x \in F_{q^3}$, with $N(x) = 1$, $x \neq 1$, such that
 \begin{align*}
  & P' = (x, x^q, x^{q^2}, c x^{-q^2}, c x^{-q}, c x^{-1}).
 \end{align*}
 If the line joining $P_{1, c}$ with $P'$ is a line of $W(5, q)$, then
 \begin{align*}
  0 & = c T(x^{-1} - x)  = c T(x^{q+1} - x) = c N(1-x)
 \end{align*}
 Therefore $x = 1$ and $P_{1, c} = P'$, a contradiction.
\end{proof}
In order to prove that $\mathcal{O}$ is maximal we need to recall some results obtained by Culbert and Ebert in \cite{CE} in terms of Sherk surfaces. A \textit{Sherk surface} $S(\alpha, \beta, \gamma, \delta)$, where $\alpha, \delta \in F_{q}$, $\beta, \gamma \in F_{q^3}$, can be seen as a hypersurface of the projective line $PG(1, q^3)$. More precisely
\begin{align*}
 S(\alpha, \beta, \gamma, \delta) = \{x \in F_{q^3} \cup \{\infty\}| \alpha N(x) + T(\beta^{q^2} x^{q+1}) + T(\gamma x) + \delta = 0\}.
\end{align*}
Furthermore, $\infty \in S(\alpha, \beta, \gamma, \delta)$ if and only if $\alpha = 0$.
\begin{lem}[\cite{CE}]\label{cullebert}
 Let $\alpha, \delta \in F_{q}$, $\beta, \gamma \in F_{q^3}$.
 \begin{enumerate}
  \item $|S(\alpha, \beta, \gamma, \delta)| \in \{1, q^2-q+1, q^2+1, q^2+q+1\}$.
  \small
  \item $|S(\alpha, \beta, \gamma, \delta)| = 1$ if and only if $(\alpha, \beta, \gamma, \delta) \in \{(1, \beta, \beta^{q^2+q}, N(\beta)), (0,0,0,1)\}$.
  \normalsize
 \end{enumerate}
\end{lem}
Note that $\lambda S(\alpha_1, \beta_1, \gamma_1, \delta_1) + \mu S(\alpha_2, \beta_2, \gamma_2, \delta_2) = S(\lambda \alpha_1 + \mu \alpha_2, \lambda \beta_1 + \mu \beta_2, \lambda \gamma_1 + \mu \beta_2, \lambda \delta_1 + \mu \delta_2)$ for any $\lambda, \mu \in F_q$, $(\lambda, \mu) \ne (0, 0)$. Therefore any two distinct Sherk surfaces give rise to a pencil containing $q+1$ distinct Sherk surfaces which together contain all the points of $PG(1, q^3)$. Moreover, any two distinct Sherk surfaces in a given pencil intersect in the same set of points, called the \textit{base locus} of the pencil, which is the intersection of all the Sherk surfaces in that pencil.
\begin{lem}\label{sherk}
 Let $\beta, \gamma \in F_{q^3}$ not both zero. If one of the following is satisfied
 \begin{itemize}
  \item $\beta \gamma = 0$,
  \item $N(\beta) \neq c^3$, for all $c \in F_q \setminus \{0\}$,
  \item $N(\beta) = c^3$, for some $c \in F_q \setminus \{0\}$, and $\gamma \beta \neq -c^2$,
 \end{itemize}
 then the base locus of the pencil generated by $S(1,0,0,-1)$ and $S(0,\beta, \gamma,0)$ is not empty.
\end{lem}
\begin{proof}
 The pencil consists of $S(0, \beta, \gamma, 0)$ and $S(1,0,0,-1) + \lambda S(0, \beta, \gamma, 0) = S(1, \lambda \beta, \lambda \gamma, -1)$, where $\lambda \in F_q$. Let $y_i$ be the number of Sherk surfaces in the pencil having $i$ points. Hence $y_i$ are non-negative integer and $y_{q^{2}+q+1} \geq 1$ since $|S(1,0,0,-1)| = q^2+q+1$. By Lemma \ref{cullebert}, it follows that $|S(0, \beta, \gamma, 0)| \neq 1$ and $|S(1, \lambda \beta, \lambda \gamma, -1)| \neq 1$, for $\lambda \in F_q$, i.e., $y_1 = 0$. Assume by contradiction that the base locus of the pencil is empty, then the following equations hold
 \begin{align*}
  & y_{q^{2}-q+1} + y_{q^{2}+1} + y_{q^{2}+q+1} = q+1, \\
  & (q^2-q+1) y_{q^{2}-q+1} + (q^2+1) y_{q^{2}+1} + (q^2+q+1) y_{q^{2}+q+1} = q^3+1.
 \end{align*}
 By substituting $y_{q^{2}-q+1} = q+1 - y_{q^{2}+1} - y_{q^{2}+q+1}$ in the second equation, we have that
 \begin{align*}
  & y_{q^{2}+q+1} = - \frac{y_{q^{2}+1}}{2}.
 \end{align*}
 Therefore necessarily
 \begin{align*}
 & y_{q^{2}+q+1} = y_{q^{2}+1} = 0,
 & y_{q^{2}-q+1} = q+1,
 \end{align*}
 which is a contradiction.
\end{proof}
\newpage
\begin{thm}
 The partial ovoid $\mathcal{O}$ of $W(5, q)$ is maximal.
\end{thm}
\begin{proof}
 It is enough to show that for each of the representatives $R$ of the $K$-orbits on points of $PG(W) \setminus \mathcal{O}$, there exists a point \\$P = (x, x^q, x^{q^2}, c x^{-q^2}, c x^{-q}, c x^{-1}) \in \mathcal{O}$ such that the line joining $R$ and $P$ is a line of $W(5, q)$. If $R$ is a point of $\pi_1$, then $R = P_{1,0}$ if $q \not\equiv 1 \pmod{3}$ or $R \in \{P_{z^i,0} \mid i = 1,2,3\}$, if $q \equiv 1 \pmod{3}$. The line $R P$ belongs to $W(5, q)$ if and only if there exists $x \in F_{q^3}$, with $N(x) = 1$ such that
 \begin{align}
  & c T(u x^{-1}) = 0, \label{conj1}
 \end{align}
 where $u = 1$ or $u \in \{z^i | i = 1,2,3\}$ according as $q \not \equiv 1 \pmod{3}$ or $q \equiv 1 \pmod{3}$. Since $c \ne 0$ and $N(x) = 1$, \eqref{conj1} is equivalent to $T(u^{q^2} x^{q+1}) = 0$, i.e., $x \in S(0, u, 0, 0)$. Hence the existence of a line $R P$ of $W(5, q)$ is equivalent to the following
 \begin{align}
  & |S(1,0,0,-1) \cap S(0, u, 0, 0)| \geq 1. \label{sherk1}
 \end{align}
 As before, let $u = 1$ or $u \in \{z^i | i = 1,2,3\}$ according as $q \not \equiv 1 \pmod{3}$ or $q \equiv 1 \pmod{3}$. Let $R$ be the point $P_{0, u} \in \pi_2$ or $P_{u, b} \in PG(W) \setminus (\pi_1 \cup \pi_2 \cup \mathcal{O})$, where $b \in F_{q^3} \setminus \{0\}$, with $b^3 \neq c^3$, if $u = 1$, otherwise $P_{1, b} \in \mathcal{O}$. Arguing in a similar way we get that the existence of a line $R P$ of $W(5, q)$ is equivalent to
 \begin{align}
  & |S(1,0,0,-1) \cap S(0, 0, u, 0)| \geq 1 \mbox{ or }  |S(1,0,0,-1) \cap S(0, cu, -b, 0)| \geq 1.  \label{sherk2}
 \end{align}
 Since $N(u)$ is a cube in $F_q$ if and only if $u = 1$ and if $u = 1$, then $b \neq c$, the inequalities in \eqref{sherk1} and \eqref{sherk2} are satisfied by Lemma \ref{sherk}.
\end{proof}
\subsection{Maximal partial ovoids of $W(5, q)$, $q$ even, of size $2q^2-q+1$}\label{subsec313}
Assume $q$ to be even and let $W(5, q)$ be the symplectic polar space of $PG(5, q)$ induced by the totally isotropic subspaces of $V(6,q)$ with respect to the non-degenerate alternating form given by
\begin{align}
 \beta(x,y)=x_0 y_1 + x_1 y_0 + x_2 y_3 + x_3 y_4 + x_4 y_5 + x_5 y_4.
\end{align}
Recall that $\mathfrak{s}$ is the symplectic polarity of $PG(5, q)$ associated with $W(5, q)$. Fix $\delta_i \in F_q$ such that the polynomial $X^2+X+\delta_i$ is irreducible over $F_q$, $i = 1,2$. Let $\Delta_1$ and $\Delta_2$ be the three-spaces given by $X_4 = X_5 = 0$ and $X_2 = X_3 = 0$, respectively; consider the three-dimensional elliptic quadrics $\mathcal{E}_i \subset \Delta_i$ defined as follows
\begin{align*}
 & \mathcal{E}_1 : X_0 X_1 + X_2^2+X_2X_3 + \delta_1 X_3^2 = 0, \\
 & \mathcal{E}_2 : X_0 X_1 + X_4^2+X_4X_5 + \delta_2 X_5^2 = 0.
\end{align*}
Denote by $\sigma$ the plane of $\Delta_1$ spanned by $U_1, U_2, U_3$; hence $\mathcal{E}_1 \cap \sigma$ is a conic and $U_3^\mathfrak{s} \cap \Delta_1 = \sigma$. In the hyperplane $U_3^\mathfrak{s}: X_3= 0$, let $\mathcal{P}_1$ be the cone having as vertex the point $U_3$ and as base $\mathcal{E}_2$ and let $\mathcal{P}_2$ be the cone having as vertex the line $\Delta_1^\mathfrak{s}$ and as base the conic $\mathcal{E}_1 \cap \sigma$. Set
\begin{align*}
 & A = \mathcal{P}_1 \cap \mathcal{P}_2 = \{(1, c^2+cd+\delta_2d^2, \sqrt{c^2+cd+\delta_2d^2}, 0, c,d) |c, d \in F_q\} \cup \{U_2\}.
\end{align*}
\begin{lem}
 The point sets $\mathcal{E}_1$ and $A$ are partial ovoids of $W(5, q)$ of size $q^2+1$.
\end{lem}
\begin{proof}
 The elliptic quadric $\mathcal{E}_1$ is an ovoid of $\Delta_1 \cap W(5, q)$. Hence it is a partial ovoid of $W(5, q)$. The set $A$ is contained in $U_3^\mathfrak{s}$ and each of the lines obtained by joining $U_3$ with a point of $A$ intersects $\Delta_2$ in a point of $\mathcal{E}_2$. In particular this gives a bijective correspondence between the points of $A$ and those of $\mathcal{E}_2$. In order to see that $A$ is a partial ovoid of $W(5, q)$ note that the elliptic quadric $\mathcal{E}_2$ is an ovoid of $\Delta_2 \cap W(5, q)$.
\end{proof}
\begin{thm}\label{partial-symp5-even}
 The point set $A \cup (\mathcal{E}_1 \setminus \sigma)$ is a maximal partial ovoids of $W(5, q)$ of size $2q^2-q+1$.
\end{thm}
\begin{proof}
 Taking into account the previous lemma, it is enough to show that if $P$ is a point of $\mathcal{E}_1 \setminus \sigma$, then $|P^\mathfrak{s} \cap A| = 0$. By contradiction let $P = (1, a^2+ab+\delta_1b^2, a, b, 0, 0) \in \mathcal{E}_1 \setminus \sigma$ and $R = (1, c^2+cd+\delta_2d^2, \sqrt{c^2+cd+\delta_2d^2}, 0, c, d) \in A$, for some $a, b, c, d \in F_q$, $b \neq 0$, with $R \in P^\mathfrak{s}$. Then $\sqrt{c^2+cd+\delta_2d^2} \in F_q$ is a root of $X^2+bX+a^2+ab+\delta_1 b^2$, which is irreducible over $F_q$, a contradiction. In order to prove the maximality, let $Q$ be a point not in $A \cup (\mathcal{E}_1 \setminus \sigma)$ and assume that $|Q^\mathfrak{s} \cap (\mathcal{E}_1 \setminus \sigma)| = 0$. There are two possibilities: either $Q^\mathfrak{s} \cap \Delta_1$ coincides with $\sigma$ or $Q^\mathfrak{s} \cap \Delta_1$ is a plane meeting $\mathcal{E}_1$ in exactly one point of the conic $\mathcal{E}_1 \cap \sigma$. In the former case $Q^\mathfrak{s}$ meets $A$ in the points $\{U_1, U_2\}$. If the latter possibility occurs, let $R$ be the point of the conic $\mathcal{E}_1 \cap \sigma$  contained in $Q^\mathfrak{s} \cap \Delta_1$; since $R^\mathfrak{s} \cap \Delta_1 \subset Q^\mathfrak{s}$, the point $Q$ lies in the plane spanned by $R$ and $\Delta_1^\mathfrak{s}$ . Hence the line joining $R$ and $Q$ is a line of $W(5, q)$ and contains a unique point of $A$.
\end{proof}
\begin{remark}
 Some computations performed with \textit{MAGMA} \cite{magma} show that the largest partial ovoid of $W(5, q)$, has size $7$ for $q = 2$ and $13$ for $q = 3$. In particular, $W(5, 2)$ has a unique partial ovoid of size $7$, up to projectivities, which coincides with the partial ovoids described above.
\end{remark}

\section{Partial ovoids of Hermitian polar spaces}\label{sec32}
Let $H(2n, q^2)$, $n \geq 2$, be a Hermitian polar space of $PG(2n, q^2)$ and let $\perp$ be its associated unitary polarity. Fix a point $P \notin H(2n, q^2)$ and a pointset $\mathcal{T} \subset P^\perp$ suitably chosen. In this section we explore the partial ovoids of $H(2n, q^2)$ that are constructed by taking the points of $H(2n, q^2)$ on the lines obtained by joining $P$ with the points of $\mathcal{T}$. It turns out that in order to get a partial ovoid, the set $\mathcal{T}$ has to intersect the lines that are tangent or contained in $P^\perp \cap H(2n, q^2)$ in at most one point.
\subsection{Tangent-sets of $H(2n-1, q^2)$, $n \in \{2, 3, 4\}$}\label{subsec321}
Let $H(2n-1, q^2)$, $n \geq 2$, be a Hermitian polar space of $PG(2n-1, q^2)$. We introduce the notion of \textit{tangent set} of $H(2n-1, q^2)$.
\begin{defn}
 A \textit{tangent-set} of a Hermitian polar space of $PG(2n-1, q^2)$ is a set $\mathcal{T}$ of points of $PG(2n-1, q^2)$ such that every line that is either tangent or contained in $H(2n-1, q^2)$ has at most one point in common with $\mathcal{T}$. A tangent-set is said to be \textit{maximal} if it is not contained in a larger tangent-set.
\end{defn}
Here the main aim is to show that starting from a (partial) ovoid of $W(2n-1, q)$ it is possible to obtain a tangent set of $H(2n-1, q^2)$. In order to do that some preliminary results are needed.
\begin{lem}\label{ovoid}
 An ovoid of $W(3, q) \subset H(3, q^2)$ is a maximal partial ovoid of $H(3, q^2)$.
\end{lem}
\begin{proof}
 Let $\mathcal{X}$ be an ovoid of $W(3, q) \subset H(3, q^2)$. Then obviously $\mathcal{X}$ is a partial ovoid of $H(3, q^2)$. In order to see that $\mathcal{X}$ is maximal as a partial ovoid of $H(3, q^2)$, observe that every point of $H(3, q^2) \setminus W(3, q)$ lies on a unique extended line of $W(3, q)$.
\end{proof}
Let us fix an element $\iota \in F_{q^2}$ such that $\iota + \iota^q = 0$. Denote by $\xi_1, \dots, \xi_q$ the $q$ elements of $F_{q^2}$ such that $\xi_i + \xi_i^q = 0$. Then $\iota (\xi_i^q - \xi_i) = -2 \iota \xi_i$ and $\iota \xi_i \in F_q$, since $(\iota \xi_i)^q = (- \iota)(- \xi_i) = \iota \xi_i$.Then
\begin{align}
 & \{\iota(\xi_i^q - \xi_i) | i = 1, \ldots, q\} = F_q.  \label{all}
\end{align}
In particular we may assume $\xi_1 = 0$. Let $H_i$, $i=1,\ldots,q$, be the Hermitian variety of $PG(2n-1, q^2)$ with equation
\begin{align*}
 & \iota (X_0 X_{n}^q - X_0^q X_{n} + \ldots + X_{n-1} X_{2n-1}^q - X_{n-1}^q X_{2n-1} + (\xi_i^q-\xi_i) X_{2n-1}^{q+1}) = 0,
\end{align*}
and denote by $\perp_i$ the polarity of $PG(2n-1, q^2)$ associated with $H_i$. Let $H_{\infty}: X_{2n-1}^{q+1} = 0$. Then $H_{\infty}$ consists of the points of the hyperplane $\Pi: X_{2n-1} = 0$ and $H_{i} = H_1 + \lambda H_{\infty}$, for some $\lambda \in F_q$. It follows that the set $\{H_i | i = 1, \ldots, q\} \cup \{\Pi\}$ is a pencil of Hermitian varieties. Since $\Pi \cap H_i = H_{\infty} \cap H_i = ( \cap_{i=1}^{q} H_i ) \cap H_{\infty}$, the base locus of the pencil is $\Pi \cap H_i$.  Let $\Sigma_i$ denote the Baer subgeometry of $PG(2n-1, q^2)$ consisting of the points
\begin{align*}
 & \{(a_1, \ldots, a_{n-1}, a_n + \xi_i, a_{n+1}, \ldots, a_{2n-1}, 1) | a_1, \dots, a_{2n-1} \in F_q \} \\
 & \cup \{(a_1, \ldots, a_{2n-1}, 0) | a_1, \ldots, a_{2n-1} \in F_q, (a_1, \ldots, a_{2n-1}) \neq (0, \ldots, 0) \} ,
\end{align*}
for  $i = 1, \ldots, q$. Note that if $i \neq j$, then $\Sigma_i \cap \Sigma_j = \Sigma_i \cap \Pi = \Sigma_j \cap \Pi$.
It is easily seen that $\Sigma_i \subset H_i$, for  $i = 1, \ldots, q$. Hence $|H_1 \cap (\Sigma_i \setminus \Pi)| = 0$, if $i \neq 1$. Moreover, since the unitary form defining $H_i$, restricted to $\Sigma_i$, is bilinear alternating, we have that $\perp_i|_{\Sigma_i}$ is a symplectic polarity of $\Sigma_i$. Hence $H_i$ induces on $\Sigma_i$ a symplectic polar space $W_i$ of $\Sigma_i$ and the lines of $H_i$ having $q+1$ points in common with $\Sigma_i$ are those of $W_i$. Let $H_i$, $i = 1, \ldots, q$, be the following $2n \times 2n$ matrix
\begin{align*}
 \begin{pmatrix}
  0 & 0 & \cdots & 0 & 0 & 1 & 0 & \cdots & 0 & 0 \\
  0 & 0 & \cdots & 0 & 0 & 0 & 1 & \cdots & 0 & 0 \\
  \vdots & \vdots & \ddots & \vdots & \vdots & \vdots & \vdots & \ddots & \vdots & \vdots \\
  0 & 0 & \cdots & 0 & 0 & 0 & 0 & \cdots & 1 & 0 \\
  0 & 0 & \cdots & 0 & 0 & 0 & 0 & \cdots & 0 & 1 \\
 -1 & 0 & \cdots & 0 & 0 & 0 & 0 & \cdots & 0 & 0 \\
  0 & -1 & \cdots & 0 & 0 & 0 & 0 & \cdots & 0 & 0 \\
  \vdots & \vdots & \ddots & \vdots & \vdots & \vdots & \vdots & \ddots & \vdots & \vdots \\
  0 & 0 & \cdots & -1 & 0 & 0 & 0 & \cdots & 0 & 0 \\
  0 & 0 & \cdots & 0 & -1 & 0 & 0 & \cdots & 0 & \xi_i^q - \xi_i \\
 \end{pmatrix},
\end{align*}
and let $K$ be the group of projectivities of $PG(2n-1, q^2)$ of order $q^{2n-1}$ induced by the matrices
\begin{align*}
 M_{a} =
 \begin{pmatrix}
  1 & 0 & \cdots & 0 & 0 & 0 & \cdots & 0 & a_1 \\
  0 & 1 & \cdots & 0 & 0 & 0 & \cdots & 0 & a_2 \\
  \vdots & \vdots & \ddots & \vdots & \vdots & \vdots & \ddots & \vdots & \vdots \\
  0 & 0 & \cdots & 1 & 0 & 0 & \cdots & 0 & a_{n-1} \\
  -a_{n+1} & -a_{n+2} & \dots & -a_{2n-1} & 1 & a_1 & \dots & a_{n-1} & a_{n} \\
  0 & 0 & \cdots & 0 & 0 & 1 & \cdots & 0 & a_{n+1} \\
  \vdots & \vdots & \ddots & \vdots & \vdots & \vdots & \ddots & \vdots & \vdots \\
  0 & 0 & \cdots & 0 & 0 & 0 & \cdots & 1 & a_{2n-1} \\
  0 & 0 & \cdots & 0 & 0 & 0 & \cdots & 0 & 1 \\
 \end{pmatrix},
\end{align*}
where $a = (a_1, \dots, a_{2n-1}) \in F_q^{2n-1}$. Since $M_{a}^T H_i M_{a}^q = H_i$, it follows that $K$ fixes $H_i$. Furthermore, if $g$ is a projectivity of $K$ induced by $M_{a}$ and $P$ is a point of $\Sigma_i$, then $P^g \in \Sigma_i$. Hence $K$ stabilizes $\Sigma_i$. We infer that $K$ leaves invariant the symplectic polar space $W_i$ induced by $H_i$ on $\Sigma_i$. The stabilizer of the point $P_i = (0, \ldots, 0, \xi_i, 0, \ldots, 0, 1) \in \Sigma_i \setminus \Pi$ in $K$ is trivial. Hence the orbit $P_i^K$ has size $q^{2n-1}$. Since $P_i^K \subseteq \Sigma_i \setminus \Pi$, it follows that $P_i^K = \Sigma_i \setminus \Pi$, i.e. $K$ permutes in a single orbit the $q^{2n-1}$ points of $\Sigma_i \setminus \Pi$, $i = 1, \ldots, q$. Let $s$ be the line joining $U_n$ with $P_1$. If $a \in F_q$, then $(a+\xi_i-\xi_1)U_n + P_1$ is a point of $s \cap \Sigma_i$. Hence $|s \cap \Sigma_i| = q+1$, $i \in 1, \ldots, q$. Since $s$ has $q$ points in common with $\Sigma_i \setminus \Pi$ and $K$ fixes $U_n$ and acts transitively on the $q^{2n-1}$ points of $\Sigma_i \setminus \Pi$, it follows that $|s^K| = q^{2n-2}$. In particular, the $q^{2n}$ points of $\bigcup_{i=1}^{q} \Sigma_i \setminus \Pi$ lie on the lines of $s^K$. Denote by $\mathcal{L}$ this set of $q^{2n-2}$ lines. In the next lemma we restrict our attention to the case $n = 2$.
\begin{lem}\label{HermCur}
 Let $\ell \in \mathcal{L}$ and let $i \in \{2, \ldots, q\}$. The intersection of $\Pi$ with the tangent lines to $H_1$ through a point of $(\ell \cap \Sigma_i) \setminus \Pi$, is a non-degenerate Hermitian curve of $\Pi$ belonging to the pencil determined by $\ell^{\perp_1}$ and $\Pi \cap H_1$. Vice versa, if $r$ is a line of $\Pi$ with $|r \cap \Sigma_1| = |r \cap H_1| = q+1$ and $\mathcal{U}_{r}$ is a non-degenerate Hermitian curve of $\Pi$ belonging to the pencil determined by $r$ and $\Pi \cap H_1$, then there exists $k$ such that the lines joining a point of $(r^{\perp_1} \cap \Sigma_k) \setminus \Pi$ with a point of $\mathcal{U}_{r}$ are tangent to $H_1$.
\end{lem}
\begin{proof}
 Let $\ell \in \mathcal{L}$ and let $i \in \{2, \ldots, q\}$. If $P$ is a point of $(\ell \cap \Sigma_i) \setminus \Pi$, since $K$ is transitive on $\mathcal{L}$, we may assume $\ell = s$ and hence $P = (0, a_2+\xi_i, 0,1)$, $\xi_i \neq 0$, $a_2 \in F_q$. Then the points of $(x, y, z, 0)$ of $\Pi$ lying on a line through $P$ and tangent to $H_1$ satisfy
 \begin{align}
  & \iota^2 y^{q+1} + \iota \alpha (x^q z - x z^q) = 0, \label{eqcurve}
 \end{align}
 where $\alpha = \iota (\xi_i^q - \xi_i)$. Hence they form a non-degenerate Hermitian curve of $\Pi$ belonging to the pencil determined by $s^{\perp_1}: y^{q+1} = 0$ and $\Pi \cap H_1: x z^q - x^q z = 0$. Vice versa, let $r$ be a line of $\Pi$ with $|r \cap \Sigma_1| = |r \cap H_1| = q+1$. Since $\perp_1$ induces a polarity on $\Sigma_1$ and $|r \cap \Sigma_1| = q+1$, then $|r^{\perp_1} \cap \Sigma_1| = q+1$. On the other hand, $r \subset \Pi$ implies $U_2 \in r^{\perp_1}$. Hence $r^{\perp_1}$ is a line of $\mathcal{L}$. Since $K$ is transitive on $\mathcal{L}$, we may assume that $r^{\perp_1} = s$ and hence $r$ is the line joining $U_1$ and $U_3$. It follows that $\mathcal{U}_r$ is given by Equation \eqref{eqcurve}, for some $\alpha \in F_q \setminus \{0\}$. Thus if $\alpha = \iota (\xi_k^q - \xi_k)$, the lines joining a point of $(r^{\perp_1} \cap \Sigma_k) \setminus \Pi$ with a point of $\mathcal{U}_{r}$ are tangent to $H_1$.
\end{proof}
\begin{lem}\label{tangenth}
 Let $\ell$ be a line that is either tangent to $H_1$ or contained in $H_1$. Then $|\ell \cap ( \bigcup_{i=1}^{q} \Sigma_i )| \in \{0, 1, q+1\}$. In particular, if $|\ell \cap ( \bigcup_{i=1}^{q} \Sigma_i )| = q+1$ then there exists $k$ such that $\ell \cap \Sigma_k = q+1$ and $\ell \subset H_k$.
\end{lem}
\begin{proof}
 If $\ell$ is contained in $H_1$, then there is nothing to prove. If $\ell$ is tangent, let $P_i, P_j$ be two points belonging to $\ell \cap \Sigma_i$ and $\ell \cap \Sigma_j$, respectively, where $i \neq 1$. Since $K$ is transitive on points of $\Sigma_i \setminus \Pi$, we may assume $P_i = (0, \ldots, 0, \xi_i, 0, \ldots, 0, 1)$. If $P_j = (a_1, \ldots, a_{n-1}, a_n + \xi_j, a_{n+1}, \ldots, a_{2n-1}, 1)$ for some $a_1, \ldots, a_{2n-1} \in F_q$ not all zero, and $\lambda \in F_{q^2}$, then the point $P_j + \lambda P_i$ belongs to $H_1$ if and only if $\lambda$ is a root of
 \begin{align*}
  \iota(\xi_i - \xi_j) X^{q+1} + \iota (\xi_j - \xi_i^q + a_n) X^q + \iota^q (\xi_j^q - \xi_i + a_n) X + \iota (\xi_j - \xi_j^q) = 0.
 \end{align*}
 Some straightforward calculations show that $\ell$ is tangent if and only if $\xi_j = \xi_i - a_n$. If $i \neq j$, then $\iota (\xi_j^q - \xi_j) = \iota (\xi_i^q - \xi_i)$, contradicting \eqref{all}. Therefore $i = j$ and $a_n = 0$. This means that $|\ell \cap \Sigma_i| = q+1$. Moreover $\ell \subset H_i$. If $P_j = (a_1, \ldots, a_{2n-1}, 0)$ for some $a_1, \ldots, a_{2n-1} \in F_q$, then $P_j \in \Sigma_i \cap \Sigma_j$. Hence $|\ell \cap \Sigma_i| = q+1$ and we may repeat the previous argument.
\end{proof}
\begin{thm}\label{tangent-set}
 Let $\mathcal{O}_i$ be a (partial) ovoid of $W_i$, $i = 1, \dots, q$, such that $\mathcal{O}_i \cap \Pi = \mathcal{O}_j \cap \Pi$, if $i \neq j$. Then $\mathcal{T} = \bigcup_{i = 1}^{q} \mathcal{O}_i$ is a tangent-set of $H_1$. Moreover if $n = 2$ and $\mathcal{O}_i$ is an ovoid of $W_i$, for every $i = 1, \ldots, q$, then $\mathcal{T}$ is a maximal tangent-set.
\end{thm}
\begin{proof}
 Assume by contradiction that $\mathcal{T}$ is not a tangent-set. If $\ell$ is a line of $H_1$ and $|\ell \cap \mathcal{T}| \geq 2$, then $\ell \cap \mathcal{T} \subset \mathcal{O}_1$, since $\mathcal{T} \cap H_1 = \mathcal{O}_1$. Hence $\ell$ is an extended line of $W_1$ containing two points of $\mathcal{O}_1$, contradicting the assumption that $\mathcal{O}_1$ is a (partial) ovoid of $W_1$. Similarly, if $\ell$ is tangent to $H_1$ and $|\ell \cap \mathcal{T}| \geq 2$, then, by Lemma \ref{tangenth}, there exists a $k$ such that $\ell$ is an extended line of $W_k$ and $\ell \cap \mathcal{T} \subset \mathcal{O}_k$. Therefore $\mathcal{O}_k$ is not a (partial) ovoid of $W_k$, since $|\ell \cap \mathcal{O}_k| \geq 2$, contradicting the hypotheses. Assume now that $n = 2$ and $\mathcal{O}_i$ is an ovoid of $W_i$, for every $i = 1, \ldots, q$. We claim that every point $R$ of $PG(3, q^2) \setminus \mathcal{T}$ lies on at least one line that is either tangent or contained in $H_1$ and contains one point of $\mathcal{T}$. If $R \not\in \Pi$, then there exactly one $H_k$ with $R \in H_k$; hence there is at least one extended line, say $r_k$ of $W_k$ containing $R$. Such a line $r_k$ has to contain exactly one point of $\mathcal{O}_k$ (see Lemma \ref{ovoid}) and by Lemma \ref{tangenth} it is contained or tangent to $H_1$ according as $k = 1$ or $k \neq 1$. Similarly if $R \in H_i \cap \Pi$. Let $R$ be a point of $\Pi \setminus H_i$. There are two possibilities: either $U_2 \in \mathcal{O}_i$ for every $i = 1, \ldots, q$, and in this case the line joining $R$ and $U_2$ is tangent to $H_1$ or $U_2 \notin \mathcal{O}_i$, for $i = 1, \ldots, q$. If the latter possibility occurs, then there exists a line $r$ of $\Pi$ such that $R \notin r$, $|r \cap \mathcal{O}_i| = 0$ and $|r \cap \Sigma_i| = |r \cap H_1| = q+1$. Let $\mathcal{U}_r$ be the (unique) non-degenerate Hermitian curve of $\Pi$ belonging to the pencil determined by $r$ and $\Pi \cap H_1$ such that $R \in \mathcal{U}_r$. By Lemma \ref{HermCur}, there exists a $k$ such that the lines joining a point of $(r^{\perp_1} \cap \Sigma_k) \setminus \Pi$ with $R$ are tangent to $H_1$. The Baer subline $r^{\perp_1} \cap \Sigma_k$ contains two points of $\mathcal{O}_k$ being the polar line of $r \cap \Sigma_k$ with respect to the symplectic polarity of $\Sigma_k$ associated with $W_k$. Therefore the line joining a point of $r^{\perp_1} \cap \Sigma_k \cap \mathcal{O}_k$ with $R$ is tangent to $H_1$.
\end{proof}
\begin{prop}
 If there is a (partial) ovoid of $W(2n-1, q)$ of size $x+y$, with $y$ points on a hyperplane, then there exists a tangent-set of $H(2n-1, q^2)$ of size $x q + y$.
\end{prop}

The next result is reached by applying Theorem \ref{tangent-set}, where $\mathcal{O}_i$ is either an ovoid of $W(3, q)$, $q$ even, or the partial ovoid of $W(3, q)$, obtained in Theorem \ref{partial-symp3}, or the partial ovoids of $W(5, q)$ described in Theorem \ref{partial-symp5} and in Theorem \ref{partial-symp5-even}, or the partial ovoid of $W(7, q)$ of size $q^3+1$ given in \cite{C, CK}.
\begin{cor}
 The following hold true.
 \begin{itemize}
  \item If $q$ is even, there exist maximal tangent-sets of $H(3, q^2)$ of size $q^3+1$ and of size $q^3-q^2+q+1$.
  \item If $q$ is an odd square and $q \not\equiv 0 \pmod{3}$, there exists a tangent-set of $H(3, q^2)$ of size $\frac{q^{\frac{5}{2}}+3q^2-q^{\frac{3}{2}}+3q}{3}$.
  \item If $q$ is even, there exists a tangent-set of $H(5, q^2)$ of size $2q^3-q^2+1$.
  \item There exists a tangent-set of $H(5, q^2)$ of size $q^3+q^2+1$.
  \item There exists a tangent-set of $H(7, q^2)$ of size $q^4+1$.
  \end{itemize}
\end{cor}
\begin{prob}
 Construct larger tangent sets of $H(2n-1, q^2)$ and provide an upper bound on their size.
\end{prob}
\subsection{Partial ovoids of $H(2n, q^2)$ from tangent-sets of $H(2n-1, q^2)$, $n \in \{2, 3, 4\}$}\label{subsec322}
The next result shows that starting from a (maximal) tangent-set of $H(2n-1, q^2)$ it is possible to obtain a (maximal) partial ovoid of $H(2n, q^2)$.
\begin{lem}\label{partial-herm}
 Let $P$ be a point of $PG(2n, q^2) \setminus H(2n, q^2)$ and let $\mathcal{T}$ be a (maximal) tangent-set of $H(2n-1, q^2) = P^\perp \cap H(2n, q^2)$. Then
 \begin{align*}
  & \mathcal{O} = \Bigg( \bigcup_{R \in \mathcal{T}} \langle P, R \rangle \Bigg) \cap H(2n, q^2)
 \end{align*}
 is a (maximal) partial ovoid of $H(2n, q^2)$ of size $(q+1) |\mathcal{T} \setminus H(2n-1, q^2)| +  |\mathcal{T} \cap H(2n-1, q^2)|$.
\end{lem}
\begin{proof}
 Assume by contradiction that $\mathcal{O}$ is not a partial ovoid, then there would exist a line $\ell$ of $H(2n, q^2)$ containing two distinct points $P_1$ and $P_2$ of $\mathcal{O}$. If at least one among $P_1, P_2$ is not in $P^\perp$, then the line $\bar{\ell} \subset P^\perp$ obtained by projecting $\ell$ from $P$ to $P^\perp$ is tangent to $H(2n-1, q^2) = P^\perp \cap H(2n, q^2)$ and contains two points of $\mathcal{T}$, a contradiction. If both $P_1$ and $P_2$ are in $P^\perp$, then $\ell$ is a line of $H(2n-1, q^2)$ and contains two points of $\mathcal{T}$, a contradiction. As for maximality, note that if $T$ is a point of $H(2n, q^2)$ such that $\mathcal{O} \cup \{T\}$ is a partial ovoid of $H(2n, q^2)$ and $\bar{T} = \langle P, T \rangle \cap P^\perp$, then $\mathcal{T} \cup \{\bar{T}\}$ is a tangent-set of $H(2n-1, q^2)$.
\end{proof}
Combining Theorem \ref{tangent-set} and Lemma \ref{partial-herm}, the following result arises.
\begin{thm}\label{thm:partial-herm}
 If there is a (partial) ovoid of $W(2n-1, q)$ of size $x+y$, with $y$ points on a hyperplane, then there exists a partial ovoid of $H(2n, q^2)$ of size $x q^2 + y$.
\end{thm}
Note that if a partial ovoid of $W(2n-1, q)$ is not maximal, then there exists a hyperplane disjoint from it. Hence, in this case it is possible to choose $y = 0$ in Theorem \ref{thm:partial-herm}. It follows that starting from (partial) ovoids of $W(2n-1, q)$, we get partial ovoids of $H(2n, q^2)$, $n \in \{2, 3, 4\}$, as described below.
\begin{cor}
 \label{ovoidiH}
 The following hold true.
 \begin{itemize}
  \item If $q$ is even, there exist maximal partial ovoids of $H(4, q^2)$ of size $q^4+1$ and of size $q^4-q^3+q+1$.
  \item If $q$ is an odd square and $q \not\equiv 0 \pmod{3}$, there exists a partial ovoid of $H(4, q^2)$ of size $\frac{q^{\frac{7}{2}}+3q^3-q^{\frac{5}{2}}+3q^2}{3}$.
  \item If $q$ is even, there exists a partial ovoid of $H(6, q^2)$ of size $2q^4-q^3+1$.
  \item There exists a partial ovoid of $H(6, q^2)$ of size $q^4+q^3+1$.
  \item There exists a partial ovoid of $H(8, q^2)$ of size $q^5+1$.
 \end{itemize}
\end{cor}
In this chapter we dealt with large partial ovoids of symplectic polar spaces and Hermitian polar spaces lying in projective spaces of even dimension. In general to obtain large partial ovoids seems to be a difficult task. Here, several subgroups of the related symplectic and unitary groups along with peculiar geometric settings are studied and used to construct examples of partial ovoids that, in some cases, are the currently best known ones. However there is a large discrepancy between the known examples and the current best theoretical upper bound. Even worse, this gap increases terribly in higher dimension. This suggests that the following two problems require further investigations:
\begin{prob}
 \begin{itemize}
  \item construct large partial ovoids by using also different techniques (as, for instance, probabilistic methods, see \cite{exRandom});
  \item find improvement on the theoretical upper bounds.
 \end{itemize}
\end{prob}
\begin{cons}
 In the case of the symplectic space $W(3,q)$, $q$ odd, some computations performed with \textit{GAP} \cite{GAP} show that it is possible to find larger maximal partial ovoids from the action of the group $G_{\epsilon}$ defined in Subsection \ref{subsec311}.
  \begin{enumerate}
   \item Consider an orbit of $G_{\epsilon}$ on the points of $W(3,q)$
   \item Some of the orbits of size $\frac{q^{\frac{3}{2}} - q^{\frac{1}{2}}}{3}$ are indeed partial ovoids
   \item Compute the non covered points, and consider the collinearity graph on this set
   \item \begin{itemize}
     \item For q=5: there is an independent set in this graph of size 141. Together with the 40 points of the orbit, we indeed get a partial ovoid of size 181, which is maximal.
     \item For q=7: there is an independent set in this graph of size 370. Together with the 112 points of the orbit, we indeed get a partial ovoid of size 482, which is maximal.
     \end{itemize}
  \end{enumerate}
\end{cons}

\part{Regular graphs arising from polar spaces}\label{pII}
\chapter{Distance regular graphs and association schemes}\label{ch4}
Polar spaces have found a lot of applications in Combinatorics. For instance, from a graph theory point of view, incidences between isotropic points and lines, points and generators, or more generally incidences between totally isotropic subspaces of polar spaces give rise to important families of (regular) graphs. In the second part of the P.h.D. Thesis we show examples of regular graphs arising from polar spaces, introducing \textit{distance regular graphs} and \textit{strongly regular graphs}. This chapter is focused on the concept of distance regular graphs, see for more details \cite{BCN, BH,Van}, and their algebraic counterpart: the structure of \textit{association schemes}, see \cite{BS, BCN, Delsarte1, Van}. In Section \ref{sec46} we give the new results on association schemes and their applications to regular systems, contained in the paper \cite{VS2}.

\section{Spectral graph theory}\label{sec41}

A graph $G$ is an ordered pair $(V(G),E(G))$, where $V(G)$ is a non-empty set of elements called \textit{vertices} and $E(G)$ is the set of \textit{edges}, together with an \textit{incidence function} $\phi$ that associates with each edge, an unordered pair of vertices. If $\phi(e)=\{u,v\}$ we say that $e$ \textit{joins} $u$ and $v$, and $u$ and $v$ are called the \textit{ends} of $e$. If there is an edge that joins the vertices $u$ and $v$, they are called \textit{adjacent vertices} or \textit{neighbours}. Here we do not allow \textit{loops}, i.e. edges of the type $\{u,u\}$, $u\in V(G)$, and \textit{multiedges}, i.e. distinct edges through the same pair of vertices.
  \begin{defn}
   An $(u,v)$-path of length $l$ on a graph $G$ is a sequence of vertices $u=x_{0},x_{1},\ldots,x_{l}=v$ such that for all $i=\{1,\ldots,l\}$, $x_{i-1}$ and $x_{i}$ are adjacent.
  \end{defn}
  Two vertices $u$ and $v$ are \textit{connected} if either $u=v$, or there exists an $(u,v)$-path. Connectedness is an equivalence relation, that provides a partition of the vertex set of the graph in \textit{connected components}. A graph with only one connected component is called \textit{connected}. The distance $d(u,v)$ between two connected vertices $u$ and $v$ is the minimum length of an $(u,v)$-path, while the maximum distance between two vertices in a connected graph is called \textit{diameter}. The set of all neighbours in $G$ of a vertex $v$ is denoted by $N_{G}(v)$.
  \begin{defn}
 An \textit{isomorphism} $\varphi$ between two graphs $G$ and $G'$ is a map from the vertex set $V(G)$ of $G$, to the vertex set $V(G')$ of $G'$ which preserves adjacencies, i.e. such that $(u,v)\in E(G)$ if and only if $(\varphi(u),\varphi(v))\in E(G')$.
\end{defn}

\begin{defn}
 \begin{itemize}
  \item A \textit{clique} is a set of pairwise adjacent vertices. A clique is said to be \textit{maximal} if it is maximal with respect to set theoretical inclusion.
  \item A \textit{coclique} is a set of pairwise non-adjacent vertices. A coclique is said to be \textit{maximal} if it is maximal with respect to set theoretical inclusion.
  \item The \textit{independence number} $\alpha(G)$ of a graph $G$ is the size of the largest coclique in $G$.
 \end{itemize}
\end{defn}

 \begin{defn}
   \label{DRG}
   A graph is called \textit{distance regular} if, for any two vertices $v$ and $w$, the number of vertices at distance $j$ from $v$ and distance $k$ from $w$ depends only upon the parameters $j$, $k$ and $i=d(v,w)$.
  \end{defn}

The spectral graph theory is a widely used combinatorial approach to the study of graphs based on the eigenvalues of their adjacency matrix. Fix an order $v_{1},\ldots,v_{n}$ on the $n$ vertices of a graph $G$, then the \textit{adjacency matrix} $A$ of $G$ is a $n\times n$ symmetric matrix whose coordinates are indexed by the vertices of $G$ and where $A(i,j)$ is the number of edges connecting $v_{i}$ and $v_{j}$. Since we do not allow loops and multiedges, the adjacency matrix is a $(0,1)$-matrix with all zeros in the main diagonal
\begin{lem}
   \label{kregeig}
    The adjacency matrix $A$ of a $k$-regular graph $G$ has $(1,\ldots,1)^{T}$ as eigenvector, with relative eigenvalue $k$ of
    multiplicity 1 if and only if $G$ is connected. All other eigenvalues $\lambda$ are such that $|\lambda|\leq k$.
   \end{lem}
   \begin{proof}
    Is easy to check that $(1,\ldots,1)^{T}$ is an eigenvector, with relative eigenvalue $k$. Let $x=(x_{1},\ldots,x_{v})^{T}$ be an eigenvector with eigenvalue $\lambda$. We show at first that $|\lambda|\leq k$. W.l.o.g. let $|x_{1}|=max\{|x_{i}|,i=1,\ldots,v\}$. The first coordinate of the equation $Av=\lambda v$ gives $\sum_{i\in N(u_{1})}x_{i}=\lambda x_{1}$, where $N(u_{1})$ is the neighbourhood of the vertex $u_{1}$, corresponding to the ones in the first row of $A$. $|N(u_{1})|=k$ for the $k$-regularity of $G$, and it gives, for the triangular inequality
$$|\lambda||x_{1}|\leq|\lambda x_{1}|=\sum_{i\in N(u_{1})}x_{i}\leq k|x_{1}|.$$
Let $G$ be disconnected, with $L$ connected components. The eigenspace of $k$ is generated by the $L$ linearly independent vectors $w_{i}$, $i=1,\ldots,L$ (where $w_{i}$ is the $(0,1)$-vector whose non-zero coordinates correspond exactly to vertices in the component $G_{i}$) i.e. the multiplicity is greater than 1. Let now $G$ be connected, and $(x_{1},\ldots,x_{v})^{T}$ eigenvector with eigenvalue $k$. We need to show that $x_{1}=\ldots=x_{v}$. As before, let $|x_{1}|=max\{|x_{i}|,i=1,\ldots,v\}$, now the first coordinate of the equation $Av=kv$ gives $\sum_{i\in N(u_{1})}x_{i}=kx_{1}$. Now, $\forall i\in N(u_{1})$ $x_{i}\leq x_{1}$, and the previous equation gives $x_{i}=x_{1}$ for all $i$. By repeating this we get $x_{1}=\ldots=x_{v}$, as we claimed
   \end{proof}

   \begin{lem}
   \label{ortheig}
    Two eigenvectors of an adjacency matrix $A$ with different eigenvalues are orthogonal.
   \end{lem}

   \begin{proof}
    Let be $x$ and $y$ eigenvectors of $A$ with respective eigenvalues $\rho$ and $\theta$ such that $\rho\neq\theta$. Since $A$ is symmetric,
    $$\langle x,Ay\rangle=x^{T}Ay=x^{T}A^{T}y=(Ax)^{T}y=\langle Ax,y\rangle,$$
    i.e. $\langle x,\theta y\rangle=\langle\rho x,y\rangle$, that gives
    $$(\theta-\rho)\langle x,y\rangle=0.$$
    Since $\rho\neq\theta$, the thesis.
   \end{proof}
  The \textit{spectrum} of a graph consists of its eigenvalues, counted with their relative  multiplicities. It can be shown that the spectrum is a graph invariant, so two isomorphic graphs are always \textit{cospectral}, i.e. they have the same spectrum.
  One might conjecture that any two cospectral graphs are also isomorphic. A counterexample of this conjecture is given by the so called \textit{Saltire}, since the two graphs viewed  together represent a picture of Scotland's flag.
\begin{exmp}
 Consider the graphs $C_{4}\cup K_{1}$ and $S_{4}$. The cyclic graph on $4$ vertices, together with an isolated vertex, has adjacency matrix,
\begin{equation*}
 \left(
     \begin{array}{ccccc}
       0 & 1 & 0 & 1 & 0 \\
       1 & 0 & 1 & 0 & 0 \\
       0 & 1 & 0 & 1 & 0 \\
       1 & 0 & 1 & 0 & 0 \\
       0 & 0 & 0 & 0 & 0 \\
     \end{array}
   \right),
 \end{equation*}
 while the star graph $S_{4}$ has adjacency matrix
 \begin{equation*}
  \left(
     \begin{array}{ccccc}
       0 & 0 & 0 & 0 & 1 \\
       0 & 0
 & 0 & 0 & 1 \\
       0 & 0 & 0 & 0 & 1 \\
       0 & 0 & 0 & 0 & 1 \\
       1 & 1 & 1 & 1 & 0 \\
     \end{array}
   \right).
 \end{equation*}
 Both matrices have characteristic polynomial $-\lambda^{3}(\lambda^{2}-4)$, hence both graphs have spectrum $(-2,0^{3},2)$. But the cospectral graphs are not isomorphic since only one is connected.
 \begin{figure}[h]
  \centering
  \includegraphics[scale=0.9]{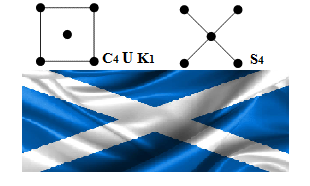}
  \caption{\footnotesize{The \textit{Saltire} counterexample}}
 \end{figure}
\end{exmp}

Two graphs cospectral but not isomorphic are usually called \textit{copsectral mate}.

\begin{prob}
 \label{cospectralgraph}
 Since cospectral graphs are not always isomorphic, the following problems arise:
 \begin{itemize}
  \item characterize graphs \textit{determined by their spectrum}, i.e. graphs without cospectral mate.
  \item study families of cospectral graphs.
 \end{itemize}
\end{prob}

\section{Association schemes}\label{sec42}
  Given a set $A$, the concept of an \textit{association scheme} is introduced to represent  certain relations between pairs of elements in $A$. From now on, the set $A$ is assumed to be finite.
  \begin{defn}
   For a set $A$ and the Cartesian product
   $$A\times A=\{(\alpha,\beta)|\alpha,\beta\in A\},$$
   an association scheme with $d$ associate classes on  $A$ is a partition of $A\times A$ into $d+1$ associate classes $C_{0}, C_{1},\ldots, C_{d}$ such that:
   \begin{enumerate}
    \item $C_{0}=Diag(A)=\{(\alpha,\alpha)|\alpha\in A\}$;
    \item for all $i$ in $\{0,1,\ldots, d\}$, $C_{i}$ is symmetric, i.e. $(\alpha,\beta)\in C_{i}$ if and only if $(\beta,\alpha)\in C_{i}$;
    \item for all $i,j,k$ in $\{0,1,\ldots,d\}$ there exists an integer $p_{ij}^{k}$ such that, for all $(\alpha,\beta)\in C_{k}$:
         $$|\{\gamma\in A|(\alpha,\gamma)\in C_{i} \wedge (\gamma,\beta)\in C_{j}\}|=p_{ij}^{k}.$$
   \end{enumerate}
  \end{defn}
  The symmetry condition also says that $p_{ij}^{0}=0$ for $i\neq j$. In a similar way $p_{i0}^{k}=0$ for $i\neq k$ and $p_{0j}^{k}=0$ for $j\neq k$, whereas $p_{i0}^{i}=1=p_{0i}^{i}$. The number $p_{ii}^{0}=a_{i}$ is called the \textit{valency} of the $i$-th associate class.
  We are now in a position to reformulate the definition (\ref{DRG}) of distance regular graphs. In a graph $G$ of vertex set $V(G)$ and diameter $d$ we introduce, for $i=0,1,\ldots,d$, the following subsets of $V(G)\times V(G)$:
  $$G_{i}=\{(x,y)\in(V(G)\times V(G))|d(x,y)=i\}.$$
  \begin{defn}
   A graph is distance regular if $G_{0},G_{1},\ldots,G_{d}$ are the $d+1$ classes of an association scheme on $V(G)$.
  \end{defn}
 \section{Adjacency matrices and Bose-Mesner Algebra}\label{sec43}
   For an association scheme $(A,\{C_{0},C_{1},\ldots,C_{d}\})$ on the set $A=\{\omega_{1},\ldots,\omega_{n}\}$, the \textit{adjacency matrix} $A_{i}$ of the class $C_{i}$, $(i=0,1,\ldots,d)$, is an $n\times n$ symmetric $(0,1)$-matrix whose coordinates are indexed by the elements of $A$ and where $A_{i}(j,k)=1$ if and only if $(\omega_{j},\omega_{k})\in C_{i}$. Those matrices act on the space $\mathbb{R}^{A}$, i.e. the space of functions from $A$ to $\mathbb{R}$. Since $|A|=n$, $\mathbb{R}^{A}\cong\mathbb{R}^{n}$. From the definition of association scheme, we find the following conditions:
   \begin{enumerate}
   \item $A_{0}=I_{n}$;
   \item each $A_{i}$ is symmetric and $A_{0}+A_{1}+\ldots+A_{d}=J_{n}$, the all-one matrix;
   \item $A_{i}A_{j}=\sum_{k=0}^{d}p_{ij}^{k}A_{k}$.
   \end{enumerate}
   From these conditions, the vector space $\langle A_{0},A_{1},\ldots,A_{d}\rangle$ is closed under multiplication and forms a $(d+1)$-dimensional commutative algebra $\mathcal{A}$ called \textit{Bose-Mesner Algebra}. Moreover, for all $i\in\{1,\ldots,s\}$, $A_{i}J_{n}=J_{n}A_{i}$ and $A_{i}A_{0}=A_{0}A_{i}=A_{i}$.
   \begin{lem}
    If $A_{0},A_{1},\ldots,A_{d}$ are adjacency matrices of an association scheme. Then for all $i,j\in\{1,\ldots,s\}$,
    \begin{equation}
     A_{i}A_{j}=A_{j}A_{i}.
    \end{equation}
   \end{lem}

   \begin{proof}
    The adjacency matrices are symmetric so
    $$A_{j}A_{i}=A_{j}^{T}A_{i}^{T}=(A_{i}A_{j})^{T}=\Big(\sum_{k=0}^{d}p_{ij}^{k}A_{k}\Big)^{T}=\sum_{k=0}^{d}p_{ij}^{k}A_{k}^{T}=\sum_{k=0}^{d}p_{ij}^{k}A_{k}=A_{i}A_{j}.$$
   \end{proof}
 \section{Orthogonal projectors}\label{sec44}
    On the space $\mathbb{R}^{A}$ there is a natural inner product $\langle,\rangle$ defined by (for $f$ and $g$ in $\mathbb{R}^{A}$):
    $$\langle f,g\rangle=\sum_{\omega\in A}f(\omega)g(\omega).$$
    Two vectors $f$ and $g$ are \textit{orthogonal} to each other $(f\perp g)$ if $\langle f,g\rangle=0$. In an analogue way, two subspaces $U$ and $V$ of $\mathbb{R}^{A}$ are orthogonal to each other $(U\perp V)$ if, for all $u$ in $U$ and for all $v$ in $V$, $u\perp v$.\\
    The \textit{orthogonal complement} of a subspace $W\subseteq\mathbb{R}^{A}$, denoted by $W^{\perp}$, is:
    $$W^{\perp}=\{v\in\mathbb{R}^{A}|\forall w\in W, \langle v,w\rangle=0\}.$$
    The following are basic facts about orthogonal complements:
    \begin{enumerate}
     \item $W^{\perp}$ is a subspace of $\mathbb{R}^{A}$;
     \item $dimW+dimW^{\perp}=|A|$;
     \item $(W^{\perp})^{\perp}=W$;
     \item $(U+V)^{\perp}=U^{\perp}\cap V^{\perp}$;
     \item $(U\cap V)^{\perp}=U^{\perp}+V^{\perp}$;
     \item $\mathbb{R}^{A}=W\oplus W^{\perp}$.
    \end{enumerate}

    \begin{defn}
     The orthogonal projector of $\mathbb{R}^{n}$ onto $W\subseteq\mathbb{R}^{n}$ is a map $P:\mathbb{R}^{n}\longrightarrow\mathbb{R}^{n}$ defined by
     \begin{itemize}
      \item $P(v)=v$ if $v\in W$;
      \item $P(v)=0$ if $v\in W^{\perp}$.
     \end{itemize}
    \end{defn}
    For $\omega\in A$, $\chi_{\omega}$ is the \textit{characteristic function} of $\omega$ i.e.
    \begin{itemize}
     \item $\chi_{\omega}(\omega)=1$;
     \item $\chi_{\omega}(\alpha)=0$ for $\alpha\in A$ and $\alpha\neq\omega$.
    \end{itemize}
    \begin{lem}
     $\{\chi_{\omega}|\omega\in A\}$ is a basis for $\mathbb{R}^{A}$.
    \end{lem}
    \begin{proof}
     For all $f\in\mathbb{R}^{A}$, we can write $f=\sum_{\omega\in A}f(\omega)\chi_{\omega}$, and so $\{\chi_{\omega}|\omega\in A\}$ forms a basis for $\mathbb{R}^{A}$, with $dim\mathbb{R}^{A}=|A|$.
    \end{proof}
    The orthogonal projector is a linear transformation, that we can identify with the matrices respect to the basis $\{\chi_{\omega}|\omega\in A\}$.
    \begin{lem}
     Let $U$ and $V$ subspaces of $\mathbb{R}^{n}$ with orthogonal projector $P$ and $Q$, respectively. Then
     \begin{enumerate}
       \item $P$ is idempotent $(P^{2}=P)$;
       \item $P$ is symmetric $(\langle P\chi_{\alpha},\chi_{\beta}\rangle=\langle\chi_{\alpha},P\chi_{\beta}\rangle)$;
       \item $dimU=tr(P)$ and $dimV=tr(Q)$;
       \item if $V=U^{\perp}$ then $Q=I_{n}-P$ (and vice versa);
       \item if $U\perp V$ then $PQ=QP=O_{n}$.
     \end{enumerate}
    \end{lem}

    The subspaces $U$ and $V$ are \textit{geometrically orthogonal} if $PQ=QP$.
    \begin{lem}
     If $U$ and $V$ are geometrically orthogonal then $PQ$ is the orthogonal projector onto $U\cap V$.
    \end{lem}
    \begin{proof}
     If $v\in(U\cap V)$ then $QP(v)=PQ(v)=P(v)=v$. If $v\in(U\cap V)^{\perp}=U^{\perp}+V^{\perp}$ then $v=x+y$ with $x\in U^{\perp}$ and $y\in V^{\perp}$. Now
     $$QP(v)=PQ(v)=PQ(x)+PQ(y)=QP(x)+PQ(y)=0.$$
    \end{proof}
    \begin{lem}
     Let $W_{1},\ldots,W_{n}$ be subspaces of $\mathbb{R}^{A}$ with orthogonal projectors $P_{1},\ldots,P_{n}$ such that:
     \begin{enumerate}
      \item $\sum_{i=1}^{n}P_{i}=I_{n}$;
      \item $P_{i}P_{j}=O_{n}$ if $i\neq j$.
     \end{enumerate}
     Then $\mathbb{R}^{A}=W_{1}\oplus\ldots\oplus W_{n}$.
    \end{lem}
    \begin{proof}
     We should find a unique expression of a vector $v\in\mathbb{R}^{A}$ as a sum of vectors in $W_{i}$. Let $v_{i}=P_{i}v$. Now $$v=Iv=\Big(\sum_{i=1}^{n}P_{i}\Big)v=\sum_{i=1}^{n}(P_{i}v)=\sum_{i=1}^{n}v_{i},$$
     with $v_{i}\in W_{i}$. We need only to show the uniqueness. Suppose that $v=\sum_{i=1}^{n}w_{i},$ with $w_{i}\in W_{i}$. Then for all $j$:
     $$v_{j}=P_{j}v=P_{j}\Big(\sum_{i=1}^{n}w_{i}\Big)=P_{j}\Big(\sum_{i=1}^{n}P_{i}w_{i}\Big)=\Big(\sum_{i=1}^{n}P_{j}P_{i}w_{i}\Big)=P_{j}P_{j}w_{j}=w_{j}.$$
    \end{proof}

   A matrix $M\in\mathcal{A}$ is symmetric, because every $A_{i}$ is symmetric, and from linear algebra is diagonalizable with certain eigenvalues $\lambda_{1},\ldots,\lambda_{r}$, with eigenspaces $W_{1},\ldots,W_{r}$, ($W_{j}=\{v\in\mathbb{R}^{n}|Mv=\lambda_{j}v\}$), where:
   \begin{enumerate}
    \item $\mathbb{R}^{n}=W_{1}\oplus\ldots\oplus W_{r}$;
    \item the minimum polynomial of $M$ is $\prod_{j=1}^{r}(x-\lambda_{j})$.
   \end{enumerate}
   The orthogonal projector onto $W_{i}$ is
   $$P_{i}=\frac{\prod_{j\neq i}(M-\lambda_{j}I)}{\prod_{j\neq i}(\lambda_{i}-\lambda_{j})}.$$
   In fact if $v\in W_{i}$, $\prod_{j\neq i}(M-\lambda_{j}I)v=\prod_{j\neq i}(\lambda_{i}-\lambda_{j})v$, whereas if $v\in W_{k}$ with $k\neq i$, $\prod_{j\neq i}(M-\lambda_{j}I)v=0$.\\

  \section{Minimal idempotent basis}\label{sec45}
   An idempotent $E^{2}=E$ in the Bose-Mesner Algebra is \textit{minimal} if it is not a sum of two non zero idempotents.

   \begin{thm}
    The Bose-Mesner Algebra of an association scheme \\$(A,\{C_{0},C_{1},\ldots,C_{d}\})$ has a unique basis of minimal idempotents $\{E_{0},E_{1},\ldots,E_{d}\}$.
   \end{thm}

   Two minimal idempotents satisfy the multiplication law $E_{i}E_{j}=\delta_{ij}E_{i}$. All $E_{i}$ define orthogonal projections onto the subspaces $V_{i}=Im(E_{i})$, and those subspaces are called \textit{strata} and form an orthogonal decomposition of $\mathbb{R}^{A}=\mathbb{R}^{n}$. One of the minimal idempotents is always the projection onto the all one vector. Matrix multiplication with respect to the basis $\{A_{0},A_{1},\ldots,A_{d}\}$ is more difficult than matrix multiplication with respect to the basis
 $\{E_{0},E_{1},\ldots,E_{d}\}$. We can introduce a second type of multiplication, that switches the role of the bases. The \textit{Schur multiplication}, denoted by $\circ$, gives a product matrix whose $(i,j)$ entry is obtained by multiplying the $(i,j)$ entries of the factors. Since $A_{i}\circ A_{j}=\delta_{ij}A_{i}$, Bose-Mesner Algebra is closed under Schur multiplication, and $\{A_{0},A_{1},\ldots,A_{d}\}$ is a basis of minimal idempotents with respect to it. In \cite{Scott} L. L. Scott introduced also the \textit{Krein parameters} $q_{ij}^{k}$ giving a multiplication law on the basis $\{E_{0},E_{1},\ldots,E_{d}\}$:
   $$E_{i}\circ E_{j}=\frac{1}{|A|}\sum_{k=0}^{d}q_{ij}^{k}E_{k}.$$

   The \textit{dual degree set} of a vector $v\in\mathbb{R}^{A}$ is the set of (non zero) indices $1\leq i\leq d$ such that $E_{i}v\neq0$, or equivalently $v\notin V_{i}^{\perp}$. Two vectors $v_{1}$ and $v_{2}$ are called \textit{design-orthogonal} if their dual degree sets are disjoint.

   If $|\{1\leq i\leq d|E_{i}v\neq0\}\cap\{1,\ldots,k\}|=0$, $v$ is called \textit{$k$-design}, whereas if $\{1\leq i\leq d|E_{i}v\neq0\}\subseteq\{1,\ldots,k\}$, $v$ is called \textit{$k$-antidesign}. Clearly $k$-designs and $k$-antidesigns are design-orthogonal sets.

The following result on the intersection size between design-othogonal subsets can be found in \cite{Delsarte3, Roos}.
   \begin{prop}
    \label{orthogonal}
    \begin{enumerate}
      \item If $v_{1}$ and $v_{2}$ are design-orthogonal vectors then
      $$v_{1}^{T}v_{2}=\frac{(v_{1}^{T}\chi_{A})(v_{2}^{T}\chi_{A})}{|A|}.$$
      \item If $S_{1}$ and $S_{2}$ are design-orthogonal subsets of $A$ then
      $$|S_{1}\cap S_{2}|=\frac{|S_{1}|\cdot|S_{2}|}{|A|}.$$
    \end{enumerate}
   \end{prop}

   \begin{proof}
    \begin{enumerate}
      \item Let $\{E_{0},E_{1},\ldots,E_{d}\}$ be the minimal idempotent basis. Every $v\in\mathbb{R}^{A}$ can be decomposed as
      $$v=E_{0}v+E_{1}v+\ldots+E_{d}v.$$
      Since $E_{0}=\frac{J}{|A|}$, $E_{0}v=\frac{v^{T}\chi_{A}}{|A|}\chi_{A}$. Now
      \small
      $$v_{1}^{T}v_{2}=(E_{0}v_{1}+E_{1}v_{1}+\ldots+E_{d}v_{1})^{T}(E_{0}v_{2}+E_{1}v_{2}+\ldots+E_{d}v_{2})=(E_{0}v_{1})^{T}(E_{0}v_{2}),$$
      \normalsize
      i.e. the thesis.
      \item It follows from the previous point, by replacing $\chi_{S_{i}}$ by $v_{i}$ and using the identities $(\chi_{S_{1}})^{T}\chi_{S_{2}}=|S_{1}\cap S_{2}|$ and $(\chi_{S_{i}})^{T}\chi_{A}=|S_{i}|$ $(i=1,2)$.
    \end{enumerate}
   \end{proof}
 \section{Distance graph of $\mathcal{P}_{d,e}$}\label{sec46}
   Let $\mathcal{D}^i_{\mathcal{P}_{d, e}}$ be the \textit{$i$-th distance graph} of $\mathcal{P}_{d, e}$, that is the graph having as vertices the members of $\mathcal{M}_{\mathcal{P}_{d, e}}$ and two vertices $x, y \in \mathcal{M}_{\mathcal{P}_{d, e}}$ are adjacent whenever $x \cap y$ is a $(d-i-1)$-space of $\mathcal{P}_{d, e}$, $0 \leq i \leq d$. The $i$-th distance graph is regular. If $i = 1$, then $\mathcal{D}^1_{\mathcal{P}_{d, e}}$ is called the \textit{dual polar graph} of $\mathcal{P}_{d, e}$ and it is symply denoted by $\mathcal{D}_{\mathcal{P}_{d, e}}$. The dual polar graph is a distance regular graph with diameter $d$, and hence it gives a $d$-class association scheme. Two vertices are adjacent if the corresponding generators meet in a $(d-2)$-space. For more information on $\mathcal{D}_{\mathcal{P}_{d, e}}$, the reader is referred to \cite[Section 9.4]{BCN}.

\subsection{Eigenvalues of the distance graph}\label{subsec461}
 Let $M$ be a matrix, we denote its column span by $Im(M)$ and its transpose by $M^T$. Fix an ordering on $\mathcal{M}_{\mathcal{P}_{d, e}}$. This provides an ordered basis for the vector space $V = \mathbb{R}^{\mathcal{M}_{\mathcal{P}_{d, e}}}$. Note that $V$ has dimension $n$, where $n = |\mathcal{M}_{\mathcal{P}_{d, e}}|$. Let $A_i$ be the adjacency matrix of the $i$-th distance graph $\mathcal{D}^i_{\mathcal{P}_{d, e}}$ of $\mathcal{P}_{d, e}$ for $0 \leq i \leq d$, the $\mathbb{R}$-vector space generated by the matrices $A_0, \dots, A_d$ is closed under matrix multiplication and so it is the Bose-Mesner algebra of the association scheme given by the graph. Consider the minimal idempotent basis $\{E_0, \dots, E_d\}$ and the orthogonal decomposition $V=V_0\perp\ldots\perp V_d$. Let $C_k$ be the incidence matrix between the $(k-1)$-spaces of $\mathcal{P}_{d, e}$ and the generators of $\mathcal{P}_{d, e}$, $1 \leq k \leq d$. This means that the rows are indexed by the $(k-1)$-spaces, the columns by the generators of the polar space and $(C_k)_{y x} = 1$ if $y \subseteq x$ and $(C_k)_{y x} = 0$ if $y \not\subseteq x$. Now up to a reordering of the indices we have that $Im(C_k^t) = V_0 \perp \dots \perp V_k$, where $V_0=Im(C_0^t)$  and $V_k = Im(C_k^t) \cap ker(C_{k-1})$. For more results on this topic the interested reader is referred to \cite{BCN, Delsarte1, Delsarte2, Van, Van2}. The \textit{dual degree set} of a set $S$ of vertices of $\mathcal{D}_{\mathcal{P}_{d, e}}$ is the set of non-zero indices $1 \leq i \leq d$ such that $E_i \chi_{S} \neq 0$ or equivalently $\chi_{S} \notin V_i^\perp$. recall that two sets of vectors $S_1$ and $S_2$ of are design-orthogonal when their dual degree sets are disjoint. Moreover, if $|\{ 1 \leq i \leq d | E_i \chi_{S} \neq0\} \cap \{1, \ldots, k\}| = 0$, then $S$ is a $k$-design, whereas if $\{ 1 \leq i \leq d | E_i \chi_{S} \neq 0\} \subseteq \{1, \ldots, k\}$, then $S$ is a $k$-antidesign. The following result has been proven in \cite[Theorem 4.4.1]{Van}, see also \cite{Delsarte2}.
\begin{prop}\cite[Theorem 4.4.1]{Van}
\label{design}
  A set $\mathcal{R} \subset \mathcal{M}_{\mathcal{P}_{d, e}}$ is an $m$-regular system with respect to $(k-1)$-spaces of $\mathcal{P}_{d, e}$, $1 \leq k \leq d-1$, if and only if $\mathcal{R}$ is a $k$-design.
\end{prop}
The eigenvalues of the $i$-th distance graph $\mathcal{D}^i_{\mathcal{P}_{d, e}}$ of $\mathcal{P}_{d, e}$ are those of its adjacency matrix $A_i$. The independence number $\alpha(\mathcal{D}^i_{\mathcal{P}_{d, e}})$ is the size of the largest coclique of $\mathcal{D}^i_{\mathcal{P}_{d, e}}$, i.e. the largest set of generators of $\mathcal{P}_{d, e}$  pairwise not intersecting in a $(d-i-1)$-space. The graph $\mathcal{D}^i_{\mathcal{P}_{d, e}}$ has at most $d+1$ distinct eigenvalues \cite{Eisfeld}, \cite[Theorem 4.3.6]{Van}. Let $k_i$ and $\lambda_{i}$ denote the largest and the smallest eigenvalue, respectively; hence $k_i$ is the valency of $\mathcal{D}^i_{\mathcal{P}_{d, e}}$. The following result is due to Hoffman, see \cite[Theorem 3.5.2]{BH}.
\begin{thm}[\textbf{Hoffman's ratio bound}]\label{hoffman0}
 Let $G$ be a $k$-regular graph with vertex set $V(G)$, largest and smallest eigenvalues $k$ and $\lambda$, respectively, and independence number $\alpha(G)$. Then
 \begin{equation}
  \alpha(G)\leq-\frac{|V(G)|\lambda}{k-\lambda}
 \end{equation}
\end{thm}

\begin{cor}\label{hoffman}
 $\alpha(\mathcal{D}^i_{\mathcal{P}_{d, e}}) \leq -\frac{|\mathcal{M}_{\mathcal{P}_{d, e}}| \lambda_i}{k_i - \lambda_i}$.
\end{cor}

The following theorem, contained in \cite{Van}, gives a formula for the eigenvalues of $\mathcal{D}_{\mathcal{P}_{d,e}}^{i}$
\begin{thm}\cite[Theorem 4.3.6]{Van}
\label{eigenformula}
 $\mathcal{D}_{\mathcal{P}_{d,e}}^{i}$ has the following $d+1$ eigenvalues, $0\leq j\leq d$:
  \begin{equation}
     \sum_{max(0,j-i)\leq u\leq min(d-i,j)}(-1)^{j+u}\begin{bmatrix}
         d-j \\
         d-i-u \\
        \end{bmatrix}_{q}\begin{bmatrix}
         j \\
         u \\
        \end{bmatrix}_{q}q^{\frac{(u+i-j)(u+i-j+2e-1)}{2}+\frac{(j-u)(j-u-1)}{2}}.
    \end{equation}
   \end{thm}

Let $\mathcal{R}$ be a $1$-regular system w.r.t. $(k - 1)$-spaces of $\mathcal{P}_{d, e}$. Then every $(k-1)$-space of $\mathcal{P}_{d, e}$ is contained in exactly one member of $\mathcal{R}$ and therefore two distinct generators of $\mathcal{R}$ intersect in at most a $(k-2)$-space of $\mathcal{P}_{d, e}$. In particular $\mathcal{R}$ is a coclique of the graph $\mathcal{D}^{d-k}_{\mathcal{P}_{d, e}}$ and hence
\begin{equation}\label{inequality}
 |\mathcal{R}| \leq -\frac{|\mathcal{M}_{\mathcal{P}_{d, e}}| \lambda_i}{k_i - \lambda_i}.
\end{equation}
In some cases Equation \eqref{inequality} yields a contradiction. We study the cases when $\mathcal{R}$ is a $1$-regular system of a polar space with rank $4$ or $5$.
\begin{thm}
 The polar spaces $Q^+(7, q)$, $H(7, q)$, $W(7, q)$, $Q(8, q)$, $H(8, q)$, $Q^-(9, q)$ do not have a $1$-regular system w.r.t. lines. The polar spaces $Q^+(9, q)$, $H(9, q)$, $W(9, q)$, $Q(10, q)$, $H(10, q)$, $Q^-(11, q)$ do not have a $1$-regular system w.r.t. planes.
 \end{thm}
\begin{proof}
 From Equation \ref{eigenformula}, the eigenvalues of $\mathcal{D}^{2}_{\mathcal{P}_{4, e}}$ are $q^{2e+1}(q^2+1)(q^2+q+1), (q^{2e+1} - q^e)(q^2+q+1), -q^e(q+1)^2+q^{2e+1}+q, -(q^e-q)(q^2+q+1), q(q^2+1)(q^2+q+1)$. Let $\mathcal{R}$ be a $1$-regular system w.r.t. lines of a polar space $\mathcal{P}_{4, e}$. By Equation \eqref{inequality} we have that
 \footnotesize
 \begin{align*}
  & (q^2+1)(q^3+1) = |\mathcal{R}| \leq 2(q^3+1), & \mbox{ for } Q^+(7, q), \\
  & (q^\frac{5}{2}+1)(q^\frac{7}{2}+1) = |\mathcal{R}| \leq \frac{ (q^\frac{1}{2}(q+1)^2-q^2-q) \prod_{i = 1}^{4} ( q^{\frac{9-2i}{2}} + 1 )}{q^2(q^2+1)(q^2+q+1) + q^\frac{1}{2}(q+1)^2-q^2-q} & \mbox{ for } H(7, q), \\
  & (q^3+1)(q^4+1) = |\mathcal{R}| \leq \frac{2(q+1)(q^2+1)(q^3+1)(q^4+1)}{q(q^2+1)(q^2+q+1) + 2} & \mbox{ for } W(7, q), Q(8, q), \\
  & ( q^\frac{7}{2} + 1)(q^\frac{9}{2} + 1) \leq \frac{( q^\frac{1}{2} - 1 )( q^\frac{3}{2} + 1)  (q^\frac{5}{2} + 1 )( q^\frac{7}{2} + 1 )( q^\frac{9}{2} + 1)}{q^3(q^2+1) + q^\frac{1}{2} - 1} & \mbox{ for } H(8, q), \\
  & (q^4+1)(q^5+1) = |\mathcal{R}| \leq \frac{(q-1)(q^2+1)(q^3+1)(q^4+1)(q^5+1)}{q^4(q^2+1)+q-1}  & \mbox{ for } Q^-(9, q).
 \end{align*}
 \normalsize
  In all cases we get a contradiction.
 Similarly the largest and the smallest eigenvalue of $\mathcal{D}^2_{\mathcal{P}_{5, e}}$ are $q^{2e+1}(q^2+1)(q^4+q^3+q^2+q+1)$ e $-q^e(q+1)(q^2+q+1)+q^{2e+1}+q(q^2+q+1)$. Hence if $\mathcal{R}$ is a $1$-regular system w.r.t. planes of $\mathcal{P}_{5, e}$, it follows that

 \footnotesize
 \begin{align*}
  & \prod_{i = 1}^{3} ( q^{5-i} + 1) = |\mathcal{R}| \leq \frac{2(q+1)(q^2+1)(q^4+1)}{q^2+q+1}  \mbox{ for } Q^+(9, q), \\
  & \prod_{i = 1}^{3} ( q^{\frac{11-2i}{2}} + 1 ) = |\mathcal{R}| \leq  \frac{ (q^\frac{1}{2}(q+1)(q^2+q+1) -q (q+1)^2 ) \prod_{i = 1}^{5} ( q^{\frac{11-2i}{2}} + 1 )}{\frac{q^2(q^2+1) (q^5-1)}{q-1} + q^\frac{1}{2}(q+1)(q^2+q+1) -q (q+1)^2}  \mbox{ for } H(9, q), \\
  & \prod_{i = 1}^{3}(q^{6-i}+1) = |\mathcal{R}| \leq \frac{(q+1)(q^2+1)(q^4+1)(q^5+1)}{q^2+q+1}  \mbox{ for } W(7, q), Q(8, q), \\
  & \prod_{i = 1}^{3} ( q^{\frac{13-2i}{2}} + 1 ) = |\mathcal{R}| \leq  \frac{(q^\frac{3}{2}(q+1)(q^2+q+1) -q (q+1)(q^2+1)) \prod_{i = 1}^{5}( q^{\frac{13-2i}{2}} + 1 )}{\frac{q^4(q^2+1) (q^5-1)}{q-1} + q^\frac{3}{2}(q+1)(q^2+q+1) - q (q+1)(q^2+1)}  \mbox{ for } H(10, q), \\
  & \prod_{i = 1}^{3} (q^{7-i}+1) = |\mathcal{R}| \leq \frac{(2q^4+q^3-q) \prod_{i = 1}^{5} (q^{7-i}+1)}{\frac{q^5(q^2+1)(q^5-1)}{q-1} + 2q^4+q^3-q} \mbox{ for } Q^-(11, q).
 \end{align*}
 \normalsize
 Again in all cases we get a contradiction.
\end{proof}
In order to obtain similar results for polar spaces of higher rank, the knowledge of the smallest eigenvalue of the $i$-th distance graph is needed.
\begin{prob}
 Determine the smallest eigenvalue of the graph $\mathcal{D}^i_{\mathcal{P}_{d, e}}$.
\end{prob}
In the following table we list some computations performed with \textit{Wolfram Mathematica} \cite{Mathematica}.
\footnotesize{\begin{table}[h!]
\centering
\begin{center}
\begin{tabular}{|c|c|c|}
\hline
$(\mathbf{d},\mathbf{i})$ & $\mathbf{k}$ &  $\mathbf{\lambda}$  \\ \hline
$(2,1)$ & $q^e(q+1)$ &  $-(q+1)$  \\ \hline
$(3,1)$ & $q^e(q^2+q+1)$ & $-(q^2+q+1)$ \\ \hline
$(3,2)$ & $q^{2e+1}(q^2+q+1)$ & $q-q^e(q+1)$ \\ \hline
$(4,1)$ & $q^e(q^3+q^2+q+1)$  & $-(q^3+q^2+q+1)$  \\ \hline
$(4,2)$ & $q^{2e+1}(q^2+1)(q^2+q+1)$  & $q^{2e+1}-q^e(q+1)^2+q$ \\ \hline
$(4,3)$ & $q^{3(e+1)}(q^3+q^2+q+1)$ & $-q^{3}(q^3+q^2+q+1)$  \\ \hline
$(5,1)$ & $q^e(q^4+q^3+q^2+q+1)$ & $-(q^4+q^3+q^2+q+1)$  \\ \hline
$(5,2)$ & $q^{2e+1}(q^2+1)(q^4+q^3+q^2+q+1)$ &  $q^{2e+1}-(q^{e+1}+q^e+q)(q^2+q+1)$ \\ \hline
$(5,3)$ & $q^{3(e+1)}(q^2+1)(q^4+q^3+q^2+q+1)$  & $-q^3(q^2+1)(q^4+q^3+q^2+q+1)$ \\ \hline
$(5,4)$ & $q^{2(2e+3)}(q^4+q^3+q^2+q+1)$ & $q^6-q^{e+3}(^4+q^3+q^2+q+1)$ \\ \hline
\end{tabular}
\caption{\footnotesize{The largest and the smallest eigenvalue of $\mathcal{D}^i_{\mathcal{P}_{d, e}}$, for $2\leq d\leq5$.}}
\end{center}
\end{table}}
\normalsize

Having a look on the structure of the eigenvalues we get that for $i$ odd the smallest eigenvalues is always the last one, i.e. the one corresponding to $j=d$ in Equation \ref{eigenformula}\footnote{We thank G. Monzillo, F. Romaniello and A. Siciliano for some interesting discussions about the smallest eigenvalue of the graph $\mathcal{D}^i_{\mathcal{P}_{d, e}}$}. In the case of the hyperbolic quadric, i.e. $e=0$, we have $k=-\lambda$, since the graph is bipartite, and the bipartition sets are the Latin and Greek generators.

\subsection{Chains of regular systems}\label{sec462}

Let $\mathcal{P}_{d, e}$ be one of the following polar spaces: $Q(2d, q)$, $Q^-(2d+1, q)$, $H(2d, q)$, $d \geq 2$, and consider a projective subspace of the ambient projective space meeting $\mathcal{P}_{d, e}$ in a cone, say $\mathcal{K}$, having as vertex a $(k-2)$-space of $\mathcal{P}_{d, e}$ and as base a $Q^+(2d-2k+1, q)$, $Q(2d-2k+2, q)$, $H(2d-2k+1, q)$, respectively, where $1 \leq k \leq d-1$. Let $\Gamma_{r}$ be the set of $(r-1)$-spaces of $\mathcal{P}_{d, e}$ contained in $\mathcal{K}$, $1 \leq r \leq d-1$, and denote by $\Omega_{r}$ the set of generators of $\mathcal{P}_{d, e}$ meeting $\mathcal{K}$ in exactly an $(r-1)$-space, $max\{k, d-k\} \leq r \leq d$. Hence $\Omega_{d}$ is the set of generators of $\mathcal{P}_{d, e}$ contained in $\mathcal{K}$.

\begin{lem}[{\cite[Theorem 2.4]{Van2}}]\label{cone1}
Let $\sigma$ be an $(r-1)$-space of $\mathcal{P}_{d, e}$, $1 \leq r \leq d-1$, and let $\mathcal{S}$ be the set of generators containing $\sigma$, then $\mathcal{S}$ is an $r$-antidesign.
\end{lem}
\begin{proof}
Note that $\chi_{\mathcal{S}} = C_r^t \chi_{\sigma}$. Since $Im(C_r^t) = V_0 \perp \dots \perp V_r$, the result follows.
\end{proof}

\begin{lem}\label{anti}
 $\Omega_{d}$ is a $k$-antidesign.
\end{lem}
\begin{proof}
 Recall that $\mathcal{P}_{d, e}$ is one of the following polar spaces: $Q(2d, q)$, $Q^-(2d+1, q)$, $H(2d, q)$, $d \geq 2$. First we see that the statement is true if $k = 1$. Indeed
 $$C_1^t \chi_{\Gamma_{1}} = \genfrac{[}{]}{0pt}{}{d}{1}_q \chi_{\Omega_{d}} + \genfrac{[}{]}{0pt}{}{d-1}{1}_q \chi_{\Omega_{d-1}} = \genfrac{[}{]}{0pt}{}{d}{1}_q \chi_{\Omega_{d}} + \genfrac{[}{]}{0pt}{}{d-1}{1}_q (\chi_{\mathcal{M}_{\mathcal{P}_{d, e}}} - \chi_{\Omega_{d}})$$
 and $C_1^t \j = \genfrac{[}{]}{0pt}{}{d}{1}_q \chi_{\mathcal{M}_{\mathcal{P}_{d, e}}}$. Hence $\chi_{\Omega_{d}} \in Im(C_1^t) = V_0 \perp V_1$. See also \cite[Theorem 5.1]{Van2}. Assume that $\chi_{\Omega_{d}} \in Im(C_s^t) = V_0 \perp \dots \perp V_s$, whenever $\mathcal{K}_s$ is a cone having as vertex an $(s-2)$-space of $\mathcal{P}_{d, e}$, and as base a $Q^+(2d-2s+1, q)$, $Q(2d-2s+2, q)$, $H(2d-2s+1, q)$, respectively, $1 < s \leq k-1$. Let $\Pi_r$ denote a subspace of the ambient space of projective dimension $2d-k+r$ if $\mathcal{P}_{d, e} \in \{Q(2d, q), H(2d, q)\}$ or of projective dimension $2d-k+r+1$ if $\mathcal{P}_{d, e} = Q^-(2d+1, q)$ and let $\mathcal{G}_{\Pi_{r}}$ be the set of generators of $\mathcal{P}_{d, e}$ contained in $\Pi_r$. If for some fixed $r$, with $1 \leq r \leq k-1$, we have that $\mathcal{K} \subset \Pi_r$, then two possibilities occur: either $\Pi_r \cap \mathcal{P}_{d, e}$ is a cone having as vertex a $(k - r - 1)$--space of $\mathcal{P}_{d, e}$ and as base a $Q(2d - 2k + 2r, q)$, $Q^-(2d - 2k + 2r +1, q)$, $H(2d - 2k + 2r, q)$, respectively, or $\Pi_r \cap \mathcal{P}_{d, e}$ is a cone having as vertex a $(k - r - 2)$--space of $\mathcal{P}_{d, e}$ and as base a $Q^+(2d - 2k + 2r+1, q)$, $Q(2d-2k + 2r + 2, q)$, $H(2d - 2k + 2r + 1, q)$, respectively. By using Lemma \ref{cone1} in the former case and our previous assumptions in the latter case, we have
 \begin{equation}\label{base}
  \chi_{\mathcal{G}_{\Pi_r}} \in Im(C_{k - r}^t) = V_0 \perp \dots \perp V_{k - r}.
 \end{equation}
 Also, observe that
 \begin{equation} \label{case1}
  C_k^t \chi_{\Gamma_k}= \genfrac{[}{]}{0pt}{}{d}{k}_q \chi_{\Omega_{d}} + \genfrac{[}{]}{0pt}{}{d-1}{k}_q \chi_{\Omega_{d-1}} + \dots + \genfrac{[}{]}{0pt}{}{d-k}{k}_q \chi_{\Omega_{d-k}}
 \end{equation}
 if $k \leq \lfloor \frac{d}{2} \rfloor$ and
 \begin{align}
  C_k^t \chi_{\Gamma_k} & = \genfrac{[}{]}{0pt}{}{d}{k}_q \chi_{\Omega_{d}} + \genfrac{[}{]}{0pt}{}{d-1}{k}_q \chi_{\Omega_{d-1}} + \dots + \genfrac{[}{]}{0pt}{}{k}{k}_q \chi_{\Omega_{k}} {\nonumber} \\ & = \genfrac{[}{]}{0pt}{}{d}{k}_q \chi_{\Omega_{d}} + \genfrac{[}{]}{0pt}{}{d-1}{k}_q \chi_{\Omega_{d-1}} + \dots + \genfrac{[}{]}{0pt}{}{d-(k-s')}{k}_q \chi_{\Omega_{d-(k-s')}} \label{case2}
 \end{align}
 if $k > \lfloor \frac{d}{2} \rfloor$, with $k = d- (k-s')$, for some $s'$, where $1 \leq s' \leq k-1$.
 Moreover
 \begin{equation}\label{recursive1}
  \chi_{\Omega_{d-r}} = \sum_{\mathcal{K} \subset \Pi_r} \chi_{\mathcal{G}_{\Pi_r}} - \sum_{i = 1}^r \genfrac{[}{]}{0pt}{}{k - r + i}{i}_q \chi_{\Omega_{d-r+i}}
 \end{equation}
 for $1 \leq r \leq k-1$ and
 \begin{equation}\label{recursive2}
  \chi_{\Omega_{d - k}} = \chi_{\mathcal{M}_{\mathcal{P}_{d, e}}} - \chi_{\Omega_d} - \chi_{\Omega_{d-1}} \ldots - \chi_{\Omega_{d-k+1}}.
 \end{equation}
 Taking into account \eqref{recursive2} and applying recursively \eqref{recursive1}, we obtain that the hand-right side of Equation \eqref{case1} can be written as a linear combination of $\chi_{\Omega_{d}}, \chi_{\mathcal{M}_{\mathcal{P}_{d, e}}}, \chi_{\mathcal{G}_{\Pi_1}}, \ldots, \chi_{\mathcal{G}_{\Pi_{k-1}}}$. Similarly, Equation \eqref{case2} can be written as a linear combination of $\chi_{\Omega_{d}}, \chi_{\mathcal{G}_{\Pi_1}}, \dots, \chi_{\mathcal{G}_{\Pi_{k-s'}}}$. The assertion follows from \eqref{base}.
\end{proof}

Let $\mathcal{R}$ be a regular system of $\mathcal{P}_{r}$. In the following lemma it is shown that if $\Pi$ is a hyperplane such that $\mathcal{P}'_r:=\Pi \cap \mathcal{P}_r$ is a polar space of rank $r$, then $\mathcal{R}$ has a constant number of elements in common with $\mathcal{M}_{\mathcal{P}'_r}$.

\begin{lem}\label{hyper}
 Let $\mathcal{P}_r$ be one of the following polar spaces: $Q(2r+2,q)$, $Q^-(2r+3,q)$, $H(2r+2,q^2)$, $r \geq 1$, let $\mathcal{R}$ be an $m$-regular systems of $\mathcal{P}_r$ and consider a hyperplane meeting $\mathcal{P}_r$ in $\mathcal{P}'_r$, where $\mathcal{P}'_r$ is $Q^+(2r+1,q)$, $Q(2r+2,q)$, $H(2r+1,q^2)$, respectively. Then,
 $$| \mathcal{R} \cap \mathcal{M}_{\mathcal{P}'_r} |=
 \begin{cases}
  2m & if \quad \mathcal{P} = Q(2r+2,q) \\
  m(q+1) & if \quad \mathcal{P} = Q^-(2r+3,q) \\
  m(q+1) & if \quad \mathcal{P} = H(2r+2,q^2) \\
 \end{cases}$$
\end{lem}
\begin{proof}
 A standard double counting argument on pairs $(P,\ell)$, where $P$ is a point of $\mathcal{P}'_r$ and $\ell$ is a member of $\mathcal{R}$, with $P \in \ell$, gives
 $$m \mathcal{O}_r' \theta_{r} = | \mathcal{R} \cap \mathcal{M}_{\mathcal{P}'_r} | \theta_{r} + | \mathcal{R} \setminus (\mathcal{R} \cap \mathcal{M}_{\mathcal{P}'_r}) | \theta_{r-1} ,$$
 where $\mathcal{O}_r'$ is the ovoid number of $\mathcal{P}'_r$. Therefore
 \begin{equation} \label{eq1}
 | \mathcal{R} \cap \mathcal{M}_{\mathcal{P}'_r} | (\theta_{r} - \theta_{r-1}) = m (\mathcal{O}_r' \theta_{r} - \mathcal{O}_r \theta_{r-1}) .
 \end{equation}
 Now the result follows by substituting in \eqref{eq1} the values of $\mathcal{O}_r, \mathcal{O}_r', \theta_{r}, \theta_{r-1}$ for the parabolic, elliptic and Hermitian case, respectively.
\end{proof}

The next result shows that, by means of Lemma \ref{anti}, it is possible to construct $m'$-regular systems starting from a given $m$-regular system.

\begin{thm}\label{chain}
 \begin{enumerate}
  \item  If $Q^-(2d + 1, q)$ has an $m$-regular system w.r.t. $(k-1)$-spaces, then $Q(2d+2, q)$ has an $m(q+1)$-regular system w.r.t $(k-1)$-spaces and if $k \geq 2$, then $Q(2d, q)$ has an $m(q+1)$-regular system w.r.t. $(k-2)$-spaces.
  \item If $Q(2d, q)$ has an $m$-regular system w.r.t. $(k-1)$-spaces, then $Q^+(2d+1, q)$ has a $2m$-regular system w.r.t. $(k-1)$-spaces and if $k \geq 2$, then $Q^+(2d-1, q)$ has an $2m$-regular system w.r.t. $(k-2)$-spaces.
  \item If $H(2d, q)$ has an $m$-regular system w.r.t. $(k-1)$-spaces, then $H(2d+1, q)$ has an $m(q+1)$-regular system w.r.t. $(k-1)$-spaces and if $k \geq 2$, then $H(2d-1, q)$ has an $m(q^{\frac{1}{2}}+1)$-regular system w.r.t. $(k-2)$-spaces.
 \end{enumerate}
\end{thm}
\begin{proof}
 Let $\mathcal{P}_{d, e}$ be one of the following polar spaces: $Q(2d, q)$, $Q^-(2d+1, q)$, $H(2d, q)$, $d \geq 2$, and let $\mathcal{R}$ be an $m$-regular system of $\mathcal{P}_{d, e}$ w.r.t. $(k-1)$-spaces. Embed $\mathcal{P}_{d, e}$ in $\mathcal{P}_{d+1, e-1}'$, where $\mathcal{P}_{d+1, e-1}'$ is $Q^+(2d+1, q)$, $Q(2d+2, q)$, $H(2d+1, q)$, respectively. Notice that a generator of $\mathcal{P}_{d+1, e-1}'$ contains exactly one generator of $\mathcal{P}_{d, e}$. On the other hand, through a generator, there pass $q^{e-1}+1$ generators of $\mathcal{P}_{d+1, e-1}'$. Let $\mathcal{R}'$ be the set of generators of $\mathcal{P}_{d+1, e-1}'$ containing a member of $\mathcal{R}$. Let $\sigma$ be a $(k-1)$-space of $\mathcal{P}_{d+1, e-1}'$ and let $\mathcal{C}$ be the set of points of $\mathcal{P}_{d+1, e-1}'$ belonging to at least a generator of $\mathcal{P}_{d+1, e-1}'$ through $\sigma$. If $\sigma$ is contained in $\mathcal{P}_{d, e}$, then through $\sigma$ there pass $m (q^{e-1}+1)$ elements of $\mathcal{R}'$. If $\sigma$ meets $\mathcal{P}_{d, e}$ in a $(k-2)$-space, then $\mathcal{K} = \mathcal{C} \cap \mathcal{P}_{d, e}$ is a cone having as vertex $\sigma \cap \mathcal{P}_{d, e}$ and as base a $Q^+(2d-2k+1, q)$, $Q(2d-2k+2, q)$, $H(2d-2k+1, q)$, respectively. Let $\Omega_d$ be the set of generators of $\mathcal{P}_{d, e}$  contained in $\mathcal{K}$. From Propositions \ref{orthogonal} and \ref{design} and Lemma \ref{anti}, we have
 $$| \mathcal{R} \cap \Omega_d | = \frac{m \prod_{i=1}^{k} (q^{d+e-i} + 1) \prod_{i=1}^{d-k+1} (q^{(d-k+1) + (e-1) -i} + 1)}{\prod_{i = 1}^{d} (q^{d+e-i} + 1)} = m (q^{e-1}+1).$$
 Hence, through $\sigma$ there pass $m(q^{e-1}+1)$ generators of $\mathcal{R}'$. Let $\mathcal{R}''$ be the set of generators of $\mathcal{R}$ contained in a polar space $\mathcal{P}_{d, e-1}''$ embedded in $\mathcal{P}_{d, e}$, where $\mathcal{P}_{d, e-1}''$ is $Q^+(2d-1, q)$, $Q(2d, q)$, $H(2d-1, q)$, respectively. We have just seen that through a $(k-2)$-space of $\mathcal{P}_{d, e-1}''$ there pass $m(q^{e-1}+1)$ members of $\mathcal{R}''$. Therefore $\mathcal{R}''$ is an $m(q^{e-1}+1)$-regular system of $\mathcal{P}_{d, e-1}''$ w.r.t. $(k-2)$-spaces.
\end{proof}

\begin{remark}
 If $q$ is even, then the polar space $Q(2d, q)$, $d \geq 2$, has an $m$-regular system if and only if the polar space $W(2d-1, q)$ has an $m$-regular system. Indeed, let $N$ be the nucleus of $Q(2d, q)$ and let $\Pi$ be a hyperplane of $PG(2d, q)$ not containing $N$. The projections from $N$ onto $\Pi$ of the totally singular $s$-dimensional subspaces of $Q(2d, q)$ are the totally isotropic $s$-dimensional subspaces of a polar space $W(2d-1, q)$ in $\Pi$. Hence an $m$-regular system of $Q(2d, q)$ is projected onto an $m$-regular system of $W(2d-1, q)$, and, conversely, any $m$-regular system of $W(2d-1, q)$ is the projection of an $m$-regular system of $Q(2d, q)$.
\end{remark}

\begin{cor}
 \begin{enumerate}
  \item If $Q^-(2d + 1, q)$ has an $m$-regular system w.r.t. $(k-1)$-spaces, then $Q^+(2d+3,q)$ has a $2m(q+1)$-regular system w.r.t. $(k-1)$-spaces.
  \item If $Q^-(2d + 1, q)$, $q$ even, has an $m$-regular system w.r.t. $(k-1)$-spaces, then $W(2d + 1, q)$ has an $m(q+1)$-regular system w.r.t. $(k-1)$-spaces.
 \end{enumerate}
\end{cor}

The generators of $Q^+(2d-1, q)$ are $(d-1)$-spaces and the set of all generators is partitioned into Greek and Latin generators and denoted by $\mathcal{M}_1=\mathcal{G}$ and $\mathcal{M}_2=\mathcal{L}$, respectively. It is straightforward to check that $\mathcal{M}_i$ is a hemisystem of $Q^+(2d-1, q)$ w.r.t. $(k-1)$-spaces, $1 \leq k \leq d-1$. More interestingly it is possible to get regular systems by a switching procedure as described in the following proposition.
\begin{prop}
 Let $\sigma$ be an $(s-1)$-space of $Q^+(2d-1, q)$, for a fixed $s$, with $1 \leq s \leq d-2$, and let $\mathcal{Z}_i$ be the set of generators of $\mathcal{M}_i$ containing $\sigma$, $i\in\{1,2\}$. Then $(\mathcal{M}_i \setminus \mathcal{Z}_i) \cup \mathcal{Z}_j$, with $i\neq j$, is a hemisystem of $Q^+(2d-1, q)$ w.r.t. $(d-s-2)$-spaces.
\end{prop}
\begin{proof}
 W.l.o.g., we can assume that $i=1$ and $j=2$. Since $|\mathcal{Z}_1| = |\mathcal{Z}_2|$, $\mathcal{Z}_1 \subseteq \mathcal{M}_1$ and $|\mathcal{Z}_2 \cap \mathcal{M}_1| = 0$, we have that $|(\mathcal{M}_1 \setminus \mathcal{Z}_1) \cup \mathcal{Z}_2| = |\mathcal{M}_1| = |\mathcal{M}_2|$. Let $\sigma$ be an $(s-1)$-space of $Q^+(2d - 1, q)$ and let $\perp$ be the orthogonal polarity induced by the quadric. The projective subspace $\sigma^\perp$ meets the quadric $Q^+(2d-1, q)$ in a cone that has as vertex $\sigma$ and as base a $Q^+(2d-2s-1, q)$. Let $\rho$ be a $(d-s-2)$-space of $Q^+(2d-1, q)$. There are two cases to consider according as $\rho \not\subseteq \sigma^{\perp}$ or $\rho \subseteq \sigma^{\perp}$, respectively. In the former case, there are no generators of $\mathcal{Z}_1 \cup \mathcal{Z}_2$ containing $\rho$. Hence $(\mathcal{M}_1 \setminus \mathcal{Z}_1) \cup \mathcal{Z}_2$ and $\mathcal{M}_1$ have the same number of elements through $\rho$. In the latter case, if $\rho \cap \sigma$ is a $(t-1)$-space, then $\langle \rho, \sigma \rangle$ is a $(d-t-2)$-space and $\langle \rho, \sigma \rangle^{\perp}$ meets the quadric $Q^+(2d-1, q)$ in a cone that has as vertex $\langle \rho, \sigma \rangle$ and as base a $Q^+(2t+1, q)$. Hence there is a bijection between the  generators of $\mathcal{Z}_1 \cup \mathcal{Z}_2$ passing through $\rho$ and the generators of $Q^+(2t+1, q)$; furthermore the members of $\mathcal{Z}_1$ containing $\rho$ and those of $\mathcal{Z}_2$ containing $\rho$ are equal in number. Note that the generators of $\mathcal{M}_1$ passing through $\rho$ and $\sigma$ are elements of $\mathcal{Z}_1$. Therefore we may conclude that the number of generators of $(\mathcal{M}_1 \setminus \mathcal{Z}_1) \cup \mathcal{Z}_2$ through $\rho$ is the same as the number of generators of $\mathcal{M}_1$ through $\rho$.
\end{proof}
In \cite{BLL} the authors proved that $W(5, q)$, (and hence $H(5, q^2)$) $q \in\{3, 5\}$, has not hemisystems w.r.t. lines. Moreover they exhibited some examples of hemisystems w.r.t. lines of $Q(6, q)$, $q \in\{3, 5,7,11\}$, \cite[Table 1]{BLL},\cite[Remark 5.3.1]{LJ}. Both results were found with the aid of a computer \footnote{The hemisystem of $Q(6, 3)$ with respect to lines was first discovered by B. De Bruyn and F. Vanhove. Their computer result also covered the uniqueness of this hemisystem as well as nonexistence of such a hemisystem in the polar space $W(5, 3)$. The result was never published but announced during (conference) talks (e.g. Combinatorics Seminar talk by Vanhove in Sendai, only a few weeks prior to his untimely death).}. As already mentioned in the introduction, if $k \geq 2$ and the polar space is not a hyperbolic quadric, up to date, these are the only known examples of regular systems w.r.t. $(k-1)$-spaces. In the rest of the section we recall the state of the art about regular systems of polar spaces of rank two and then we apply Theorem \ref{chain} to obtain regular systems in polar spaces of higher rank. See Chapter \ref{ch3} for considerations about $Q(4,q)$ and $W(3,q)$. The point line dual of $H(3,q^2)$ is $Q^-(5,q)$, the elliptic quadric of $PG(5,q)$. Therefore $m$-regular systems of $Q^-(5,q)$ and $m$-ovoids of $H(3,q^2)$ are equivalent objects and $m$-regular systems of $H(3,q^2)$ and $m$-ovoids of $Q^-(5,q)$ are equivalent objects. From Chapter \ref{ch2} we know that regular systems of $H(3,q^2)$ are hemisystems, and so we get $\frac{q+1}{2}$-ovoids of $Q^{-}(5,q)$. On the other hand, $m$-ovoids of $H(3,q^2)$ exist for all possible $m$. Indeed, $H(3,q^2)$ admits a \textit{fan}, i.e., a partition of the pointset into ovoids. We briefly mention the construction of a fan of $H(3,q^2)$ due to Brouwer and Wilbrink \cite{BW}: let $\perp$ be the unitary polarity of $PG(3,q^2)$ induced by $H(3,q^2)$, let $P$ be a point of $PG(3,q^2) \setminus H(3,q^2)$, let $X$ be a point of $P^\perp \cap H(3,q^2)$ and let $t$ be the unique tangent line at $X$ contained in $P^\perp$. Put $\mathcal{O}_X = P^\perp \cap H(3,q^2)$, and $\mathcal{O}_Y = ((Y^\perp \cap H(3,q^2)) \setminus P^\perp) \cup (PY \cap H(3,q^2))$ for $Y \in t \setminus \{ X \}$. Then each $\mathcal{O}_Z$, is an ovoid, and $\{ \mathcal{O}_Z |Z \in t \}$ is a fan. Finally, up to date, no infinite family of $m$-regular systems of $H(4,q^2)$ is known, although several examples have been found when $q \in \{2,3,4,5\}$ \cite{BDS}. Applying Theorem \ref{chain}, we obtain the following results.
 \begin{cor}
  If $q$ is even, then $Q^+(5,q)$ has a $2m$-regular system w.r.t. points for all $m$, $1 \leq m \leq q+1$. If $q$ is odd, then $Q^+(5,q)$ has a $2m$-regular system w.r.t. points for all even $m$, $1 \le m \le q+1$.
 \end{cor}
\begin{cor}
 $Q(6,q)$ has an $m(q+1)$-regular system w.r.t. points for all $m$, $1 \leq m \leq q^2+1$.
\end{cor}

\begin{cor}
 $Q^+(7,q)$ has a $2m(q+1)$-regular system w.r.t. points for all $m$, $1 \leq m \leq q^2+1$.
\end{cor}

\begin{cor}
 $H(5,q^2)$ has an $m(q+1)$-regular system w.r.t. points for $q=2$ and $3 \leq m \leq 6$, $q=3$ and $3 \leq m \leq 25$, $q=4$ and $5 \leq m \leq 60$, $q=5$ and $m=30$.
\end{cor}

\chapter{Strongly regular graphs}\label{ch5}
 A relevant family of distance regular graph consists of  \textit{strongly regular graphs}. In this chapter we treat some properties of strongly regular graphs. The approach used is that introduced in \cite{Bishnoi}. For more details on strongly regular graphs, see also \cite{Bose, BH, brvan}. Final sections of the chapter provide a few families of strongly regular graphs arising from polar spaces and incidence structures, originally described in the papers \cite{VS5, VS4}.

 A strongly regular graph $G$ with parameters $(v,k,\lambda,\mu)$ is a graph with $v$ vertices, each vertex lies on $k$ edges, any two adjacent vertices have $\lambda$ common neighbours and any two non-adjacent vertices
   have $\mu$ common neighbours.

   \begin{lem}
   \label{fundamentalsrg}
    $k(k-\lambda-1)=\mu(v-k-1)$.
   \end{lem}

   \begin{proof}
    Consider in $G$ all the configurations of three vertices $x$, $y$ and $z$, where $y$ is adjacent to both $x$ and $z$, and $x$ and $z$ are not
    adjacent. Fixing $x$, we may double count all those configurations, starting respectively from $y$ and from $z$. In fact, we have $k$ choice for a
    neighbour $y$ of $x$, and $z$ now may be chosen in the $k-\lambda-1$ neighbours of $y$ different than the $\lambda$ common neighbours with $x$, and
    excluding $x$ itself. On the other hand, $z$ may be chosen in all the $v-k-1$ vertices not adjacent to $x$, while $y$ is one of the $\mu$ common
    neighbours of $x$ and $z$.
   \end{proof}

   \begin{lem}
    The complementary graph $\overline{G}$ of a strongly regular graph $G$ with parameters $(v,k,\lambda,\mu)$ is also strongly regular, with parameters\\ $(v,v-k-1,v-2k+\mu-2,v-2k+\lambda)$.
   \end{lem}

   \begin{proof}
    It can be seen that $|\overline{G}|=|G|=v$ and every vertex of $G$ has $k'=n-k-1$ non-neighbours. Two non-adjacent vertices $x$ and $y$ have $\mu$ common neighbours, and so the number of vertices adjacent to at least one of $x$ and $y$ is $2k-\mu$, and therefore the common non-neighbours (i.e. the common neighbours of two adjacent vertices in $\overline{G}$) are $\lambda'=v-2-(2k-\mu)$. On the other hand, two adjacent vertices $x$ and $y$ have $\lambda$ common neighbours, and so the number of vertices adjacent to at least one of $x$ and $y$ is $2k-\lambda-2$, where the last 2 vertices excluded are $x$ and $y$ themselves. Therefore the common non-neighbours (i.e. the common neighbours of two non-adjacent vertices in $\overline{G}$) are $\mu'=v-2-(2k-\lambda-2)=v-2k+\lambda$.
   \end{proof}

   \begin{lem}
    \label{eqnA}
    Let $A$ be the adjacency matrix of a strongly regular graph $G$ with parameters $(v,k,\lambda,\mu)$. Let $I$ be the identity matrix, and $J$ be the
    all-one matrix, both of size $v$. Then:
    \begin{enumerate}
      \item $AJ=JA=kJ$
      \item $A^{2}=kI+\lambda A+\mu(J-I-A)$
    \end{enumerate}
   \end{lem}

   \begin{proof}
    The first claim comes from to the $k$-regularity of $G$ together with Lemma \ref{kregeig}. Now consider the matrix $A^{2}$, in which the $ij$-entries are $(A^{2})_{ij}=A_{i}A^{j}$. Since $A$ is symmetric, the diagonal elements are the $k$ entries of a row of $A$: in fact $(A^{2})_{ii}=A_{i}A^{i}$. The entries corresponding to the ones in $A$, i.e. where the vertices $v_{i}$ and $v_{j}$ are adjacent, are equal to $\lambda$, since $A_{i}$ and $A_{j}$ have $\lambda$ common non-zero entries, i.e. the $\lambda$ common neighbours of $v_{i}$ and $v_{j}$. In an analogue way all other entries (corresponding to $J-I-A$) are equal to $\mu$, since now $A_{i}$ and $A_{j}$ have $\mu$ common non-zero entries.
   \end{proof}
   These two conditions on the matrix $A$ determine the full spectrum of strongly regular graphs, that only depends on the parameters $v$, $k$, $\lambda$ and $\mu.$

   \begin{thm}
    A strongly regular graph $G$ with parameters $(v,k,\lambda,\mu)$ has exactly three eigenvalues: $k$, $\theta_{1}$ and $\theta_{2}$ of multiplicity, respectively, $1$, $m_{1}$ and $m_{2}$, where:
    $$\theta_{1}=\frac{1}{2}\big[(\lambda-\mu)+\sqrt{(\lambda-\mu)^{2}+4(k-\mu)}\big],$$
    $$\theta_{2}=\frac{1}{2}\big[(\lambda-\mu)-\sqrt{(\lambda-\mu)^{2}+4(k-\mu)}\big],$$
    $$m_{1}=\frac{1}{2}\Big[(v-1)-\frac{2k-(v-1)(\lambda-\mu)}{\sqrt{(\lambda-\mu)^{2}+4(k-\mu)}}\Big],$$
    $$m_{2}=\frac{1}{2}\Big[(v-1)+\frac{2k-(v-1)(\lambda-\mu)}{\sqrt{(\lambda-\mu)^{2}+4(k-\mu)}}\Big].$$
   \end{thm}
We write the spectrum as $k,\theta_{1}^{m_{1}},\theta_{2}^{m_{2}}$.
   \begin{proof}
    Let $v$ be an eigenvector with eigenvalue $\rho$ different than $k$. $Jv=0$, in fact from Lemma \ref{ortheig} two eigenvectors with different eigenvalues are orthogonal
    $$Jv=\begin{pmatrix}
           1 & \ldots & 1 \\
           \vdots & \ddots & \vdots \\
           1 &\ldots & 1
         \end{pmatrix}v=\underline{0}.$$
   Now consider the action of the operator $A^{2}$ on $v$, where $A^{2}v=A\rho v=\rho^{2}v$. From Lemma \ref{eqnA} we also have
   $$A^{2}v=[kI+\lambda A+\mu(J-I-A)]v=kv+\lambda Av+\mu Jv-\mu v-\mu Av=[(\lambda-\mu)\rho+(k-\mu)]v,$$
   i.e. the eigenvalues different then $k$ must be roots of the quadratic equation
   $$\rho^{2}=(\lambda-\mu)\rho+(k-\mu).$$
   Computing the solutions we get the form of $\theta_{1}$ and $\theta_{2}$. The respective multiplicities are given by the following system of equations:
   \begin{equation}
    \begin{cases}
     k+\theta_{1}m_{1}+\theta_{2}m_{2}=tr(A)=0 \\
     1+m_{1}+m_{2}=v.
    \end{cases}
    \end{equation}
   \end{proof}

   \begin{cor}
    Given a strongly regular graph $G$ with parameters $(v,k,\lambda,\mu)$,
    $$m_{1,2}=\frac{1}{2}\Big[(v-1)\mp \frac{2k-(v-1)(\lambda-\mu)}{\sqrt{(\lambda-\mu)^{2}+4(k-\mu)}}\Big]$$
    must be integers.
   \end{cor}
 Moreover, the parameters $(v,k,\lambda,\mu)$ may be derived from the spectrum of a strongly regular graph, since
 \begin{equation}
  \begin{cases}
   \lambda=k+\theta_{1}+\theta_{2}+\theta_{1}\theta_{2}\\
   \mu=k+\theta_{1}\theta_{2}\\
   v=1+m_{1}+m_{2}.
  \end{cases}
 \end{equation}

Since spectrum and parameters of strongly regular graphs are equivalent, Problem \ref{cospectralgraph} may be restated as follows in terms of strongly regular graphs.
\begin{prob}
 \begin{itemize}
  \item Characterize strongly regular graphs \textit{determined by their parameters}.
  \item Study families of strongly regular graphs with the same parameters which are not isomorphic.
 \end{itemize}
\end{prob}

\section{Linear representation graph and two-weight codes}\label{sec51}
Let fix a \textit{hyperplane at infinity} $\pi_{\infty}$ in the projective space $PG(n,q)$. We call \textit{affine space} the set $AG(n,q)=PG(n,q)\setminus\pi_{\infty}$, and its points are called \textit{affine points}.
\begin{defn}
  Take now $PG(n,q)$ embedded as hyperplane at infinity in $PG(n+1,q)$, and take a set of points $\mathcal{A}$ in $PG(n,q)$. The \textit{linear representation graph} $\Gamma_{\mathcal{A}}$ is the graph whose vertex set is given by the affine points in $PG(n+1,q)\setminus PG(n,q)$, and two vertices $x$ and $y$ are adjacent if the line $\langle x,y\rangle$ meets the hyperplane at infinity in a point of $\mathcal{A}$.
\end{defn}
Let $Q^{-}(5,q)$ be the dual of $H(3,q^2)$, image via the Klein correspondence. Then a hemisystem of $H(3,q^2)$ corresponds to a set $\mathcal{O}$ consisting of $\frac{(q^3 +1)(q + 1)}{2}$ points of $Q^{-}(5,q)$ such that every line of $Q^{-}(5,q)$ has $\frac{q+1}{2}$ points in common with $\mathcal{O}$. We call such a set $\Big(\frac{q+1}{2}\Big)$-ovoid. More generally, $m$-regular systems and $m$-ovoids are dual objects. A hyperplane of $PG(5,q)$ meets $\mathcal{O}$ in either $\frac{(q+1)(q^2+1)}{2}$ or $\frac{q^3-q^2+q+1}{2}$ points, i.e. $\mathcal{O}$ it is called projective $(\frac{(q^{3}+1)(q+1)}{2},6,\frac{(q^{2}+1)(q+1)}{2},\frac{q^{3}-q^{2}+q+1}{2})$-set, since a projective $(n,k,h_1,h_2)$-set is a set of $n$ points in $PG(k-1,q)$ meeting hyperplanes in $h_1$ or $h_2$ points.
\begin{prop}\cite[Theorem 11]{bamberg2007tight}
 The linear representation graph $\Gamma_{\mathcal{O}}$ is a strongly regular graph with parameters
 $$v=q^{6}$$
 $$k=\frac{(q^{3}+1)(q^{2}-1)}{2}$$
 $$\lambda=\frac{q^{4}-5}{4}$$
 $$\mu=\frac{q^{4}-1}{4}.$$
\end{prop}
An $[n,k,d]_{q}$-linear code $C$ is a $k$-dimensional subspace of $V(n,q)$, in which two different vectors differ in at most $d$ entries, i.e. $d$ is the minimum Hamming distance. Vectors in $C$ are called \textit{codewords}, and the weight $w(v)$ of $v\in C$ is the number of non-zero entries in $v$. A \textit{two-weight code} is an $[n,k,d]$-linear code $C$ such that $|\{W|\exists v\in C\setminus\{0\}\hspace{2 mm} w(v)=W\}|=2$. For more details on linear codes see Appendix \ref{apB}. Here we use a result of \cite{calderbank1986geometry} that allows us to construct two-weight codes from hemisystems of $H(3,q^2)$.

\begin{thm}\cite[Theorems 3.1 and 3.2]{calderbank1986geometry}\label{CK1}
 Let $\Omega$ a subset of $\mathbb{F}_{q}^{k}$, with $\Omega=-\Omega$ and $0\not\in \Omega$, and define $G(\Omega)$ to be the graph whose vertices are the vectors of $\mathbb{F}_{q}^{k}$, and two vertices are adjacent if and only if their difference is in $\Omega$. If $\Sigma=\{\langle v_i\rangle\,i=1,\ldots,n\}$ is a proper subset of $PG(k-1,q)$ that spans $PG(k-1,q)$, then the following are equivalent:
 \begin{enumerate}
  \item $G(\Omega)$ is a strongly regular graph;
  \item $\Sigma$ is a projective $(n,k,n-w_{1},n-w_{2})$-set for some $w_{1}$ and $w_{2}$;
  \item the linear code $ C=\{(x\cdot v_1, x\cdot v_2, \ldots, x\cdot v_n)\,:\, x\in \mathbb{F}_{q}^{k}\}$ (here $x\cdot v$ is the classical scalar product) is an $[n,k]$-linear two-weight code with weights $w_{1}$ and $w_{2}$.
 \end{enumerate}
\end{thm}
\begin{proof}
 We give the proof of $2\iff3$. Take $x\in \mathbb{F}_{q}^{k}\setminus\{0\}$. Then consider the orthogonal hyperplane w.r.t. the scalar product: $x^{\perp}=\{y\in \mathbb{F}_{q}^{k}|x\cdot y=0\}$, then the weight of the codeword $(x\cdot y_1,\ldots,x\cdot y_k)$ is $n-|x^{\perp}\cap\{y_{1},\ldots,y_{k}\}|$ Since we have only two intersection orders with hyperplanes, the codes have only two weights, and vice versa.
\end{proof}
\begin{cor}
 Let $q$ be as in the hypothesis of Theorem \ref{th:1}. It exists a family of $[\frac{(q^{3}+1)(q+1)}{2},6]_q$-linear two-weight codes with weights $w_{1}=\frac{q^{2}(q^{2}-1)}{2}$ and $w_{2}=\frac{q^{2}(q^{2}+1)}{2}$. The minimum distance between two codewords is $\frac{q^{2}(q^{2}-1)}{2}$.
\end{cor}

In the case $q=5$ we get a $[375,6,300]_{5}$-linear two-weight code with weights $w_{1}=300$ and $w_{2}=325$.

\section{Strongly regular graph on the lines of $m$-regular systems of $H(3,q^{2})$}\label{sec52}
J. A. Thas in \cite{4} constructed a strongly regular graph $\Gamma$ from any $m$-regular system $\mathcal{S}$ on the Hermitian surface $\mathcal{H}(3,q^{2})$, $q$ odd. Vertices of $\Gamma$ are the lines on the surface, not contained in $\mathcal{S}$. Two vertices are joint by an edge if the lines are incident in $PG(3,q^{2})$.
\begin{prop}
 \label{ThasH}
 $\Gamma$ is a strongly regular graph with parameters $((q^{3}+1)(q+1-m),(q^{2}+1)(q-m),q-1-m,q^{2}+1-m(q+1))$.
\end{prop}
\begin{proof}
 The number of vertices of $\Gamma$ is $v=(q^{3}+1)(q+1-m)$, while every such a line has $q-m$ neighbours in on each point, i.e. $k=(q^{2}+1)(q-m)$. Since $H(3,q^{2})$ is a generalized quadrangle, it does not contain triangles, i.e. two incident lines may have common neighbours only in the intersection point, so $\lambda=q-1-m$. The $m$-regular system on the Hermitian surface also provides an $m$-ovoid $\mathcal{O}$ on the elliptic quadric $Q^{-}(5,q)$, which is the image of $H(3,q^{2})$ via the Klein correspondence. For two points $x$ and $y$ on $\mathcal{O}$, there are $q^{2}+1$ points collinear with both $x$ and $y$. Those $q^{2}+1$ points form an elliptic quadric $Q^{-}(3,q)$ in $PG(3,q)$ such that $\mu$ of them are on $\mathcal{O}$. Now take the $q+1$ hyperplanes $PG(4,q)$ containing $PG(3,q)$, and their intersection with the $m(q^{3}+1)$ points of the $m$-ovoid. It turns out
 $$m(q^{3}+1)=(q+1)(m(q^{2}+1)-(q^{2}+1-\mu))-q^{2}+1-\mu,$$
 i.e. $\mu=q^{2}+1-m(q+1)$.
\end{proof}
Theorem \ref{SegreH} is a corollary of Proposition \ref{ThasH} since parameters $((q^{3}+1)(q+1-m),(q^{2}+1)(q-m),q-1-m,q^{2}+1-m(q+1))$ contradict Lemma \ref{fundamentalsrg}, except from the case $m=\frac{q+1}{2}$. It turns out that each $m$-regular system on the Hermitian surface needs to be an hemisystem, and each of the two hemisystems defined in the $H(3,q^{2})$ define a strongly regular graph. Since $m=\frac{q+1}{2}$, the parameters are $v=\frac{(q^{3}+1)(q+1)}{2}$, $k=\frac{(q^{2}+1)(q-1)}{2}$, $\lambda=\frac{q-3}{2}$, $\mu=\frac{(q-1)^{2}}{2}$. Therefore the spectrum of $\Gamma$ is determined: the first eigenvalue is $k$, of multiplicity 1, and we have two other eigenvalues:
$$\theta_{1}=\frac{1}{2}\big[(\lambda-\mu)+\sqrt{(\lambda-\mu)^{2}+4(k-\mu)}\big]=q-1,$$
$$\theta_{2}=\frac{1}{2}\big[(\lambda-\mu)-\sqrt{(\lambda-\mu)^{2}+4(k-\mu)}\big]=\frac{-q^{2}+q-2}{2},$$
of multiplicity, respectively,
$$m_{1}=\frac{1}{2}\Big[(v-1)-\frac{2k+(v-1)(\lambda-\mu)}{\sqrt{(\lambda-\mu)^{2}+4(k-\mu)}}\Big]=\frac{q^{4}-q^{3}+2q^{2}-q+1}{2},$$
$$m_{2}=\frac{1}{2}\Big[(v-1)+\frac{2k+(v-1)(\lambda-\mu)}{\sqrt{(\lambda-\mu)^{2}+4(k-\mu)}}\Big]=(q^{2}+1)(q-1)=2k.$$
The hemisystems on the Hermitian surface $H(3,p^{2})$, for $p=1+4a^{2}$, constructed in Chapter \ref{ch2} produce strongly regular graphs $\Gamma$ with the above parameters for $q=p$. We point out that, in the smallest case $p=5$, the graph $\Gamma$, arising from the $3.A_{7}$-stabilized hemisystem constructed in \cite{cossidente2005hemisystems} and in Chapter \ref{ch2}, has parameters $(378,52,1,8)$ and spectrum $52, 4^{273}, -11^{104}$, the same spectrum of the Cossidente-Penttila strongly regular graph, i.e. the graph arising from the Cossidente-Penttila hemisystem stabilized by a group isomorphic to  $PSL(2,25)$. We now want to prove the following theorem.
\begin{thm}\label{hemicospectral}
 The two strongly regular graphs with parameters $(378,52,1,8)$ arising from the Cossidente-Penttila hemisystem, and from the $3.A_{7}$-stabilized hemisystem are cospectral but not isomorphic.
\end{thm}

To prove Theorem \ref{hemicospectral} we show that the two graphs have different automorphism groups. A \textit{graph automorphism} is a permutation on vertices which preserves adjacencies, while an \textit{automorphism} of an hemisystem is a permutation on generators preserving the hemisystem structure.

\begin{lem}\label{hemiauto}
 Let $\mathcal{E}$ be an hemisystem of the Hermitian surface $H(3,q^{2})$, $q>3$. Then the automorphism group of the graph $\Gamma_{\mathcal{E}}$ is isomorphic to the automorphism group of $\mathcal{E}$.
\end{lem}
\begin{proof}
 $Aut(\mathcal{E})\leq Aut(\Gamma_{\mathcal{E}})$ since an automorphism of $\mathcal{E}$ obviously stabilizes the arising graph. We need only to prove that an automorphism of the graph stabilizes the hemisystem. Since $H(3,q^{2})$ does not contain triangles, thus maximal cliques of the graphs are made of the $\frac{q+1}{2}$ lines through a point. An automorphism of $\Gamma_{\mathcal{E}}$ permutes all maximal cliques, i.e. must preserve the hemisystem structure.
\end{proof}

The proof of Theorem \ref{hemicospectral} arises naturally from Lemma \ref{hemiauto}, since the hemisystems have different automorphism groups.

\section{Collinearity graph of $\mathcal{P}_{d,e}$}\label{sec53}
An other example of graph arising from polar spaces is the collinearity graph.
\begin{defn}
 The \textit{collinearity graph} $\mathcal{C}(\mathcal{P}_{d,e})$ of a polar space $\mathcal{P}_{d,e}$ is the graph having the pointset of $\mathcal{P}_{d,e}$ as vertex set, and in which two vertices $x$ and $y$ are adjacent if and only if the line $\langle x,y\rangle$ is totally isotropic.
\end{defn}
Since (partial) ovoids of $\mathcal{P}_{d,e}$ are equivalent to cocliques in $\mathcal{C}(\mathcal{P}_{d,e})$, new results on partial ovoids give upper bound on the independence number of $\mathcal{C}(\mathcal{P}_{d,e})$.
\begin{prop}\cite[Theorem 2.2.12]{brvan}
 The collinearity graph $\mathcal{C}(\mathcal{P}_{d,e})$ is a strongly regular graph with parameters
 $$v=(q^{d+e-1}+1)\frac{q^{d}-1}{q-1}$$
 $$k=q(q^{d+e-2}+1)\frac{q^{d-1}-1}{q-1}$$
 $$\lambda=q-1+q^{2}(q^{d+e-3}+1)\frac{q^{d-2}-1}{q-1}$$
 $$\mu=(q^{d+e-2}+1)\frac{q^{d-1}-1}{q-1}=\frac{k}{q}.$$
 The spectrum is $k,\theta_1^{m_1},\theta_2^{m_2}$ where
 $$\theta_1=-q^{d+e-2}-1$$
 $$\theta_2=q^{d-1}-1$$
 $$m_1=\frac{q^{e+1}(q^{d}-1)(q^{d+e-2}+1)}{(q-1)(q+q^{e})}$$
 $$m_2=\frac{q^{2}(q^{d-1}-1)(q^{d+e-1}+1)}{(q-1)(q+q^{e})}$$
\end{prop}
From Theorem \ref{hoffman0}, the independence number of the collinearity graph must be bounded by
$$\alpha(\mathcal{C}(\mathcal{P}_{d,e}))\leq-\frac{v\theta_1}{k-\theta_1}=q^{d+e-1}+1.$$
We recall that $\mathcal{O}_{d}=q^{d+e-1}+1 $ is the ovoid number of $\mathcal{P}_{d,e}$. Now take the polar spaces $W(3,q)$ and $H(4,q^{2})$. Since generators are lines, in both cases a coclique of the collinearity graph is a partial ovoid of the polar space, i.e. the independence number of the graph is the size of the largest maximal partial ovoid. The following (new) result arises from Theorems \ref{hoffman0} and \ref{partial-symp3} and Corollary \ref{ovoidiH}.
\begin{thm}
 \begin{itemize}
  \item If $q$ be an odd square with $q \not\equiv 0 \pmod{3}$,
        \begin{equation}
         \frac{q^{\frac{3}{2}} +3q - q^{\frac{1}{2}} +3}{3}<\alpha(\mathcal{C}(W(3,q))<q^{2}+1.
        \end{equation}
  \item If $q$ is even,
        \begin{equation}
         q^4+1<\alpha(\mathcal{C}(H(4,q^2))<q^{5}+1.
        \end{equation}
  \item If $q$ is an odd square and $q \not\equiv 0 \pmod{3}$,
        \begin{equation}
         \frac{q^{\frac{7}{2}}+3q^3-q^{\frac{5}{2}}+3q^2}{3}<\alpha(\mathcal{C}(H(4,q^2))<q^{5}+1.
        \end{equation}
 \end{itemize}
\end{thm}

\chapter{Graphs cospectral with $NU(n+1, q^2)$, $n \neq 3$}\label{ch6}
The last example of strongly regular graph arising from polar spaces we want to introduce is the \textit{tangent graph} $NU(n+1,q^2)$. The tangent graph $NU(n+1,q^{2})$ is the strongly regular with vertex set $PG(n,q^{2})\setminus H(n,q^2)$ where two vertices are adjacent if they lie on the same tangent line to the Hermitian variety $H(n,q^2)$. Set $\varepsilon=(-1)^{n+1}$ and $r=q^{2}-q-1$. Then $NU(n+1,q^{2})$ has the following parameters:
$$v=\frac{q^{n}(q^{n+1}-\varepsilon)}{q+1}$$
$$k=(q^{n}+\varepsilon)(q^{n-1}-\varepsilon)$$
$$\lambda=q^{2n-3}(q+1)-\varepsilon q^{n-1}(q-1)-2$$
$$\mu=q^{n-2}(q+1)(q^{n-1}-\varepsilon).$$
Furthermore, its complementary graph $\overline{NU(n+1,q^{2})}$ is also strongly regular and has parameters:\\
$$v'=\frac{q^{n}(q^{n+1}-\varepsilon)}{q+1}$$
$$k'=\frac{q^{n-1}r(q^{n}+\varepsilon)}{q+1}$$
$$\lambda'=\mu'+\varepsilon q^{n-2}r-\varepsilon q^{n-1}$$
$$\mu'=\frac{q^{n-1}r(q^{n-2}r+\varepsilon)}{q+1}.$$
For further details and information about $NU(n+1,q^2)$ see \cite{brvan,Chakravarti}. A variation of Godsil-McKay switching on strongly regular graphs was described by Wang, Qiu, and Hu in \cite{WQH}. In \cite{IM}, by using a simplified version of the Wang-Qiu-Hu switching, the authors exhibited a strongly regular graph having the same parameters as $NU(n+1, 4)$, $n \geq 5$, but not isomorphic to it. Moreover they asked whether or not the graph $NU(n+1, q^2)$, is determined by its spectrum. In this paper we show that the graph $NU(n + 1, q^2)$, $n \neq 3$, is not determined by its spectrum. In design theory, an \textit{unital} $\mathcal{U}$ is defined as a $2-(a^{3}+1,a+1,1)$ design, $a\geq3$, i.e. a set of $a^{3}+1$ points arranged into blocks of size $a+1$, such that each pair of distinct points is contained in exactly one block. See more details on designs in Appendix \ref{apC}. There exist unitals which are embedded in a projective plane of order $a^{2}$. The incidence structure made of points and secants of the Hermitian curve $H(2,q^{2})$ in $PG(2,q^{2})$ is called the \textit{classical unital}
 An \textit{embedded unital} $\mathcal{U}$ of a finite projective plane $\Pi$ of order $q^2$ is a set of $q^3+1$ points such that every line of $\Pi$ meets $\mathcal{U}$ in either $1$ or $q+1$ points. Besides the classical one, there are other known embedded unitals in $PG(2, q^2)$ and they arise from a construction due to F. Buekenhout, see \cite{BE, B, M}. Let $\mathcal{U}$ be a unital of $PG(2, q^2)$ and let $\Gamma_{\mathcal{U}}$ be the graph whose vertices are the points of $PG(2, q^2) \setminus \mathcal{U}$ and two vertices $P_1, P_2$ are adjacent if the line joining $P_1$ and $P_2$ contains exactly one point of $\mathcal{U}$. The graph $\Gamma_{\mathcal{U}}$ is strongly regular with the same parameters as $NU(3, q^2)$. Here we observe the (well-known) fact that if $\mathcal{U}$ is a non-classical unital, then $\Gamma_{\mathcal{U}}$ is not isomorphic to $NU(3, q^2)$. Moreover, in $PG(n, q^2)$, $n \geq 4$, by applying a simplified version of the Wang-Qiu-Hu switching to the graph $NU(n+1, q^2)$, we obtain two graphs $G'_n$ and $G''_n$, which are strongly regular and have the same parameters as $NU(n+1, q^2)$. Moreover we show that $G'_n$ (if $q > 2$) and $G''_n$ are not isomorphic to $NU(n+1, q^2)$. All results in the chapter belong to the paper \cite{VS1}.

\section{Preliminary results}\label{sec61}
For a point $P$ of $PG(2, q^2)$, its polar line meets $H(2, q^2)$ exactly in $P$ or in the $q+1$ points of a Baer subline, according as $P$ belongs to $H(2, q^2)$ or not. It is well known that $H(2, q^2)$ contains no \textit{O’Nan configuration} \cite[p. 507]{O}, which is a configuration consisting of four distinct lines intersecting in six distinct points of $H(2, q^2)$. Dually, $H(2, q^2)$ cannot contain a \textit{dual O'Nan configuration}, that is a configuration formed by four points of $PG(2, q^2) \setminus H(2, q^2)$ no three on a line, such that the six lines connecting two of them are tangent lines. In $PG(2, q^2)$, a \textit{Hermitian pencil of lines} is a cone having as vertex a point $V$ and as base a Baer subline of a line $\ell$, where $V \notin \ell$. See \cite{BE} for more details. A plane of $PG(4, q^2)$ meets $H(4, q^2)$ in a line, a Hermitian pencil of lines or a non-degenerate Hermitian curve, according as its polar line has $q^2+1$, $1$ or $q+1$ points in common with $H(4, q^2)$, respectively.
\begin{lem}\label{HermCurve1}
 In $PG(2, q^2)$, let $P_1, P_2, P_3$ be three points on a line $\gamma$ tangent to a non-degenerate Hermitian curve $H(2, q^2)$ at the point $T$, where $T \neq P_i$, $i = 1,2,3$. Let $s$ be the unique Baer subline of $\gamma$ containing $P_1, P_2, P_3$.
 \begin{itemize}
  \item If $T \in s$, then there is no point $P \in PG(2, q^2) \setminus (H(2, q^2) \cup \gamma)$ such that the lines $P P_i$, $i= 1,2,3$, are tangent to $H(2, q^2)$.
  \item If $T \notin s$, then there are $q$ points $P \in PG(2, q^2) \setminus (H(2, q^2) \cup \gamma )$ such that the lines $P P_i$, $i= 1,2,3$, are tangent to $H(2, q^2)$.
 \end{itemize}
\end{lem}
\begin{proof}
 Let $\perp$ be the unitary polarity of $PG(2, q^2)$ defining $H(2, q^2)$. Assume first that $T \in s$. Suppose by contradiction that there is a point $P \in PG(2, q^2) \setminus ( H(2, q^2) \cup \gamma )$ such that the lines $P P_i$, $i= 1,2,3$, are tangent to $H(2, q^2)$. Then by projecting the $q+1$ points of $P^{\perp} \cap H(2, q^2)$ from $P$ onto $\gamma$ we obtain a Baer subline of $\gamma$ containing $P_1, P_2, P_3$, i.e., the Baer subline $s$. Since $T \in s$, we have that the line $P T$ has to be tangent to $H(2, q^2)$ at $T$. Hence $P T = \gamma$, contradicting the fact that $P \notin \gamma$. Assume now that $T \notin s$. By using \cite[Theorem 15.3.11]{Hirschfeld2} is can be seen that there are $q^2(q-1)$ Baer sublines in $\gamma$ such that $T \notin s$, and these are permuted in a single orbit by the stabilizer $PGU(3, q)_\gamma$ of $\gamma$ in $PGU(3, q)$. On the other hand, $PGU(3, q)_\gamma$ permutes in a unique orbit the $q^3(q-1)$ points of $PG(2, q^2) \setminus ( H(2, q^2) \cup \gamma )$, see for instance \cite[Lemma 2.7]{O}. Consider the incidence structure whose point set $\mathcal{P}$ consists of the $q^4-q^3$ points of $PG(2, q^2) \setminus ( H(2, q^2) \cup \gamma)$ and whose block set $\mathcal{L}$ of the $q^3-q^2$ Baer sublines of $\gamma$, where a point $P \in \mathcal{P}$ is incident with a block $r \in \mathcal{L}$ if $r$ is obtained by projecting the $q+1$ points of $P^{\perp} \cap H(2, q^2)$ from $P$ onto $\gamma$. Since both, $\mathcal{P}$ and $\mathcal{L}$, are orbits of the same group, the considered incidence structure is a tactical configuration \cite[Lemma 1.45]{HP} and hence we have that a block of $\mathcal{L}$ is incident with $q$ points of $\mathcal{P}$. In other words, there are $q$ points of $PG(2, q^2) \setminus ( H(2, q^2) \cup \gamma )$ that are projected onto the same Baer subline of $\gamma$, as required.
\end{proof}
\begin{lem} \label{HermCurve0}
 In $PG(2, q^2)$, let $t, t'$ be two lines tangent to a non-degenerate Hermitian curve $H(2, q^2)$ and let $R = t \cap t'$. There is a bijection between the points $P$ of $t' \setminus ( H(2, q^2) \cup \{R\} )$ and the Baer sublines $s_{P}$ of $t$ passing through $R$, not containing $H(2, q^2) \cap t$ and such that a line joining $P$ with a point of $s_P$ is tangent to $H(2, q^2)$.
\end{lem}
\begin{proof}
 It is enough to note that the stabilizer of $t$ and $t'$ in $PGU(3, q)$ acts transitively on the $q^2-1$ points of $t' \setminus ( H(2, q^2) \cup \{R\} )$ and the $q^2-1$ Baer sublines of $t$ through $R$ not containing $H(2, q^2) \cap t$.
\end{proof}
\begin{lem} \label{HermCurve-1}
 In $PG(2, q^2)$, let $t$ be a line tangent to a non-degenerate Hermitian curve $H(2, q^2)$ at $T$ and let $u, u_1, u_2$ be three distinct points of $t$ such that the Baer subline of $t$ containing them does not pass through $T$. For a point $u' \in PG(2, q^2) \setminus H(2, q^2)$ such that $\langle u, u' \rangle$ is secant to $H(2, q^2)$ and $\langle u', u_1 \rangle$, $\langle u', u_2 \rangle$ are tangents, there are exactly two points $R_i$, $i = 1, 2$, of $PG(2, q^2) \setminus ( H(2, q^2) \cup t )$ such that the lines $\langle R_i, u \rangle, \langle R_i, u_1 \rangle, \langle R_i, u_2 \rangle, \langle R_i, u' \rangle$ are tangent to $H(2, q^2)$. Moreover $\langle R_i, u' \rangle \cap t = u_i$.
\end{lem}
\begin{proof}
 We may assume w.l.o.g. that $H(2, q^2): X_0^q X_1 + X_0 X_1^q + X_2^{q+1} = 0$, $T = (1, 0, 0)$, $u = (0, 0, 1)$, $u_i = (x_i, 0, 1)$, with $x_1^q x_2 - x_1 x_2^q \neq 0$, otherwise $T$ would belong to the Baer subline of $t$ containing $u, u_1, u_2$. Let $R = (c, 1, d)$. By \cite[Lemma 2.3]{Hirschfeld3}, the lines $\langle R, u \rangle$, $\langle R, u_i \rangle$, $i = 1,2$, are tangents if and only if $c + c^q = 0$ and $d = \frac{x_1^{q+1} x_2^q - x_1^q x_2^{q+1}}{x_1^q x_2 - x_1 x_2^q}$. Let $u' = (a, 1, b)$, where $a + a^q \neq 0$, since $\langle u, u' \rangle$ is secant, $a+a^q+b^{q+1} \neq 0$, since $u' \notin H(2, q^2)$ and $b - d = \frac{(a+a^q) (x_1^q - x_2^q)}{x_1^q x_2 - x_1 x_2^q}$, since $|\langle u', u_i \rangle \cap H(2, q^2)| = 1$, $i = 1, 2$. Thus the line $\langle R, u' \rangle$ is tangent if and only if $c$ satisfies $X^2 + X (a^q - a + b^qd - bd^q) + ad (b - d)^q + a^qd^q (b - d) - a^{q+1} = 0$. The two solutions of the previous quadratic equation are $\frac{a x_1^q (x_2 - x_1) + a^q x_1 (x_2^q - x_1^q)}{x_1^q x_2 - x_1 x_2^q}$ and $\frac{a x_2^q (x_2 - x_1) + a^q x_2 (x_2^q - x_1^q)}{x_1^q x_2 - x_1 x_2^q}$. Moreover it can be easily checked that $\langle u', R \rangle \cap t \in \{ u_1, u_2 \}$.
\end{proof}
\begin{lem}\label{HermCurve2}
 Let $\ell$ be a line of $PG(2, q^2)$ and let $P_1, P_2, P_3$ be three non-collinear points of $PG(2, q^2) \setminus \ell$. Then there are $q^2+2q$ Hermitian pencils of lines meeting $\ell$ in one point and containing $P_1, P_2, P_3$.
\end{lem}
\begin{proof}
 Let $\mathcal{C}$ be a Hermitian pencil of lines consisting of $q+1$ lines through a point of $\ell$ and containing the points $P_1, P_2, P_3$. If the lines $P_1 P_2, P_1 P_3, P_2 P_3$ are not lines of $\mathcal{C}$, then there are $q^2-2$ of such curves. Note that by projecting the Baer subline of $P_2 P_3$ containing $\ell \cap (P_2 P_3), P_2, P_3$ from $P_1$ onto $\ell$ and by projecting the Baer subline of $P_1 P_3$ containing $\ell \cap (P_1 P_3), P_1, P_3$ from $P_2$ onto $\ell$ and by projecting the Baer subline of $P_1 P_2$ containing $\ell \cap (P_1 P_2), P_1, P_2$ from $P_3$ onto $\ell$, we obtain the same Baer subline $s$ of $\ell$. If a Hermitian pencil of lines contains $\ell$ and the points $P_1, P_2, P_3$, but does not contain the lines $P_1 P_2, P_1 P_3, P_2 P_3$, then  its vertex is a point of $s \setminus \{ \ell \cap (P_1 P_2), \ell \cap (P_1 P_3), \ell \cap (P_2 P_3) \}$. Hence, among the $q^2-2$ Hermitian pencils of lines described above there are exactly $q-2$ that contain $\ell$. There are $q+1$ Hermitian pencils of lines consisting of $q+1$ lines through the point of $\ell \cap (P_1 P_2)$, containing the points $P_1, P_2, P_3$ and the line $P_1 P_2$. Exactly one of these contains $\ell$. Similarly for $P_1 P_3$ and $P_2 P_3$. Therefore there are $q^2-2 - (q-2) + 3(q+1) - 3 = q^2+2q$ Hermitian pencils of lines meeting $\ell$ in one point and containing $P_1, P_2, P_3$.
\end{proof}
\begin{lem}\label{HermCurve3}
 In $PG(2, q^2)$, let $H(2, q^2)$ be a non-degenerate Hermitian curve and let $P_1, P_2, P_3$ be three non--collinear points of $PG(2, q^2) \setminus H(2, q^2)$ such that the line $P_i P_j$ is tangent to $H(2, q^2)$, $1 \leq i < j \leq 3$. Then there are exactly $3q$ Hermitian pencils of lines meeting $H(2, q^2)$ in $q+1$ points and containing $P_1, P_2, P_3$.
\end{lem}
\begin{proof}
 Let $\perp$ be the unitary polarity of $PG(2, q^2)$ defining $H(2, q^2)$. Let $\mathcal{C}$ be a Hermitian pencil of lines such that $|\mathcal{C} \cap H(2, q^2)| = q+1$ and $P_i \in \mathcal{C}$, $i = 1,2,3$. Let $V$ be the vertex of $\mathcal{C}$. Then necessarily $V^{\perp} = \langle \mathcal{C} \cap H(2, q^2) \rangle$ and each of the $q+1$ lines of $\mathcal{C}$ is tangent to $H(2, q^2)$. If $V$ is one among $P_1, P_2, P_3$, then $\mathcal{C}$ is uniquely determined. If $V \notin \{P_1, P_2, P_3\}$, then $V$ must lie on one of the lines $P_1 P_2, P_1 P_3, P_2 P_3$, otherwise $H(2, q^2)$ would contain the dual of an O'Nan configuration, a contradiction. On the other hand the point $V$ can be chosen in $q-1$ ways on each of the three lines $P_1 P_2, P_1 P_3, P_2 P_3$ and for every choice of $V$ the curve $\mathcal{C}$ is uniquely determined.
\end{proof}
\begin{lem}\label{HermCurve4}
 In $PG(2, q^2)$, let $H$ and $H'$ be two non-degenerate Hermitian curves with associated polarity $\perp$ and $\perp'$, respectively. Let $|H \cap H'| \in \{1, q+1\}$ and let $\ell$ be the (unique) line of $PG(2, q^2)$ such that $\ell \cap H = \ell \cap H' = H \cap H'$. Thus $\ell^\perp = \ell^{\perp'}$ and if $P \in H \setminus H'$, then $P^{\perp'} \cap P^\perp \in \ell$.
\end{lem}
\begin{proof}
 We may assume w.l.o.g. that $H$ is given by $X_0^qX_1 + X_0X_1^q + X_2^{q+1} = 0$ and that $H'$ is given by $X_0^qX_1 + X_0X_1^q + \lambda X_2^{q+1} = 0$, $\lambda \in F_q \setminus \{0, 1\}$, if $|H \cap H'| = q+1$, or that $H'$ is given by $\lambda X_0^{q+1} + X_0^qX_1 + X_0X_1^q + X_2^{q+1} = 0$, $\lambda \in F_q \setminus \{0\}$, if $|H \cap H'| = 1$. For more details on the intersection of two non-degenerate Hermitian curves see \cite{BEUn, G, K}. In the former case $\ell$ is the line $X_2 = 0$ and $\ell^\perp = \ell^{\perp'} = (0, 0, 1)$. Let $P = (a, b, 1) \in H \setminus H'$, then $P^\perp \cap P^{\perp'} = (a^q, -b^q, 0) \in \ell$. In the latter case $\ell$ is the line $X_0 = 0$ and $\ell^\perp = \ell^{\perp'} = (0, 1, 0)$. Let $P = (1, a, b) \in H \setminus H'$, then $P^\perp \cap P^{\perp'} = (0, -b^q, 1) \in \ell$.
\end{proof}
\begin{prop}\label{HermCurve5}
 In $PG(2, q^2)$, let $H(2, q^2)$ be a non-degenerate Hermitian curve and let $P_1, P_2, P_3$ be three non-collinear points of $PG(2, q^2) \setminus H(2, q^2)$ such that the line $P_i P_j$ is tangent to $H(2, q^2)$, $1 \leq i < j \leq 3$. Then there are exactly $q^2-q+1$ non-degenerate Hermitian curves meeting $H(2, q^2)$ in one or $q+1$ points and containing $P_1, P_2, P_3$.
\end{prop}
\begin{proof}
 Assume that $P_1 = (1,0,0)$, $P_2 = (0,1,0)$, $P_3 = (0,0,1)$. If $H$ is a non-degenerate Hermitian curve containing $P_1, P_2, P_3$, then $H$ is given by
 $$a_{01} X_0 X_1^q + a_{01}^q X_0^qX_1 + a_{02} X_0X_2^q + a_{02}^q X_0^qX_2 + a_{12} X_1X_2^q + a_{12}^q X_1^qX_2 = 0,$$
 for some $a_{01}, a_{02}, a_{12} \in F_{q^2}$, with
 \begin{equation} \label{cond1}
  a_{01} a_{02}^q a_{12} + a_{01}^q a_{02} a_{12}^q \neq 0.
 \end{equation}
 On the other hand, if $H(2, q^2)$ is a non-degenerate Hermitian curve such that $P_1, P_2, P_3 \in PG(2, q^2) \setminus H(2, q^2)$ and the line $P_i P_j$ is tangent to $H(2, q^2)$, $1 \leq i < j \leq 3$, then $H(2, q^2)$ is given by
 $$0=y^{q+1} X_0^{q+1} + z^{q+1} X_1^{q+1} + X_2^{q+1} - xz^{q+1} X_0X_1^q+ $$
 $$- x^qz^{q+1} X_0^qX_1 - y X_0X_2^q - y^q X_0^qX_2 - z X_1X_2^q - z^q X_1^qX_2 ,$$
 \normalsize
 for some $x, y, z \in F_{q^2}$ such that $y^{q+1} = x^{q+1} z^{q+1}$ and $y \neq - x z$. By using the projectivity $X'_0 = y X_0, X'_1 = zX_1, X'_2 = X_2$, we may assume that $H(2, q^2)$ is given by
 $$X_0^{q+1} + X_1^{q+1} + X_2^{q+1} - \alpha X_0X_1^q - \alpha^q X_0^qX_1 - X_0X_2^q - X_0^qX_2 - X_1X_2^q - X_1^qX_2 = 0,$$
 for some $\alpha \in F_{q^2}$ such that $\alpha^{q+1} = 1$ and $\alpha \neq -1$. Then the lines $P_1 P_2$, $P_1 P_3$ and $P_2 P_3$ are tangent to $H(2, q^2)$ at the points $L_3 = (1, \alpha, 0)$, $L_2 = (1, 0, 1)$ and $L_1 = (0, 1, 1)$, respectively. Let $\perp$ and $\perp'$ be the polarities defining $H(2, q^2)$ and $H$, respectively. If $|H(2, q^2) \cap H| \in \{1, q+1\}$, according to Lemma \ref{HermCurve4}, we have that $L_i^\perp \cap L_i^{\perp'} \in \ell$, $i = 1,2,3$, where $\ell$ is the unique line of $PG(2, q^2)$ such that $\ell \cap H(2, q^2) = \ell \cap H = H(2, q^2) \cap H$. Denote by $Z_i$ the point $L_i^\perp \cap L_i^{\perp'}$, $i = 1,2,3$ and let $L$ be the point $\ell^\perp$. Then
 $$Z_1 = (0, -a_{12}^q, a_{12}), Z_2 = (-a_{02}^q, 0, a_{02}), Z_3 = (-a_{01}^q, \alpha^q a_{01}, 0)$$
 are on a line if and only if
 \begin{equation}\label{cond2}
  a_{01}^q a_{02} a_{12}^q - \alpha^q a_{01} a_{02}^q a_{12} = 0.
 \end{equation}
 Note that $L = Z_1^\perp \cap Z_2^\perp = (a_{02}(a_{12}+a_{12}^q), \alpha a_{12} (a_{02}+a_{02}^q), \alpha a_{02}^q a_{12} + a_{02} a_{12}^q)$ and since $\ell^\perp = \ell^{\perp'}$, we have that $Z_1, Z_2 \in L^{\perp'}$. Some calculations show that $Z_1, Z_2 \in L^{\perp'}$ if and only if
 \begin{align}
 & (a_{12} + a_{12}^q)(a_{12} a_{02}^q - a_{12}^q a_{01}^q) + a_{12}^{q+1}(\alpha^q a_{12}^q - a_{12}) = 0, \label{cond3}\\
 & (a_{02} + a_{02}^q)(a_{12}^q a_{02} - a_{02}^q a_{01}) + a_{02}^{q+1}(\alpha a_{02}^q - a_{02}) = 0. \label{cond4}
 \end{align}
 Since $\alpha \neq -1$, it follows from \eqref{cond4} that
 \begin{equation} \label{cond0}
  a_{02} + a_{02}^q \neq 0.
 \end{equation}
 Hence from \eqref{cond4}, we have that
 \begin{equation}\label{coeff1}
  a_{01} = \frac{a_{02} a_{12}^q  (a_{02} + a_{02}^q) + a_{02}^{q+1}( \alpha a_{02}^q - a_{02} )}{a_{02}^q ( a_{02}+a_{02}^q) }.
 \end{equation}
 By substituting \eqref{coeff1} in \eqref{cond2}, we get that one of the following has to be satisfied:
 \begin{align}
  & a_{02}^q - \alpha^q a_{02} = 0, \label{cond5}\\
  & a_{12}^{q+1}(a_{02} + a_{02}^q) - a_{02}^{q+1}(a_{12} + a_{12}^q) = 0. \label{cond6}
 \end{align}
 Assume that \eqref{cond6} is satisfied. Since $a_{12} + a_{12} \neq 0$, then
 \begin{equation}\label{coeff2}
  a_{02} = \frac{a_{12}^{q+1}}{a_{12} + a_{12}^q} ( \xi + 1 ),
 \end{equation}
 where $\xi \in F_{q^2}, \xi^{q+1} = 1$, with $\xi \neq -1$, otherwise $a_{02} = 0$, contradicting \eqref{cond1}. By taking into account \eqref{coeff2}, equation \eqref{coeff1} becomes
 \begin{equation}\label{coeff1_new}
  a_{01} = \frac{a_{12}^q( \xi a_{12}^q + \alpha a_{12} )}{a_{12} + a_{12}^q}.
 \end{equation}
 Moreover, if \eqref{coeff2} and \eqref{coeff1_new} hold true, then \eqref{cond3} is satisfied, whereas some calculations show that condition \eqref{cond1} is satisfied whenever $\xi a_{12}^q + \alpha a_{12} \neq 0$, i.e,
 \begin{equation} \label{cond7}
  a_{12}^{q-1} \neq -\xi^q \alpha.
 \end{equation}
 The equation $X^{q-1} = -\xi^q \alpha$ has $q-1$ solutions in $F_{q^2}$ since $(- \xi^q \alpha )^{q+1} = 1$. Assume that \eqref{cond5} holds true. Thus it turns out that necessarily $a_{01} = \alpha a_{12}^q$ and $a_{02} = \frac{a_{12}^{q+1}(\alpha+1)}{ a_{12} +a_{12}^q }$ and hence we retrieve a particular solution of the previous case (the one with $\xi = \alpha$). Observe that for a fixed $\xi \in F_{q^2}$, with $\xi^{q+1} = 1$, $\xi \neq -1$, we have that
 $$|\{a_{12} \in F_{q^2} |  a_{12} + a_{12}^q \neq 0, a_{12}^{q-1} \neq -\xi^q \alpha \}| =\begin{cases}
  q(q-1) & \mbox{ if } \xi = \alpha, \\
  (q-1)^2 & \mbox{ if } \xi \ne \alpha.
 \end{cases}$$
 Furthermore for a fixed $a_{12} \in F_{q^2}$ satisfying \eqref{cond0} and \eqref{cond7}, then $a_{01}$ and $a_{02}$ are uniquely determined by \eqref{coeff1_new} and \eqref{coeff2}, respectively, and, if $\rho \in F_q \setminus \{0\}$, the values $a_{12}, \rho a_{12}$ give rise to the same Hermitian curve.
\end{proof}
\begin{lem}\label{HermSur}
 In $PG(3, q^2)$, let $H(3, q^2)$ be a non-degenerate Hermitian surface with associated polarity $\perp$ and let $\gamma = P^\perp$ be a plane secant to $H(3, q^2)$. If $P'$ is a point of $PG(3, q^2) \setminus ( H(3, q^2) \cup \gamma \cup \{P\} )$, then the points of $\gamma$ lying on a line passing through $P'$ and tangent to $H(3, q^2)$ form a non-degenerate Hermitian curve $H'$ of $\gamma$ such that $\gamma \cap H' \cap H(3, q^2) \subset \langle P P' \rangle^\perp$. Vice versa, if $H'$ is a non-degenerate Hermitian curve of $\gamma$ meeting $\gamma \cap H(3, q^2)$ in one or $q+1$ points and $\ell$ is the (unique) line of $\gamma$ such that $\ell \cap H' = \ell \cap H(3, q^2) = \gamma \cap H' \cap H(3, q^2)$, then there are exactly $q+1$ points $P'$ of $\ell^\perp \setminus (H(3, q^2) \cup \gamma \cup \{P\} )$ such that the lines joining $P'$ with a point of $H'$ are tangent to $H(3, q^2)$.
\end{lem}
\begin{proof}
 Let $H(3, q^2)$ be the Hermitian surface of $PG(3, q^2)$ given by $X_0^{q+1}+X_1^{q+1}+X_2^{q+1}+X_3^{q+1} = 0$ and let $\gamma$ be the plane $X_3 = 0$. Let $P'$ be a point of $PG(3, q^2) \setminus ( H(3, q^2) \cup \gamma \cup \{P\} )$. If $PP'$ is tangent to $H(3, q^2)$, then we may assume that $P P' \cap \gamma = (1, \xi, 0, 0)$, for some $\xi \in F_{q^2}$ such that $\xi^{q+1} = -1$. Then $P' = (1, \xi, 0, \lambda)$ where $\lambda \in F_{q^2} \setminus \{0\}$. Similarly if $PP'$ is secant to $H(3, q^2)$, then we may assume that $P P' \cap \gamma = (1, 0, 0, 0)$. In this case $P' = (1, 0, 0, \lambda)$ where $\lambda \in F_{q^2} \setminus \{0\}$ and $\lambda^{q+1} \neq -1$. From \cite[Lemma 2.3]{Hirschfeld3}, the line joining $P'$ and $Q = (x_1, x_2, x_3, 0) \in \gamma$ is tangent to $H(3, q^2)$ if and only if $Q \in H'$, where $H'$ is the non-degenerate Hermitian curve of $\gamma$ given by
 $$( \lambda^{q+1} - 1 ) x_0^{q+1} + ( \lambda^{q+1} + 1) x_1^{q+1} + \lambda^{q+1} x_2^{q+1} - \xi^q x_0^q x_1 - \xi x_0 x_1^q = 0$$
 if $|PP' \cap H(3, q^2)| = 1$ or by
 $$\lambda^{q+1} x_0^{q+1} +( \lambda^{q+1} + 1 ) ( x_1^{q+1} + x_2^{q+1}) = 0$$
 if $|PP' \cap H(3, q^2)| = q+1$. Note that $\gamma \cap H' \cap H(3, q^2) \subset \langle P P' \rangle^\perp$ and by replacing $\lambda$ with $\lambda'$, where $\lambda^{q+1} = \lambda'^{q+1}$ we retrieve the same Hermitian curve $H'$.
\end{proof}

\section{Graphs cospectral with $NU(3, q^2)$}\label{sec62}

Let $\Pi$ be a finite projective plane of order $q^2$ and let $\mathcal{U}$ be a unital of $\Pi$. There are $q^3 + 1$ tangent lines and $q^4 - q^3 + q^2$ secant lines. In particular, through each point of $\mathcal{U}$ there pass $q^2$ secant lines and one tangent line, while through each point of $\Pi \setminus \mathcal{U}$ there pass $q + 1$ tangent lines and $q^2 - q$ secant lines. Starting from $\mathcal{U}$, it is possible to define a unital $\mathcal{U}^*$ of the dual plane $\Pi^*$ of $\Pi$. This can be done by taking the tangent lines to $\mathcal{U}$ as the points of $\mathcal{U}^*$ and the points of $\Pi \setminus \mathcal{U}$ as the secant lines to $\mathcal{U}^*$, where incidence is given by reverse containment. Thus $\mathcal{U}^*$ is a unital of $\Pi^*$, called the \textit{dual unital} to $\mathcal{U}$. Let $\Gamma_{\mathcal{U}}$ be the graph whose vertices are the points of $\Pi \setminus \mathcal{U}$ and two vertices $P_1, P_2$ are adjacent if the line joining $P_1$ and $P_2$ is tangent to $\mathcal{U}$.
\begin{prop}\label{SRGU}
 The graph $\Gamma_{\mathcal{U}}$ is strongly regular with parameters $(q^2(q^2-q+1), (q+1)(q^2-1), 2(q^2-1), (q+1)^2)$.
\end{prop}
\begin{proof}
 Since every point of $\Pi \setminus \mathcal{U}$ lies on $q+1$ lines that are tangent to $\mathcal{U}$, the graph $\Gamma_{\mathcal{U}}$ is $(q+1)(q^2-1)$-regular. Consider two distinct vertices $u_1, u_2$ of $\Gamma_{\mathcal{U}}$ corresponding to the points $P_1, P_2$ of $\Pi \setminus \mathcal{U}$. If $u_1, u_2$ are adjacent, then the line $P_1 P_2$ is tangent to $\mathcal{U}$ and $u_1, u_2$ have $2(q^2-1)$ common neighbours. If $u_1, u_2$ are not adjacent, then the line $P_1 P_2$ is secant to $\mathcal{U}$ and $u_1, u_2$ have $(q+1)^2$ common neighbours.
\end{proof}
The following proposition gives a characterization of maximal cliques of the graph $\Gamma_{\mathcal{U}}$.

\begin{prop}\label{cliques}
The graph $\Gamma_{\mathcal{U}}$ contains a class of $q^3+1$ maximal cliques of size $q^2$ and a class of $q^2(q^3+1)(q^2-q+1)$ maximal cliques of size $q+2$.
\end{prop}
\begin{proof}
 In the first case, a clique corresponds to the points on a tangent line with the tangent point excluded. In the second case, it corresponds to $q+1$ points on a tangent line $\ell$, and the other point off $\ell$. In both cases these cliques are maximals.
\end{proof}

Let $\Pi = PG(2, q^2)$ and let $\mathcal{U}$ be the classical unital $H(2, q^2)$. In this case $\Gamma_{\mathcal{U}}$ is denoted by $NU(3, q^2)$. Since a dual O'Nan configuration cannot be embedded in $H(2, q^2)$, we have the following (see also \cite[Corollary 3]{Bruen}).

\begin{cor}\label{cor1}
 The maximal cliques of the graph $NU(3,q^2)$ are exactly those described in Proposition \ref{cliques}.
\end{cor}

In what follows we recall few known properties of some classes unitals of $PG(2, q^2)$. Besides the classical one, there are other known unitals in $PG(2, q^2)$ and all the known unitals of $PG(2, q^2)$ arise from a construction due to F. Buekenhout, see \cite{BE, B, M}. A \textit{Buekenhout unital} of $PG(2, q^2)$ is either an orthogonal Buekenhout-Metz unital or a Buekenhout-Tits unital. An \textit{orthogonal Buekenhout-Metz unital} of $PG(2, q^2)$, $q > 2$, is projectively equivalent to one of the following form:
$$\mathcal{U}_{\alpha, \beta} = \{(x, \alpha x^2 + \beta x^{q+1} + z, 1) |z \in F_q, x \in F_{q^2}\} \cup \{(0,1,0)\} ,$$
where $\alpha, \beta$ are elements in $F_{q^2}$ such that $(\beta - \beta^q)^2 + 4 \alpha^{q+1}$ is a non-square in $F_q$ if $q$ is odd, or $\beta \notin F_q$ and $\frac{\alpha^{q+1}}{(\beta+\beta^q)^2}$ has absolute trace $0$ if $q$ is even and $q > 2$. The unital $\mathcal{U}_{\alpha, \beta}$ is classical if and only if $\alpha = 0$. If $q$ is odd, $\beta \in F_q$ and $\alpha \neq 0$, then $\mathcal{U}_{\alpha, \beta}$ can be obtained glueing together $q$ conics of $PG(2,q^2)$ having in common the point $(0, 1, 0)$. Let $m > 1$ be an odd integer and let $q = 2^m$. A \textit{Buekenhout-Tits unital} of $PG(2, q^2)$ is projectively equivalent to
$$\mathcal{U}_{T} = \{(x_0 + x_1 \beta, (x_0^{\delta + 2} + x_0 x_1 + x_1^{\delta}) \beta + z, 1)  | x_0, x_1, z \in F_q\} \cup \{(0,1,0)\} ,$$
where $\beta$ is an element of $F_{q^2} \setminus F_q$ and $\delta = 2^{\frac{m+1}{2}}$. See \cite{FL} for more details. As the Desarguesian plane is self-dual, that is, it is isomorphic to its dual plane, the dual of a unital in $PG(2, q^2)$ is a unital in $PG(2, q^2)$. The non-classical orthogonal Buekenhout-Metz unital $\mathcal{U}_{\alpha, \beta}$ as well as the Buekenhout-Tits unital $\mathcal{U}_{T}$  of $PG(2, q^2)$ are self-dual \cite[Theorem 4.17, Theorem 4.28, Corollary 4.35]{BE}.  Moreover the O'Nan configuration can be embedded in each non-classical Buekenhout unital of $PG(2, q^2)$ \cite{FL}. Therefore it follows that there exists a dual O'Nan configuration embedded in both unitals $\mathcal{U}_{\alpha, \beta}$ and $\mathcal{U}_{T}$\footnote{We thank A. Aguglia for pointing this out.}.

\begin{prop}
 There exists a dual O'Nan configuration embedded in a non-classical Buekenhout unital of $PG(2, q^2)$.
\end{prop}

We illustrate the previous statement with an example.
\begin{exmp}
 Let $q$ be odd and let us fix $\alpha, \beta$ in $F_{q^2}$, $\alpha \neq 0$, such that $(\beta - \beta^q)^2 + 4 \alpha^{q+1}$ is a non-square in $F_q$. Set $\mathcal{U} = \{(-2 \alpha x + (\beta^q - \beta) x^q, 1, \alpha x^2 - \beta^q x^{q+1} - z) | x \in F_{q^2}, z \in F_q \} \cup \{(0, 0, 1)\}$. From the proof of \cite[Theorem 4.17]{BE} $\mathcal{U}$ is projectively equivalent to $\mathcal{U}_{\alpha, \beta}$. Consider the six lines given by: $[0, 0, 1]$, $[0, -2 x_{\lambda_1} x_{\lambda_2}, x_{\lambda_1} + x_{\lambda_2}]$, $[x_{\lambda_1}, - x_{\lambda_1}, 1]$, $[x_{\lambda_2}, - x_{\lambda_2}, 1]$, $[- x_{\lambda_1}, - x_{\lambda_1}, 1]$, $[- x_{\lambda_2}, - x_{\lambda_2}, 1]$, where $x_{\lambda_i} = - \frac{\alpha^q + \lambda_i - \beta}{\alpha^{q+1} - (\lambda_i - \beta)^{q+1}}$, $i = 1,2$, and $\lambda_1, \lambda_2 \in F_q$ are such that $(\lambda_1 + \alpha^q - \beta)^{q+1}(\alpha^{q+1} - \beta^{q+1} - \lambda_2^2) + (\lambda_2 + \alpha^q - \beta)^{q+1}(\alpha^{q+1} - \beta^{q+1} -\lambda_1^2) = 0$. By \cite[Theorem 4.17]{BE} and \cite[Corollary 3.6]{FL} these lines are tangent to $\mathcal{U}$ and there are four non-collinear points of $PG(2, q^2) \setminus \mathcal{U}$ such that the lines joining two of these four points are exactly the six lines given above.
\end{exmp}
An easy consequence of the previous proposition is the following result.

\begin{cor}\label{cor2}
 If $\mathcal{U}$ is a non-classical Buekenhout unital of $PG(2, q^2)$, then the graph $\Gamma_{\mathcal{U}}$ has at least a further class of cliques than those described in Proposition \ref{cliques}. A clique in this class contains four points no three on a line.
\end{cor}

Comparing Corollary \ref{cor1} with Corollary \ref{cor2}, we see that $\Gamma_{\mathcal{U}}$ has more maximal cliques than $\Gamma_{\mathcal{U}}$ and we obtain the desired result for $q>2$.

\begin{thm}
 Let $q>2$. If $\mathcal{U}$ is a non-classical Buekenhout unital of $PG(2, q^2)$, then the graph $\Gamma_{\mathcal{U}}$ is not isomorphic to $NU(3, q^2)$.
\end{thm}

\section{Graphs cospectral with $NU(n+1, q^2)$, $n \geq 4$}\label{sec63}
 In this section we introduce two classes of graphs cospectral with $NU(n+1, q^2)$, $n \geq 4$. These graphs are obtained by modifying few adjacencies in $NU(n+1, q^2)$ following an idea proposed by Wang, Qiu, and Hu. For a graph $G$ and for a vertex $u$ of $G$, let $N_{G}(u)$ be the set of neighbours of $u$ in $G$. Let us recall the Wang-Qiu-Hu switching in a simplified version as stated in \cite{IM}.%
\begin{thm}[\textbf{WQH switching}]\label{thm:wqh}
 Let $G$ be a graph whose vertex set is partitioned as $\ell_1 \cup \ell_2 \cup D$. Assume that the induced subgraphs on $\ell_1, \ell_2,$ and $\ell_1 \cup \ell_2$ are regular, and that the induced subgraphs on $\ell_1$ and $\ell_2$ have the same size and degree. Suppose that each $x \in D$ either has the same number of neighbours in $\ell_1$ and $\ell_2$, or $N_{G}(x) \cap (\ell_1 \cup \ell_2) \in \{ \ell_1, \ell_2 \}$. Construct a new graph $G'$ by switching adjacency and non-adjacency between $x \in D$ and $\ell_1 \cup \ell_2$ when $N_{G}(x) \cap (\ell_1 \cup \ell_2) \in \{ \ell_1, \ell_2 \}$. Then $G$ and $G'$ are cospectral.
\end{thm}
\subsection{The graphs $G'_n$ and $G''_n$}\label{subsec631}
 Let $H(n, q^2)$ be a non-degenerate Hermitian variety of $PG(n, q^2)$, $n \geq 4$, with associated polarity $\perp$. Let us denote by $G_n$ the graph $NU(n+1, q^2)$. Fix a point $P \in H(n, q^2)$ and let $\bar{\ell}_1, \bar{\ell}_2$ be two lines of $PG(n, q^2)$ such that $\bar{\ell}_1 \cap H(n, q^2) = \bar{\ell}_2 \cap H(n, q^2) = P$. Let $\pi$ be the plane spanned by $\bar{\ell}_1$ and $\bar{\ell}_2$ and let $\ell_i = \bar{\ell}_i \setminus \{P\}$, $i = 1, 2$. Thus $\ell_i$ consists of $q^2$ vertices of $G_n$. There are two possibilities: either $\pi \cap H(n, q^2)$ is a Hermitian pencil of lines or is a line. Let us define the following sets:
\begin{align*}
  & A = ( \bigcap_{u \in \ell_1} N_{G_n}(u) ) \cap ( \cap_{u \in \ell_2} N_{G_n}(u) ), \\
  & A_1 = ( \bigcap_{u \in \ell_1} N_{G_n}(u) ) \setminus (A \cup \ell_2), \\
  & A_2 = ( \bigcap_{u \in \ell_2} N_{G_n}(u) ) \setminus (A \cup \ell_1).
\end{align*}
\begin{lem}\label{sets12}
 Let $n = 4$.
 \begin{itemize}
  \item If $\pi \cap H(4, q^2)$ is a Hermitian pencil of lines, then the set $A$ or $A_i$, $i= 1,2$, consists of the points lying on $(q+1)^2$ or $(q+1)(q^2-q-2)$ lines tangent to $H(4, q^2)$ at $P$, minus $P$ itself. In particular $|A| = q^2(q+1)^2$ and $|A_1| = |A_2| = q^2(q+1)(q^2-q-2)$.
  \item If $\pi \cap H(4, q^2)$ is a line, then the set $A$ or $A_i$, $i= 1,2$, consists of the points lying on $2(q^2-1)$ or $q(q^2-q-1)$ lines tangent to $H(4, q^2)$ at $P$, minus $P$ itself. In particular $|A| = 2q^2(q^2-1)$ and $|A_1| = |A_2| = q^3(q^2-q-1)$.
 \end{itemize}
\end{lem}
\begin{proof}
 Assume that $\pi \cap H(4, q^2)$ is a Hermitian pencil of lines. Let $u'$ be a vertex of $G_{4}$ such that it is adjacent to any vertex of $\ell_1$. Then $u' \notin \ell$, otherwise $u' \in N_{G_{4}}(u')$, a contradiction. If $u \in \ell_1$, the line determined by $u$ and $u'$ contains a unique point of $H(4, q^2)$ and hence the plane $\sigma$ spanned by $\bar{\ell}_1$ and $u'$ has exactly $q^2+1$ points of $H(4, q^2)$. Therefore $\sigma \cap H(4, q^2)$ is a line and $\sigma$ contains $q^4-q^2$ vertices of $G_{4}$ adjacent to every vertex of $\ell_1$. Note that there are $q+1$ planes through $\ell_1$ meeting $H(4, q^2)$ in a line. Hence $|A_1 \cup A| = (q+1)(q^4-q^2)$. A similar argument holds for $\ell_2$. Moreover if $\sigma_i$ is a plane through $\ell_i$ having $q^2+1$ points in common with $H(4, q^2)$, $i = 1,2$, then $\sigma_1 \cap \sigma_2$ is a line tangent to $H(4, q^2)$ at $P$. Therefore $A = q^2(q+1)^2$ and the first part of the statement follows. If $\pi \cap H(n, q^2)$ is a line, one can argues as before.
\end{proof}
It can be easily deduced that $A, A_1, A_2 \subseteq P^\perp$. If $\pi \cap H(n, q^2)$ is a Hermitian pencil of lines and $(n, q) = (4, 2)$, then $|A_1| = |A_2| = 0$.
\begin{cons}\label{cons}
 Define a graph $G'_n$ or $G''_n$, according as $\pi \cap H(n, q^2)$ is a Hermitian pencil of lines or a line. The vertices of $G'_n$ and $G''_n$ are the vertices of $G_n$ and the edges are as follows:
 $$\begin{cases}
   N_{G_n}(u) & \mbox{ if } u \notin A_1 \cup A_2 \cup \ell_1 \cup \ell_2, \\
  ( N_{G_n}(u) \setminus A_1 ) \cup A_2 & \mbox{ if } u \in \ell_1, \\
  ( N_{G_n}(u) \setminus A_2 ) \cup A_1 & \mbox{ if } u \in \ell_2, \\
  ( N_{G_n}(u) \setminus \ell_1  \cup \ell_2 & \mbox{ if } u \in A_1, \\
  ( N_{G_n}(u) \setminus \ell_2 ) \cup \ell_1 & \mbox{ if } u \in A_2.
  \end{cases}$$
\end{cons}
Note that the graphs $G'_n$ and $G_n''$ are obtained from $G_n$ by applying the WQH switching as described in Theorem \ref{thm:wqh}. Indeed, the vertex set of $G_n$ consists of $\ell_1 \cup \ell_2 \cup D$, where $A, A_1, A_2 \subset D$. Furthermore, if $x \in D \setminus (A_1 \cup A_2)$, then $N_{G_n}(x) \cap \ell_1 = N_{G_n}(x) \cap \ell_2 \in \{0, 1, q^2\}$, whereas if $x \in A_i$, then $N_{G_n}(x) \cap \ell_i = q^2$ and $N_{G_n}(x) \cap \ell_j = 0$, where $\{i, j\} = \{1, 2\}$.
\begin{thm}
 The graphs $G'_n$ and $G''_n$ are strongly regular and have the same parameters as $G_n$.
\end{thm}
\begin{proof}
 It is sufficient to show that the hypotheses of Theorem \ref{thm:wqh} are satisfied. Consider the graph $G'_n$. The induced subgraph on $\ell_i$ is the complete graph on $q^2$ vertices; as there is no edge between a vertex of $\ell_1$ and a vertex of $\ell_2$, the induced subgraph on $\ell_1 \cup \ell_2$ is the union of the two complete graphs. Let $u$ be a vertex of $G_n$ with $u \notin \ell_1 \cup \ell_2$. If $u$ is not in $\pi$, let $\sigma_i$ be the span of $u$ and $\ell_i$ for $i=1,2$. If $\sigma_i \cap H(n, q^2)$ is a line, then $N_{G_n}(u) \cap \ell_i = \ell_i$. If $\sigma_i \cap H(n, q^2)$ is a Hermitian pencil of lines, then all lines of $\sigma_i$ that are tangent to $H(n, q^2)$ go through $P$, so $|N_{G_n}(u) \cap \ell_i| = 0$. If $\sigma_i \cap H(n, q^2)$ is a non-degenerate Hermitian curve, then $\ell_i$ is the unique line of $\sigma_i$ that is tangent to $H(n, q^2)$ at $P$. In this case, as $u$ lies on $q+1$ lines of $\sigma_i$ that are tangent to $H(n, q^2)$, we have that $|N_{G_n}(u) \cap \ell_i| = q+1$. The curve $\sigma_1 \cap H(n, q^2)$ is non-degenerate if and only if the line $\sigma_1 \cap \sigma_2$ is secant to $H(n, q^2)$, i.e., if and only if the curve $\sigma_2 \cap H(n, q^2)$ is non-degenerate. Assume that $u \in \pi$. Thus $|N_{G_n}(u) \cap \ell_i| = 0$, $i = 1,2$. Hence, we have shown that we can apply Theorem \ref{thm:wqh} and obtain a strongly regular graph with the same parameters as $G_n$. By using the same arguments used for $G'$, it is possible to see that the graph $G''_n$ is strongly regular and has the same parameters as $G_n$.
\end{proof}
\subsection{The vertices adjacent to three pairwise adjacent vertices of $G_4$}\label{subsec632}
For a point $R$ of $PG(4, q^2) \setminus H(4, q^2)$ let $\mathcal{T}_{R}$ be the set of points lying on a line tangent to $H(4, q^2)$ at $R$. Thus $\mathcal{T}_{R} = N_{G_4}(R) \cup (R^\perp \cap H(4, q^2)) \cup \{R\}$. Let $P_1, P_2, P_3$ be three points of $PG(4, q^2) \setminus H(4, q^2)$ such that the line $P_i P_j$ is tangent to $H(4, q^2)$, $1 \leq i < j \leq 3$, or, in other words, $P_1, P_2, P_3$ are three pairwise adjacent vertices of $G_4$. Let $\gamma = \langle P_1, P_2, P_3 \rangle$. Note that if $\gamma$ is a plane, then $\gamma \cap H(4, q^2)$ is not a Hermitian pencil of lines. Indeed every tangent line contained in a plane meeting $H(4, q^2)$ in a Hermitian pencil of lines pass through the same point. Hence there are three possibilities:
\begin{enumerate}
 \item $\gamma$ is a line tangent to $H(4, q^2)$ and hence $\gamma \cap H(4, q^2)$ is a point;
 \item $\gamma$ is a plane and $\gamma \cap H(4, q^2)$ is a line;
 \item $\gamma$ is a plane and $\gamma \cap H(4, q^2)$ is a non-degenerate Hermitian curve.
\end{enumerate}
\begin{lem}\label{tanPlane}
 Let $P_1, P_2, P_3$ be three points of $PG(4, q^2) \setminus H(4, q^2)$ such that the line $P_i P_j$ is tangent to $H(4, q^2)$, $1 \leq i < j \leq 3$, $\gamma = \langle P_1, P_2, P_3 \rangle$ is a plane and $\gamma \cap H(4, q^2)$ is a line $\ell$.
 \begin{itemize}
  \item If $R$ is a point of $PG(4, q^2) \setminus (H(4, q^2) \cup \gamma )$, then $\mathcal{T}_{R} \cap \gamma$ is a Hermitian pencil of lines meeting $\ell$ in one point.
  \item If $\mathcal{C}$ is a Hermitian pencil of lines of $\gamma$ meeting $\ell$ in one point, then there are $q^3$ points $R \in PG(4, q^2) \setminus (H(4, q^2) \cup \gamma )$ such that $\mathcal{T}_{R} \cap \gamma = \mathcal{C}$.
 \end{itemize}
\end{lem}
\begin{proof}
 Let $R$ be a point of $PG(4, q^2) \setminus (H(4, q^2) \cup \gamma )$. The plane $\langle R, \ell \rangle$ meets $H(4, q^2)$ in a Hermitian pencil of lines and hence the line $t = R^\perp \cap \gamma = \langle R, \gamma^\perp \rangle^\perp = \langle R, \ell \rangle^\perp$ is tangent to $H(4, q^2)$. By projecting the plane $\gamma$ from $R$ onto $R^\perp$, we get a plane $\gamma'$, where $\gamma \cap \gamma' = t$. Moreover, since $R^\perp \cap H(4, q^2)$ is a non-degenerate Hermitian surface and $\gamma' \subset R^\perp$, it follows that $|\gamma' \cap H(4, q^2)| \in \{q^3+q^2+1, q^3+1\}$. Note that if $L$ is a point of $\ell$, then the line $R L$ shares with $H(4, q^2)$ either one or $q+1$ points according as $L \in R^\perp$ or $L \notin R^\perp$, respectively, and $\ell \cap R^\perp = \ell \cap t$, since $\ell \subset \gamma$. Hence $|\ell \cap t| = 1$. By projecting the line $\ell$ from $R$ onto $R^\perp$, we obtain a line $\ell'$ that is tangent to $H(4, q^2)$, with $\ell' \subset \gamma'$ and $\ell' \cap H(4, q^2) = t \cap H(4, q^2)$. In particular in $\gamma'$, there are two distinct tangent lines, namely $t$ and $\ell'$, such that $t \cap \ell' \in \ell$ and hence $\gamma' \cap H(4, q^2)$ is a Hermitian pencil of lines Therefore, we deduce that there are $q^3+q^2$ points of $\gamma \setminus H(4, q^2)$ belonging to $N_{G_4}(R)$. This shows the first part of the claim. In the plane $\gamma$, there are $q^2(q-1)(q^2+1)$ Hermitian pencils of lines meeting $\ell$ in one point, and these are permuted in a single orbit by the group $PGU(5, q)_\gamma$, the stabilizer of $\gamma$ in $PGU(5, q)$. On the other hand, $PGU(5, q)_\gamma$ permutes in a unique orbit the $q^5(q-1)(q^2+1)$ points of $PG(4, q^2) \setminus ( H(4, q^2) \cup \gamma )$. Consider the incidence structure whose point set $\mathcal{P}$ is formed by the points of $PG(4, q^2) \setminus ( H(4, q^2) \cup \gamma)$ and whose block set $\mathcal{L}$ consists of the $q^2(q-1)(q^2+1)$ degenerate Hermitian curves described above, where a point $R \in \mathcal{P}$ is incident with a block $\mathcal{C} \in \mathcal{L}$ if $\mathcal{T}_R \cap \gamma = \mathcal{C}$. Since both, $\mathcal{P}$ and $\mathcal{L}$, are orbits of the same group, the considered incidence structure is a tactical configuration and hence we have that a block of $\mathcal{L}$ is incident with $q^3$ points of $\mathcal{P}$, as required.
\end{proof}
\begin{lem}\label{secPlane}
 Let $P_1, P_2, P_3$ be three points of $PG(4, q^2) \setminus H(4, q^2)$ such that the line $P_i P_j$ is tangent to $H(4, q^2)$, $1 \leq i < j \leq 3$, $\gamma = \langle P_1, P_2, P_3 \rangle$ is a plane and $\gamma \cap H(4, q^2)$ is a non-degenerate Hermitian curve.
 \begin{itemize}
  \item If $R$ is a point of $\gamma^\perp \setminus H(4, q^2)$, then $\mathcal{T}_{R} \cap \gamma \subset H(4, q^2)$.
  \item If $R$ is a point of $PG(4, q^2) \setminus (H(4, q^2) \cup \gamma \cup \gamma^\perp )$ and $\langle R, \gamma \rangle \cap \gamma^\perp \in H(4, q^2)$, then $\mathcal{T}_{R} \cap \gamma$ is a Hermitian pencil of lines, say $\mathcal{C}$, meeting $H(4, q^2)$ in $q+1$ points. In this case there are $(q+1)(q^2-1)$ points $R \in PG(4, q^2) \setminus (H(4, q^2) \cup \gamma \cup \gamma^\perp )$ such that $\mathcal{T}_{R} \cap \gamma = \mathcal{C}$.
  \item If $R$ is a point of $PG(4, q^2) \setminus (H(4, q^2) \cup \gamma \cup \gamma^\perp )$ and $\langle R, \gamma \rangle \cap \gamma^\perp \notin H(4, q^2)$, then $\mathcal{T}_{R} \cap \gamma$ is a non-degenerate Hermitian curve, say $\mathcal{C}$, meeting $H(4, q^2)$ in $1$ or $q+1$ points. In this case there are $(q+1)(q^2-q)$ points $R \in PG(4, q^2) \setminus (H(4, q^2) \cup \gamma \cup \gamma^\perp )$ such that $\mathcal{T}_{R} \cap \gamma = \mathcal{C}$.
 \end{itemize}
\end{lem}
\begin{proof}
 Let $R$ be a point of $\gamma^\perp \setminus H(4, q^2)$, then a line $t$ through $R$ that is tangent to $H(4, q^2)$ meets $\gamma$ in a point of $H(4, q^2)$. Hence $\mathcal{T}_{R} \cap \gamma \subset H(4, q^2)$. Let $R$ be a point of $PG(4, q^2) \setminus ( H(4, q^2) \cup \gamma \cup \gamma^\perp )$. The solid $\Gamma = \langle R, \gamma \rangle$ meets $\gamma^\perp$ in a point, say $R'$, and let $R'' = \Gamma^\perp = R^\perp \cap \gamma^\perp$. Let us define $\gamma'$ as the plane $R^\perp \cap \Gamma$ and $r$ as the line $\gamma \cap \gamma'$. Note that $\gamma'$ is the plane obtained by projecting $\gamma$ from $R$ onto $R^\perp$. If $R' \in H(4, q^2)$, then $R' = R''$, $R'^\perp = \Gamma$, the line $RR'$ is tangent to $H(4, q^2)$ at $R'$ and the plane $\langle \gamma^\perp, R R' \rangle$ meets $H(4, q^2)$ in $q^3+1$ points. Hence $\gamma' = R^\perp \cap \Gamma = R^\perp \cap R'^\perp = \langle R R' \rangle^\perp$ is a plane meeting $H(4, q^2)$ in $q^3+q^2+1$ points. Moreover, $r = \langle \gamma^\perp, R R' \rangle^\perp$ has $q+1$ points in common with $H(4, q^2)$, since $|\langle \gamma^\perp, R R' \rangle \cap H(4, q^2)| = q^3+1$. In this case by projecting the $q^3+q^2+1$ points of $\gamma' \cap H(4, q^2)$ from $R$ onto $\gamma$, we get a Hermitian pencil of lines of $\gamma$, say $\mathcal{C}$, such that $\mathcal{C} \cap r = \mathcal{C} \cap H(4, q^2) = r \cap H(4, q^2)$. Further the vertex of $\mathcal{C}$ is the point $V = RR' \cap \gamma$. Varying the point $R'$ in $\gamma^\perp \cap H(4, q^2)$ and by replacing $R$ with one of the $q^2-1$ points of $V R' \setminus \{ V, R' \}$, we obtain the same curve $\mathcal{C}$. If $R' \notin H(4, q^2)$, then $R' \neq R''$, $R''^\perp = \Gamma$ and the line $RR''$ is secant to $H(4, q^2)$, since $R \in R''^\perp$, with $R \notin H(4, q^2)$. Hence $\gamma' = R^\perp \cap \Gamma = R^\perp \cap R''^\perp = \langle R R'' \rangle^\perp$ is a plane meeting $H(4, q^2)$ in $q^3+1$ points. Note that $r = \gamma \cap \gamma' = \langle R', R'' \rangle^\perp \cap \langle R, R'' \rangle^\perp = \langle R, R', R'' \rangle^\perp$. Two possibilities arise: either $R R'$ is a line tangent to $H(4, q^2)$ at the point $RR' \cap \gamma$ or $RR'$ is secant to $H(4,q^2)$ and $RR' \cap \gamma \notin H(4, q^2)$. In the former case $|r \cap H(4, q^2)| = 1$, since $|\langle R, R', R'' \rangle \cap H(4, q^2)| = q^3+q^2+1$, whereas in the latter case $|r \cap H(4, q^2)| = q+1$, since $|\langle R, R', R'' \rangle \cap H(4, q^2)| = q^3+1$. By projecting the $q^3+1$ points of $\gamma' \cap H(4, q^2)$ from $R$ onto $\gamma$, we get a non-degenerate Hermitian curve of $\gamma$, say $\mathcal{C}$, such that $\mathcal{C} \cap r = \mathcal{C} \cap H(4, q^2) = r \cap H(4, q^2)$. By Lemma \ref{HermSur} there are exactly $q+1$ points of $R R' \setminus ( (4, q^2) \cup \gamma \cup \gamma' )$ which give rise to the same curve $\mathcal{C}$. Since $|\gamma^\perp \setminus H(4, q^2)| = q^2-q$, it follows that there are $(q+1)(q^2-q)$ points $R \in PG(4, q^2) \setminus (H(4, q^2) \cup \gamma \cup \gamma^\perp )$ such that $\mathcal{T}_{R} \cap \gamma = \mathcal{C}$.
\end{proof}
\begin{thm}\label{char}
 Let $P_1, P_2, P_3$ be three pairwise adjacent vertices of $G_4$ and let $\gamma = \langle P_1, P_2, P_3 \rangle$.
 \begin{enumerate}
  \item If $\gamma$ is a line tangent to $H(4, q^2)$, let $s$ be the unique Baer subline of $\gamma$ containing $P_1, P_2, P_3$ and let $T = \gamma \cap H(4, q^2)$. Then
      $$|N_{G_4}(P_1) \cap N_{G_4}(P_2) \cap N_{G_4}(P_3)|=
       \begin{cases}
       q^5+q^4-q^3-3 & \mbox{ if } T \in s, \\
       2q^5+q^4-q^3-3 & \mbox{ if } T \notin s.
       \end{cases}$$
  \item If $\gamma$ is a plane and $\gamma \cap H(4, q^2)$ is a line, then $|N_{G_4}(P_1) \cap N_{G_4}(P_2) \cap N_{G_4}(P_3)| = q^5+3q^4-3$.
  \item If $\gamma$ is a plane and $\gamma \cap H(4, q^2)$ is a non-degenerate Hermitian curve, then $|N_{G_4}(P_1) \cap N_{G_4}(P_2) \cap N_{G_4}(P_3)| = q^5+2q^4+3q^3-2q^2-q-3$.
 \end{enumerate}
\end{thm}
\begin{proof}
 Assume that $\gamma$ is a line tangent to $H(4, q^2)$ and let $\sigma$ be a plane through $\gamma$. If $\sigma \cap H(4, q^2)$ is a line, then $\sigma$ contains $q^4-q^2$ points $R$ not in $\gamma$ such that the lines $R P_i$, $i = 1,2,3$, are tangent to $H(4, q^2)$. If $\sigma \cap H(4, q^2)$ is Hermitian pencil of lines, then $\sigma$ contains no point $R$ not in $\gamma$ such that the lines $R P_i$, $i = 1,2,3$, are tangent to $H(4, q^2)$, whereas if $\sigma \cap H(4, q^2)$ is a non-degenerate Hermitian curve, then from Lemma \ref{HermCurve1}, $\sigma$ contains either $0$ or $q$ points $R$ not in $\gamma$ such that the lines $R P_i$, $i = 1,2,3$, are tangent to $H(4, q^2)$ according as $T  \in s$ or $T \notin s$, respectively. Since there are $q+1$ planes meeting $H(4, q^2)$ in a line, $q^2-q$ planes intersecting $H(4, q^2)$ in a Hermitian pencil of lines and $q^4$ planes having $q^3+1$ points in common with $H(4, q^2)$, we have that $|N_{G_4}(P_1) \cap N_{G_4}(P_2) \cap N_{G_4}(P_3)| = (q+1)(q^4-q^2) + |\gamma \setminus \{ T, P_1, P_2, P_3\}| = q^5+q^4-q^3-3$, if $T \in s$, and $|N_{G_4}(P_1) \cap N_{G_4}(P_2) \cap N_{G_4}(P_3)| = (q+1)(q^4-q^2) + q^4 \cdot q + |\gamma \setminus \{ T, P_1, P_2, P_3\}| = 2q^5+q^4-q^3-3$, if $T \notin s$. Assume that $\gamma$ is a plane and that $\gamma \cap H(4, q^2)$ is a line, say $\ell$. If $R$ is a point of $PG(4, q^2) \setminus ( H(4, q^2) \cup \gamma )$ such that the lines $R P_i$, $i = 1,2,3$, are tangent to $H(4, q^2)$, then from Lemma \ref{tanPlane}, $\mathcal{T}_{R} \cap \gamma$ is a Hermitian pencil of lines meeting $\ell$ in one point. From Lemma \ref{HermCurve2}, in $\gamma$ there are $q^2+2q$ Hermitian pencils of lines meeting $\ell$ in one point and containing $P_1, P_2, P_3$. From Lemma \ref{tanPlane}, for each of these $q^2+2q$ Hermitian pencils of lines, there are $q^3$ points $R$ of $PG(4, q^2) \setminus (H(4, q^2) \cup \gamma)$ such that the lines $R P_i$, $i = 1,2,3$, are tangent to $H(4, q^2)$. Therefore in this case $|N_{G_4}(P_1) \cap N_{G_4}(P_2) \cap N_{G_4}(P_3)| = q^3(q^2+2q) + |\gamma \setminus (\ell \cup \{P_1, P_2, P_3\})| = q^5+3q^4-3$. Assume that $\gamma$ is a plane and that $\gamma \cap H(4, q^2)$ is a non-degenerate Hermitian curve. If $R$ is a point of $PG(4, q^2) \setminus ( H(4, q^2) \cup \gamma )$ such that the lines $R P_i$, $i = 1,2,3$, are tangent to $H(4, q^2)$, then from Lemma \ref{secPlane}, we have that $R \in PG(4, q^2) \setminus ( H(4, q^2) \cup \gamma \cup \gamma^\perp )$ and either $\mathcal{T}_{R} \cap \gamma$ is a Hermitian pencil of lines meeting $H(4, q^2)$ in $q+1$ points or $\mathcal{T}_{R} \cap \gamma$ is a non-degenerate Hermitian curve meeting $H(4, q^2)$ in one or $q+1$ points. From Lemma \ref{HermCurve3}, in $\gamma$ there are $3q$ Hermitian pencils of lines meeting $\gamma \cap H(4, q^2)$ in $q+1$ points and containing $P_1, P_2, P_3$. Similarly from Proposition \ref{HermCurve5}, in $\gamma$ there are $q^2-q+1$ non--degenerate Hermitian curves meeting $\gamma \cap H(4, q^2)$ in one or $q+1$ points and containing $P_1, P_2, P_3$. From Lemma \ref{secPlane}, for each of these Hermitian curves, there are either $(q+1)(q^2-1)$ or $(q+1)(q^2-q)$ points $R$ of $PG(4, q^2) \setminus (H(4, q^2) \cup \gamma \cup \gamma^\perp )$ such that the lines $R P_i$, $i = 1,2,3$ are tangent to $H(4, q^2)$ according as the Hermitian curve is degenerate or not, respectively. Observe that if $R$ belongs to $\gamma$ then necessarily $R$ has to lie on a line joining two of the three points $P_1, P_2, P_3$, otherwise $\gamma \cap H(4, q^2)$ would contain the dual of an O'Nan configuration, a contradiction. Therefore in this case $|N_{G_4}(P_1) \cap N_{G_4}(P_2) \cap N_{G_4}(P_3)| = 3q (q+1)(q^2-1) + (q^2-q+1) (q+1)(q^2-q) + 3 (q-1) = q^5+2q^4+3q^3-2q^2-q-3$.
\end{proof}
\subsection{The isomorphism issue}\label{subsec633}
 Set $n = 4$. Let $u$ be a point of $\ell_1$ and let $t$ be a line through $u$ tangent to $H(4, q^2)$ at $T$ and contained in $P^\perp$, with $t \ne \bar{\ell}_1$. If $\pi \cap H(4, q^2)$ is a Hermitian pencil of lines, then $|t \cap A_1| = q^2 - q - 2$, $|t \cap A_2| = 0$ and $t$ shares with $A$ the $q+1$ points of a Baer subline. If $\pi \cap H(4, q^2)$ is a line, then $|t \cap A_1| = q^2 - q - 1$, $|t \cap A_2| = 0$, $|t \cap A| = q$ and $(t \cap A) \cup \{u\}$ is Baer subline. Let $u_1$ and $u_2$ be two distinct points of $A$ and let $b$ be the Baer subline of $t$ containing $u, u_1, u_2$. In the case when $|\pi \cap H(4, q^2)| = q^3+q^2+1$, assume that $T \notin b$. Observe that $u, u_1, u_2$ are three pairwise adjacent vertices of $G_4$, $G'_4$ and $G_4''$.

\begin{prop}\label{char_new}
 $|N_{G'_4}(u) \cap N_{G'_4}(u_1) \cap N_{G'_4}(u_2)| = |N_{G''_4}(u) \cap N_{G''_4}(u_1) \cap N_{G''_4}(u_2)| = 2q^5+q^3-3$.
\end{prop}
\begin{proof}
 From the definition of $G'_4$ and $G''_4$ we have that
 \begin{align*}
  N_{G_4'}(u) \cap N_{G_4'}(u_1) \cap N_{G_4'}(u_2)  =  N_{G_4''}(u) \cap N_{G_4''}(u_1) \cap N_{G_4''}(u_2) \\
    = ( ( N_{G_4}(u) \setminus A_1 ) \cup A_2 ) \cap N_{G_4}(u_1) \cap N_{G_4}(u_2) \\
   =  ( ( N_{G_4}(u) \cap N_{G_4}(u_1) \cap N_{G_4}(u_2) ) \setminus A_1 ) \cup ( N_{G_4}(u_1) \cap N_{G_4}(u_2) \cap A_2 ) &.
 \end{align*}
 Let $\sigma$ be the plane spanned by $t$ and $P$. Since $\sigma \cap H(4, q^2)$ is a line, it follows that every point of $\sigma \cap A_1$ spans with both $u_1$ and $u_2$ a line that is tangent to $H(4, q^2)$. Hence $A_1 \cap \sigma \subseteq N_{G_4}(u) \cap N_{G_4}(u_1) \cap N_{G_4}(u_2)$. Note that $|A_2 \cap \sigma| = 0$, whereas $|A_1 \cap \sigma|$ equals $q^2 (q^2-q-2)$ or $q^2(q^2-q-1)$ according as $\pi \cap H(4, q^2)$ is a Hermitian pencil of lines or a line. Denote by $\sigma'$ a plane through $t$ contained in $P^\perp$, with $\sigma' \neq \sigma$, and let $u' = \sigma' \cap \ell_2$. Then $\sigma' \cap H(4, q^2)$ is a non-degenerate Hermitian curve. If $|\pi \cap H(4, q^2)|$ equals $q^3+q^2+1$ (respectively $q^2+1$), the points of $A_1$ contained in $\sigma' \setminus \sigma$ are on the $q$ (respectively $q-1$) lines of $\sigma'$ tangent to $\sigma' \cap H(4, q^2)$ passing through $u$ and distinct from $t$ (respectively $t$ and $\sigma' \cap \pi$). By Lemma \ref{HermCurve0}, there is a unique point $R$ on each of these tangent lines such that $R \in N_{G_4}(u) \cap N_{G_4}(u_1) \cap N_{G_4}(u_2)$ and if $R \in N_{G_4}(u) \cap N_{G_4}(u_1) \cap N_{G_4}(u_2)$, then $R \in (A \cup A_1) \cap (\sigma' \setminus \sigma)$. In particular, if the case $\langle R, u' \rangle \cap t \notin \{u_1, u_2\}$ occurs, then $R \notin A$, otherwise $R, u_1, u_2, u'$ would be a dual O'Nan configuration of $\sigma' \cap H(4, q^2)$, which is impossible. Hence $R \in A_1$. Assume that $\langle R, u' \rangle \cap t \in \{u_1, u_2\}$. If $\pi \cap H(4, q^2)$ is a Hermitian pencil of lines, then $R \in A$, by Lemma \ref{HermCurve-1}. If $\pi \cap H(4, q^2)$ is a line, then $R \in A_1$, otherwise $R, u_1, u, u'$ or $R, u_2, u, u'$ would be a dual O'Nan configuration of $\sigma' \cap H(4, q^2)$. Therefore $|N_{G_4}(u) \cap N_{G_4}(u_1) \cap N_{G_4}(u_2) \cap A_1|$ equals $q^2(q^2-q-2)+q^2(q-2) = q^4-4q^2$ or $q^2(q^2-q-1)+q^2(q-1) = q^4-2q^2$, according as $\pi \cap H(4, q^2)$ is a Hermitian pencil of lines or a line, respectively. In a similar way, if $|\pi \cap H(4, q^2)|$ equals $q^3+q^2+1$ (respectively $q^2+1$), the points of $A_2$ contained in $\sigma' \setminus \sigma$ are on the $q+1$ (respectively $q$) lines of $\sigma'$ tangent to $\sigma' \cap H(4, q^2)$ passing through $u'$ (and distinct from $\sigma' \cap \pi$). Each of these lines meets $t$ in a point of $A$. Let $R$ be a point on one of these $q+1$ (respectively $q$) lines, with $R \in N_{G_4}(u_1) \cap N_{G_4}(u_2) \setminus (t \cup \{u'\})$. Hence $R \in (A \cup A_2) \cap (\sigma' \setminus \sigma)$. If $\langle R, u' \rangle \cap t \notin \{u_1, u_2\}$, then $R \notin A_2$, otherwise $R, u_1, u_2, u'$ would be a dual O'Nan configuration of $\sigma' \cap H(4, q^2)$, a contradiction. Assume that $\langle R, u' \rangle \cap t \in \{u_1, u_2\}$. There are $q-1$ possibilities for $R$, since the line $\langle u', u_i \rangle$ contains $q-1$ points distinct from $u'$ belonging to $N_{G_4}(u_1) \cap N_{G_4}(u_2)$, by Lemma \ref{HermCurve0}. If $\pi \cap H(4, q^2)$ is a Hermitian pencil of lines, then exactly one of these $q-1$ points is in $A$ by Lemma \ref{HermCurve-1}. If $\pi \cap H(4, q^2)$ is a line, then for all the $q-1$ possibilities we have that $R \in A_2$, otherwise $R, u_1, u, u'$ or $R, u_2, u, u'$ would be a dual O'Nan configuration of $\sigma' \cap H(4, q^2)$. Therefore $|N_{G_4}(u_1) \cap N_{G_4}(u_2) \cap A_2 \cap \sigma'|$ is $2(q-2)$ or $2(q-1)$ and hence $|N_{G_4}(u_1) \cap N_{G_4}(u_2) \cap A_2|$ equals $2(q-2) q^2$ or $2(q-1)q^2$, according as $\pi \cap H(4, q^2)$ is a Hermitian pencil of lines or a line, respectively. It follows that $|N_{G_4'}(u) \cap N_{G_4'}(u_1) \cap N_{G_4'}(u_2)| = 2q^5+q^4-q^3-3 - (q^4-4q^2) + 2(q^3 - 2q^2) = 2q^5+q^3-3$ and $|N_{G_4''}(u) \cap N_{G_4''}(u_1) \cap N_{G_4''}(u_2)| = 2q^5+q^4-q^3-3 - (q^4-2q^2) + 2(q^3 - q^2) = 2q^5+q^3-3$.
\end{proof}
As a consequence of Theorem \ref{char} and Proposition \ref{char_new}, we have the following.
\begin{thm}
 Let $n \geq 4$. The graph $G_n$ is not isomorphic neither to $G_n'$ (if $q >2$), nor to $G''_n$.
\end{thm}
\begin{proof}
 Let $\Pi$ be a four-space of $PG(n, q)$ such that $\Pi \cap H(n, q^2) = H(4, q^2)$ and $\bar{\ell}_i \subset \Pi$, $i = 1, 2$. Then the induced subgraph by $G_n$ or $G_n'$ or $G_n''$ on $\Pi \setminus H(4, q^2)$ is $G_4$ or $G_4'$ or $G_4''$. Let $u, u_1, u_2$ be three pairwise adjacent vertices of $G_4$, $G_4'$ and $G_4''$. By Theorem \ref{char}, $|N_{G_4}(u) \cap N_{G_4}(u_1) \cap N_{G_4}(u_2)|$ takes four different values. By Proposition \ref{char_new}, both $|N_{G_4'}(u) \cap N_{G_4'}(u_1) \cap N_{G_4'}(u_2)|$ and $|N_{G_4''}(u) \cap N_{G_4''}(u_1) \cap N_{G_4''}(u_2) |$ take at least five different values.
\end{proof}
\begin{prob}
 We have seen that if $n \neq 3$, then $NU(n+1, q^2)$ is not determined by its spectrum. Regarding the case $n = 3$, there are exactly $28$ strongly regular graphs cospectral with $NU(4, 4)$, see \cite{Spence}. A strongly regular graph having the same parameters as $NU(4, 9)$ and admitting $2 \times  U(4, 4)$ as a full automorphism group has been constructed in \cite{CRS}. We ask whether or not the graph $NU(4, q^2)$, $q \geq 4$, is determined by its spectrum.
\end{prob}

\chapter{The Automorphism Group of $NU(3,q^2)$}\label{ch7}
In Chapter \ref{ch6} is shown that $NU(n+1,q^2)$ is not uniquely determined by its spectrum for $n\neq3 $. In particular, for $n=2$, it is given a construction of strongly regular graphs cospectral but non isomorphic to $NU(n+1,q^2)$. The construction relied on the existence of non-classical unitals embedded in $PG(2,q^2)$ for $q>2$. The goal in this chapter is to determine the automorphism group of $NU(3,q^{2})$, following the proof contained in the paper \cite{VS3}, that is to prove the following result.
\begin{thm}
  \label{mmain}	
     Let $\mathfrak{G}_{2}=Aut(NU(3,q^{2}))$ be the automorphism group of the graph $NU(3,q^{2})$:
     \begin{itemize}
      \item if $q\neq2$, $\mathfrak{G}_{2}\cong P\Gamma U(3,q)$, the semilinear collineation group stabilizing the Hermitian variety $H(2,q^{2})$;
           \item if $q=2$, $\mathfrak{G}_{2}\cong S_{3} \wr S_4 \cong S_3^4 \rtimes S_4$.
      \end{itemize}
  \end{thm}
  The following lemma is a straightforward consequence of the construction of $NU(n+1,q^{2})$ since the group $P\Gamma U(n+1,q)$ is transitive on the external points of $H(n,q^2)$ and it preserves the adjacency properties of the graph.
  \begin{lem}
   \label{subgr}
   The subgroup $P\Gamma U(n+1,q)$ of $P\Gamma L(n+1,q^2)$ is an automorphism group of $NU(n+1,q^{2})$.
  \end{lem}
  \begin{proof}
   The result is immediate, because the group $P\Gamma U(n+1,q)$ is transitive on the external points to $H(n,q^{2})$, i.e. the veritces of $NU(n+1,q^{2})$, and it preserves the lines, i.e. the adjacency properties on the graph.
  \end{proof}

  In the planar case, the graph $NU(3,q^{2})$ arises from a Hermitian curve $H(2,q^{2})$. This curve has $q^{3}+1$ isotropic points, so that the vertex set of $NU(3,q^{2})$ has size  $q^{4}+q^{2}+1-(q^{3}+1)=q^{4}-q^{3}+q^{2}$, while the other parameters are $k=(q^2-1)(q+1)$, $\lambda=2(q^2-1)$ and $\mu=(q+1)^2$ .

  \begin{lem}
   \label{curve}
   If $q\neq2$, $Aut(NU(3,q^{2}))$ sends $q^{2}$ vertices on a tangent line in other $q^{2}$ vertices on a tangent line.
  \end{lem}

  \begin{proof}
   The $q^{2}$ non-isotropic points on a tangent line, form a maximal clique on the graph $NU(3,q^{2})$. From \cite{Bruen} we know that the maximal cliques of $NU(3,q^{2})$ are given by either $q^{2}$ points on a tangent line, or $q+2$ points on two different lines. Since the automorphism group of a graph sends maximal cliques in maximal cliques then, if $q\neq2$, i.e. $q^{2}\neq q+2$, the image of a set of $q^{2}$ non-isotropic aligned points consists on other $q^{2}$ non-isotropic aligned points.
  \end{proof}

 \section{The action of $\mathfrak{G}_{n}$ on isotropic points}\label{sec71}
  To understand how the automorphism group of $NU(n+1,q^{2})$ acts on the points of $H(2,q^{2})$, it is useful to introduce another graph $\Gamma_{n}$, whose vertex set consists of all points in $PG(n,q^{2})$, and with the same vertex-edge incidence relation of $NU(n+1,q^{2})$, namely two vertices are adjacent if the points are on the same tangent. $\Gamma_{n}$ has $\frac{q^{2n}-1}{q^{2}-1}$ vertices, and its automorphism group has two orbits of vertices, one is the orbit $\mathcal{O}_{1}$ consisting of the $v=\frac{q^{n}(q^{n+1}-\varepsilon)}{q+1}$ vertices of $NU(n+1,q^{2})$, each vertex of this type having $k=(q^{n}+\varepsilon)(q^{n-1}-\varepsilon)$ neighbours in $NU(n+1,q^{2})$ and others in $H(n,q^{2})$, and the orbit $\mathcal{O}_{2}$ comprises the points of the Hermitian variety $H(n,q^{2})$, each vertex of this type having neighbours in $NU(n+1,q^{2})$ but no neighbour in $H(n,q^{2})$. In other words, we added an extra point, the $(q^{2}+1)$-th, to each of the maximal cliques arising from tangent lines. In fact, since it is possible to prove that $P\Gamma U(n+1,q)\leq Aut(\Gamma_{n})$ with the same arguments of Lemma \ref{subgr}, the automorphism group of $\Gamma_{n}$ is transitive on both $\mathcal{O}_{1}$ and $\mathcal{O}_{2}$, see \cite{Segre}. However, it should be observed that the resulting graph is not strongly regular anymore.
\begin{lem}
   \label{curve2}
   If $q\neq2$, $Aut(\Gamma_{2})\cong \mathfrak{G}_{2}$.
\end{lem}
  \begin{proof}
   From Lemma \ref{curve}, maximal cliques corresponding to a tangent are fixed setwise under the action of $\mathfrak{G}_2$.
   The projection
    \begin{equation*}
    \begin{array}{lccc}
     \pi: & Aut(\Gamma_{2}) & \longrightarrow & \mathfrak{G}_{2} \\
     &\lambda & \mapsto & \overline	{\lambda}=\lambda\big|_{NU(3,q^{2})},\\
   \end{array}
   \end{equation*}
   is surjective since every automorphism $\overline{\lambda}$ in $\mathfrak{G}_2$ lifts to an element $\lambda \in Aut(\Gamma_2)$ by extending the action in a \textit{natural} way to the isotropic points (i.e. by considering the image of the tangent lines under the action of $\lambda$ and mapping the corresponding tangency points into each other). To prove the injectivity of the projection $\pi$, consider $\lambda\in Aut(\Gamma_{2})$ such that $\pi(\lambda)$ is the identity on $NU(3,q^{2})$. When all the points of $NU(3,q^{2})$ are fixed, the tangency point is fixed as well, since it is adjacent in $\Gamma_2$ to each of the $q^{2}$ non-isotropic points on a tangent. Therefore $\lambda$ is the identity on $\Gamma_{2}$, thus the kernel of the projection $\pi$ is trivial, and $\pi$ is a monomorphism.
  \end{proof}
  \begin{lem}
   The automorphism group $\mathfrak{G}_{2}$ of $NU(3,q^{2})$, $q\neq2$, acts 2-transitively on tangents lines.
  \end{lem}

  \begin{proof}
   Since the action of $P\Gamma U(3,q)$ is $2$-transitive on the set of isotropic points of the Hermitian curve, the assertion follows from Lemma \ref{subgr} and Lemma \ref{curve2} taking into account the fact that every isotropic point is the tangency point of a unique tangent to $H(2,q^{2})$.
  \end{proof}

  \begin{lem}
  \label{no3tr}
   $\mathfrak{G}_{2}$ does not act $3$-transitively on the set of tangents of $NU(3,q^2)$.	
  \end{lem}
  \begin{proof}
   In $PG(2,q^2)$ let us consider the following two different configurations of $3$-tangent lines: three concurrent tangents or three tangents which intersect pairwise in three points $P_1, P_2, P_3$ on a self-polar triangle. In the graph $NU(3,q^2)$ the former set corresponds to three maximal cliques intersecting in their common vertex, whereas in the latter it corresponds to three maximal cliques intersecting pairwise in three different vertices. It is clear that any automorphism of the graph cannot map the former set to the latter, since they are different adjacence configurations in the graph $NU(3,q^2)$.
  \end{proof}	
\section{Proof of Theorem \ref{mmain}}\label{sec72}
\subsection{The case $q > 2$}\label{subsec721}
From now on, assume that $q>2$. Let $M$ be the minimal normal subgroup of $\mathfrak{G}_2$ (it is unique by the $2$-transitivity of $\mathfrak{G}_{2}$). Observe that $M$ cannot be elementary abelian since its order is not a power of a prime (as it contains $P\Gamma U(3,q)$ as its subgroup), so $M$ must be a simple group. We show that among the finite simple groups which are minimal normal subgroups of $2$-transitive groups, only a few can actually occur in our case. For this purpose we refer to Cameron's list \cite{Cameron} reported in Table \ref{sgroups} and to \cite{Dixon} and \cite{Wilson} for the details about standard group theory arguments. Table \ref{sgroups} belong directly to the known Classification of Finite Simple Groups, see \cite{Wilson}, and consider the theory of finite permutation groups with the assumption that the finite simple groups are known.
\begin{table}[h]
\begin{center}
\begin{tabular}{|c|c|c|}
\hline
$M$         & Degree                  & Remarks         \\ \hline
$A_d$       & $d$                  & $d \geq 5$      \\ \hline
$PSL(d,u)$  & $\frac{u^d-1}{u-1}$  & $d \geq 2$      \\ \hline
$PSU(3,u)$  & $u^3+1$              & $u >2$          \\ \hline
$Sz(u)$     & $u^2+1$              & $u=2^{2a+1} >2$ \\ \hline
$Ree(u)$    & $u^3+1$              & $u=3^{2a+1} >3$ \\ \hline
$PSp(2d,2)$ & $2^{2d-1} + 2^{d-1}$ & $d >2$          \\ \hline
$PSp(2d,2)$ & $2^{2d-1} - 2^{d-1}$ & $d >2$          \\ \hline
$PSL(2,11)$ & $11$                 &                 \\ \hline
$PSL(2,8)$  & $28$                 &                 \\ \hline
$A_7$       & $15$                 &                 \\ \hline
$M_{11}$    & $11$                 &                 \\ \hline
$M_{11}$    & $12$                 &                 \\ \hline
$M_{12}$    & $12$                 &                 \\ \hline
$M_{22}$    & $22$                 &                 \\ \hline
$M_{23}$    & $23$                 &                 \\ \hline
$M_{24}$    & $24$                 &                 \\ \hline
$HS$        & $176$                &                 \\ \hline
$Co_3$      & $276$                &                 \\ \hline
\end{tabular}
\caption{\footnotesize{The simple groups $M$ which can occur as minimal normal subgroups of $2$-transitive groups of degree $d$.}}
\label{sgroups}
\end{center}
\end{table}


The Alternating Group $A_d$ is ruled out by Lemma \ref{no3tr}, as it acts $3$-transitively on the sets of tangents.

It is easily seen that $\mathfrak{G}_2$ is not any of the sporadic groups in Table \ref{sgroups} since their degree is different from $q^3+1$.

Moreover, for the same reason, $PSL(d,u),~ d > 2$ has also to be excluded by $\frac{u^d-1}{u-1} \neq q^3+1$. In fact, if $\frac{u^d-1}{u-1} = q^3+1$ were true we would have
$$ u^{d-1}+ u^{d-2}+ \dots + u = q^3.$$
Observe that $u$ and $q$ must be powers of the same prime $p$, and this would lead to a contradiction because the right side of the equation is still a power of $p$ whereas
$$u^{d-1}+ u^{d-2}+ \dots + u = u(u^{d-2} +u^{d-3}+ \dots + 1) \neq p^k, $$ for any $k \in \mathbb{Z}$.
Finally, when $d=2$, a $PSL(2,u)$ is not a subgroup of a $PSU(3,q)$, and this dismiss $PSL(2,8)$.

Before analysing the remaining candidates, take a normal subgroup $M$ of  $\mathfrak{G}_2$.
Then $\mathfrak{G}_2$ acts on $M$ in its natural way:
 \begin{equation*}
    \begin{array}{lccc}
     \varphi_{g}: & M & \longrightarrow & M \\
     & m & \mapsto & m^{g}:= g^{-1}mg,\\
   \end{array}
   \end{equation*}
for all $g \in \mathfrak{G}_2.$\\

\begin{lem}
\label{phi}
$\varphi_{g}=id \Leftrightarrow g =1$	
\end{lem}
\begin{proof}
It is enough to show that $	\varphi_{g}=id \Rightarrow g =1$. From $\varphi_{g}=id$ it follows $g \in C_{\mathfrak{G}_2}(M)$ where $C_{\mathfrak{G}_2}(M)$ is the centraliser of $\mathfrak{G}_2$ in $M$. Suppose by contradiction that $C_{\mathfrak{G}_2}(M)$ is not trivial. Then it must be primitive since $M$ is $2$-transitive (see \cite[Theorem 4.B]{Dixon}). Let $t$ be a tangent line, and consider $c \in C_{\mathfrak{G}_2}(M)$. Then, there exists $\alpha \in \mathfrak{G}_2$ which fixes only $t$. In our case this implies that $t$ is also fixed by $c$, that is,
$$ c(\alpha(t))=c(t),\quad \alpha(c(t))= c(\alpha(t))=c(t).$$
Therefore, every tangent line is fixed by $c$, and this implies that $c$ fixes every point as tangent lines are the maximal cliques, so
$C_{\mathfrak{G}_2}(M)$ is trivial, a contradiction.
\end{proof}

Lemma \ref{phi} shows that $\mathfrak{G}_2$ is a subgroup of $Aut(M)$.

$Sz(u),~Ree(u),~PSp(2d,2)$ cannot be the minimal normal subgroup of $\mathfrak{G}_2$ as their automorphism group does not contain $P\Gamma U(3,q)$, we refer the reader to \cite{Wilson} for all the details.

Our discussion on the groups in Table \ref{sgroups} together with lemma \ref{phi} yield that $M=PSU(3,q)$. Since $\mathfrak{G}_2$ is isomorphic to a subgroup of $Aut(M) \cong P\Gamma U(3,q)$ (see \cite[section 3.6.3]{Wilson}), it follows from Lemma \ref{subgr} that $\mathfrak{G}_2 \cong P\Gamma U(3,q)$.
\subsection{The case $q=2$}\label{subsec722}
In the smallest case $q=2$, consider the Hermitian curve $H(2,4)$ with $q^{3}+1=9$ isotropic points. The graph $NU(3,4)$ has $q^{4}+q^{2}+1-q^{3}+1=12$ vertices. Through an external point $P\in PG(2,4)\setminus H(2,4)$ there are $q+1=3$ tangent lines, and each line is incident with $q^{2}-1=3$ non-isotropic points other than $P$. Hence the graph is 9-regular. Now let $P$ and $Q$ be two non-adjacent vertices, then $PQ$ is a $(q+1)$-secant line. Through $P$ there are $q+1$ tangent lines, and the same number through $Q$, and they meet each other in $(q+1)^{2}$ other common neighbours, i.e. $\mu=9$. Now let $P$ and $Q$ be two adjacent vertices. Then $PQ$ is tangent to $H$ at $T$. On $PQ$ there are $q^{2}-2$ non-isotropic points other than $P$ and $Q$, so that they have at least $q^{2}-2$ common neighbours. Moreover, through $P$ there are $q$ tangents other than $PQ$, and the same number through $Q$, and they meet each other in $q^{2}$ other common neighbours, i.e. $\lambda=6$. The complementary graph $\overline{NU(3,4)}$ has parameters $(v',k',\lambda',\mu')=(12,2,1,0)$. Therefore it is a trivial strongly regular graph with four connected components isomorphic to the complete graph $K_{3}$. Since $Aut(NU(3,4))=Aut(\overline{NU(3,4)})$, it is enough to find the automorphism group of the complementary graph. Observe that $Aut(\overline{NU(3,4)})$ is the wreath product of four copies of the automorphism group of $K_{3}$, which is the dihedral group $D_{3}\cong S_{3}$, by the Symmetric group of degree $4$ acting as a permutation group on the four connected components:
 $$\mathfrak{G}_{2}=Aut(NU(3,4))\cong S_{3} \wr S_{4} \cong S_3^4 \rtimes S_4.$$
  In particular, $|\mathfrak{G}_{2}|=31104.$
  It should be noted that in the graph $\Gamma_{2}$ we make distinction between the two idempotent kinds of maximal cliques, whether we add or do not the fifth point on the Hermitian curve $H(2,4)$, i.e. $\mathfrak{G}_{2}\ncong Aut(\Gamma_{2})\cong P\Gamma U(3,2)$.
  \begin{figure}[h]
    \label{NU34}
    \centering
    \includegraphics[scale=0.5]{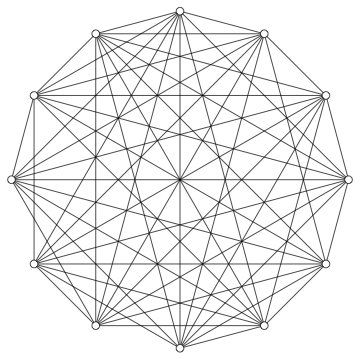}
    \caption{\footnotesize{$NU(3,4)$.}}
   \end{figure}

 \section{Non-classical unitals}\label{sec73}
  As said in Section \ref{sec72}, besides the classical one, there are other known unitals in $PG(2, q^2)$ and all the known unitals of $PG(2, q^2)$ arise from a construction due to Buekenhout: a Buekenhout unital of $PG(2, q^2)$ is either an orthogonal Buekenhout-Metz unital or a Buekenhout-Tits unital. recall that an orthogonal Buekenhout-Metz unital of $PG(2, q^2)$, $q > 2$, is projectively equivalent to one of the following form:
$$\mathcal{U}_{\alpha, \beta} = \{(x, \alpha x^2 + \beta x^{q+1} + z, 1) | z \in F_q, x \in F_{q^2}\} \cup \{(0,1,0)\}, $$
where $\alpha, \beta$ are elements in $GF(q^2)$ such that $(\beta - \beta^q)^2 + 4 \alpha^{q+1}$ is a non-square in $F_q$ if $q$ is odd, or $\beta \notin F_q$ and $\frac{\alpha^{q+1}}{(\beta+\beta^q)^2}$ has absolute trace $0$ if $q$ is even and $q > 2$, while a Buekenhout-Tits unital of $PG(2, q^2)$ is projectively equivalent to
$$\mathcal{U}_{T} = \{(x_0 + x_1 \beta, (x_0^{\delta + 2} + x_0 x_1 + x_1^{\delta}) \beta + z, 1) | x_0, x_1, z \in F\} \cup \{(0,1,0)\},$$
where $\beta$ is an element of $F_{q^{2}} \setminus F_q$ and $\delta = 2^{\frac{m+1}{2}}$. Another class of abstract unitals of order $q=3^{2n+1}$ was discovered by H. L\"{u}neburg in \cite{Ree}. Let $q=3^{2n+1}$, and $Ree(q)$ be the Ree group of order $(q^3+1)q^{3}(q-1)$. The group $Ree(q)$ has a 2-transitive action on $q^3+1$ points, by conjugation on the set of all 3-Sylow. The pointwise stabilizer of two points is cyclic of order $q-1$ and contains a unique involution. $Ree(q)$ has a unique conjugacy class of involutions, and any involution fixes $q+1$ points. The blocks of the Ree unital $R(q)$ are the sets of fixed points of the involutions of $Ree(q)$. For $q\geq27$, $R(q)$ has no embeddings, while all embeddings of $R(3)$ are completely classified, see \cite{Nagy} for all details. Moreover in \cite{Bagchi} the authors constructed examples of cyclic unitals for $a=4,6$, the latter being the first example of unital in which $a$ is not a prime power. Let $E(s)$ be the cyclic group $C_{s}$ if $s$ is prime or the direct product $C_{s_{1}}\times C_{s_{2}}\times\ldots\times C_{s_{n}}$ if $s=s_{1}s_{2}\ldots s_{n}$ with all $s_{i}$'s primes. Let $p,q$ be prime powers such that $p-1$ divides $q-1$, then there exists a Steiner system $S(2,p,pq)$ which admits $E(pq)$ as point-regular automorphism group. Moreover if $p$ and $q$ are distinct primes, the design is cyclic. Unitals of order 4 and 6 arise respectively from $(p,q)=(5,13)$ and $(p,q)=(7,31)$. Finding the automorphism group of non-classical unital is an open, difficult, problem. Theoretical results and computations performed with \textit{GAP} \cite{GAP} gives us results about the explicit factors of $Aut(\mathcal{U})$ for small values of $a$:
\begin{table}[h]
\begin{center}
\begin{tabular}{|c|c|c|}
\hline
$\mathcal{U}$                                   & $a$                  & $Aut(\mathcal{U})$        \\ \hline
Hermitian curve $H(2,q^{2})$                    & 2                    & $P\Gamma U(3,2)$      \\ \hline
Hermitian curve $H(2,q^{2})$                    & $q=p^{h}$            & $P\Gamma U(3,q)$      \\ \hline
Ree unital $R(q)$                               & 3                    & $Ree(3)$      \\ \hline
Buekenhout-Metz orthogonal unital               & 3                    & $(C_{3}\times C_{3}\times C_{3})\rtimes Q_8$, \cite{GAP}         \\ \hline
Buekenhout-Tits unital                          & 3                    & $((C_{4}\times C_{4})\rtimes C_{8})\rtimes C_{6}$, \cite{GAP}   \\ \hline
Bagchi-Bagchi unital                            & 6                    & $C_{7}\rtimes(C_{31}\rtimes C_{30})$, \cite{GAP}   \\ \hline
\end{tabular}
\caption{\footnotesize{Groups acting as automorphism groups on unitals of order $a$.}}
\end{center}
\end{table}

Now let $\Gamma_{\mathcal{U}}$ be the graph whose vertices are the points of the projective plane $\Pi$ not in the unital and two vertices are adjacent if the line joining them is tangent to $\mathcal{U}$. From Proposition \ref{SRGU}, the graph $\Gamma_{\mathcal{U}}$ is strongly regular with parameters $(q^2(q^2-q+1), (q+1)(q^2-1), 2(q^2-1), (q+1)^2)$. If the unital is not embeddable, we need the definition of dual unital defined on the dual projective plane, where the $q^{3}+1$ tangents are considered as points, and the blocks are the pencil of tangent lines from an external point. In the dual unital, we have a one-to-one correspondance between points of the unital and tangents, and between points in $PG(2,q^{2})\setminus H(2,q^{2})$ and secants, given by the unitary polarity. The vertex-set of the dual graph is now the set of all secants, i.e. blocks of the unital. As in the tangent graph two vertices are adjacent if the points lie on the same tangent, in the dual \textit{block graphs} two vertices are adjacent if the secants intersect in an isotropic point, i.e. adjacency between two blocks is equivalent to the intersection in the unital.
\begin{prop}
 Let $\mathcal{U}$ be an unital of order $a$ (classical or not, embeddable or not). The \textit{block graph} $\Gamma_{\mathcal{U}}$ is strongly regular with parameters $(a^2(a^2-a+1), (a+1)(a^2-1), 2(a^2-1), (a+1)^2)$.
\end{prop}

\section{EKR Theorems in unitals}\label{sec74}
 In their seminal paper \cite{EKR}, the authors proved the following theorem.
 \begin{thm}\cite{EKR}
 Choose integers $k$ and $n$ such that $0<2k\leq n$, let $\mathcal{F}$ be a family of subsets
of size $k$ of $\{1,\ldots, n\}$ such that for each $F,F'\in\mathcal{F}$, $F\cup F'\neq\emptyset$. Then
$$\mathcal{F}\leq\binom{n-1}{k-1}.$$
Moreover, if $2k<n$, equality holds if and only if $\mathcal{F}$ is the family of all $k$-sets through a fixed element
of $\{1,\ldots, n\}$.
\end{thm}
Consider in the set of $a^{3}+1$ points of $\mathcal{U}$, the family of $(a+1)$-subsets defined by the blocks. A maximal clique in the block graph is equivalent to an EKR-intersecting family.
\begin{thm}\cite[Theorem 6.4]{DeBoeck}
    \label{DeB}
    Let $\mathcal{U}$ be a unital of order $a$ and let $S$ be a maximal Erd\H{o}s-Ko-Rado set on $\mathcal{U}$.
    \begin{enumerate}
     \item If $a\geq5$ then either $|S|=a^2$ and $S$ is a point-pencil, or else $|S|\leq a^2-a+a^{\frac{2}{3}}-\frac{2}{3}a^{\frac{1}{3}}+1$.\\
     \item If $a=4$ then either $|S|=16=a^2$ and $S$ is a point-pencil, or else $|S|\leq 13=a^2-a+1$.\\
     \item If $a=3$ then either $|S|=9=a^2$ and $S$ is a point-pencil, or else $|S|\leq 8$.\\
    \end{enumerate}
   \end{thm}
   Note that in the dual representation, the point pencil of $q^2$ blocks arises from $q^2$ points on a same tangent. Therefore, Theorem \ref{DeB} is a generalization of Lemma \ref{curve}. As in Lemma \ref{curve}, from this result we get the needed informations on the automorphisms.
\begin{thm}
 Let $\mathcal{U}$ be an unital of order $a>2$. The automorphism group of the block graph $\Gamma_{\mathcal{U}}$ is isomorphic to the automorphism group of $\mathcal{U}$.
\end{thm}
\begin{proof}
 Since the structure of $\Gamma_{\mathcal{U}}$ belongs to the intersections in the unital, each $\lambda\in Aut(\mathcal{U})$ stabilizes the graph. To see the converse, consider $\mu\in Aut(\Gamma_{\mathcal{U}})$. The action of $\mu$ permutes all maximal cliques of size $a^{2}$, preserving incidences in both the graph and the design. From Theorem \ref{DeB}, since $a>2$, maximal cliques of size $a^{2}$ are only the \textit{canonical cliques}, i.e. point-pencils. Since points and canonical cliques are in one-to-one correspondance, the induced action of $\mu$ on the unital is described considering the composition $\mu$ with the bijection $\pi$ between points and point-pencils of blocks.\footnote{The proof of this theorem was found with S. Adriaensen, J. De Beule and F. Romaniello, that we are very grateful to.}
\end{proof}

\appendix
\chapter{Axioms of polar spaces}\label{ap0}
F. Veldkamp in \cite{Vel1, Vel2} introduced the following axioms of polar spaces. A space $S$ is called polar if
\begin{itemize}
 \item[I] $\forall x, y\in S$, $x\leq y$ and $y\leq x \iff x=y$.
 \item[II] $\forall x, y, z\in S$, if $x\leq y$ and $y\leq z$, then $x\leq z$.
 \item[III] Every nonempty set of elements of $S$ has an intersection.
 \item[IV] For every element $a\in S$, the set of all $x\leq a$ is a projective space of finite rank.
 \item[V] Every element of $S$ is contained in some maximal element, all maximal elements have the same (finite) rank which is called the index of $S$.
 \item[VI] If $x$ and $y$ are elements of S with intersection 0, then there exist maximal elements $a$ and $b$ in $S$ with intersection 0 and such that $x\leq a$, $y\leq b$.
 \item[VI] Let $a$ and $b$ be maximal with intersection 0. For every $x\leq a$ there exists one and only one $x^{\delta}\leq b$ such that $x\cup x^{\delta}$ is maximal.
 \item[VIII] There exists a finite number of points such that $S$ is the only flat subset of $S$ that contains all those points.
 \item[IX] If $u$ and $v$ are maximal elements of $S$ and are considered as projective spaces as in IV, then there exists a projectivity of $u$ upon $v$.
 \item[X] If $a\in S$, then the projective geometry $x\leq a$ admits the \textit{Desargues configuration}, see Figure \ref{Desargues}.
\end{itemize}

\begin{figure}[h]
\label{Desargues}
    \centering
    \includegraphics[scale=0.4]{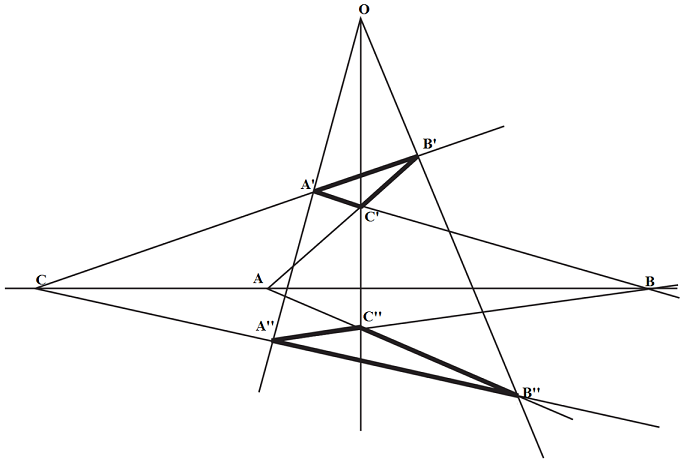}
    \caption{\footnotesize{Desargues Configuration.}}
   \end{figure}

Later the axioms were simplified by J. Tits in \cite[Chapters 7-9]{buildings} as follows. Let start with a set $S$ in which some subsets are called \textit{subspaces}. There is an integer $n\geq1$ called the rank of $S$. $S$ must respect the four following axioms P1)-P4).
\begin{itemize}
 \item[P1)] A subspace $L$, together with the subspaces it contains, is a $d$-dimensional projective space, $-1\leq d\leq n-1$.
 \item[P2)] The intersection of two subspaces is a subspace.
 \item[P3)] Given a subspace $L$ of dimension $n-1$ and a point $p\in S\setminus L$, there exists a unique subspace $M$ containing $p$ and such that $M\cap L$ is of dimension $n-2$ and it contains all points of $L$ collinear with $p$.
 \item[P4)] There exist two disjoint subspaces of dimension $n-1$.
\end{itemize}

From the point of view of incidence geometry, axioms of polar spaces were also given by F. Buekenhout and E. Shult in \cite{BSh} in 1974. In their approach a polar space is an incidence structure of points and lines $(\mathcal{P},\mathcal{L},I)$ with $I\subseteq\mathcal{P}\times\mathcal{L}$ such that
\begin{enumerate}
 \item Any line contains at least three points.
 \item No point is incident with all others.
 \item Any chain of singular subspaces is of finite length.
 \item Let $x\in\mathcal{P}$, and $\ell\in\mathcal{L}$, with $x\not\in\ell$, then $x$ is collinear with one or all points of $\ell$.
\end{enumerate}
For more details on history of polar spaces, we refer the reader to \cite{history}.

\chapter{Maximal curves over finite fields}\label{apA}
We give some definitions and results about $F_{q}$-maximal curves used in Chapter \ref{ch2}. For details and proofs; see  \cite{HKT}. By an \textit{algebraic curve} defined over $F_{q}$ we mean a projective, geometrically irreducible, non-singular algebraic curve $\mathcal{X}$ of $PG(r,q)$ viewed as a curve of $PG(r,\overline{F_{q}})$ where $ \overline{F_{q}} $ is the algebraic closure of $F_{q}$. Each curve  $\mathcal{X}$ is birationally isomorphic to a plane curve $\mathcal{X}$ with only ordinary singularities.

Let $K$ be an algebraic closed field, consider a polynomial in two indeterminate $P\in K[X,Y]$ of polynomial degree $d$. We associate to $P$ the homogeneous polynomial $P^*\in K[X_0,X_1,X_2]$ where $X=\frac{X_1}{X_0}$, $Y=\frac{X_2}{X_0}$ and $P^{*}(X_0,X_1,X_2)=X_{0}^{d}P(\frac{X_1}{X_0},\frac{X_2}{X_0})$.

\begin{defn}
 The \textit{algebraic plane projective curve} of \textit{affine equation} $P(X,Y)=0$, or \textit{homogeneous equation} $P^*(X_0,X_1,X_2)=0$, is the variety
 \begin{equation*}
  \mathcal{X}=v(P)=v(P^*)=\{(x_0,x_1,x_2)\in PG(2,K)|P^*(x_0,x_1,x_2)=0\}.
 \end{equation*}
 The \textit{degree} of $\mathcal{X}$ is the polynomial degree of $P$.
\end{defn}

Let $\mathcal{X}$ a curve defined over the algebraic closure of a finite field $\overline{F_q}$. Points of $\mathcal{X}$ lying on $F_q$ are usually called \textit{$F_q$-rational points}.

\begin{defn}\cite[Proposition 17]{Vaccaro}
 Let $\mathcal{C}$ be an absolutely irreducible plane curve of degree $d$ with only ordinary singularities, if $P_1, \ldots, P_k$ are the singular points of $\mathcal{C}$ and they have multiplicities $r_1,\ldots, r_k$, then the genus $\mathfrak{g}$ of $\mathcal{C}$ is given by
 \begin{equation}
  \mathfrak{g}=\frac{1}{2} (n-1)(n-2)- \frac{1}{2}\sum_{i=1}^k r_i(r_i-1).
 \end{equation}
\end{defn}
The genus of a curve $\mathcal{X}$ is defined by the genus of any plane curve $\mathcal{C}$ birationally isomorphic to $\mathcal{X}$.
\begin{thm}[\textbf{Hasse-Weil bound}]
  Consider an algebraic curve $\mathcal{X}$ defined over $F_{q}$, and let $N_{q}(\mathcal{X})$ be the set of its $F_{q}$-rational points. Then
  \begin{equation}
   |N_{q}(\mathcal{X})-(q+1)|\leq2\mathfrak{g}\sqrt{q}.
  \end{equation}
\end{thm}

An algebraic curve $\mathcal{X}$ defined over $F_{q}$ is \textit{$F_{q}$-maximal} if the number of its $F_{q}$-rational points attains the Hasse-Weil upper bound.
By the \textit{Natural Embedding Theorem}\cite{korchmaros2001embedding}, each $F_{q^2}$-maximal curve is $F_{q^2}$-isomorphic to a degree $q+1$ curve lying in the Hermitian variety of $PG(r,q^2)$ where $r$ is bounded by the Frobenius dimension $d$ of $\mathcal{X}$. Here, the $d$ is defined to be the dimension of a divisor $|(q+1)P_0|$ where $P_0$ is any  $F_{q^2}$-rational point of $\mathcal{X}$.
\begin{thm}[\textbf{Natural Embedding Theorem}]
 \label{NET}
 Every $F_{q^{2}}$ maximal curve $\mathcal{X}$ of genus $g\geq0$ is isomorphic over $F_{q^{2}}$ to a curve of $PG(m,K)$ of degree $q+1$ lying
 on a non-degenerate Hermitian variety defined over $F_{q^{2}}$.
\end{thm}

\chapter{$[n,k]_q$-linear codes and Hamming distance}\label{apB}
In this Appendix we list first properties of linear codes, see for more details \cite{Codici}.
\begin{defn}
 An $[n,k]_{q}$-\textit{linear code} $\mathcal{C}$ is a subspace of $\mathbb{F}_{q}^{n}$ of dimension $k$. The elements of $\mathcal{C}$ are said \textit{codewords}.
\end{defn}
\begin{defn}
 \begin{itemize}
   \item The \textit{Hamming distance} between two codewords $x=(x_{1},\ldots,x_{n})$ and $y=(y_{1},\ldots,y_{n})$ is the number of entries in which $x$ and $y$ differ: $d(x,y)=|\{i|x_{i}\neq y_{i}\}|$.
   \item The \textit{minimum (Hamming) distance} of a code $\mathcal{C}$ is
   $$d=d(\mathcal{C})=min\{d(x,y)|x,y\in\mathcal{C},x\neq y\}.$$
  \end{itemize}
  In this case we say $\mathcal{C}$ is a $[n,k,d]_{q}$-linear code.
\end{defn}
\begin{prop}
 The Hamming distance is a metric, i.e.  $\forall x,y,z\in\mathcal{C}$:
 \begin{enumerate}
   \item $d(x,y)\geq0$ and equality holds if and only if $x=y$;
   \item $d(x,y)=d(y,x)$;
   \item $d(x,y)+d(y,z)\geq d(x,z)$.
 \end{enumerate}
\end{prop}

\begin{proof}
 1 and 2 are quite trivial, we focus on 3. If $d(x,z)=0$, $x=z$ and we have done since $2d(x,y)\geq d(x,x)=0$. If $x\neq z$, say w.l.o.g. $x_j\neq z_j$ for a fixed $1\leq j\leq n$. Then it must occur at least one condition between $x_j\neq y_j$ and $y_j\neq z_j$. If we repeat the argument for all entries $i$ in which $x_i\neq z_i$ we get the thesis.
\end{proof}

\begin{thm}
 \begin{itemize}
  \item Let $\mathcal{C}$ be a $[n,k,d]_{q}$-linear code. Then, $\mathcal{C}$ can correct $\lfloor\frac{d-1}{2}\rfloor$ errors.
  \item If is used for detection, $\mathcal{C}$ can detect $d-1$ errors.
 \end{itemize}
\end{thm}

\begin{proof}
 \begin{itemize}
  \item Consider balls centered in a codeword of radius $\lfloor\frac{d-1}{2}\rfloor$. Since the minimum distance is $d$, we have the thesis, considering the metric properties of Hamming distance.
  \item Trivial, since the minimum distance is $d$, $d-1$ errors cannot send one codeword in another one.
 \end{itemize}
\end{proof}

\begin{defn}
 Let $\mathcal{C}$ be a $[n,k]_{q}$-linear code.
 \begin{itemize}
  \item The \textit{Hamming weight} of a codeword $c$ is the number of non-zero entries of $c$, i.e. $w(c)=d(c,0)$.
  \item The \textit{minimum weight} of a code $\mathcal{C}$ is $w(\mathcal{C})=min\{w(c)|c\in\mathcal{C},c\neq 0\}$.
 \end{itemize}
\end{defn}

The following proposition allows us to study the minimum distance of a code studying weights of codewords.

\begin{prop}
 Let $\mathcal{C}$ be a $[n,k]_{q}$-linear code, then $d(\mathcal{C})=w(\mathcal{C})$.
\end{prop}

\begin{proof}
 Since $d(x,y)=w(x-y)$ and the linear subspaces are closed under addition we have
 $$d(\mathcal{C})=min\{d(x,y)|x,y\in\mathcal{C},x\neq y\}=min\{w(c)|c\in\mathcal{C},c\neq 0\}=w(\mathcal{C}).$$
\end{proof}

\begin{defn}
 The \textit{weight distribution} of a $[n,k]_{q}$-code $\mathcal{C}$ is the $(n+1)$-tuple $(1=W_{\mathcal{C}}(0),W_{\mathcal{C}}(1),\ldots,W_{\mathcal{C}}(n))$.
\end{defn}
\begin{defn}
  The \textit{weight enumerator} of a $[n,k]_{q}$-code $\mathcal{C}$ is the polynomial  $$W_{\mathcal{C}}(X,Y)=\sum_{i=0}^{n}W_{\mathcal{C}}(i)X^{n-i}Y^{i}=X^{n}+W_{\mathcal{C}}(1)X^{n-1}Y+\ldots+W_{\mathcal{C}}(n)Y^{n}.$$
\end{defn}

\begin{defn}
A \textit{two-weight code} is an $[n,k,d]$-linear code $C$ such that $|\{W|\exists v\in C\setminus\{0\}\hspace{2 mm}, w(v)=W\}|=2$, i.e. the weight distribution have $3$ non-zero entries and the weight enumerator have $3$ non-zero coefficients.
\end{defn}

\chapter{Combinatorial designs}\label{apC}
In this section we report some basic notions from design theory. The approach is same as in \cite[Chapter 6]{Mazzocca}. For more details on combinatorial designs; see \cite{design}.
\begin{defn}
 For a finite set $\mathcal{P}$ with $v$ elements, called \textit{points}, together with a family $\mathcal{B}$ of $b$ nonempty subsets of $\mathcal{P}$ of size $k$, called \textit{blocks}, the pair $\mathcal{D}=(\mathcal{P},\mathcal{B})$ is called $t-(v,k,\lambda)$ \textit{design}, or $t$-\textit{design}, or also \textit{design}, if each subset of $\mathcal{P}$ of size $t$ is contained in exactly $\lambda$ blocks. In case $v=b$, it is called a \textit{symmetric design}.
\end{defn}

\begin{defn}
 An \textit{automorphism} of the design $\mathcal{D}$ is a permutation on points which induces a permutation on $\mathcal{B}$.
\end{defn}

\begin{prop}
\label{prop0602}
 Let $\mathcal{D}=(\mathcal{P},\mathcal{B})$ be a $t-(v,k,\lambda)$ design. For each integer $1\geq i\geq t$, let
 $$\lambda_{i}=\lambda\frac{\binom{v-i}{t-i}}{\binom{k-i}{t-i}}.$$
 Then $\mathcal{D}$ be a $i-(v,k,\lambda_i)$ design for each such $i$.
\end{prop}

\begin{proof}
 An equivalent reformulation of the claim is that through each $i$-subset of $\mathcal{P}$ there pass exactly $\lambda_i$ blocks. Fix an $i$-subset $X$, and count twice the pairs $(A,B)$ when $A\subset\mathcal{P}\setminus X$, $|A|=t-i$ and $B$ is a block containing $A\cup X$. For a fixed block $B$ through the $\lambda_i$ blocks containing $X$, there are exactly $\binom{k-i}{t-i}$ way to choose the $t-i$ points of $A$ in $B\setminus X$, so that the number of pairs $(A,B)$ equals $\nu=\lambda_i\binom{k-i}{t-i}$. On the other hand we have $\binom{v-i}{t-i}$ choices for $A$. Since $|A\cup X|=t$, $B$ may be chosen in $\lambda$ ways, since $\mathcal{D}$ is a $t-(v,k,\lambda)$ design. So $\nu=\lambda\binom{v-i}{t-i}=\lambda_i\binom{k-i}{t-i}$, and this proves the claim.
\end{proof}

A corollary of Proposition \ref{prop0602} is that the number of blocks equals $b=\lambda_0=\lambda\frac{\binom{v}{t}}{\binom{k}{t}}$. Moreover through each point there pass exactly $r=\lambda_{1}=\lambda\frac{\binom{v-1}{t-1}}{\binom{k-1}{t-1}}$ blocks. Also, $\lambda_t=\lambda$. Therefore, any $t-(v,k,\lambda)$ design is a \textit{tactical configuration}, i.e. an incidence structure with $v$ points and $b$ blocks, with $k$ points on each block and $r$ blocks through each point. A straightforward double counting argument gives the following identity for tactical configurations.

\begin{cor}[\textbf{Tactical configuration}]
 $$vr=bk.$$
\end{cor}

\begin{exmp}
 Points and $d$-subspaces of $PG(n,q)$ forms a 2-design with $b=\begin{bmatrix}
         n+1 \\
         d+1\\
        \end{bmatrix}_{q}$, whose parameters are
 $$v=\begin{bmatrix}
         n+1 \\
         1 \\
        \end{bmatrix}_{q}=q^{n}+q^{n-1}+\ldots+1=\theta_{n},$$
 $$k=\begin{bmatrix}
         d+1 \\
         1 \\
        \end{bmatrix}_{q}=\theta_{d},$$
 $$\lambda= \begin{cases}
       \begin{bmatrix}
         n-2 \\
         d-2\\
        \end{bmatrix}_{q}=\frac{(q^{n-1}-1)(q^{n-2}-1)\ldots(q^{n-d+1}-1)}{(q^{d-1}-1)(q^{d-2}-1)\ldots(q-1)} & \mbox{ if } d>1, \\
        \hspace{2cm} 1 & \mbox{ if } d=1.
       \end{cases}$$
 The design is denoted by $PG_d(n,q)$, and by $PG(n,q)$ for $d=1$. Its automorphism group is $P\Gamma L(n+1,q)$ for all values of $d$. $PG_d(n,q)$ is symmetric if and only if $d=n-1$.
\end{exmp}

 \begin{exmp}
  In the previous construction, the Fano plane $PG(2,2)$, see Figure \ref{Fano}, is a $2-(7,3,1)$ design.
 \end{exmp}

\begin{defn}
 Let $a$ be a positive integer, a $2-(a^3+1,a+1,1)$ design is called \textit{unital}.
\end{defn}

\begin{prop}
 Points and intersections with secants of a Hermitian curve $H(2,q^{2})$ define an unital.
\end{prop}

\begin{proof}
 $H(2,q^{2})$ has $q^{3}+1$ points, and the secant lines meet the curve in $q+1$ points. Since lines of $PG(2,q^{2})$ are either tangent or secant to $H(2,q^2)$, through a pair of isotropic points there pass exactly one secant.
\end{proof}

\listoffigures
\listoftables

\cleardoublepage

\end{document}